\documentclass[aos,authoryear]{imsart}

\RequirePackage{amsthm,amsmath,amsfonts,amssymb, mathtools}
\RequirePackage[numbers,sort&compress]{natbib}
\RequirePackage[colorlinks,citecolor=blue,urlcolor=blue]{hyperref}
\RequirePackage{graphicx}
\RequirePackage[T1]{fontenc}
\RequirePackage{newtxtext,newtxmath}
\RequirePackage{hyperref}   % load before cleveref!
\RequirePackage[capitalize,nameinlink,noabbrev]{cleveref}
\RequirePackage{enumitem}
\RequirePackage{tikz}
\usetikzlibrary{arrows}
\usetikzlibrary{shapes.geometric}

\DeclareMathOperator*{\argmin}{arg\,min}

\DeclareMathOperator{\tr}{tr}
\newcommand{\trs}[1]{\operatorname{tr}(#1)}    
\startlocaldefs
\theoremstyle{plain}

\newtheorem{theorem}{Theorem}[section]
\newtheorem{lemma}[theorem]{Lemma}
\newtheorem{proposition}[theorem]{Proposition}

\newtheorem*{assumptions}{Assumptions}

\theoremstyle{definition}

\newtheorem{remark}[theorem]{Remark}
\endlocaldefs

\begin{document}

\begin{frontmatter}
\title{Implicit vs. explicit regularization for high-dimensional gradient descent}
%\title{A sample article title with some additional note\thanksref{t1}}
\runtitle{Implict vs. Explicit Regularization}
\runauthor{T. Stark and L. Steinberger}
%\thankstext{T1}{A sample additional note to the title.}

\begin{aug}
%%%%%%%%%%%%%%%%%%%%%%%%%%%%%%%%%%%%%%%%%%%%%%%
%% Only one address is permitted per author. %%
%% Only division, organization and e-mail is %%
%% included in the address.                  %%
%% Additional information such as            %%
%% identifying the corresponding author must %%
%% be included in in the Acknowledgments     %%
%% section if necessary.                     %%
\author[A]{\fnms{Thomas}~\snm{Stark} \ead[label=e1]{thomas.stark@math.au.dk}\orcid{0009-0001-8154-1129}}\and
\author[B]{\fnms{Lukas}~\snm{Steinberger}\ead[label=e2]{lukas.steinberger@univie.ac.at}\orcid{0000-0002-2376-114X}}
%%%%%%%%%%%%%%%%%%%%%%%%%%%%%%%%%%%%%%%%%%%%%%
%% Addresses                                %%
%%%%%%%%%%%%%%%%%%%%%%%%%%%%%%%%%%%%%%%%%%%%%%
\address[A]{Department of Mathematics,
Aarhus University\printead[presep={ ,\ }]{e1}}

\address[B]{Department of Statistics and Operations Research, University of Vienna \printead[presep={,\ }]{e2}}
\end{aug}

\begin{abstract}
In this paper we investigate the generalization error of gradient descent (GD) applied to an $\ell_2$-regularized OLS objective function in the linear model. Based on our analysis we develop new methodology for computationally tractable and statistically efficient linear prediction in a high-dimensional and massive data scenario (large-$n$, large-$p$). Our results are based on the surprising observation that the generalization error of optimally tuned regularized gradient descent approaches that of an optimal benchmark procedure \emph{monotonically} in the iteration number $t$. On the other hand standard GD for OLS (without explicit regularization) can achieve the benchmark only in degenerate cases. This shows that (optimal) explicit regularization can be nearly statistically efficient (for large $t$) whereas implicit regularization by (optimal) early stopping can not. 

To complete our methodology, we provide a fully data driven and computationally tractable choice of the $\ell_2$ regularization parameter $\lambda$ that is computationally cheaper than cross-validation. On this way, we follow and extend ideas of \citet{Dicker2014} to the non-gaussian case, which requires new results on high-dimensional sample covariance matrices that might be of independent interest.
\end{abstract}



\end{frontmatter}
%%%%%%%%%%%%%%%%%%%%%%%%%%%%%%%%%%%%%%%%%%%%%%
%% Please use \tableofcontents for articles %%
%% with 50 pages and more                   %%
%%%%%%%%%%%%%%%%%%%%%%%%%%%%%%%%%%%%%%%%%%%%%%
%\tableofcontents

\section{Introduction}

A common observation in applications of modern high-dimensional machine learning methods is the fact that terminating a learning algorithm early often leads to better generalization performance than running the algorithm until convergence \citep[cf.][]{Bengio2016, HasTibFri2009}. The benefits of \emph{early stopping} have attracted substantial theoretical interest, and a number of optimal, data-driven stopping rules have been devised \citep[see, for example,][]{HuckRei2024, BlaHofRei2018, zhang2005boosting, SonLokReb2024, hu2022early}. Intuitively, particularly in overparameterized settings, \emph{early stopping} acts as a form of implicit regularization by preventing the algorithm from overfitting the data.
More formally, in a linear model, one can observe that the bias of iterates of, for example, full-batch gradient descent (GD) for solving the least squares problem decreases with increasing iteration number, while the variance increases. This interplay results in the typical U-shaped behavior of the quadratic generalization error (see the blue curve in Figure~\ref{fig:main}).

The common intuition that overfitting or even interpolation of the training data will have detrimental effects on generalization performance -- and therefore has to be avoided, for instance, by early stopping -- has recently been challenged by a rapidly growing literature on benign overfitting \citep[see, for example,][]{Belkin2019, Oravkin2021, Hastie2022, Bartlett2020}. Here, we do not follow this intriguing line of research, which, to some extent, is in opposition with the idea of early stopping, but we rather study a scenario of dense signals in linear data generating models, where $\ell_2$-regularized least squares regression (aka. ridge regression) with a certain non-vanishing regularization parameter $\lambda^*>0$ is provably optimal in terms of the generalization risk \citep[see, for instance,][]{AliKolTib2019}. Hence, we consider a scenario in which the natural benchmark procedure is not interpolating. Our goal is to develop methodology that is both computationally feasible in high-dimensional, large-scale settings (large $n$, large $p$) and statistically efficient in the sense of approaching the generalization performance of the benchmark procedure.

Another reason why early stopping of GD can be understood as a kind of \emph{implicit regularization} is the fact that for an appropriate choice of iteration number its risk is very close to that of the optimal benchmark, which in our setting is explicitly $\ell_2$ regularized ridge regression. More precisely, \citet{AliKolTib2019} showed that for certain dense signals in a linear model the generalization risk of gradient descent applied to ordinary least squares when stopped at an appropriate iteration is at most $1.69$ times that of ridge regression.
However, aside from the practical difficulty of implementing this theoretical stopping rule, it can also be shown (see Proposition~\ref{corTh1} below) that the generalization error of GD-OLS, except in trivial cases, never reaches that of the benchmark procedure. Moreover, an exact implementation of the ridge regression benchmark procedure -- e.g., via LU decomposition combined with cross-validation for tuning parameter selection -- is often computationally prohibitive in large-$n$, large-$p$ settings, which is precisely why iterative algorithms are employed in the first place.

In this paper, we move from implicit to explicit regularization to develop a computationally tractable iterative algorithm that provably approximates the benchmark generalization error to arbitrary precision. Our approach relies on the surprising phenomenon that the typical U-shape of the generalization error of GD-OLS as a function of the iteration number disappears when there is an appropriately chosen explicit $\ell_2$-regularization term included in the objective function being minimized by GD (see the green curve in Figure~\ref{fig:main}). Thus, we shift the problem from finding an optimal data driven stopping rule to the optimal (and computationally tractable) selection of the regularization parameter. Moreover, the generalization error of the proposed procedure monotonically approaches that of optimal ridge regression in the large iteration limit, thereby making early stopping unnecessary and statistically inefficient. In other words, we provide a formal argument for the naturally appealing intuition that increasing computation time should also lead to increased accuracy of a learning algorithm.

\begin{figure}
	\centering
	\includegraphics[scale=0.55]{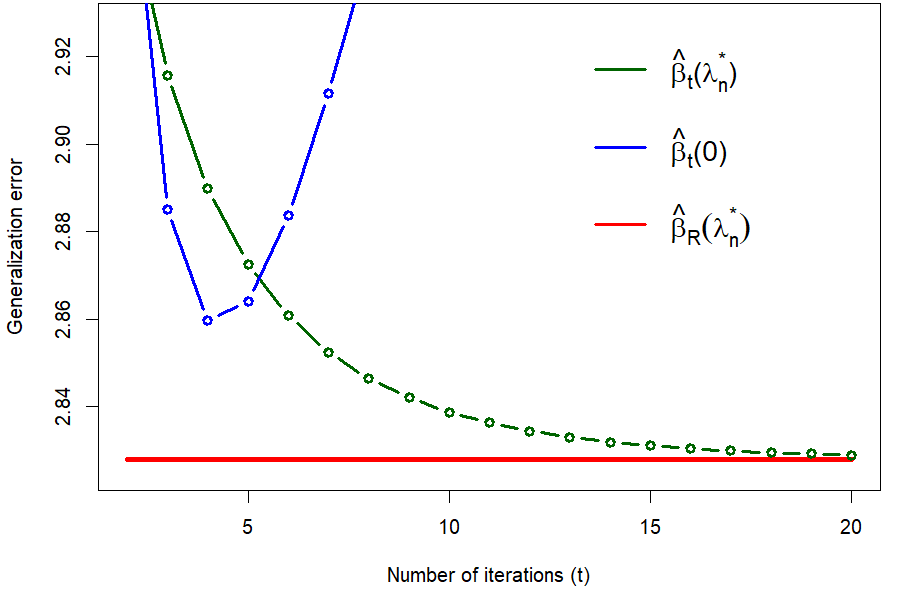}
	\caption{
			Generalization errors of different estimators plotted against the number of iterations $t$ from 1000 Monte-Carlo runs. 
			\textbf{Red}: Generalization error of the \emph{optimally-tuned} ridge estimator. 
			\textbf{Blue}: Generalization error of the gradient descent estimator with no penalization parameter ($\lambda = 0$). 
			\textbf{Green}: Generalization error of the \emph{optimally-tuned} gradient descent estimator ($\lambda_n^* = \frac{\sigma^2}{\tau^2} \frac{p}{n}$). 
			The simulation was performed for $\beta \sim \mathcal{N}(0, \tau^2 I_p)$ with $\tau^2 = 2$, $\sigma^2 = 4$, $p = 1000$, $n = 500$, and $\lambda_n^* = 1$. 
			The entries of $X$ are i.i.d.\ standard normal. 
			The step-size is set as $\hat{\eta}_n(\lambda) = (\hat{s}_1 + \lambda)^{-1}$ for some $\lambda \ge 0$, where $\hat{s}_1$ is an estimate of the largest eigenvalue of $n^{-1}X^\top X $ obtained via power iterations.}
	\label{fig:main}
\end{figure}

\subsection{Our contributions}

In this paper, we investigate the generalization risk of constant step-size, full-batch regularized gradient descent (RGD), that is, gradient descent applied to the $\ell_2$-regularized objective function
\begin{align}\label{eq:obj-f}
	L(b) = \frac{1}{2n}\|y-Xb\|_2^2 + \frac{\lambda}{2}\|b\|_2^2, \quad \lambda \ge 0,
\end{align}
for independent data $(x_i,y_i)_{i=1}^n$ following the standard linear model
\begin{equation*}
	y_i = \beta^\top x_i + u_i, \quad i=1,\dots,n,
\end{equation*}
with $\mathbb{E}(u_i)=0$ and $\mathbb{E}(u_i^2)=\sigma^2$ for all $i \in \{1,\dots,n\}$.
Notice that we analyze the actual numerical iterates of RGD directly, rather than a continuous-time approximation, which typically conflates the effects of step-size (learning rate) and iteration number \citep{AliDobTib2020,Lolas2020,HuckReiSta2025}. 
Accordingly, our methods are fully data-driven and can be implemented without additional tuning.

We focus on the scenario where the true signal $\beta\in \mathbb{R}^p$ is not correlated with an extreme eigenvector of the covariance matrix of the features, which we formalize using a common \emph{random effects assumption} which states that $\mathbb{E}(\beta)=0$ and $\mathbb{E}(\beta\beta^\top)=\tau^2p^{-1} I_p$, for some $\tau>0$. See Section~\ref{sec:RandEff} for more context on this assumption. 

In particular, our contributions are the following.
\begin{itemize}
	\item We provide a precise finite-sample analysis of the (out-of-sample) generalization error of RGD (cf. Section~\ref{sec:GenError}).
	
	\item In particular, we show that the generalization error of RGD decreases monotonically in the iteration index $t$ whenever the tuning parameter $\lambda$ satisfies $\lambda \ge \lambda_n^*$, where the optimal full-ridge choice is $\lambda_n^* \coloneqq \frac{\sigma^{2}}{\tau^{2}}\,\gamma_n$ with $\gamma_n \coloneqq \frac{p}{n}$ (cf. Theorem~\ref{Th1}). Furthermore, for $\lambda = \lambda_n^*$, the generalization error of RGD converges monotonically to that of the optimal ridge benchmark as $t \to \infty$.

	\item Extending results of \citet{Dicker2014} we develop consistent estimators for the error variance $\sigma^2$ and the signal strength $\tau^2=\mathbb{E}(\|\beta\|_2^2)$. In particular, we completely drop the assumption of Gaussian design. This readily leads to an estimator $\hat{\lambda}_n$ for $\lambda_n^*$ that is consistent under mild assumptions on the design distribution and in the full range of $\gamma_n\to \gamma \in(0,\infty)$. The computational bottleneck of this estimation is the computation of $\trs{\hat{\Sigma}_n^2}$, where $\hat{\Sigma}_n\coloneqq n^{-1}X^\top X$ is the sample covariance matrix (cf. Section~\ref{sec:ConsistRand}).
	
	\item We show that the generalization error of RGD tuned with $\hat{\lambda}_n$ is uniformly close (in $t$) to that of RGD tuned with $\lambda_n^*$ as $\gamma_n\to \gamma\in(0,\infty)$. Analogously, we show that ridge regression tuned with $\hat{\lambda}_n$ asymptotically achieves the lower bound of optimally tuned ridge regression (cf. Section~\ref{sec:TunedRGD}).
	
	\item We replace the random effects assumption by a more intuitive deterministic condition on $\beta$, also used by \citet{Dicker2014}, and reprove the consistency result for $\sigma^2$ and $\tau^2 := \|\beta\|_2^2$ without relying on Gaussianity. Consequently, in the setting where $p=o(n^{5/4})$, we provide an extension of the results by \citet{Dicker2014} to non-Gaussian designs. This is achieved through novel approximations of $\tr(\hat{\Sigma}_n^2)$ and $\beta^\top \hat{\Sigma}_n^2 \beta$ in a non-Gaussian setting, which may be of independent interest. (cf. Section~\ref{sec:NoRand}).
	
\item Throughout this work, we consider a setting in which both the training and test data are independent but not necessarily identically distributed. Our assumptions require only certain uniform moment bounds and a common covariance structure among the features in the training and test set. This framework can be interpreted in the context of transfer learning in the absence of a covariance shift (cf. \citet{sur2024}). Along the way, we generalize a result of \citet[Theorem~2.1]{DobWag2018} under substantially weaker distributional assumptions and also provide a version of the classical result of \citet[Theorem~2]{LedPeche2011}.

	%    \item \textcolor{red}{Ledoit and Peche result is also new and interesting?}
\end{itemize}

\subsection{Notation and definitions}

We denote the $p \times p$ identity matrix by $I_p$, the Moore-Penrose pseudoinverse of the symmetric $p\times p$ matrix $A=(a_1,...,a_p)$ with $A^\dagger$, where $a_1,...,a_p$ are the column vectors of $A$. The largest and smallest nonzero eigenvalues of a matrix $A$ are denoted by $s_{\max}(A)$ and $s_{\min}(A)$, respectively. The indicator function on the set B is denoted by $x \mapsto \mathbf{1}_B(x)$. The column space of a matrix $A$ is denoted by $im(A)$ and the nullspace by $\mathcal{N}(A)$. We represent the $A$-seminorm by $\|x\|_A^2=x^\top A x$ for a positive semidefinite matrix $A$ and a vector $x$. We write $\|A\|_2$ for the spectral norm of a symmetric matrix $A$ and denote the $2$-norm of a vector $x$ by $\|x\|_2$. The expression $\|A\|_F$ denotes the Frobenius norm, $x\land y=\min\{x,y\}$ and $x\lor y=\max\{x,y\}$ for $x,y\in \mathbb{R}$. For a cumulative distribution function $F$, we write $F(\{x\})=F(x)-F(x^-)$, where $F(x^-)\coloneqq\lim_{t\uparrow x}F(t)$.

 \section{Finite sample properties of regularized gradient descent}
\label{sec:FiniteSample}

\subsection{Definition and deterministic convergence}

For an $n$-dimensional response vector $y$ and an $n\times p$ design matrix $X$ the (full) ridge-estimator $\hat{\beta}_R(\lambda) =(X^\top X+n\lambda I_p)^{-1}X^\top y$ is the unique minimizer of \eqref{eq:obj-f} whenever $\lambda>0$. In the case where $\lambda=0$, we define $\hat{\beta}_R(0) =(X^\top X)^{\dagger}X^\top y$ (i.e., the minimum-norm estimator). This is a reasonable extension, since $\hat{\beta}_R(0)$ minimizes $b\mapsto\|y-Xb\|_2^2$ and $\hat{\beta}_R(0)= (X^\top X)^\dagger X^\top y =\lim_{\lambda \downarrow 0}(X^\top X+n\lambda I_p)^{-1}X^\top y $. For instance, using the LU decomposition, computing the ridge regression solution requires $O(np \,(n \wedge p))$ floating-point operations (flops).
Here, the minimum $n\land p$ arises because, when $p>n$, a practitioner would typically consider the dual representation of the ridge estimator $\hat{\beta}_R(\lambda)=X^\top (XX^\top +n\lambda I_n)^{-1} y$, which follows directly from a rearrangement of the ridge normal equations. Applying gradient descent with a constant step-size $\eta>0$ and initialized at $\hat{\beta}_{0}(\lambda,\eta)=\theta\in \mathbb{R}^p$ to \eqref{eq:obj-f}, the iterations take the following form:

\setlength{\jot}{15pt}
\begin{align}\label{Def_GD}
	\begin{split}
		\hat{\beta}_t(\lambda,\eta) &= \hat{\beta}_{t-1}(\lambda,\eta) - \eta \, \nabla L\big(\hat{\beta}_{t-1}(\lambda,\eta)\big),\\
		\nabla L(\beta) &= \frac{1}{n} \Big( -X^\top (y - X \beta) + n \lambda \beta \Big).
	\end{split}
\end{align}
As seen from \eqref{Def_GD}, each RGD iteration requires $O(np)$ flops. Hence, for $t < n \wedge p$, computing the RGD iterates is computationally more efficient than directly computing the full ridge estimator. We want to point out that the RGD-estimator $\hat{\beta}_{t}(\lambda,\eta)$ depends on three tuning parameters: the number of iterations $t$, the step-size or learning rate $\eta$ and the penalty parameter $\lambda$. Our objective is to find data-driven choices for $t,\eta$ and $\lambda$ that are both computationally tractable and statistically optimal. Statistical optimality, of course, depends on the performance measure and the data-generating process we introduce below. Before addressing these issues, we begin with a purely deterministic result on the RGD iterates. 

\begin{proposition}\label{prop1}
If we initialize $\hat{\beta}_0(\lambda,\eta) = \theta \in \mathbb{R}^p$ and perform gradient descent on \eqref{eq:obj-f} with constant step size $\eta > 0$ and regularization parameter $\lambda \ge 0$, then, for all $t \in \mathbb{N}$, the iterates can be written as
\begin{align}\label{repGD}
	\begin{split}
		\hat{\beta}_t(\lambda,\eta) 
		= \tilde{\beta}_t(\lambda,\eta) + A^t \theta,
	\end{split}
\end{align}
where
$A = A(\lambda,\eta) \coloneqq I_p - \eta (\hat{\Sigma}_n + \lambda I_p)$ and $\tilde{\beta}_t(\lambda,\eta) \coloneqq (I_p - A^t) \hat{\beta}_R(\lambda)$.

\end{proposition}
    \begin{remark}[On the convergence of the iterates]\label{rem:GD}\,
	\begin{enumerate}
		\item Note that, for arbitrary \(\lambda \ge 0\) and \(\eta > 0\), the eigenvalues of $A(\lambda,\eta)$ have the form $a_j=a_j(\lambda,\eta)\coloneqq 1-\eta(s_j+\lambda)$ for $j\in\{1,...,p\}$, where $0\leq s_p\leq...\leq s_1=s_{max}(\hat{\Sigma}_n)$ are the ordered eigenvalues of $\hat{\Sigma}_n=n^{-1}X^\top X$. Assume for the moment that at least $s_1>0$. Then, whenever \(0 < \eta < 2/(s_1+\lambda)\), we have \(|a_j| < 1\) for all \(j \in \{1,\dots,p\}\), and consequently \(a_j^t \to 0\) as \(t \to \infty\).
		In particular, using \eqref{repGD}, for a fixed $\lambda > 0$, the RGD iterates converge to the corresponding ridge estimator as $t \to \infty$; that is, $
		\hat{\beta}_t(\lambda,\eta) \to \hat{\beta}_R(\lambda)$,
		provided that the step-size satisfies $0 < \eta < 2/(s_1 + \lambda)$.

		\item For the case $\lambda=0$, consider the spectral decomposition $X/\sqrt{n}= V\Lambda^{1/2}U^\top$. Note that for $\lambda=0$ we have $a_j(0,\eta)=1-\eta s_j$, if $s_j>0$ and $a_j(0,\eta)=1$, otherwise. If we define a diagonal matrix $B$, with the $j$-th diagonal entry equal to one if $s_j=0$ and zero if $s_j>0$, we then have that $A(0,\eta)^t\to P\coloneqq UBU^\top$ as $t\to\infty$ and $\eta\in(0,2/s_1)$. We notice that $P=I_p-X^\dagger X$ and  $P$ is the orthogonal projector onto $\mathcal{N}(X)$, hence $\hat{\beta}_t(0,\eta)\to\hat{\beta}_R(0)+P\theta$ as $t\to\infty$, $\eta\in(0,2/s_1)$ and  $P \theta =0$ if $\theta\in im(X^\top)$.
		
		\item The step size that minimizes the number of iterations $t$ required for convergence, for a given $\lambda \ge 0$, is obtained by minimizing the spectral radius of the matrix
		$A(\lambda, \eta)$
		with respect to $\eta$ over the admissible interval
		$0 < \eta < 2/(s_1 + \lambda)$.
	Minimizing the spectral radius of $A(\lambda,\eta)$ in $\eta$ is equivalent to
	\[
	\argmin_{0 < \eta < 2/(s_1+\lambda)} \max_{1 \le j \le p} |a_j| = \argmin_{0 < \eta < 2/(s_1+\lambda)} \max \{|a_1|, |a_p|\},
	\]
	and the minimum is attained when $|a_1| = |a_p|$. Two cases arise. If $a_1$ and $a_p$ have the same sign, then $s_1 = s_p$, $a_j = 1 - \eta(s_1 + \lambda)$ for $j=1,\dots,p$, and choosing $\eta_{\mathrm{opt}}(\lambda) = 1/(s_1 + \lambda)=2/(s_1+s_p+2\lambda)$ yields $A(\lambda, \eta_{\mathrm{opt}}(\lambda)) = 0$. If instead $a_1 < 0 < a_p$, the condition $|a_1| = |a_p|$ implies that $-a_1=a_p$, which after straightforward algebra gives
		\[
		\eta_{\mathrm{opt}}(\lambda) = \frac{2}{ s_1 + s_p+2\lambda} \le \frac{2}{s_1 + \lambda},
		\]
		where the inequality is strict whenever at least one of $\lambda$ or $s_p$ is strictly positive.
	\end{enumerate}
\end{remark}

	\subsection{Risk measure}
Consider independent training observations $\{(x_i, y_i)\}_{i=1}^n$ and a test pair $(x_0, y_0)$  independent of the training data $\{(x_i, y_i)\}_{i=1}^n$. For each $i \in \{0,1,\dots,n\}$, the feature vector $x_i$ takes values in $\mathbb{R}^p$ and the response $y_i$ takes values in $\mathbb{R}$. The data is generated according to the linear model
\begin{equation}\label{set-1}
	y_i = x_i^{\top} \beta + u_i, \qquad i \in \{0,1,\dots,n\},
\end{equation}
where $\beta \in \mathbb{R}^p$ is an unknown parameter vector. The noise variables $\{u_i\}_{i=0}^n$ satisfy $\mathbb{E}(u_i) = 0$ and $\mathbb{E}(u_i^2) = \sigma^2$. The test feature vector $x_0$ satisfies $\mathbb{E}(x_0) = 0$ and $\mathbb{E}(x_0 x_0^\top) = \Sigma$, where $\Sigma \in \mathbb{R}^{p \times p}$ is a symmetric positive semidefinite matrix. Moreover, $x_0$ is independent of $\{u_i\}_{i=0}^n$. Stacking together the training observations, let $y = (y_1,\dots,y_n)^\top$ denote the response vector, $X = (x_1,\dots,x_n)^\top$ the design matrix, and $u = (u_1,\dots,u_n)^\top$ the noise vector. In this section on finite-sample properties (Section~\ref{sec:FiniteSample}), we treat the design matrix $X$ as fixed and non-random, and we emphasize this throughout by conditioning on $X$. 
We assess the performance of an estimator  $\hat{\beta}:=\hat{\beta}(X,y)$ based on the training data $(X,y)$ in terms of the (out-of-sample) generalization error

\begin{align*}
	\begin{split}
\mathrm{Risk}_{\mathrm{out}}(\hat{\beta})
&= \mathbb{E}\!\left[
(x_0^{\top}\hat{\beta} - y_0)^{2}
\,\middle|\, X
\right] 
= \mathbb{E}\!\left[
\bigl(x_0^{\top}(\hat{\beta} - \beta) - u_0\bigr)^{2}
\,\middle|\, X
\right] \\
&= \mathbb{E}\!\left[
(\hat{\beta} - \beta)^{\top}\Sigma(\hat{\beta} - \beta)
\,\middle|\, X
\right] + \sigma^{2},
	\end{split}
\end{align*}
where the expectation is taken with respect to all random quantities, conditional on $X$.
Since the irreducible error term $\sigma^2$ does not depend on $\hat{\beta}$, we will analyze 
\begin{equation}\label{pred_risk}	
	\mathrm{R}_{\Sigma}(\hat{\beta})\coloneqq\mathbb{E}((\hat{\beta}-\beta)^\top\Sigma(\hat{\beta}-\beta)|X),
\end{equation}  
which captures the part of the generalization error attributable to estimating $\beta$. In the rest of the paper, we will refer to $\mathrm{R}_{\Sigma}(\hat{\beta})$ as the generalization error, keeping in mind that it excludes the irreducible error term $\sigma^2$.

\subsection{The random effects assumption}			
\label{sec:RandEff}
In the main part of our work, we analyze the quantity $\mathrm{R}_{\Sigma}(\hat{\beta})$ under the \emph{random-effects} assumption on the coefficient vector $\beta$, which has become standard in the high-dimensional learning literature; see, for example, \citet{DobWag2018, AliKolTib2019, Hastie2022, dobriban2020wonder}. This assumption states that the unknown signal $\beta = (b_1, \dots, b_p)^\top$ is random, $\beta$, 
$X$ and $u$ are jointly independent, and $\beta$ follows an isotropic prior distribution:
\begin{equation}
	\mathbb{E}(\beta) = 0, \quad \mathbb{E}(\beta\beta^\top) = \frac{\tau^2}{p} I_p
	\tag{A*}  \label{set-2}
\end{equation}
where $\tau^2 = \mathbb{E}(\|\beta\|_2^2)$ can be interpreted as the expected signal strength. 
\setcounter{equation}{6}

The significance of this condition in the literature seems to be somewhat ambiguous. For example, it is used in \citet{DobWag2018} and \citet{Hastie2022} -- and a deterministic version of it in \citet{Dicker2014} (see also Section~\ref{sec:tuning}, below) -- as a technical aid to analyze convergence of quadratic forms $\beta^\top M \beta$, which under Assumption~\eqref{set-2} behave like $\tau^2p^{-1}\tr(M)$ in expectation. We also run into this kind of challenge here. Essentially, what is required, both in our work and in the cited literature, is to avoid situations where the true signal $\beta$ is strongly aligned with the eigenvectors corresponding to the extreme eigenvalues of $\hat{\Sigma}_n$. Indeed, the boundedness and monotonicity results in Proposition~\ref{corTh1} and Theorem~\ref{Th1} no longer hold if, for example, $\beta$ is parallel to the leading eigenvector of $\hat{\Sigma}_n$. Extensions addressing such scenarios will be considered and discussed in future work.

The random effects assumption is also crucial in \citet{AliKolTib2019} to relate the Bayes risk of ridge regression to that of \emph{gradient-flow}. In particular, under Assumption~\eqref{set-2} on $\beta$, ridge regression with the optimal tuning parameter $\lambda_n^* = \frac{\sigma^2}{\tau^2}\gamma_n$ constitutes a Bayes estimator. Consequently, its Bayes risk provides a lower bound for the Bayes risk of any other estimator of $\beta$.
\citep[cf. the proof of Theorem~3 in][]{AliKolTib2019}. This fact makes optimally tuned ridge regression a natural benchmark in the present setting. We also use the generalization error of optimally tuned ridge regression as a benchmark for the generalization error of RGD.

Finally, we point out that the random-effects assumption can also be seen as quantifying the size of a set $B \subseteq \mathbb{R}^p$ of favorable signals $\beta$ for which our results hold. Suppose, for example, that Assumption~\eqref{set-2} allows us to establish the asymptotic negligibility of a remainder term $R_n(X,\beta,u)$, that is,
$R_n(X,\beta,u) = o_{\mathbb{P}}(1)$ as $n \to \infty$.
As a concrete example, let $\hat{\lambda}_n = \hat{\lambda}_n(X,\beta,u)$ be a consistent estimator of $\lambda_n^*$; in this case, one may take
\(
R_n(X,\beta,u) = |\hat{\lambda}_n - \lambda_n^*|
\)
(see Subsection~\ref{sec:TunedRGD} below). %Another relevant example is
%\[
%R_n(X,\beta,u) = \|\hat{\beta}_R(\hat{\lambda}_n) - \beta\|_\Sigma^2 %- \|\hat{\beta}_R(\lambda_n^*) - \beta\|_\Sigma^2.
%\]
Let $\nu$ denote the marginal distribution of $\beta$. For some $\varepsilon>0$ and $\delta>0$, define
\[
B := \bigl\{ b \in \mathbb{R}^p : \mathbb{P}(R_n(X,b,u)>\varepsilon) < \delta \bigr\},
\]
the set of all deterministic signals for which the remainder term is small with high probability. Now, Markov's inequality yields that this set is large in terms of the measure $\nu$, that is,
\[
\nu(B^c) \le \frac{1}{\delta}\, \mathbb{P}(R_n(X,\beta,u)>\varepsilon) \longrightarrow 0\quad\text{as }n\to\infty.
\]
In other words, for most deterministic signals $\beta$, the remainder term is small with high probability. Furthermore, since $\mathrm{R}_{\Sigma_n}(\hat{\beta})$ involves the true signal $\beta$ only through quadratic forms (see \eqref{pred_risk}), one could also quantify sets of favorable $\beta$'s using concentration inequalities for quadratic forms. Consider, for example,
\[
R_n(X,\beta,u) = \|\hat{\beta}(\hat{\lambda}_n) - \beta\|_{\Sigma_n}^2 - \|\hat{\beta}(\lambda_n^*) - \beta\|_{\Sigma_n}^2,
\]
where $\hat{\beta}$ is an estimator tuned with $\hat{\lambda}_n$ and $\lambda_n^*$
(cf. Theorem~\ref{Ridgetuned} and Theorem~\ref{GDtuned} below).
Such an approach would even allow for a finite-sample analysis. Since these are technical, but conceptually straightforward, alternative views on the random-effects assumption, we do not include the details here.

\subsection{Generalization error of RGD}	
\label{sec:GenError}

In this section we present our first main result on the generalization properties of RGD. Among other things, it establishes that, for a suitable choice of the regularization parameter $\lambda$ and step-size $\eta$, the generalization error of the RGD estimator decreases monotonically with the number of iterations $t$. Related results were also obtained by \citet[Corollary~3]{Lolas2020} and \citet[Proposition~3.11]{HuckReiSta2025} in the idealized setting of \emph{gradient-flow}. Moreover, the monotonicity appears to contradict common intuition about the benefits of early stopping, which is motivated by the reasoning that the bias is decreasing in $t$ while the variance is increasing in $t$, leading to the characteristic U-shaped behaviour of the generalization error. This apparent contradiction is resolved as follows. Initializing the RGD procedure with $\hat{\beta}_0(\lambda,\eta)=\theta\in\mathbb{R}^p$ 
and using the representation of $\hat{\beta}_t(\lambda,\eta)$ from Proposition~\ref{prop1}, under the data model \eqref{set-1} and Assumption~\eqref{set-2}, we obtain,
\begin{align}\label{decA1}
	\begin{split}
		\mathrm{R}_\Sigma(\hat{\beta}_t(\lambda,\eta)) &= \mathrm{R}_\Sigma(\tilde{\beta}_t(\lambda,\eta)) + \mathbb{E}(\theta^\top A^t \Sigma A^t \theta \mid X)\\[2mm]
		&= \mathbb{E}\!\left[
		\bigl\|
		\mathbb{E}\!\left[\tilde{\beta}_t(\lambda,\eta) \mid X, \beta\right] - \beta
		\bigr\|_{\Sigma}^{2}
		\,\middle|\, X
		\right]\\[1mm]
		&\quad + \operatorname{tr}\!\left\{
		\Sigma \,
		\mathbb{E}\!\left[
		\operatorname{Cov}\!\left(\tilde{\beta}_t(\lambda,\eta) \mid X, \beta\right)
		\,\middle|\, X
		\right]
		\right\} + \mathbb{E}(\theta^\top A^t \Sigma A^t \theta \mid X)\\[1mm]
		&= B^2_\Sigma(\lambda, \eta, t) + V_\Sigma(\lambda, \eta, t) + \mathbb{E}(\theta^\top A^t \Sigma A^t \theta \mid X).
	\end{split}
\end{align}
In this decomposition, \(B_\Sigma(\lambda,\eta,t)\) and \(V_\Sigma(\lambda,\eta,t)\) are referred to as the bias and variance components of the generalization error associated with the zero-initialized estimator \(\tilde{\beta}_t(\lambda,\eta)\), while the remaining term in the last line of \eqref{decA1} captures the additional contribution due to an arbitrary initialization \(\theta\). It can be shown that, for a suitably chosen step size \(\eta\) and for every \(\lambda \ge 0\), the squared bias \(B_\Sigma^2(\lambda,\eta,t)\) is monotonically decreasing in \(t\), whereas the variance term \(V_\Sigma(\lambda,\eta,t)\) is monotonically increasing in \(t\). Nevertheless, the sum \(B_\Sigma(\lambda,\eta,t) + V_\Sigma(\lambda,\eta,t)\) is monotonically decreasing in \(t\) when RGD operates in the over-regularized regime, namely when \(\lambda \ge \lambda_n^* = \frac{\sigma^2}{\tau^2}\gamma_n\) (cf. Theorem~\ref{Th1}\ref{Th1.i} below). In this regime, the bias term decreases more rapidly than the variance term increases as \(t\) grows. In order to ensure that the additional contribution induced by a general initialization \(\theta\) of \(\hat{\beta}_t(\lambda,\eta)\) in \eqref{decA1} also decreases monotonically in \(t\), we require an assumption on \(\theta\) analogous to the random effects Assumption~\eqref{set-2} on \(\beta\). Under zero initialization, this term is zero for all \(t\).

Throughout the chapter, we omit the dependence of $\lambda^*_n$ on $n$ for the sake of readability. Before stating the main theorem of this section, we present a lemma that plays a central role in the proof of Proposition~\ref{corTh1} and Theorem~\ref{Th1}.

\begin{lemma}\label{LemforTh1}
Under the data model \eqref{set-1} and Assumption~\eqref{set-2}, and using the notation of Remark~\ref{rem:GD} with $\gamma_n = p/n$, if we initialize $\hat{\beta}_{0}(\lambda,\eta) = 0$, then for every $t\in \mathbb{N}$, $\lambda \ge 0$ and $\eta>0$ the following holds: 
	\begin{enumerate}[label=(\alph*)]
		\item \label{dec1} The generalization error of the RGD estimator can be expressed as
		\begin{align*}
			\mathrm{R}_\Sigma(\hat{\beta}_{t}(\lambda,\eta)) &= \tr(\Sigma E_t), 
		\end{align*}
		where
		\begin{align*}
			E_t=E_t(\lambda,\eta) = \frac{\tau^2}{p} \Bigl( P + (\hat{\Sigma}_n + \lambda I_p)^\dagger (\lambda I_p + A^t \hat{\Sigma}_n)^2 \Bigr)
			+ \frac{\sigma^2}{n} \Bigl( (\hat{\Sigma}_n + \lambda I_p)^\dagger \Bigr)^2 (I_p - A^t)^2 \hat{\Sigma}_n,
		\end{align*}
		and $P = P(\lambda) = I_p - (\hat{\Sigma}_n + \lambda I_p)^\dagger (\hat{\Sigma}_n + \lambda I_p)$.
		
		The eigenvalues $\{e_i(t,\lambda,\eta)\}_{i=1}^p$ of $E_t$ are given by
		\begin{align*}
			e_i(t,\lambda,\eta)
			= \begin{cases}
				\displaystyle \frac{1}{p}\frac{\sigma^2\gamma_n}{s_i+\lambda^*}
				+
				\frac{1}{p} \frac{s_i}{(s_i+\lambda)^2}
				\biggl(
				\frac{\bigl(\frac{\lambda}{\lambda^*}-1\bigr)\sigma^2\gamma_n}
				{\sqrt{\tau^2 (s_i+\lambda^*)}}
				\;+\;
				\sqrt{\tau^2 (s_i+\lambda^*)}\, a_i^t
				\biggr)^{\!2}, & s_i > 0,\\[1.5mm]
				\displaystyle \frac{\tau^2}{p}, & s_i = 0.
			\end{cases}
		\end{align*}
		\item\label{dec2} The generalization error of the ridge estimator can be expressed as
		\begin{align*}
			\mathrm{R}_\Sigma(\hat{\beta}_{R}(\lambda)) = \tr(\Sigma F),
		\end{align*}
		where
		\begin{align*}
			F=F(\lambda) = \frac{\tau^2}{p} \bigl( P + \lambda (\hat{\Sigma}_n + \lambda I_p)^\dagger \bigr)^2 
			+ \frac{\sigma^2}{n} \bigl( (\hat{\Sigma}_n + \lambda I_p)^\dagger \bigr)^2 \hat{\Sigma}_n,
		\end{align*}
		and $P = P(\lambda) = I_p - (\hat{\Sigma}_n + \lambda I_p)^\dagger (\hat{\Sigma}_n + \lambda I_p)$.
		
		The eigenvalues $\{f_i\}_{i=1}^p$ of $F$ are given by
		\begin{align*}
			f_i=f_i(\lambda) =
			\begin{cases}
				\displaystyle \frac{1}{p} \frac{\sigma^2 \gamma_n}{s_i + \lambda^*} 
				+ \frac{1}{p} \frac{s_i}{(s_i + \lambda)^2} 
				\frac{\bigl(\frac{\lambda}{\lambda^*} - 1\bigr)^2 (\sigma^2 \gamma_n)^2}{\tau^2 (s_i + \lambda^*)}, & s_i > 0,\\[1.5mm]
				\displaystyle \frac{\tau^2}{p}, & s_i = 0.
			\end{cases}
		\end{align*}
	\end{enumerate}
\end{lemma}

\begin{proposition}\label{corTh1}
	Consider the data model \eqref{set-1} and Assumption~\eqref{set-2}, and using the notation of Remark~\ref{rem:GD}, suppose we initialize 
	$\hat{\beta}_0(\lambda,\eta) = \theta \in \mathbb{R}^p$. Then, for all $t \in \mathbb{N}$, $\lambda \ge 0$, and $\eta > 0$, it holds that
	\begin{align*}
		\mathrm{R}_\Sigma(\hat{\beta}_t(\lambda,\eta)) \ge \mathrm{R}_\Sigma(\hat{\beta}_R(\lambda^*)).
	\end{align*}
	
	In the special case $\lambda = 0$, the inequality is strict provided that 
	$\hat{\Sigma}_n$ has at least two positive and distinct eigenvalues, 
	$s_j, s_k > 0$ with $s_j \neq s_k$, and that the corresponding diagonal entries of
	$(U^\top \Sigma U)_{jj}$ and $(U^\top \Sigma U)_{kk}$ are strictly positive, 
	where $U$ is an orthogonal matrix whose columns are the normalized eigenvectors of $\hat{\Sigma}_n$.
\end{proposition}

	\begin{theorem}\label{Th1}
	Consider the same assumptions as in Proposition~\ref{corTh1}, together with the additional assumption that $\theta$ is random, $\theta$ is conditionally independent of $\beta$ and $u$ given $X$, and satisfies $\mathbb{E}(\theta\theta^\top|X)=\rho I_p$, for some $\rho\geq 0$. Then it holds that 		
	\begin{enumerate}[label=(\alph*), itemsep=0.7em, leftmargin=*]
		\item\label{Th1.i}  $\mathrm{R}_\Sigma(\hat{\beta}_{t}(\lambda,\eta))$ is monotonically decreasing in $t$, if $\lambda \geq \lambda^*$ and $\eta \in \left(0,1/(s_{1}+\lambda)\right]$,
		
		\item \label{Th1.ii} $\mathrm{R}_\Sigma(\hat{\beta}_{t}(\lambda,\eta)) \to \mathrm{R}_\Sigma(\hat{\beta}_R(\lambda)) $ for $t\to \infty$, if $\lambda \geq 0$ and $\eta \in (0,2/(s_{1}+\lambda))$.

	\end{enumerate}
	
\end{theorem}

By Proposition~\ref{corTh1}, we observe that ---regardless of the design $X$ and the choice of tuning parameters $t$, $\lambda$ and  $\eta$--- 
$\mathrm{R}_\Sigma(\hat{\beta}_R(\lambda^*))$ serves as a natural benchmark for the performance of RGD, including the unregularized case $\lambda=0$.\footnote{In the model \eqref{set-1} and under the random effects Assumption~\eqref{set-2}, \citet{AliKolTib2019} proved the even stronger statement that $\hat{\beta}_R(\lambda^*)$ minimizes $\mathrm{R}_\Sigma(\hat{\beta})$ over all measurable estimators $\hat{\beta} = \hat{\beta}(X,y)$; see \citet[Theorem~3]{AliKolTib2019}.}
Moreover, the second statement of Proposition~\ref{corTh1} implies that, when $\Sigma$ is positive definite, the generalization error of unregularized gradient descent, $\mathrm{R}_\Sigma(\hat{\beta}_t(0,\eta))$, can never exactly attain this lower bound, except in two trivial cases: either all eigenvalues $s_j$ are zero, that is, when the design matrix $X$ is identically zero, or there is exactly one nonzero eigenvalue, possibly with algebraic multiplicity greater than one.

It follows from Theorem~\ref{Th1}\ref{Th1.ii} that, in general, the lower bound can only be attained exactly in the high-iteration limit ($t\to\infty$) by an RGD-estimator with $\lambda=\lambda^*$ and step-size $\eta\in(0,2/(s_1+\lambda^*))$. Hence, if we want statistical optimality, we have to use some explicit regularization $\lambda>0$. 
Furthermore, the monotonicity result in Theorem~\ref{Th1}\ref{Th1.i} shows that for certain levels of regularization ---including the optimal one--- stopping the RGD algorithm early is superfluous. Early stopping in this setting does not improve statistical accuracy relative to utilizing the full computational resources. Therefore, to achieve statistical optimality under computational constraints, one should set $\lambda = \lambda^*$, select an appropriate step-size $\eta$ and run the RGD algorithm until the entire available computational budget is exhausted.

Another interesting observation is the following. Although \(\hat{\beta}_t(0,\eta)\) converges to \(\hat{\beta}(0) + P \theta\) as \(t \to \infty\) (see Remark~\ref{rem:GD}), where \(P = I_p - X^\dagger X\), the generalization error satisfies
\[
\mathrm{R}_\Sigma(\hat{\beta}_t(0,\eta)) \;\longrightarrow\; 
\mathrm{R}_\Sigma(\hat{\beta}(0)) \quad \text{as } t \to \infty,
\]
whenever \(0 < \eta < 2/s_1\) (cf. Theorem~\ref{Th1}\ref{Th1.ii}). We emphasize that this convergence holds not only in the idealized setting of Theorem~\ref{Th1}, where $\theta$ is treated as random—a device used solely to ensure that $\mathbb{E}(\theta^\top A^t \Sigma A^t \theta \mid X)$ decreases monotonically in $t$—but also for arbitrary fixed $\theta \in \mathbb{R}^p$. To see this, fix an arbitrary initialization $\theta \in \mathbb{R}^p$ and note that, by \eqref{decA1}, the initialization $\theta$ enters the error decomposition of $\mathrm{R}_\Sigma(\hat{\beta}_t(0,\eta))$ only through the term
\[
\mathbb{E}\!\left(\theta^\top A^t \Sigma A^t \theta \mid X\right)
= \theta^\top A^t \Sigma A^t \theta.
\]
Moreover, since $\theta^\top A^t \Sigma A^t \theta \to 0$ as $t \to \infty$ provided that $0 < \eta < 2/(s_1+\lambda)$, the effect of the initialization vanishes in the large-iteration limit. In other words, the null-space component $P\theta \in \mathcal{N}(X) = im(X^\top)^\perp$ does not influence predictive performance. From a statistical perspective, this observation is important, as it shows that RGD with $\lambda = 0$ automatically concentrates on the learnable directions $im(X^\top)$ while ignoring components in the null space that cannot be identified from the data.

% Note that by \eqref{decA1} the initialization vector $\theta$ enters the decomposition, for arbitrary $\lambda\geq0,\eta>0 $ and $t\in \mathbb{N}$, of $R_\Sigma(\hat{\beta}_t(\lambda,\eta))$ only through $\mathbb{E}(\theta^\top A^t\Sigma A^t\theta)=\theta^\top A^t\Sigma A^t\theta$ and $\theta^\top A^t\Sigma A^t\theta\to0$, as $t\to\infty$. Hence, In remark~\ref{rem:GD} we argued that for arbitrary $\theta\in \mathbb{R}^p$, $\hat{\beta}_t(0,\eta)\to \hat{\beta}_R(0,\eta)+P\eta$ as $t\to\infty$, where $P=I_p-X^\dagger X$ and $P\theta=0$ whenever $\theta\in im(X^\top)$. Theorem~\ref{Th1}\ref{Th1.i} shows that in the high iteration limit the effect 

\begin{remark}[On the choice of step size]\label{rem:stepGD}\,
	\begin{enumerate}
		\item Note that Theorem~\ref{Th1}\ref{Th1.i} was stated for \(\eta \in (0, 1/(s_1 + \lambda)]\), whereas, for any \(\lambda > 0\) and \(\eta \in (0, 2/(s_1 + \lambda))\), we have \(\hat{\beta}_t(\lambda, \eta) \to \hat{\beta}_R(\lambda)\) as \(t \to \infty\) (cf. Remark~\ref{rem:GD}). The monotonicity result of Theorem~\ref{Th1}\ref{Th1.i} can be extended to \(\eta \in (0, 2/(s_1 + \lambda))\) by restricting to even iteration indices \(t\); this follows from the decomposition in Lemma~\ref{LemforTh1}\ref{dec1} and the fact that \(a_i^{2t} = (1 - \eta (s_i + \lambda))^{2t} \in (0,1)\) for all \(i = 1, \dots, p\), \(\eta \in (0, 2/(s_1 + \lambda))\), and \(t \in \mathbb{N}\).

		\item The same arguments as in Remark~\ref{rem:GD}, regarding the minimal number of iterations \(t\) required for the RGD estimator \(\hat{\beta}_t(\lambda^*,\eta)\) to converge to the corresponding ridge solution \(\hat{\beta}_R(\lambda^*)\), can be applied to the convergence of the generalization error \(\mathrm{R}_\Sigma(\hat{\beta}_t(\lambda^*,\eta))\), thanks to the decomposition in Lemma~\ref{LemforTh1}\ref{dec1}.
		This decomposition leads to the same optimal choice of step size \(\eta_{\mathrm{opt}}(\lambda^*)\).
		However, since \(\eta_{\mathrm{opt}}(\lambda^*) > (s_1 + \lambda^*)^{-1}\) except in the degenerate case \(s_1 = s_p\), Theorem~\ref{Th1}\ref{Th1.i} is not applicable to \(\mathrm{R}_\Sigma(\hat{\beta}_t(\lambda^*, \eta_{\mathrm{opt}}(\lambda^*)))\) outside of this trivial scenario.
		 Nevertheless, as argued above, monotonicity still holds along even \(t \in \mathbb{N}\). If we restrict \(\eta\) to the interval \((0, 1/(s_1+\lambda^*)]\), where Theorem~\ref{Th1}\ref{Th1.i} applies, the fastest convergence in \(t\) of \(\mathrm{R}_\Sigma(\hat{\beta}_t(\lambda^*,\eta))\) toward the corresponding ridge generalization error \(\mathrm{R}_\Sigma(\hat{\beta}_R(\lambda^*))\) is achieved by choosing
		 \[
		 \eta = \argmin_{\eta \in (0, (s_i + \lambda^*)^{-1}]} \max \{|a_1(\lambda^*,\eta)|,\dots, |a_p(\lambda^*,\eta)|\}= \frac{1}{s_1 + \lambda^*}.
		 \]
	\end{enumerate} 
\end{remark}
Regarding the monotonicity result in Theorem~\ref{Th1}\ref{Th1.i}, we note that \citet{Lolas2020} establish a similar property for the continuous-time analogue of RGD, namely the \textit{penalized gradient-flow} estimator. To fully appreciate the differences between the result of \citet{Lolas2020} and the present work, we introduce the continuous-time analogue of RGD formally. 

Applying gradient descent to the objective function in \eqref{eq:obj-f} admits a simple numerical interpretation. In particular, for any $\lambda \ge 0$, gradient descent corresponds to an explicit Euler discretization of the initial value problem
\begin{equation}\label{Showalter}
	\dot{\beta}(\alpha) = -\nabla L(\beta(\alpha)), \qquad \beta(0) = 0, \qquad \alpha \ge 0.
\end{equation}
In the case $\lambda=0$, the initial value problem \eqref{Showalter} is referred to in the inverse problems literature as \emph{Showalter's method} or \emph{asymptotic regularization} (see, e.g., Example~4.7 of \citet{EngHanNeu1996}). 
A closed-form solution to \eqref{Showalter} is given by
\begin{equation}\label{gradientflow}
	\hat{\beta}^{GF}_\lambda(\alpha) \coloneqq 
	\frac{1}{n} (\hat{\Sigma}_n + \lambda I_p)^{\dagger} 
	\Bigl(I_p - \exp\{-\alpha (\hat{\Sigma}_n + \lambda I_p)\}\Bigr) X^\top y,
\end{equation}
which we refer to as the \emph{regularized gradient-flow} (RGF) estimator.  
Comparing $\hat{\beta}^{GF}_\lambda(\alpha)$ in \eqref{gradientflow} with the RGD estimator
$\hat{\beta}_t(\lambda,\eta)$ in \eqref{repGD}, we observe that the two parameters,
the step size $\eta$ and the iteration number $t$, are absorbed into a single continuous-time
parameter $\alpha$ in the gradient-flow formulation.
This reformulation places RGD within the broader framework of
continuous-time dynamical systems.
More generally, replacing discrete-time algorithms by a continuous-time analogues
is a common device in numerical analysis and optimization, as the resulting differential
equations often retain the essential structure of the original algorithm while admitting
a more tractable theoretical analysis; see, for example, \citet{Helmke}, \citet{Candes2016},
and \citet{Bach2017}.

\citet{Lolas2020} show that the map
\[
\alpha \mapsto \mathrm{R}_\Sigma(\hat{\beta}^{GF}_\lambda(\alpha))
\]
is monotonically decreasing in $\alpha$ whenever $\lambda \ge \lambda^*$.
 This result also relies on the Assumption~\eqref{set-2}. Our result provides a more nuanced understanding of the generalization error of RGD, as the decomposition derived in Lemma~\ref{LemforTh1}---which is the key ingredient in the proof of Proposition~\ref{corTh1} and Theorem~\ref{Th1}\ref{Th1.i}---enables an exact analysis of all three tuning parameters $t,\lambda$ and $\eta$, that effect the generalization error of RGD.

For completeness, we also mention the recent work of \citet{HuckReiSta2025}. In contrast to \citet{Lolas2020} and the present study, \citet[Proposition~3.11]{HuckReiSta2025} establish a monotonicity result for the penalized gradient-flow estimator without relying on the random effects Assumption~\eqref{set-2}. Their result differs from the aforementioned works in that, rather than using $\Sigma$ in \eqref{pred_risk}, they consider $\hat{\Sigma} + \lambda I_p$; that is, they analyze a penalized version of the \emph{in-sample} generalization error. Specifically, they show that the map
\[
\alpha \mapsto \mathrm{R}_{\hat{\Sigma}+\lambda I_p}(\hat{\beta}^{GF}_\lambda(\alpha))
\]
is monotonically decreasing in $\alpha$ whenever
\begin{align*}
	\lambda \sum_{j=i}^p s_j (v_j^\top \beta)^2 \ge \frac{\sigma^2}{n} \sum_{j=i}^p s_j, \quad \forall i \in \{1,\dots,p\},
\end{align*}
where $v_1,\dots,v_p$ denote the normalized eigenvectors corresponding to the eigenvalues $s_1,\dots,s_p$.

A key advantage of the methodology developed here is that RGF is inherently a continuous-time object and, as such, cannot be implemented directly on a computer. Consequently, the results established in this work can be applied directly, without depending on approximations that guarantee the \textit{closeness} of the RGF estimator evaluated at a given time $\alpha$ to a corresponding RGD iterate (cf. \citet[Lemma 2]{AliKolTib2019}).

\section{Computationally efficient tuning parameter selection}\label{sec:tuning}
As discussed in Section~\ref{sec:FiniteSample}, approximating the benchmark
generalization error $\mathrm{R}_{\Sigma}(\hat{\beta}_{R}(\lambda^{*}))$
requires implementing RGD with the optimal tuning parameter
\(
\lambda_{n}^{*} = \tfrac{\sigma^2}{\tau^2}\,\gamma_{n},
\)
and running the algorithm for as many iterations $t$ as computationally
feasible.
Since the optimal tuning parameter $\lambda_n^{*}$ depends on unknown population
quantities, these must be estimated from the observed data. Such an estimation
should, however, remain compatible with computational tractability, which
motivates the use of an iterative algorithm in the first place.

A classical approach for tuning parameter selection is cross-validation. \citet{Hastie2022} prove that leave-one-out cross-validation (LOOCV) achieves optimal tuning of the ridge penalty in a large-$p$, large-$n$ setting where $\gamma_n = p/n\to \gamma \in(0,\infty)$ (see \citet[Theorem~7]{Hastie2022}). More precisely, they choose the ridge tuning parameter $\lambda$ minimizing the LOOCV error
\begin{align}\label{CV}
	\text{CV}_n(\lambda)=\frac{1}{n}\sum_{i=1}^n(y_i-\hat{f}^{-i}_\lambda(x_i))^2=\frac{1}{n}\sum_{i=1}^n\bigg(\frac{y_i-\hat{f}_{\lambda}(x_i)}{1-(S_{\lambda})_{ii}}\bigg)^2,
\end{align}
where $\hat{f}_\lambda(x_i)=x_i^\top \hat{\beta}_R(\lambda)$ is the ridge predictor, $\hat{f}^{-i}_\lambda(x_i)$ is the ridge predictor trained on the whole training set except the $i$-th observation and $S_{\lambda}=X(X^\top X +n\lambda I_p)^{-1}X^\top$. In the high-dimensional case $p>n$, it is more convenient to consider the dual-representation of the ridge estimator, $\hat{\beta}_R(\lambda)=X^\top (XX^\top +n\lambda I_n)^{-1}y$, in which case $S_\lambda=XX^\top (XX^\top +n\lambda I_n)^{-1}$. The second equality of \eqref{CV} is the well-known short-cut formula of the LOOCV error, which can be easily derived using the Sherman-Morrison-Woodbury formula. \citet{Hastie2022} then show that
\begin{align*}
	\bigl| \mathrm{R}_{I_p}(\hat{\beta}_R(\hat{\lambda}_{CV})) 
	- \mathrm{R}_{I_p}(\hat{\beta}_R(\lambda^*)) \bigr|
	\;\overset{a.s.}{\longrightarrow}\; 0,
	\quad \text{ as }\gamma_n \to \gamma \in (0,\infty),
\end{align*} 
where $\hat{\lambda}_{CV}=\arg \min_{\lambda\in [\lambda_1,\lambda_2]}CV_n(\lambda)$ and $\lambda^*\coloneqq \frac{\sigma^2}{\tau^2}\gamma\in [\lambda_1,\lambda_2]$. Thus, the minimization of $CV_n$ is done on a pre-specified compact interval that is known to contain the optimal tuning parameter $\lambda^*$. 
A general problem with cross-validation -- even in this simple setup where there is the short-cut formula available -- is the computational cost. 

We distinguish two cases. For $n \ge p$, one first computes $X^\top X$, which costs $O(n p^2)$ flops; for each $\lambda$ in a grid of $r$ candidate values, one inverts $X^\top X + n \lambda I_p$, which incurs $O(r p^3)$ flops, and computes the diagonal entries of the matrix $S_\lambda$, $(X (X^\top X + n \lambda I_p)^{-1} X^\top)_{ii} = x_i^\top (X^\top X + n \lambda I_p)^{-1} x_i$, for $i=1,\dots,n$, at a cost of $O(r n p^2)$; all other operations are of a lower order. Hence, the total cost is $O( r n p^2)$. For $p > n$, one first computes $X X^\top$, costing $O(n^2 p)$ flops; for each $\lambda$ in the grid, the $n \times n$ matrix $X X^\top + n \lambda I_n$ must be inverted, incurring $O(r n^3)$ flops, and using the identity $S_\lambda = X X^\top (X X^\top + n \lambda I_n)^{-1} = I_n - n \lambda (X X^\top + n \lambda I_n)^{-1}$, the diagonal entries can then be computed at a negligible cost, while computing $X^\top (X X^\top + n \lambda I_n)^{-1}$ $r$ times costs $O(r n^2 p)$. All remaining operations are of a lower order, giving a total cost of $O(r n^2 p)$.
In practice, 
$r$ is typically large when using a fine grid of candidate $\lambda$ values. 

In this paper we take a more direct approach to estimate $\lambda^*$. Since $\lambda_n^* = \frac{\sigma^2}{\tau^2}\gamma_n$ depends explicitly on the signal strength $\tau^2$ and on the noise level $\sigma^2$, we propose to estimate these parameters consistently from the training data $(X,y)$. To this end, we build on and extend results of \citet{Dicker2014}. For a response vector $y$ and a design matrix $X$, we define
\begin{align*}
	\hat{m}_1 &\coloneqq \frac{1}{p}\,\text{tr}(\hat{\Sigma}_n), \quad
	\hat{m}_2 \coloneqq \frac{1}{p}\,\text{tr}(\hat{\Sigma}_n^2), \quad
	\tilde{m}_2 \coloneqq \hat{m}_2 - \gamma_n \hat{m}_1^2.
\end{align*}
Provided that $\tilde{m}_2 > 0$, we define the estimators
\begin{align}\label{est:def}
	\begin{split}
	\hat{\tau}_n^2 &\coloneqq
	\frac{1}{\tilde{m}_2}\frac{\|X^\top y\|_2^2}{n^2}
	- \frac{\gamma_n \hat{m}_1}{\tilde{m}_2}\frac{\|y\|_2^2}{n}, \\[0.5em]
	\hat{\sigma}_n^2 &\coloneqq
	\frac{\|y\|_2^2}{n}\!\left(1+\gamma_n \frac{\hat{m}_1^2}{\tilde{m}_2}\right)
	- \frac{\|X^\top y\|_2^2}{n^2}\frac{\hat{m}_1}{\tilde{m}_2}
	= \frac{1}{\tilde{m}_2}\!\left(
	\hat{m}_2\frac{\|y\|_2^2}{n}
	- \frac{\hat{m}_1\|X^\top y\|_2^2}{n^2}
	\right).
	\end{split}
\end{align}
These estimators for $\tau^2$ and $\sigma^2$ are linear combinations of $n^{-2}\|X^Ty\|_2^2$ and $n^{-1}\|y\|_2^2$, with coefficients determined by $\gamma_n$, $\hat{m}_1$ and $\hat{m}_2$. Thus, the highest computational cost of the proposed estimators arises from calculating $(n^{-1} X^\top X)^2$ or $(n^{-1} X X^\top)^2$, depending on whether $p$ or $n$ is larger, which requires $O(n p \, (n \wedge p))$ flops and therefore has the same computational complexity as computing a single ridge estimator.
In comparison, LOOCV incurs a total computational cost of $O(r np (n\wedge p))$ flops, as established above; hence, the proposed procedure achieves roughly an $r$-fold reduction in the leading computational cost relative to LOOCV.

To establish the consistency of these estimators in a large-$n$, large-$p$ setting, we impose several technical assumptions on the data-generating process, which are listed below. While we provide the complete set of assumptions for clarity, not all of them are required in every result. In particular, in Subsection~\ref{sec:NoRand}, we replace Assumption~\ref{A5} with the deterministic condition in Assumption~\ref{A6} and the error condition in Assumption~\ref{A7}.
In the subsequent arguments, we consider, for each $n \in \mathbb{N}$, the linear model
\begin{equation}\label{linmod}
	y=X\beta +u,
\end{equation}
where $X=Z\Sigma_n^{1/2}$,  $\Sigma_n^{1/2}$ is the unique symmetric positive semidefinite square-root of the covariance matrix $\Sigma_n$ and $Z$ is an $n\times p$ matrix. Throughout the remainder of the paper, we assume that $n, p\geq 2$. We define the empirical spectral distribution function of a symmetric matrix $A$ as,
\begin{equation}\label{empdist}
	F_{A}(x)=\frac{1}{p}\sum_{i=1}^p1_{\{s_i(A)\leq x\}},
\end{equation} 
where $\{s_i(A)\}_{i=1}^p$ are the eigenvalues of $A$.
\begin{assumptions} In the following, all random elements are defined on a common probability space $(\Omega,\mathcal{A},\mathbb{P})$ and we consider the assumptions:
	\begin{enumerate}[label=(A\arabic*)]
		\item\label{A1} $p=p_n$ and $\gamma_n = p/n\to \gamma \in (0,\infty)$ as $n \to \infty$.
		
		\item\label{A2} $\{\Sigma_n\}_{n\geq 1}$ is a sequence of $p \times p$ symmetric positive semidefinite matrices with uniformly bounded eigenvalues from above (i.e., $\sup_{n \in \mathbb{N}}\|\Sigma_n\|_{2}<\infty$) and $F_{\Sigma_n}(0)\neq 1$, for every $n\in\mathbb N$.
		
		\item\label{A3} We assume that the entries of $Z=Z_n=(z_{ij})_{1\leq i\leq n, 1\leq j\leq p}$ are, for each $n\in \mathbb{N}$, real-valued independently distributed random variables satisfying $\mathbb{E}(z_{ij})=0$, $\mathbb{E}(z_{ij}^2)=1$ and \begin{align}\label{X:unif4}\sup_{n\ge 1}\,\max_{1\le i\le n,\,1\le j\le p_n} \mathbb E\bigl(|z_{ij}|^{4+\varepsilon}\bigr) < \infty, 
		\end{align} for some $\varepsilon>0$. Furthermore, the distributions of the $z_{ij}$ are assumed to be absolutely continuous with respect to Lebesgue measure. 
		
		\item\label{A4} The spectral distribution function $F_{\Sigma_n}$ converges weakly to a cumulative distribution function $H$ supported on $[0,\infty)$, as n $\to \infty$. Additionally we assume that $H(0)\neq 1$.

		\item\label{A5} For each $n\in\mathbb N$, let $\beta$ be a $p$-dimensional random vector with independent entries satisfying 
		$\mathbb E(\beta_i)=0$, $\mathbb E(\beta_i^2)=p^{-1}\tau^2$ for some $\tau>0$ independent of $n$, and
		\[
		\sup_{n\ge 1} \max_{1\le i\le p_n} p^2 \,\mathbb E(\beta_i^4) \le \nu_{4,\beta}.
		\]
		 Similarly, for each $n\in\mathbb N$, $u$ is an $n$-dimensional random vector with independent entries satisfying $\mathbb{E}(u_i)=0$, $\mathbb{E}(u_i^2)=\sigma^2$ for some $\sigma>0$ independent of $n$, and
		 \[  \sup_{n\geq 1}\max_{1\leq i\leq n}\mathbb{E}(u_i^{4})\leq \nu_{4,u}<\infty.\] Furthermore, we assume that $u$, $Z$ and $\beta$ are jointly independent for every $n$.
		
		\item\label{A6} For the sequence of vectors $\beta = \beta_n$ appearing in \eqref{linmod}, we write $\tau_n = \|\beta\|_2$. We assume that $\tau_n>0$ for all $n$ and $\{\tau_n\}_{n\geq 1}$ is uniformly bounded, that is, $\sup_{n\ge 1} \tau_n < \infty$.
		Moreover, for $k \in \{1,2\}$,
		\begin{equation*}
		\Delta_k \coloneqq \frac{\beta^\top \Sigma_n^{\,k} \beta}{\|\beta\|_2^{2}}
		- \frac{1}{p}\operatorname{tr}\!\left(\Sigma_n^{\,k}\right)
		\longrightarrow\; 0, \quad\text{as } n\to\infty.
		\end{equation*}

		\item\label{A7} For each $n\in\mathbb N$, $u$ is an $n$-dimensional random vector with independent entries satisfying $\mathbb{E}(u_i)=0$, $\mathbb{E}(u_i^2)=\sigma_n^2$, where $\sigma_n> 0$ and
		\[ \sup_{n\geq 1}\max_{1\leq i\leq n}\mathbb{E}(u_i^{4})\leq \nu_{4,u} < \infty.\] Additionally, we assume that $u$ is independent of $Z$ for all $n$.		
	\end{enumerate}		
\end{assumptions}

 \begin{remark}[Discussion of the assumption set]\label{rem:Assumption}
Assumptions \ref{A1}--\ref{A4} are fairly standard in the random matrix literature, with few exceptions. First, the absolute continuity of the entries of $Z_n$ with respect to Lebesgue measure, as required in Assumption~\ref{A3}, is not widely used. This condition, together with Assumption~\ref{A2}, ensures that $\tilde m_2 = \hat m_2 - \gamma_n \hat m_1^2 > 0$, almost surely (see Lemma~\ref{denomnonzero}), and hence guarantees that the estimators $\hat \tau_n^2$ and $\hat \sigma_n^2$ are well-defined with probability one.
Second, the condition $F_{\Sigma_n}(0) \neq 1$ for all $n$ in Assumption~\ref{A2} is somewhat uncommon and is often replaced by the stronger requirement that the sequence $\{\Sigma_n\}_{n\ge 1}$ consists of symmetric positive definite matrices whose eigenvalues are uniformly bounded away from zero (see, for example, \citet{LedPeche2011} and \citet{BODNAR2014}).
Third, we impose a similar restriction in the limit in Assumption~\ref{A4}, namely $H(0) \neq 1$. We also note that Assumption~\ref{A3} allows for the distributions of the entries $z_{ij}$ of $Z_n$ to vary across $i$, $j$, and $n$. Moreover, by Assumptions~\ref{A2} and \ref{A4}, and the Portmanteau theorem, it follows that for any continuous function $g:\mathbb{R} \to \mathbb{R}$,
\begin{equation}\label{convtrgSig}
\frac{1}{p} \tr(g(\Sigma_n)) = \frac{1}{p} \sum_{i=1}^p g(t_i) = \int_0^C g(t)\, dF_{\Sigma_n}(t) \;\longrightarrow\; \int_0^C g(x)\, dH(x) < \infty \quad \text{as } n \to \infty,
\end{equation}
where $\{t_i\}_{i=1}^p$ denote the eigenvalues of $\Sigma_n$.
It follows directly from Assumption~\ref{A5} that Assumption~\eqref{set-2} holds for all $n$. 
Assumptions~\ref{A6} and \ref{A7} are used only in Subsection~\ref{sec:NoRand} to establish the consistency of $\hat{\sigma}_n^2$ and $\hat{\tau}_n^2$ in a setting where the coefficient vector $\beta$ is deterministic.  
Similarly to the random effects case, we make no distributional assumptions on the error terms $u_i$, and we even allow the variances $\sigma_n^2$ and $\tau_n^2 = \|\beta\|_2^2$ to depend on $n$ in Assumptions~\ref{A6} and~\ref{A7}.  
Finally, observe that when $\Sigma_n = I_p$ for all $n$, the condition imposed on
$\Delta_k$ in Assumption~\ref{A6} is trivially satisfied.
\end{remark} 
First, we verify that the estimators $\hat{\sigma}_n^2$ and $\hat{\tau}_n^2$ are well defined in our setting by showing that $\tilde{m}_2>0$, almost surely.
\begin{lemma}\label{denomnonzero}
	Under Assumptions \ref{A2} and \ref{A3} 
	\begin{align*}
		\tilde{m}_2 = \hat{m}_2-\gamma_n \hat{m}_1^2 > 0,\,\, \text{almost surely}.
	\end{align*}		
	\end{lemma}
We emphasize that Lemma~\ref{denomnonzero} does not require any assumption on $\gamma_n$; it only relies on the condition $n,p\geq 2$.
%A key ingredient in the proof of the consistency of the estimators $\hat{\sigma}_n^2$ and $\hat{\tau}_n^2$ in Subsection~\ref{sec:ConsistRand} and Subsection~\ref{sec:NoRand} is the following lemma:
%\begin{lemma}\label{expest}
%	Under Assumptions~\ref{A3} and \ref{A7}, we have 
%	\begin{align*}
%		\mathbb{E}(\hat{\sigma}_n^2|X,\beta)&=\frac{1}{\tilde{m}_2} \bigg(\hat{m}_2\beta^\top \hat{\Sigma}_n\beta - \hat{m}_1\beta^\top \hat{\Sigma}_n^2\beta\bigg)+\sigma_n^2  \\
%		\mathbb{E}(\hat{\tau}_n^2|X,\beta)&=\frac{1}{\tilde{m}_2} \bigg(\beta^\top \hat{\Sigma}_n^2\beta - \gamma_n \hat{m}_1 \beta^\top \hat{\Sigma}_n \beta \bigg).
%	\end{align*}
%\end{lemma}	
%Observe that under Assumption~\ref{A5}, we have $\mathbb{E}(\hat{\tau}_n^2 \mid X) = \tau^2$ and $\mathbb{E}(\hat{\sigma}_n^2 \mid X) = \sigma^2$.  
%This substantially simplifies the proof in Subsection~\ref{sec:ConsistRand}, since $|\hat{\tau}_n^2 - \tau^2| = |\hat{\tau}_n^2 - \mathbb{E}(\hat{\tau}_n^2 \mid X)|$ and $|\hat{\sigma}_n^2 - \sigma^2| = |\hat{\sigma}_n^2 - \mathbb{E}(\hat{\sigma}_n^2 \mid X)|$.
%Consequently, it suffices to show that the differences between certain quadratic and bilinear forms and their corresponding conditional expectations converge in probability to zero. 
In the next step, we establish the consistency of $\hat{\sigma}^2_n$ and $\hat{\tau}^2_n$ under Assumption~\ref{A5} (cf. Subsection~\ref{sec:ConsistRand}). In order to be consistent with our results in Subsection~\ref{sec:GenError}, 
we show that the difference between the generalization error of RGD tuned with 
an estimator $\hat{\lambda}_n$ (based on $\hat{\sigma}_n^2$ and $\hat{\tau}_n^2$) 
and the generalization error of optimally tuned RGD, with 
$\lambda_n^* = \tfrac{\sigma^2}{\tau^2}\gamma_n$, converges to zero as $n \to \infty$.
 (cf.\ Subsection~\ref{sec:TunedRGD}).   
On this way, we also prove a \textit{plug-in} result for the generalization error of ridge regression tuned with $\lambda_n^*$. Specifically, we show that the generalization error of ridge regression, tuned with $\hat{\lambda}_n$, attains the asymptotically optimal benchmark generalization error in a setting where $\gamma_n \to \gamma \in (0,\infty)$ as $n \to \infty$.
Finally, in Subsection~\ref{sec:NoRand}, we replace Assumption~\ref{A5} with Assumptions~\ref{A6} and \ref{A7} and show that consistent estimation of the noise variance $\sigma_n^2$ and the signal strength $\tau_n^2 = \|\beta\|_2^2$ remains possible, even in settings where $\gamma_n \to 0$ or $\gamma_n \to \infty$.

\subsection{Consistent estimators with random effects}
\label{sec:ConsistRand}
In this section we proof consistency of the estimators $\hat{\sigma}_n^2$ and $\hat{\tau}_n^2$ under Assumption~\ref{A5}. In addition, we derive an auxiliary result of independent interest that serves as a key component in the consistency proof. Specifically, we relate the first two moments of the empirical spectral distribution function of $\hat{\Sigma}_n$ to their deterministic counterparts, which depend only on $\Sigma_n$ and $\gamma_n$. 
\begin{lemma}\label{momqf1}
	Under Assumptions \ref{A2} and \ref{A3}, we have 
	\begin{align*}
		\Bigg|\frac{1}{p^{1/2}}\operatorname{tr}(\hat{\Sigma}_n)
		- \frac{1}{p^{1/2}}\operatorname{tr}(\Sigma_n)\Bigg|
		&= O_{\mathbb{P}}\left(\frac{1}{n^{1/2}}\right), \\[4pt]
		\Bigg|\frac{1}{p}\operatorname{tr}(\hat{\Sigma}_n^2)
		- \Bigg(\frac{1}{p}\operatorname{tr}(\Sigma_n^2)
		+ \gamma_n \Big(\frac{1}{p}\operatorname{tr}(\Sigma_n)\Big)^2\Bigg)\Bigg|
		&= O_{\mathbb{P}}\left(\frac{1}{n^{1/2}} \,\vee\, \frac{p}{n^{3/2}} \right).
	\end{align*}
\end{lemma}	
\citet[Theorem 3.2]{BODNAR2014} prove a similar result  although under stronger assumptions. They obtain convergence almost surely to zero for both expressions in Lemma~\ref{momqf1}, in a setting where $n\to \infty$ and $\gamma_n\to \gamma\in(0,\infty)$. Here, we obtain convergence rates and most notably we do not require $\limsup_{n\to\infty}\gamma_n<\infty$, to prove convergence in probability to zero. It is somewhat surprising that, irrespective of the growth rate of $\gamma_n$, 
the convergence rate in the first display of Lemma~\ref{momqf1} is $n^{-1/2}$, 
which ensures that $p^{-1/2}\operatorname{tr}(\hat{\Sigma}_n)$ is asymptotically 
equivalent to its population counterpart. In the second statement, this is not the case in general. To be more specific, $p^{-1}\operatorname{tr}(\hat{\Sigma}_n^2)$ is asymptotically 
equivalent to $p^{-1}\operatorname{tr}(\Sigma_n^2)$ when $\gamma_n \to 0$, 
whereas if $\gamma_n \to \gamma \in (0,\infty]$ (for example, in regimes with 
$\gamma_n = o(n^{1/2})$), the excess term
$\gamma_n\big(p^{-1}\operatorname{tr}(\Sigma_n)\big)^2$ need not to vanish asymptotically.

A trivial consequence of Lemma~\ref{momqf1}, together with Assumption~\ref{A4}, is that 
\begin{equation}\label{convhatm1}
	\hat{m}_1=\frac{1}{p} \tr(\hat{\Sigma}_n)\overset{\mathbb{P}}{\longrightarrow}\int_0^\infty t\, dH(t)\neq 0\quad \text{as}\,n\to\infty.
\end{equation}
Similarly, using both statements of Lemma~\ref{momqf1}, we obtain
\begin{equation*}
	\tilde{m}_2 = \frac{1}{p}\tr(\hat{\Sigma}_n^2)
- \gamma_n \Biggl(\frac{1}{p}\tr(\hat{\Sigma}_n)\Biggr)^2
= \frac{1}{p}\tr(\Sigma_n^2) + o_{\mathbb{P}}(1),
\end{equation*}
under the condition $\gamma_n = o(n^{1/2})$; a detailed proof is deferred to
Appendix~\ref{App:C}. Furthermore, together with Assumption~\ref{A4}, the preceding
display implies that
\begin{equation}\label{convtildem2}
\tilde{m}_2 \;\overset{\mathbb{P}}{\longrightarrow}\; \int_0^\infty t^2 \, dH(t) \;\neq 0
\quad \text{as } n\to\infty.
\end{equation}
We are now going to present the main contribution of this subsection. 
\begin{theorem}\label{est1}
	Consider Assumptions  \ref{A1}--\ref{A5}. Then it holds that
	\begin{equation*}
		\hat{\sigma}_n^2 \;\overset{\mathbb{P}}{\longrightarrow}\; \sigma^2>0
		\quad\text{and}\quad
		\hat{\tau}_n^2 \;\overset{\mathbb{P}}{\longrightarrow}\; \tau^2>0 \quad \text{as}\,\,n\to\infty.
	\end{equation*}
\end{theorem}	
Theorem~\ref{est1} can be viewed as a version of \citet[Proposition 2.(i) and Lemma~2]{Dicker2014} under Assumption~\ref{A5}. For a more detailed discussion on this result, we refer to Subsection~\ref{sec:NoRand}, where we prove a similar result in a different setting where Assumption~\ref{A5} is replaced by Assumptions~\ref{A6} and~\ref{A7}. Nevertheless, Theorem~\ref{est1} will play a prominent role in the next chapter, where we discuss fully data-driven \textit{plug-in} procedures for ridge regression and RGD.

\subsection{Optimally tuned RGD}
\label{sec:TunedRGD}
Since $\lambda_n^*=\frac{\sigma^2}{\tau^2}\gamma_n$ we get by Theorem~\ref{est1} and the continuous mapping theorem that \[\biggl|\frac{\hat{\sigma}_n^2}{\hat{\tau}_n^2}\gamma_n-\lambda_n^*\biggr|\overset{\mathbb{P}}{\to} 0.\] Based on this fact, we present \emph{plug-in} results in two settings: first, for the 
generalization error of ridge regression with tuning parameter $\lambda_n^*$, and 
second, for the generalization error of RGD with tuning parameter $\lambda_n^*$. In addition, we provide two results of independent interest that play a crucial role 
in the proof of the \emph{plug-in} procedure for ridge regression.
Before presenting the main results of this section, we briefly review some key concepts from random matrix theory.

We define the Stieltjes transform for any cumulative distribution function $G$ supported on $[0,\infty)$ as
\begin{align*}
	m_G(z)\coloneqq\int_0^\infty \frac{dG(s)}{s-z},
\end{align*}
where $z\in \mathbb{C}\setminus \mathbb{R}_{\geq 0}$. The Stieltjes transform of the empirical spectral distribution $F_{\underline{\hat{\Sigma}}_n}$, with $\underline{\hat{\Sigma}}_n\coloneqq n^{-1}XX^\top$, is given by
\begin{align*}
	v_n(z)\coloneqq m_{F_{\underline{\hat{\Sigma}}_n}}(z)=\int_0^\infty \frac{1}{s-z} dF_{\underline{\hat{\Sigma}}_n}(s), \,\,\text{for}\,\,z \in \mathbb{C}\setminus\mathbb{R}_{\geq 0}. 
\end{align*}     
From the connection between weak convergence and the pointwise convergence of 
Stieltjes transforms (see \citet[Proposition~2.2]{Hachem2007} 
or \citet[Theorem~B.9]{Silversteinbook2010}), weak convergence of the sequence of cumulative distribution functions 
$\{F_{\underline{\hat{\Sigma}}_n}\}_{n\ge 1}$ to a nonrandom limiting distribution 
function $\bar{F}$, except on a null set $N$, can be established by showing 
\begin{equation*}
v_n(z) = m_{F_{\underline{\hat{\Sigma}}_n}}(z,\omega) 
	\;\longrightarrow\; m_{\bar{F}}(z)\quad \text{as}\,\, n\to\infty,
\end{equation*} 
for every $z \in \mathbb{C}^+ \coloneqq \{ z \in \mathbb{C} : \Im(z) > 0 \}$ 
and fixed $\omega \in N^c$.

\begin{theorem}[\citet{MarPas1967}, \citet{Pan2010}]\label{Pan2010}
	Consider Assumptions \ref{A1}--\ref{A4}.
	Then it holds, with probability one, that the sequence $\{F_{\underline{\hat{\Sigma}}_n}\}_{n \geq1}$ converges weakly to a nonrandom limiting distribution function $\bar{F}=\bar{F}_{\gamma,H}$ supported on $[0,\infty)$. 
	The corresponding Stieltjes transform $m_{\bar{F}}(z)$ for $z\in \mathbb{C}^+$ is the unique solution to 
	\begin{align}\label{stielt}
		m_{\bar{F}}(z)=-\bigg(z-\gamma\int_{0}^{\infty}\frac{t dH(t)}{1+t m_{\bar{F}}(z)}\bigg)^{-1}.
	\end{align}
\end{theorem}	
\citet{Pan2010} proves Theorem~\ref{Pan2010} in a slightly more general setting. \citet{Pan2010} assumes that $\underline{\hat{\Sigma}}_n=n^{-1}Z_n\Sigma_nZ_n^\top$, where the entries of $Z_n$ are assumed to be independent, have a common mean $\mu\in \mathbb{R}$, variance $\phi^2>0$ and satisfy 
\begin{align}\label{PanLind}
	\lim_{n\to\infty}\frac{1}{n^2\varepsilon_n^2}\sum_{i=1}^n\sum_{j=1}^p\mathbb{E}(Z_{ij}^2 \mathbf{1}\{|Z_{ij}|\geq\varepsilon_n\sqrt{n}\})=0,
\end{align}
for some positive sequence $\{\varepsilon_n\}_{n\geq 1}$ converging to zero. Thus, to prove Theorem~\ref{Pan2010}, it remains to verify that
Assumption~\ref{A3} entails the condition stated in~\eqref{PanLind}.

\begin{proposition}\label{prop4m}
Assume that Assumption~\ref{A1} holds. Then, under Assumption~\ref{A3}, the condition
in \eqref{PanLind} is satisfied.
\end{proposition}
The proof of this statement is provided in Appendix~\ref{App:B}.
We next state a theorem that might be of independent interest. 
\begin{theorem}\label{LPTh1}
Consider Assumptions~\ref{A1}--\ref{A4}, and let $g : [0,\infty) \to \mathbb{R}$ be a function with finitely many points of discontinuity. Then, for each $z \in \mathbb{C} \setminus \mathbb{R}_{\ge 0}$, 
\begin{align*}
	\frac{1}{p} \tr\bigl(g(\Sigma_n) (\hat{\Sigma}_n - z I_p)^{-1}\bigr) 
	\;\overset{\text{a.s.}}{\longrightarrow}\; 
	\int_0^\infty \frac{g(t)}{-z v(z) t - z} \, dH(t) \quad \text{as } n \to \infty,
\end{align*}
	where $v(z)$ is the Stieltjes transform of $\bar{F}$ defined for $z \in \mathbb{C} \setminus \mathbb{R}_{\ge 0}$.

\end{theorem}
Theorem~\ref{LPTh1} is a version of \citet[Theorem~2]{LedPeche2011}, and the proof is provided in Appendix~\ref{App:B}. 
In \citet{LedPeche2011}, the result is established for a function $g : [h_1, h_2] \to \mathbb{R}$, with $0 < h_1 \le h_2 < \infty$, where the interval $[h_1, h_2]$ contains the support of $H$.
Additionally, they assume that $Z_n$ denotes the $n \times p$ upper-left block of a doubly infinite i.i.d.\ array with mean zero, variance one, and finite $12$th moments, and that $\Sigma_n$ is positive definite for all $n$, without requiring the largest eigenvalue of $\Sigma_n$ to be uniformly bounded from above.

Another result of independent interest is the following lemma, which builds on Theorem~\ref{LPTh1}.
\begin{lemma}\label{LPmod}
	Consider Assumptions \ref{A1}--\ref{A4}. Then, almost surely, for all $\lambda >0$,
	\begin{align*}
		\frac{1}{p}\tr\big(\Sigma_n(\hat{\Sigma}_n+\lambda I_p)^{-1}\big)&{\longrightarrow}\frac{1}{\gamma}\frac{1-\lambda v(-\lambda)}{\lambda v(-\lambda)},\\
		\frac{1}{p}\tr\big(\Sigma_n(\hat{\Sigma}_n+\lambda I_p)^{-2}\big)&{\longrightarrow}\,\frac{1}{\gamma}\frac{v(-\lambda)+\lambda \tfrac{d}{d \lambda} v(-\lambda)}{(\lambda v(-\lambda))^2},
	\end{align*}
	as $n \to \infty$, where $\tfrac{d}{d \lambda} v(-\lambda)$ denotes the derivative of the function $\lambda \mapsto v(-\lambda)$ with respect to $\lambda>0$.
\end{lemma}
The first statement of Lemma~\ref{LPmod} is a version of \citet[Lemma 2] {LedPeche2011} with the difference that we focus on $\lambda\in \mathbb{R}_{>0}$ instead of $z \in \mathbb{C}^+$. The second statement is a version of \citet[Lemma 2.2]{DobWag2018}, which relies on \citet[Lemma 2]{LedPeche2011} and a \textit{derivative trick} similar to results of \citet{RubMesPal2011} and \citet{Zhang2013}. Our proof technique in Theorem~\ref{LPTh1} and Lemma~\ref{LPmod}, builds on the idea of deterministic equivalents (cf. \citet{Hachem2007}), resulting in weaker distributional assumptions on $Z$. Furthermore, the almost sure convergence in \citet[Lemma 2]{LedPeche2011} and \citet[Lemma 2.2]{DobWag2018} both hold for a null-set that may depend on $\lambda$. In contrast, Lemma~\ref{LPmod} states almost sure convergence for all $\lambda>0$, that is, the null-set is independent of the choice of $\lambda$. These results play a crucial role in the next theorem. 

Before stating the plug-in result for ridge regression, we introduce an estimator for $\lambda_n^*$. Under the current assumptions, we set
\begin{equation}\label{defestimatorlambda}
	\hat{\lambda}_n \coloneqq \biggl(\frac{\hat{\sigma}_n^2}{\hat{\tau}_n^2}\gamma_n \;\lor\; \lambda_1\biggr) 
	\;\land\; \lambda_2,
\end{equation}
where $0 < \lambda_1 \le \lambda_2 < \infty$, and such that
\[
\frac{\sigma^2}{\tau^2}\gamma_n \longrightarrow \frac{\sigma^2}{\tau^2}\gamma \;\eqqcolon\; \lambda^* \in (\lambda_1, \lambda_2).
\]
The truncation is imposed for technical convenience. While
\[
\frac{\hat{\sigma}_n^2}{\hat{\tau}_n^2}\gamma_n \overset{\mathbb{P}}{\longrightarrow} \lambda^*,
\]
the ratio need not be uniformly integrable, which complicates expectation-level arguments. The truncation ensures that $\hat{\lambda}_n$ remains bounded and does not affect the asymptotic limit.

\begin{theorem}\label{Ridgetuned}
	Consider Assumptions \ref{A1}--\ref{A5}. Then, almost surely, for all $\lambda >0$,
	\begin{equation*}
		\mathrm{R}_{\Sigma_n}(\hat{\beta}_R(\lambda))\longrightarrow\mathrm{R}(\lambda)\quad \text{as} \,\,n\to \infty, 
	\end{equation*}
	where
	\begin{align*}
		\mathrm{R}(\lambda)\coloneqq\lambda(\lambda \tau^2-  \sigma^2\gamma)\bigg(\frac{1}{\gamma}\frac{v(-\lambda)+\lambda \tfrac{d}{d \lambda} v(-\lambda)}{(\lambda v(-\lambda))^2}\bigg)+ \sigma^2\gamma \bigg(\frac{1}{\gamma}\frac{1-\lambda v(-\lambda)}{\lambda v(-\lambda)}\bigg).
	\end{align*}
	Furthermore,
	\begin{equation*}
		\mathrm{R}_{\Sigma_n}(\hat{\beta}_R(\hat{\lambda}_n))\overset{\mathbb{P}}{\longrightarrow}\mathrm{R}(\lambda^*)= \min_{\lambda >0}\mathrm{R}(\lambda)\quad \text{as}\,\,n\to\infty,
	\end{equation*}
	 where $\lambda^*= \frac{\sigma^2}{\tau^2}\gamma\in(\lambda_1,\lambda_2)$.
\end{theorem}
Consider the decomposition in Lemma~\ref{LemforTh1}\ref{dec2} and note that for $\lambda>0$, $\tr(\Sigma_nF)$ can be equivalently written as 
\begin{equation}\label{altridge}
	\tr(\Sigma_n F)=(\lambda \tau^2- \sigma^2\gamma_n )\frac{\lambda}{p}\tr(\Sigma_n(\hat{\Sigma}_n+\lambda I_p)^{-2})+ \frac{\sigma^2\gamma_n}{p}\tr(\Sigma_n(\hat{\Sigma}_n+\lambda I_p)^{-1}).
\end{equation}
The first statement of Theorem~\ref{Ridgetuned} is an immediate consequence of \eqref{altridge} and Lemma~\ref{LPmod}. Moreover, we note that the first statement of Theorem~\ref{Ridgetuned} is a generalization of \citet[Theorem 2.1]{DobWag2018}, who establish the corresponding claim under the more restrictive assumption that the entries of $Z_n$ are i.i.d. with finite 12-th moments. The second statement shows that tuning the ridge estimator with $\hat{\lambda}_n$ is asymptotically optimal, in the sense that it achieves the oracle ridge generalization error in the limit. As noted in the introduction of Section~\ref{sec:tuning}, \citet[Theorem~7]{Hastie2022} provide a similar result for a CV-tuned ridge penalty. Hence, Theorem~\ref{Ridgetuned} demonstrates that asymptotic statistical optimality can be achieved in ridge regression by a computationally more efficient alternative than cross-validation (see the discussion at the beginning of Section~\ref{sec:tuning}). 

In this regard, we also want to mention the result of \citet[Theorem 4.1]{patil21}, who establish consistency of a CV-tuned ridge estimator in a setting that does not rely on the random-effects assumption. As we do not consider cross-validation, but instead focus on a computationally more efficient procedure, we do not pursue any extensions in this direction. Similarly to Theorem~\ref{Ridgetuned}, we also consider a fully data-driven plug-in result for RGD. 
\begin{theorem}\label{GDtuned}
	Consider Assumptions \ref{A1}--\ref{A5}. Suppose RGD is initialized at $\hat{\beta}_0(\lambda,\eta)=\theta \in \mathbb{R}^p$, where $\hat{\eta}_n(\lambda):=(\hat{s}_1+\lambda)^{-1}$ and $\|\theta\|_2$ is uniformly bounded in $n$. Then, for every $t\in \mathbb{N}$, 
	\begin{equation*}\label{GDtuned1}
	\big|\mathrm{R}_{\Sigma_n}\big(\hat{\beta}_{t}(\hat{\lambda}_n, \hat{\eta}_n(\hat{\lambda}_n))\big)
	- \mathrm{R}_{\Sigma_n}\big(\hat{\beta}_{t}(\lambda_n^*, \hat{\eta}_n(\lambda_n^*))\big)\big| 
	\overset{\mathbb{P}}{\longrightarrow} 0 \quad \text{as}\,\,n\to\infty,
	\end{equation*}
	where $\hat{s}_1= \hat{s}_1(X)\geq 0$ is any measurable function of $X$. \footnote{For example, $\hat{s}_1$ could be an approximation of the largest eigenvalue of $\hat{\Sigma}_n$.} 
	
	Moreover, if $ \mathbb{P}(s_{max}(\hat{\Sigma}_n)<\hat{s}_1<K)\to 1$ as $n\to\infty$, for some constant $K>0$, the convergence is even uniform in $t$, that is,
	\begin{equation*}\label{GDtuned2}
	\sup_{t \in \mathbb{N}} 
	\big|\mathrm{R}_{\Sigma_n}\big(\hat{\beta}_{t}(\hat{\lambda}_n, \hat{\eta}_n(\hat{\lambda}_n))\big)
	- \mathrm{R}_{\Sigma_n}\big(\hat{\beta}_{t}(\lambda_n^*, \hat{\eta}_n(\lambda_n^*))\big)\big| 
	\overset{\mathbb{P}}{\longrightarrow} 0\quad \text{as}\,\,n\to\infty.
	\end{equation*}
\end{theorem}
Analogously to Theorem~\ref{Ridgetuned}, the key implication of Theorem~\ref{GDtuned} is that RGD tuned with $\hat{\lambda}_n$ has the same asymptotic performance as RGD tuned with the statistically optimal tuning parameter $\lambda_n^*$. We also emphasize that Theorem~\ref{GDtuned} allows for a data-dependent step-size that depends on $\hat{\lambda}_n$ and an approximation of the largest eigenvalue of $\hat{\Sigma}_n$. For instance, one may obtain such an approximation using power iterations on the matrix $\hat{\Sigma}_n$. Although $\hat{\eta}_n(\hat{\lambda}_n)=1/(\hat{s}_1+\hat{\lambda}_n)$ does not necessarily coincide with the optimal step-size $\eta_{\mathrm{opt}}(\lambda_n^*)=2/(2\lambda_n^*+s_1+s_p)$ (cf. Remark~\ref{rem:GD}),  $\hat{\eta}_n(\hat{\lambda}_n)$ approximates the \textit{optimal step-size} on the interval $(0, 1/(s_1+\lambda_n^*)]$, the interval where Theorem~\ref{Th1}\ref{Th1.i} applies (cf. Remark~\ref{rem:stepGD}). 

By the second statement of Theorem~\ref{GDtuned}, the generalization error of RGD with the data-driven tuning parameter $\hat{\lambda}_n$ can be made arbitrarily close, uniformly over $t$, to that of RGD with the oracle parameter $\lambda_n^*$ for all sufficiently large $n$. Recall that, by Theorem~\ref{Th1}\ref{Th1.i} and Theorem~\ref{Th1}\ref{Th1.ii}, the generalization error of RGD with $\lambda_n^*$ decreases monotonically in $t$ and converges, as $t \to \infty$, to the benchmark ridge generalization error $\mathrm{R}_{\Sigma_n}(\hat{\beta}_R(\lambda_n^*))$. Altogether, these results establish that, for all sufficiently large $n$, the generalization error of the data-driven procedure along the iterates
$\{\mathrm{R}_{\Sigma_n}(\hat{\beta}_{t}(\hat{\lambda}_n, \hat{\eta}_n(\hat{\lambda}_n))) : t \in \mathbb{N}\}$
closely follows that of the oracle iterates
$\{\mathrm{R}_{\Sigma_n}(\hat{\beta}_{t}(\lambda_n^*, \hat{\eta}_n(\lambda_n^*))) : t \in \mathbb{N}\}$.
In Figure~\ref{fig:main}, the green curve shows the oracle generalization error for a selected subset of iterates, with the dots marking these points. 
For large $n$, these oracle iterates can be replaced by their data-driven counterparts while achieving essentially the same predictive accuracy.

\subsection{Consistency without random effects}
\label{sec:NoRand}
In this section, we replace Assumption~\ref{A5}, which treats \(\beta\) as random, 
with Assumptions~\ref{A6} and~\ref{A7} in the consistency proof of 
Theorem~\ref{est1}. In particular, under Assumption~\ref{A6}, \(\tau_n^2\) is 
no longer random. Furthermore, we establish consistency of $\hat{\sigma}_n^2$ and $\hat{\tau}^2_n$ in a setting where $\gamma_n$ is not restricted to converge to a finite constant $\gamma \in (0,\infty)$ as $n\to \infty$. In particular, we also allow the boundary cases $\gamma=0$ and $\gamma=\infty$. The only restriction on the relative growth of $p$ and $n$ is that $p=o(n^{5/4})$. Before we state the main theorem of this section, we present a lemma that might be of independent interest. 

\begin{lemma}\label{momqf2}
Under Assumptions \ref{A2} and \ref{A3}, and for every sequence $\tilde{\beta}=\tilde{\beta}_n$ with $\|\tilde{\beta}\|_2 = 1$, we have
\begin{align*}
	\big|\tilde{\beta}^\top \hat{\Sigma}_n \tilde{\beta} - \tilde{\beta}^\top \Sigma_n \tilde{\beta}\big| 
	&= O_{\mathbb{P}}\left(\frac{1}{n^{1/2}}\right),\\[2mm]
	\Bigg|\tilde{\beta}^\top \hat{\Sigma}_n^2 \tilde{\beta} - \Big(\tilde{\beta}^\top \Sigma_n^2 \tilde{\beta} + \frac{1}{n} \tr(\Sigma_n) \tilde{\beta}^\top \Sigma_n \tilde{\beta} \Big) \Bigg| 
	&= O_{\mathbb{P}}\left(\frac{1}{n^{1/2}} \lor \frac{p}{ n^{3/2}} \right).
\end{align*}
\end{lemma}	
Lemma~\ref{momqf2} is the analogous result for quadratic forms to Lemma~\ref{momqf1}. As pointed out in Subsection~\ref{sec:RandEff}, Assumption~\ref{A5} is mainly a technical tool. For example, it simplifies the analysis of quadratic forms of the type $\beta^\top \hat{\Sigma}_n^k \beta$, $k \in \{1,2\}$, and avoids dealing with problematic alignments of $\beta$ with the eigenvectors corresponding to the extreme eigenvalues of $\hat{\Sigma}_n$. More specifically, under Assumption~\ref{A5}, results of the form
\[
\frac{1}{\tau^{2}} \beta^\top \hat{\Sigma}_n^k \beta - \frac{1}{p} \operatorname{tr}(\hat{\Sigma}_n^k) = o_{\mathbb{P}}(1), \quad k \in \{1,2\},
\]
can be obtained by using concentration inequalities to control deviations of $\tau^{-2} \beta^\top \hat{\Sigma}_n^k \beta$ from its conditional expectation 
$\mathbb{E}\bigl(\tau^{-2} \beta^\top \hat{\Sigma}_n^k \beta \,\big|\, X\bigr) = p^{-1} \operatorname{tr}(\hat{\Sigma}_n^k)$, 
in the regime where $\gamma_n \to \gamma \in (0,\infty)$ (see the proof of Theorem~\ref{est1} in Appendix~\ref{App:B}).

In contrast, in a setting without random effects, similar convergence results can still be achieved. In the case where $\Sigma_n = I_p$ for all $n$ and for the sequence of $\beta$ vectors arising from the model in \eqref{linmod}, Lemma~\ref{momqf1} and Lemma~\ref{momqf2} imply that
\begin{equation}\label{quadfromtotr}
	\frac{\beta^\top \hat{\Sigma}_n^k \beta}{\|\beta\|_2^2} - \frac{1}{p} \operatorname{tr}(\hat{\Sigma}_n^k) = o_{\mathbb{P}}(1), \quad k \in \{1,2\},
\end{equation}
provided that $\gamma_n = o(n^{1/2})$. For general $\Sigma_n$, the same convergence results as in \eqref{quadfromtotr} can be obtained using Assumption~\ref{A6}, given that $\gamma_n = o(n^{1/2})$. The relevant arguments for this statement are deferred to Appendix~\ref{App:C}.

\begin{theorem}\label{consestnonran}
	Under Assumptions \ref{A2}, \ref{A3}, \ref{A4}, \ref{A6}, and \ref{A7}, and the additional assumption that $n\to \infty$ with $p=o(n^{5/4})$, we have
	\begin{align*}
		\hat{\sigma}_n^2 - \sigma_n^2 \;\overset{\mathbb{P}}{\longrightarrow}\; 0 
		\quad \text{and} \quad 
		\hat{\tau}_n^2 - \tau_n^2 \;\overset{\mathbb{P}}{\longrightarrow}\; 0 
		\quad \text{as } n \to \infty.
	\end{align*}
\end{theorem}
The proof of Theorem~\ref{consestnonran} is much more involved than the proof of 
Theorem~\ref{est1}, due to the nonrandomness of \(\beta\) and the fact that we 
work in a setting where \(\gamma_n\) may even tend to infinity. We note that 
\citet{Dicker2014} defines the signal strength as \(\beta^\top \Sigma_n \beta\), 
which results in a slightly different estimator of \(\hat{\tau}_n^2\). In 
particular, the estimator proposed in \citet{Dicker2014} differs from that in 
\eqref{est:def} in that it is additionally multiplied by \(\hat{m}_1\). The 
estimator for \(\sigma_n^2\) is unchanged. 
\citet[Proposition~2(i)]{Dicker2014} proves Theorem~\ref{consestnonran} with this 
adapted estimator in a setting where \(p = o(n^2)\), the entries of \(Z_n\), as 
well as \(u\), are Gaussian for each $n$, and Assumption~\ref{A6} holds. In contrast, the 
distributional assumptions in Theorem~\ref{consestnonran} are considerably 
weaker. Using Assumption~\ref{A6}, equations~\eqref{convtrgSig} and \eqref{convhatm1}, and 
Theorem~\ref{consestnonran}, we obtain
\begin{align}\label{Dickergen}
	\begin{split}
	\hat{m}_1 \hat{\tau}_n^2 - \beta^\top \Sigma_n \beta
	&= \hat{m}_1 \hat{\tau}_n^2 - \tau_n^2 \frac{1}{p} \tr(\Sigma_n) + o(1) \\
	&= \hat{m}_1\bigl(\hat{\tau}_n^2 - \tau_n^2\bigr)
	+ \tau_n^2\left(\hat{m}_1 - \frac{1}{p} \tr(\Sigma_n)\right)
	+ o(1) \\
	&= o_{\mathbb{P}}(1).
	\end{split}
\end{align}
Hence, Theorem~\ref{consestnonran} together with \eqref{Dickergen} generalizes 
\citet[Proposition~2(i)]{Dicker2014} in the case where \(p = o(n^{5/4})\).

\begin{acks}[Acknowledgments]

This research was supported by the Austrian Science Fund (FWF): I~5484-N, as part of the Research Unit 5381 of the German Research Foundation and by the Aarhus University Research Foundation (AUFF, 47221 and 47388). 
\end{acks}

\bibliographystyle{imsart-nameyear}
\bibliography{bibliography}       

\newpage

%%%%%%%%%%%%%%%%%%%%%%%%%%%%%%%%%%%%%%%%%%%%%%
%% Funding information, if any,             %%
%% should be provided in the                %%
%% funding section.                         %%
%%%%%%%%%%%%%%%%%%%%%%%%%%%%%%%%%%%%%%%%%%%%%%

%%%%%%%%%%%%%%%%%%%%%%%%%%%%%%%%%%%%%%%%%%%%%%
%% Supplementary Material, including data   %%
%% sets and code, should be provided in     %%
%% {supplement} environment with title      %%
%% and short description. It cannot be      %%
%% available exclusively as external link.  %%
%% All Supplementary Material must be       %%
%% available to the reader on Project       %%
%% Euclid with the published article.       %%
%%%%%%%%%%%%%%%%%%%%%%%%%%%%%%%%%%%%%%%%%%%%%%

\appendix

\section{Proofs of Section~\ref{sec:FiniteSample}}\label{App:A}

\begin{proof}[Proof of Proposition~\ref{prop1}]
Recall that $A = I_p - \eta(\hat{\Sigma}_n+\lambda I_p)$.  
For notational convenience, we write $\hat{\beta}_t$ in place of $\hat{\beta}_t(\lambda,\eta)$ when the dependence on $(\lambda,\eta)$ is clear from context.  
For $t=1$, $\eta>0$ and $\lambda>0$, we have 
\begin{align*}
	\hat{\beta}_1 
	&= \hat{\beta}_0 - \eta \, \nabla L(\hat{\beta}_0) \\
	&= \theta - \frac{\eta}{n} \bigl(-X^\top y + X^\top X \theta + \lambda n \theta\bigr) = \frac{\eta}{n} X^\top y + A \theta \\
	&= (I_p - A) \hat{\beta}_R(\lambda) + A \theta.
\end{align*}
Thus, the first equality holds for $t=1$.  We now prove that it holds for all 
$t \in \mathbb{N}$ by induction. Suppose that the claim holds for some $t\in \mathbb{N}$. 
Then, using the gradient descent recursion, we obtain for $t+1$:
\begin{align*}
	\hat{\beta}_{t+1} 
	&= \hat{\beta}_{t} - \eta \nabla L(\hat{\beta}_{t}) \\
	&= \hat{\beta}_{t}- \frac{\eta}{n} \bigl(-X^\top y + X^\top X \hat{\beta}_{t} + \lambda n \hat{\beta}_{t}\bigr)\\
	&=\frac{\eta}{n} X^\top y +A\hat{\beta}_{t}.
\end{align*}
Using $\frac{\eta}{n}X^\top y=(I_p-A)\hat{\beta}_R(\lambda)$ and the induction hypothesis in the previous display, we obtain 
\begin{align*}
	\hat{\beta}_{t+1}&=(I_p-A)\hat{\beta}_R(\lambda)+ A(I_p-A^t)\hat{\beta}_R(\lambda)+A^{t+1}\theta
	\\&=(I_p-A^{t+1})\hat{\beta}_R(\lambda)+A^{t+1}\theta
\end{align*}
which completes the induction step. Finally, since $\lim_{\lambda \downarrow 0}\hat{\beta}_R(\lambda)=\hat{\beta}_R(0)$ the result extends to the case $\lambda=0$.
\end{proof}	

    \begin{proof}[Proof of Lemma~\ref{LemforTh1}]
Since $\theta = 0$, the representation in \eqref{repGD} of $\hat{\beta}_t(\lambda,\eta)$ yields $\hat{\beta}_t(\lambda,\eta) = \tilde{\beta}_t(\lambda,\eta)=(I_p-A^t)\hat{\beta}_R(\lambda)$. Under the data model \eqref{set-1}, for $t \in \mathbb{N}$, $\lambda \ge 0$, and $\eta > 0$, we obtain
\begin{align*}
	\mathbb{E}\!\left(\hat{\beta}_t(\lambda,\eta) \mid X,\beta\right)
	&= \mathbb{E}\!\left((I_p - A^t) \hat{\beta}_R(\lambda) \mid X,\beta\right) \\
	&= (I_p - A^t)(X^\top X + n \lambda I_p)^\dagger X^\top X \beta \\
	&= (\hat{\Sigma}_n + \lambda I_p)^\dagger (I_p - A^t) \hat{\Sigma}_n \beta,
\end{align*}
where in the last equality we use that all matrices involved are simultaneously 
diagonalizable and therefore commute.
We define 
\begin{equation*}
C_t\coloneqq (\hat{\Sigma}_n + \lambda I_p)^\dagger (I_p - A^t)\hat{\Sigma}_n.
\end{equation*} 
Using $\beta-\mathbb{E}\!\left(\tilde{\beta}_t(\lambda,\eta)\mid X,\beta\right)=(I_p-C_t)\beta$, 
the fact that $(I_p-C_t)$ is symmetric, and assumption~\eqref{set-2}, we obtain
\begin{align}\label{B_t}
	\begin{split}
		B_\Sigma(t,\lambda,\eta)
		&= \mathbb{E} \Bigl[
		\bigl\|
		\mathbb{E} \bigl[\tilde{\beta}_t(\lambda,\eta) \mid X,\beta\bigr] - \beta
		\bigr\|_{\!\Sigma}^{2}
		\,\Bigm|\, X
		\Bigr] \\
		&= \mathbb{E} \Bigl[
		\beta^\top (I_p - C_t) \Sigma (I_p - C_t) \beta
		\,\Bigm|\, X
		\Bigr] \\
		&= \frac{\tau^2}{p} \tr\Bigl(\Sigma (I_p - C_t)^2 \Bigr).
	\end{split}
\end{align}
Similarly, we define
\begin{equation*}
D_t\coloneqq \frac{1}{n}(\hat{\Sigma}_n + \lambda I_p)^\dagger (I_p - A^t)X^\top.
\end{equation*}
Noting that $\tilde{\beta}_t(\lambda,\eta)-\mathbb{E}\!\left(\tilde{\beta}_t(\lambda,\eta)\mid X,\beta\right)=D_t u$, we obtain under the data model \eqref{set-1} that
\begin{align}\label{V_t}
	\begin{split}
		V_\Sigma(t,\lambda,\eta)
		&= \operatorname{tr}\Bigl\{
		\Sigma\,
		\mathbb{E} \Bigl[
		\operatorname{Cov}\!\left(\tilde{\beta}_t(\lambda,\eta) \mid X,\beta\right)
		\,\Bigm|\, X
		\Bigr]
		\Bigr\} \\
		&= \tr\bigl(\Sigma \, \mathbb{E}(D_t u u^\top D_t^\top \mid X) \bigr) \\
		&= \sigma^2 \tr(\Sigma D_t D_t^\top).
	\end{split}
\end{align}
We write $P=P(\lambda)=I_p-(\hat{\Sigma}_n+\lambda I_p)^\dagger (\hat{\Sigma}_n+\lambda I_p)$. Note that in the case $\lambda = 0$, $P=I_p-\hat{\Sigma}_n^\dagger\hat{\Sigma}_n=I_p-X^\dagger X$ is the orthogonal projector onto $\mathcal{N}(X)$; for $\lambda>0$, $P$ is the zero matrix. Therefore,
\begin{align}\label{I-C_t}
	\begin{split}
		I_p - C_t
		&= I_p - (\hat{\Sigma}_n + \lambda I_p)^\dagger (I_p - A^t)\hat{\Sigma}_n \\
		&= I_p - (\hat{\Sigma}_n + \lambda I_p)^\dagger 
		(\hat{\Sigma}_n + \lambda I_p - \lambda I_p - A^t \hat{\Sigma}_n) \\
		&= P + (\hat{\Sigma}_n + \lambda I_p)^\dagger (\lambda I_p + A^t \hat{\Sigma}_n).
	\end{split}
\end{align}
Next, we consider
\begin{align}\label{D_t^2}
	\begin{split}
		D_t D_t^\top
		&= \frac{1}{n^2} (\hat{\Sigma}_n + \lambda I_p)^\dagger (I_p - A^t) X^\top X (I_p - A^t) (\hat{\Sigma}_n + \lambda I_p)^\dagger \\
		&= \frac{1}{n} \bigl((\hat{\Sigma}_n + \lambda I_p)^\dagger\bigr)^2 (I_p - A^t)^2 \hat{\Sigma}_n,
	\end{split}
\end{align}
where in the second equality we use that all involved matrices are simultaneously diagonalizable. Combining the decomposition in \eqref{decA1} with \eqref{B_t}--\eqref{D_t^2}, we obtain
\begin{align*}
		\mathrm{R}_\Sigma(\hat{\beta}_t(\lambda,\eta))&= B_\Sigma(t,\lambda,\eta) + V_\Sigma(t,\lambda,\eta) \\
		&= \tr\Biggl(
		\Sigma \Bigl(
		\frac{\tau^2}{p} \bigl(P + (\hat{\Sigma}_n + \lambda I_p)^\dagger (\lambda I_p + A^t \hat{\Sigma}_n)^2\bigr)
		+ \frac{\sigma^2}{n} \bigl((\hat{\Sigma}_n + \lambda I_p)^\dagger\bigr)^2 (I_p - A^t)^2 \hat{\Sigma}_n
		\Bigr)
		\Biggr)
		\\&=\tr(\Sigma E_t),
\end{align*}
where
\begin{equation}\label{E_t}
E_t = 
\frac{\tau^2}{p} \Bigl( P + (\hat{\Sigma}_n + \lambda I_p)^\dagger (\lambda I_p + A^t \hat{\Sigma}_n)^2 \Bigr)
+ \frac{\sigma^2}{n} \Bigl( (\hat{\Sigma}_n + \lambda I_p)^\dagger \Bigr)^2 (I_p - A^t)^2 \hat{\Sigma}_n.
\end{equation}
We denote by $\{e_i\}_{i=1}^p$ the eigenvalues of $E_t$.  
Note that all matrices involved in \eqref{E_t} are simultaneously diagonalizable.  
Therefore, in the basis of common eigenvectors $E_t$ is diagonal and its eigenvalues, in the case $\lambda > 0$, can be written as
\begin{align*}
	e_i=e_i(t,\lambda,\eta) 
	= 	\begin{cases}\frac{\tau^2}{p} \left( \frac{\lambda + a_i^t s_i}{s_i + \lambda} \right)^2
	+ \frac{\sigma^2 }{n} \left( \frac{(1 - a_i^t)^2 s_i}{(s_i + \lambda)^2} \right),  & s_i > 0,\\[1.5mm]
	\frac{\tau^2}{p}, & s_i = 0,
	\end{cases}
\end{align*}
where $a_i = a_i(\lambda,\eta) = 1 - \eta (s_i + \lambda)$ are the eigenvalues of $A=I_p-\eta(\hat{\Sigma}_n+\lambda I_p)$ and $s_i$ are the eigenvalues of $\hat{\Sigma}_n$.  

In the case $\lambda = 0$, we have 
$(\hat{\Sigma}_n + \lambda I_p)^\dagger = \hat{\Sigma}_n^\dagger$, 
where the eigenvalues of $\hat{\Sigma}_n^\dagger$ are $1/s_i$ for $s_i > 0$ and $0$ for $s_i = 0$. Since in this case $P$ is the orthogonal projector onto $\mathcal{N}(X)$, its 
eigenvalues are $0$ for $s_i > 0$ and $1$ for $s_i = 0$. The eigenvalues of $E_t$ simplify in this case to
\begin{align*}
e_i(t,0,\eta) =
\begin{cases}
	\frac{\tau^2}{p}\, a_i^{2t}
	\;+\;
	\frac{\sigma^2 }{n}\, \frac{(1 - a_i^t)^2}{s_i},
	& s_i > 0,\\[2mm]
	\frac{\tau^2}{p},
	& s_i = 0,
\end{cases}
\end{align*}
where $a_i = a_i(0,\eta) = 1 - \eta s_i$.
 Therefore,
for all $\lambda \geq0$, we have 
\begin{align}\label{repe_i}
	\begin{split}
		e_i=e_i(t,\lambda,\eta) 
	= 	\begin{cases}\frac{\tau^2}{p} \left( \frac{\lambda + a_i^t s_i}{s_i + \lambda} \right)^2
		+ \frac{\sigma^2 }{n} \left( \frac{(1 - a_i^t)^2 s_i}{(s_i + \lambda)^2} \right),  & s_i > 0,\\[2mm]
		\frac{\tau^2}{p}, & s_i = 0.
	\end{cases}
\end{split}
\end{align}
It remains to show that, for $s_i>0$, 
\begin{equation*}
	e_i=\frac{1}{p}\frac{\sigma^2\gamma_n}{s_i+\lambda^*} +\frac{1}{p}\frac{s_i}{(s_i+\lambda)^2}\biggl(\frac{(\frac{\lambda}{\lambda^*}-1) \sigma^2\gamma_n}{\sqrt{\tau^2(s_i+\lambda^*)}}+\sqrt{\tau^2(s_i+\lambda^*)}a_i^t\biggr)^2,
\end{equation*}
where $\lambda^*=\frac{\sigma^2}{\tau^2}\gamma_n$. 
Using the representation in \eqref{repe_i}, we can decompose $e_i$ into three terms:
\begin{align*}
	e_i 
	&= \frac{\tau^2}{p} \left( \frac{\lambda + a_i^t s_i}{s_i + \lambda} \right)^2
	+ \frac{\sigma^2 \gamma_n}{p} \left( \frac{(1 - a_i^t)^2 s_i}{(s_i + \lambda)^2} \right) \\[1mm]
	&= \frac{\tau^2}{p} \left( 
	\frac{\lambda^2}{(s_i+\lambda)^2}
	+ \frac{2 a_i^t s_i \lambda}{(s_i+\lambda)^2}
	+ \frac{s_i^2 a_i^{2t}}{(s_i+\lambda)^2}
	\right)
	+ \frac{\sigma^2 \gamma_n}{p} \left(
	\frac{s_i}{(s_i+\lambda)^2}
	- \frac{2 a_i^t s_i}{(s_i+\lambda)^2}
	+ \frac{s_i a_i^{2t}}{(s_i+\lambda)^2}
	\right) \\[1mm]
	&= 
	\frac{\tau^2}{p} \frac{\lambda^2}{(s_i+\lambda)^2}
	+ \frac{\sigma^2 \gamma_n}{p} \frac{s_i}{(s_i+\lambda)^2} 
	\hspace{1em} (*) \\[0.5mm]
	&\quad + \frac{2\tau^2}{p} \frac{a_i^t s_i \lambda}{(s_i+\lambda)^2}
	- \frac{2\sigma^2 \gamma_n}{p} \frac{a_i^t s_i}{(s_i+\lambda)^2} 
	\hspace{1em} (**) \\[0.5mm]
	&\quad + \frac{\tau^2}{p} \frac{s_i^2 a_i^{2t}}{(s_i+\lambda)^2}
	+ \frac{\sigma^2 \gamma_n}{p} \frac{s_i a_i^{2t}}{(s_i+\lambda)^2} 
	\hspace{1em} (***).
\end{align*}
To simplify the first term in the decomposition, we use $\tau^2=\frac{\sigma^2\gamma_n}{\lambda^*}$ and write
\begin{align}
	(*) &= \frac{\tau^2}{p} \frac{\lambda^2}{(s_i+\lambda)^2} + \frac{\sigma^2\gamma_n}{p} \frac{s_i}{(s_i+\lambda)^2} \notag \\
	&= \frac{\frac{\sigma^2\gamma_n}{p} \left( \frac{\lambda^2}{\lambda^*} + s_i \right) (s_i+\lambda^*)}{(s_i+\lambda)^2 (s_i+\lambda^*)} \notag \\
	&= \frac{\frac{\sigma^2\gamma_n}{p} \left( \lambda^2 + \frac{\lambda^2}{\lambda^*} s_i + s_i \lambda^* + s_i^2 + 2 s_i \lambda - 2 s_i \lambda \right)}{(s_i+\lambda)^2 (s_i+\lambda^*)} \notag \\
	&= \frac{1}{p} \left( \frac{\sigma^2\gamma_n}{s_i+\lambda^*} + \frac{\sigma^2\gamma_n \left( \frac{\lambda^2}{\lambda^*} s_i + s_i \lambda^* - 2 s_i \lambda \right)}{(s_i+\lambda)^2 (s_i+\lambda^*)} \right) \notag \\
	&= \frac{1}{p} \left( \frac{\sigma^2\gamma_n}{s_i+\lambda^*} + \frac{\sigma^2\gamma_n s_i \lambda^*}{(s_i+\lambda)^2 (s_i+\lambda^*)} \left( \frac{\lambda}{\lambda^*} - 1 \right)^2 \right) \notag \\
	&= \frac{1}{p} \left( \frac{\sigma^2\gamma_n}{s_i+\lambda^*} + \frac{(\sigma^2\gamma_n)^2 s_i}{\tau^2 (s_i+\lambda)^2 (s_i+\lambda^*)} \left( \frac{\lambda}{\lambda^*} - 1 \right)^2 \right). \label{eq:star}
\end{align}
For the second term in the decomposition of $e_i$, we again use 
$\tau^2 = \frac{\sigma^2 \gamma_n}{\lambda^*}$ and obtain
\begin{align}
	(**) 
	&= \left(
	\frac{2\tau^2}{p} \frac{a_i^t s_i \lambda}{(s_i+\lambda)^2}
	- \frac{2\sigma^2\gamma_n}{p} \frac{a_i^t s_i}{(s_i+\lambda)^2}
	\right) \notag\\
	&= \frac{2}{p} \frac{\sigma^2\gamma_n\, s_i}{(s_i+\lambda)^2}
	\biggl( \frac{\lambda}{\lambda^*} - 1 \biggr)a_i^t .
	\label{eq:star2}
\end{align}
For the third term, we use $\tau^2 \lambda^* = \sigma^2 \gamma_n$ and compute
\begin{align}
	(***) 
	&= \frac{\tau^2}{p} \frac{s_i^2 a_i^{2t}}{(s_i+\lambda)^2}
	+ \frac{\sigma^2\gamma_n}{p} \frac{s_i a_i^{2t}}{(s_i+\lambda)^2} \notag\\
	&= \frac{1}{p} \frac{s_i}{(s_i+\lambda)^2}
	\,\tau^2 (s_i + \lambda^*)\, a_i^{2t}.\label{eq:star3}
\end{align}
Hence, using \eqref{eq:star}--\eqref{eq:star3}, we obtain the following expression for $e_i$:
\begin{align}
	e_i 
	&= \frac{1}{p}\biggl(
	\frac{\sigma^2\gamma_n}{s_i+\lambda^*}
	+ \frac{(\sigma^2\gamma_n)^2 s_i}{\tau^2 (s_i+\lambda)^2 (s_i+\lambda^*)}
	\biggl( \frac{\lambda}{\lambda^*} - 1 \biggr)^{\!2}
	\biggr) \notag
	\\ &\quad
	+ \frac{2}{p}\frac{\sigma^2\gamma_n s_i}{(s_i+\lambda)^2}
	\biggl( \frac{\lambda}{\lambda^*} - 1 \biggr) a_i^t
	+ \frac{1}{p}\frac{s_i}{(s_i+\lambda)^2}
	\tau^2 (s_i+\lambda^*) a_i^{2t} \notag
	\\
	&= \frac{1}{p}\frac{\sigma^2\gamma_n}{s_i+\lambda^*}
	\;+\;
	\frac{1}{p} \frac{s_i}{(s_i+\lambda)^2}
	\biggl(
	\frac{\bigl(\frac{\lambda}{\lambda^*}-1\bigr)\sigma^2\gamma_n}
	{\sqrt{\tau^2 (s_i+\lambda^*)}}
	\;+\;
	\sqrt{\tau^2 (s_i+\lambda^*)}\, a_i^t
	\biggr)^{\!2}.
\end{align}
Therfore, for all $\lambda\geq 0$, the eigenvalues $\{e_i\}_{i=1}^p$ of $E_t$ can be written as
\begin{align*}
	\begin{split}
		e_i=e_i(t,\lambda,\eta) 
		= 	\begin{cases}\frac{1}{p}\frac{\sigma^2\gamma_n}{s_i+\lambda^*}
			\;+\;
			\frac{1}{p} \frac{s_i}{(s_i+\lambda)^2}
			\biggl(
			\frac{\bigl(\frac{\lambda}{\lambda^*}-1\bigr)\sigma^2\gamma_n}
			{\sqrt{\tau^2 (s_i+\lambda^*)}}
			\;+\;
			\sqrt{\tau^2 (s_i+\lambda^*)}\, a_i^t
			\biggr)^{\!2},  & s_i > 0,\\[2 mm]
			 \frac{\tau^2}{p}, & s_i = 0.
		\end{cases}
	\end{split}
\end{align*}
This completes the first statement of the lemma.

For the second statement of the lemma, consider for $\lambda \ge 0$ the decomposition
\begin{align*}
	\beta-\hat{\beta}_R(\lambda) 
	&= \bigl(I_p - (X^\top X + n\lambda I_p)^\dagger X^\top X \bigr)\beta
	\;-\; (X^\top X + n\lambda I_p)^\dagger X^\top u \\[1mm]
	&= \bigl(I_p - (\hat{\Sigma}_n + \lambda I_p)^\dagger \hat{\Sigma}_n \bigr)\beta
	\;-\; \frac{1}{n}(\hat{\Sigma}_n + \lambda I_p)^\dagger X^\top u \\[1mm]
	&= \bigl(P + \lambda (\hat{\Sigma}_n + \lambda I_p)^\dagger\bigr)\beta
	\;-\; \frac{1}{n}(\hat{\Sigma}_n + \lambda I_p)^\dagger X^\top u.
\end{align*}
Note that $P + \lambda (\hat{\Sigma}_n + \lambda I_p)^\dagger$ is a symmetric matrix. Using the assumption on the data model in \eqref{set-1} and \eqref{set-2}, the previous display implies that the generalization error of 
$\hat{\beta}_R(\lambda)$ for $\lambda \ge 0$ is given by
\begin{align}\label{Fdec}
	\begin{split}
	\mathrm{R}_\Sigma(\hat{\beta}_R(\lambda)) 
	&= \mathbb{E} \bigl[ \|\hat{\beta}_R(\lambda) - \beta\|_\Sigma^2 \,\big|\, X \bigr] \\[1mm]
	&= \mathbb{E} \Bigl[ 
	\beta^\top \bigl( P + \lambda (\hat{\Sigma}_n + \lambda I_p)^\dagger \bigr) 
	\Sigma 
	\bigl( P + \lambda (\hat{\Sigma}_n + \lambda I_p)^\dagger \bigr) 
	\beta \,\Big|\, X \Bigr] \\[1mm]
	&\quad + \frac{1}{n^2} \mathbb{E} \Bigl[ 
	u^\top X (\hat{\Sigma}_n + \lambda I_p)^\dagger 
	\Sigma (\hat{\Sigma}_n + \lambda I_p)^\dagger X^\top u 
	\,\Big|\, X \Bigr] \\[1mm]
	&= \frac{\tau^2}{p} \, \tr \Bigl( \Sigma \bigl[ P + \lambda (\hat{\Sigma}_n + \lambda I_p)^\dagger \bigr]^2 \Bigr) 
	+ \frac{\sigma^2}{n} \, \tr \Bigl( \Sigma \bigl[ (\hat{\Sigma}_n + \lambda I_p)^\dagger \bigr]^2 \hat{\Sigma}_n \Bigr).
	\end{split}
\end{align}
Using \eqref{Fdec}, we obtain
\begin{align*}
	\mathrm{R}_\Sigma(\hat{\beta}_R(\lambda)) 
	&= \tr \left( \Sigma \Bigl\{\frac{\tau^2}{p} \bigl[ P + \lambda (\hat{\Sigma}_n + \lambda I_p)^\dagger \bigr]^2 
	+ \frac{\sigma^2}{n} \bigl[ (\hat{\Sigma}_n + \lambda I_p)^\dagger \bigr]^2 \hat{\Sigma}_n \Bigr\} \right) \notag \\
	&= \tr(\Sigma F),
\end{align*}
where 
\begin{equation}\label{Fmat}
	F = \frac{\tau^2}{p} \bigl[ P + \lambda (\hat{\Sigma}_n + \lambda I_p)^\dagger \bigr]^2 
	+ \frac{\sigma^2}{n} \bigl[ (\hat{\Sigma}_n + \lambda I_p)^\dagger \bigr]^2 \hat{\Sigma}_n.
\end{equation}
We denote by $\{f_i\}_{i=1}^p$ the eigenvalues of $F$.  
Note that all matrices involved in \eqref{Fmat} are simultaneously diagonalizable.  
Hence, in the basis of common eigenvectors, $F$ is diagonal and its eigenvalues can be expressed as
\begin{align*}
	f_i = f_i(\lambda) =
	\begin{cases}
		 \frac{\tau^2}{p} \frac{\lambda^2 + s_i}{(s_i + \lambda)^2}, & s_i > 0,\\[2 mm]
		\frac{\tau^2}{p}, & s_i = 0,
	\end{cases}
\end{align*}
for $\lambda \ge 0$. Using $\frac{\sigma^2}{n} = \frac{\sigma^2 \gamma_n}{p}$ and the derivation in \eqref{eq:star}, we obtain for $s_i > 0$,
\begin{equation*}
	f_i = (*) = \frac{1}{p} \frac{\sigma^2 \gamma_n}{s_i + \lambda^*} 
	+ \frac{1}{p} \frac{s_i}{(s_i + \lambda)^2} \frac{\bigl(\frac{\lambda}{\lambda^*} - 1\bigr)^2 (\sigma^2 \gamma_n)^2}{\tau^2 (s_i + \lambda^*)}.
\end{equation*}
Therefore, we conclude that for $\lambda \geq 0$ the eigenvalues $\{f_i\}_{i=1}^p$ of $F$ are
\begin{align*}
	f_i = f_i(\lambda) =
	\begin{cases}
		 \frac{1}{p} \frac{\sigma^2 \gamma_n}{s_i + \lambda^*} 
		+ \frac{1}{p} \frac{s_i}{(s_i + \lambda)^2} \frac{\bigl(\frac{\lambda}{\lambda^*} - 1\bigr)^2 (\sigma^2 \gamma_n)^2}{\tau^2 (s_i + \lambda^*)}, & s_i > 0,\\[2 mm]
		 \frac{\tau^2}{p}, & s_i = 0.
	\end{cases}
\end{align*}

\end{proof}

\begin{proof}[Proof of Proposition~\ref{corTh1}]
First, note that by Lemma~\ref{LemforTh1}\ref{dec2} we have 
$\mathrm{R}_{\Sigma}(\hat{\beta}_R(\lambda^*)) = \tr(\Sigma F(\lambda^*))$, 
and the eigenvalues of $F$ are given by 
$\frac{1}{p}\frac{\sigma^2 \gamma_n}{s_i + \lambda^*}$ for 
$i \in \{1,\dots,p\}$. 
Using Lemma~\ref{LemforTh1}\ref{dec1} and the decomposition in \eqref{decA1} yields
\begin{align*}
	\mathrm{R}_\Sigma(\hat{\beta}_t(\lambda,\eta))
	&= \mathrm{R}_\Sigma(\tilde{\beta}_t(\lambda,\eta))	+ \mathbb{E}\!\left(\theta^\top A^t \Sigma A^t \theta \,\middle|\, X\right) \\
	&= \tr(\Sigma E_t)
	+ \mathbb{E}\!\left(\theta^\top A^t \Sigma A^t \theta \,\middle|\, X\right).
\end{align*}
Recall that for all $t \in \mathbb{N}$, $\lambda \ge 0$ and $\eta > 0$ the eigenvalues of $E_t$ are given by
\begin{equation*}
	e_i(t,\lambda,\eta)
	= \begin{cases}
		\frac{1}{p}\frac{\sigma^2\gamma_n}{s_i + \lambda^*}
		+
		\frac{1}{p} \frac{s_i}{(s_i + \lambda)^2} 
		\biggl(
		\frac{
			\bigl(\frac{\lambda}{\lambda^*}-1\bigr)\sigma^2 \gamma_n
		}{
			\sqrt{\tau^2 (s_i + \lambda^*)}
		}
		+
		\sqrt{\tau^2 (s_i + \lambda^*)}\, a_i^t
		\biggr)^{\!2},
		& s_i > 0,
		\\[2 mm]
		 \frac{\tau^2}{p},
		& s_i = 0.
	\end{cases}
\end{equation*}
	In the case $s_i > 0$, the first term of $e_i(t,\lambda,\eta)$ is also given by $\frac{1}{p}\frac{\sigma^2\gamma_n}{s_i+\lambda^*}$ and the second term is clearly non-negative, because of the square. Hence, for all  $t \in \mathbb{N}$, $\lambda \ge 0$ and $\eta > 0$, we have 
	\begin{align}\label{e_igeqf_i}
		e_i(\lambda,t,\eta)\geq f_i(\lambda^*)\quad i\in\{1,\dots,p\}.
	\end{align} 
In the case $s_i = 0$, the inequality in \eqref{e_igeqf_i} is trivially satisfied.
 Furthermore, note that for all  $t \in \mathbb{N}$, $\lambda \ge 0$ and $\eta > 0$, we have
	$\theta^\top A^t \Sigma A^t \theta \;\ge\; 0$,
	since $\Sigma$ is positive semidefinite.
	 Since both matrices $E_t(\lambda,\eta)$ and $F(\lambda^*)$ depend only on $I_p$ and $\hat{\Sigma}_n$, we may write $E_t(\lambda,\eta)=U\Lambda_{t,\lambda,\eta} U^\top$ and $F=U\Lambda^* U^\top$, where $U$ is the orthogonal matrix whose columns are the normalized eigenvectors of $\hat{\Sigma}_n$ and $\Lambda_{t,\lambda,\eta}$ and $\Lambda^*$ are diagonal matrices containing the corresponding eigenvalues $\{e_i(t,\lambda,\eta)\}_{i=1}^p$ and $\{f_i(\lambda^*)\}_{i=1}^p$. Therefore, using \eqref{e_igeqf_i} and the fact that $U^\top \Sigma U$ is positive semidefinite, we obtain, for all  $t \in \mathbb{N}$, $\lambda \ge 0$ and $\eta > 0$,
	 \begin{align}\label{diffGDoptR}
	 	\begin{split}
	 		\mathrm{R}_\Sigma(\hat{\beta}_t(\lambda,\eta))
	 		- \mathrm{R}_\Sigma(\hat{\beta}_R(\lambda^*))
	 		&= \tr\bigl(\Sigma E_t(\lambda,\eta)\bigr) - \tr\bigl(\Sigma F(\lambda^*)\bigr) \\
	 		&= \tr\bigl[U^\top \Sigma U \, (\Lambda_{t,\lambda,\eta} - \Lambda^*)\bigr]
	 		\;\ge\; 0.
	 	\end{split}
	 \end{align}
	This concludes the first statement.
	
For the second statement, it suffices to show that, for all $\eta > 0$ and $t \in \mathbb{N}$,
\begin{align}
	\mathrm{R}_\Sigma(\hat{\beta}_t(0,\eta)) - \mathrm{R}_\Sigma(\hat{\beta}_R(\lambda^*)) 
	&= \operatorname{tr}\bigl(\Sigma E_t\bigr) - \operatorname{tr}\bigl(\Sigma F(\lambda^*)\bigr) \notag \\
	&= \operatorname{tr} \bigl[ U^\top \Sigma U \, (\Lambda_{t,0,\eta} - \Lambda^*) \bigr] \notag \\
	&= \sum_{l=1}^p (U^\top \Sigma U)_{ll} \, \bigl[ e_l(t,0, \eta) - f_l(\lambda^*) \bigr] \;>\; 0,
\end{align}
under the assumption that $\hat{\Sigma}_n$ has at least two positive and distinct eigenvalues, $s_j, s_k > 0$ with $s_j \neq s_k$, and that the corresponding diagonal entries $(U^\top \Sigma U)_{jj}$ and $(U^\top \Sigma U)_{kk}$ are strictly positive.
We prove this by contradiction. Suppose that
\begin{equation}\label{contra}
	\mathrm{R}_\Sigma(\hat{\beta}_t(0,\eta)) - \mathrm{R}_\Sigma(\hat{\beta}_R(\lambda^*)) 
	= \sum_{l=1}^p (U^\top \Sigma U)_{ll} \, \bigl[ e_l(t,0, \eta) - f_l(\lambda^*) \bigr] = 0.
\end{equation}
Note that $(U^\top \Sigma U)_{ll} \ge 0$ for all $l \in \{1,\dots,p\}$, and that $(U^\top \Sigma U)_{jj}$ and $(U^\top \Sigma U)_{kk}$ are strictly positive by assumption. Recalling \eqref{e_igeqf_i}, in order for \eqref{contra} to hold, it is necessary that
\[
e_j(t,0,\eta) - f_j(\lambda^*) = 0 
\quad \text{and} \quad 
e_k(t,0,\eta) - f_k(\lambda^*) = 0.
\]
For $s_j > 0$, we have
\begin{align*}
	e_j(t,0,\eta) - f_j(\lambda^*) 
	&= \frac{1}{p} \frac{1}{s_j} \biggl(
	- \frac{\sigma^2 \gamma_n}{\sqrt{\tau^2 (s_j + \lambda^*)}}
	+ \sqrt{\tau^2 (s_j + \lambda^*)}\, a_j^t
	\biggr)^2,
\end{align*}
where $a_j = 1 - \eta s_j$. The term inside the square is zero if and only if
\begin{align*}
	\sqrt{\tau^2 (s_j + \lambda^*)} \, a_j^t = \frac{\sigma^2 \gamma_n}{\sqrt{\tau^2 (s_j + \lambda^*)}}
	\quad \Longleftrightarrow \quad
	a_j^t = \frac{\lambda^*}{s_j + \lambda^*} 
	\quad \Longleftrightarrow \quad
	\eta = \frac{1}{s_j} \left( 1 - \biggl( \frac{\lambda^*}{s_j + \lambda^*} \biggr)^{1/t} \right).
\end{align*}
Similarly, for $e_k(t,0,\eta) - f_k(\lambda^*) = 0$, we would require
\[
\eta = \frac{1}{s_k} \left( 1 - \biggl( \frac{\lambda^*}{s_k + \lambda^*} \biggr)^{1/t} \right).
\]
Since $s_j \neq s_k$, the two conditions on $\eta$ cannot hold simultaneously, leading to a contradiction and thereby establishing the claim.

\end{proof}

 \begin{proof}[Proof of Theorem~\ref{Th1}]
We note that under the assumption on $\theta$, the decomposition of RGD given in \eqref{decA1} still remains valid, that is,
\begin{equation*}
	\mathrm{R}_{\Sigma}(\hat{\beta}_t(\lambda,\eta))
	= \mathrm{R}_\Sigma(\tilde{\beta}_t(\lambda,\eta)) + \mathbb{E}(\theta^\top A^t \Sigma A^t \theta).
\end{equation*}  
Note that, $\mathrm{R}_{\Sigma}(\hat{\beta}_t(\lambda,\eta))$ is monotonically decreasing in $t$ for all $\lambda \ge \lambda^* = \frac{\sigma^2}{\tau^2} \gamma_n$ and corresponding step-sizes $\eta \in (0, 1/(s_1+\lambda)]$, provided that both $\mathrm{R}(\tilde{\beta}_t(\lambda,\eta))$ and $\mathbb{E}(\theta^\top A^t \Sigma A^t \theta)$ are monotonically decreasing in $t$.

For $\mathbb{E}(\theta^\top A^t \Sigma A^t \theta)$, we prove the stronger statement that, under the assumption on $\theta$ stated in the theorem,  $\mathbb{E}(\theta^\top A^t \Sigma A^t \theta)$ is monotonically decreasing in $t$ for all $\lambda \ge 0$ and corresponding step sizes $\eta \in (0, 2/(s_1+\lambda))$, where $s_1$ denotes the largest eigenvalue of $\hat{\Sigma}_n$. We have,
	\begin{align*}
		\mathbb{E}\!\left[
		\theta^\top\!\left(
		A^{t}\Sigma A^{t} - A^{t+1}\Sigma A^{t+1}
		\right)\theta
		\,\middle|\, X
		\right]
		&= \rho\, \mathrm{tr}\!\left(
		A^{t}\Sigma A^{t} - A^{t+1}\Sigma A^{t+1}
		\right) \\
		&= \rho\bigl(
		\mathrm{tr}(\Sigma A^{2t})
		- \mathrm{tr}(\Sigma A^{2(t+1)})
		\bigr).
	\end{align*}
Recall that $A=I_p-\eta(\hat{\Sigma}_n+\lambda I_p)$ is symmetric and the matrices involved in its definition are simultaneously diagonalizable. Therefore, we can write $A = U \Lambda_1 U^\top$,
where $U$ is the orthogonal matrix whose columns are the normalized eigenvectors of $\hat{\Sigma}_n$ and $\Lambda_1$ is a diagonal matrix containing the corresponding eigenvalues
$a_j = a_j(\lambda,\eta) = 1 - \eta (s_j + \lambda)$ for $j \in \{ 1,\dots,p\}$.
From Remark~\ref{rem:GD}, we have $|a_j| < 1$ whenever $\eta \in (0, 2/(s_1 + \lambda))$. Using these considerations together with the previous display, we obtain
\begin{align*}
	\mathbb{E}\!\left[
	\theta^\top \!\left( A^t \Sigma A^t - A^{t+1} \Sigma A^{t+1} \right) \theta
	\,\middle|\, X
	\right]
	&= \rho \bigl[ \mathrm{tr}\bigl\{\Sigma\, U(\Lambda_1^{2t} -  \Lambda^{2(t+1)})\, U^\top\bigr\} \bigr] \\
	&= \rho \, \mathrm{tr} \Big[ U^\top\,\Sigma \, U \big( \Lambda_1^{2t} - \Lambda_1^{2(t+1)} \big)  \Big] \ge 0,
\end{align*}
where the inequality follows because each $(a_j)^{2t} - (a_j)^{2(t+1)} \ge 0$ for $j\in \{1,\dots,p\}$ and $U^\top \Sigma U$ is positive semidefinite, which ensures that its diagonal entries are nonnegative.

It remains to show that $\mathrm{R}(\tilde{\beta}_t(\lambda,\eta))$ is monotonically decreasing in $t$ for all $\lambda\geq 0$ and corresponding step-sizes $\eta\in (0,2/(s_1+\lambda))$. Using Lemma~\ref{LemforTh1}\ref{dec1}, we have for all $t\in \mathbb{N}$, $\lambda\geq 0$ and $\eta>0$ that $\mathrm{R}(\tilde{\beta}_t(\lambda,\eta))=\tr(\Sigma E_t)$, where
\begin{equation*}
	E_t=E_t(\lambda,\eta)= \frac{\tau^2}{p} \Bigl[ P + (\hat{\Sigma}_n + \lambda I_p)^\dagger (\lambda I_p + A^t \hat{\Sigma}_n)^2 \Bigr]
	+ \frac{\sigma^2}{n} \Bigl[ (\hat{\Sigma}_n + \lambda I_p)^\dagger \Bigr]^2 (I_p - A^t)^2 \hat{\Sigma}_n
\end{equation*}
and $P=P(\lambda)=I_p-(\hat{\Sigma}_n+\lambda I_p)^{\dagger}(\hat{\Sigma}_n+\lambda I_p)$. Since, for any fixed $t \in \mathbb{N}$, the matrices $E_t$ and $E_{t+1}$ depend only on $I_p$ and $\hat{\Sigma}_n$, we may write
$E_t=U\Lambda_tU^\top$ and $E_{t+1}=U\Lambda_{t+1}U^\top$, where $\Lambda_t$ and $\Lambda_{t+1}$ are diagonal matrices containing the corresponding eigenvalues $\{e_i(t,\lambda,\eta)\}_{i=1}^p$ and  $\{e_i(t+1,\lambda,\eta)\}_{i=1}^p$, respectively. We now show that each $e_i(t,\lambda,\eta)$, for $i\in\{ 1, \dots, p\}$, is monotonically decreasing in $t$ for all $\lambda \ge \lambda^* = \frac{\sigma^2}{\tau^2} \gamma_n$ and $\eta \in (0, 1/(s_1 + \lambda)]$. First, note that $(\lambda / \lambda^* - 1) \ge 0$ for all $\lambda \ge \lambda^*$. Second, since $a_i=a_i(\lambda,\eta) \in (0,1]$ for all $\lambda \ge 0$ and corresponding $\eta \in (0, 1/(s_1 + \lambda)]$, it follows that $a_i^t$ is monotonically decreasing in $t$. We recall that, for $i\in \{1,\dots,p\}$, the eigenvalues of $E_t$ are of the form 
\begin{equation*}
	e_i(t,\lambda,\eta) =
	\begin{cases}
		 \frac{1}{p} \frac{\sigma^2 \gamma_n}{s_i + \lambda^*} 
		+ \frac{1}{p} \frac{s_i}{(s_i + \lambda)^2} 
		\Biggl(
		\frac{\bigl(\frac{\lambda}{\lambda^*} - 1\bigr) \sigma^2 \gamma_n}{\sqrt{\tau^2 (s_i + \lambda^*)}}
	\;+\; \sqrt{\tau^2 (s_i + \lambda^*)}\, a_i^t
	\Biggr)^2, & s_i > 0, \\[2 mm]
	 \frac{\tau^2}{p}, & s_i = 0,
\end{cases}
\end{equation*}
where, in the case $s_i > 0$, the expression inside the square consists of two nonnegative terms whenever $\lambda \ge \lambda^*$ and $\eta \in (0, 1/(s_1+\lambda)]$. Moreover, one of these terms is monotonically decreasing in $t$, while the other does not depend on $t$, implying that the entire squared expression is also monotonically decreasing in $t$.

Consequently, we conclude that 
\begin{equation*}
e_i(t,\lambda,\eta) \ge e_i(t+1,\lambda,\eta),
\end{equation*}
for $\lambda \ge \lambda^*$ and $\eta \in (0, 1/(s_1 + \lambda)]$.
Since $U^\top \Sigma U$ is positive semidefinite, it follows that \[
\mathrm{R}_\Sigma(\tilde{\beta}_t(\lambda,\eta))-\mathrm{R}_\Sigma(\tilde{\beta}_{t+1}(\lambda,\eta)) = \mathrm{tr}(\Sigma E_t)-\mathrm{tr}(\Sigma E_{t+1})\geq 0,\] whenever $\lambda \ge \lambda^*$ and $\eta \in (0, 1/(s_1 + \lambda)]$. This completes the first statement.

Comparing the decompositions of the eigenvalues of $E_t$ and $F$ in
Lemma~\ref{LemforTh1}\ref{dec1} and Lemma~\ref{LemforTh1}\ref{dec2}, we observe that, for each $i\in \{1,\dots,p\}$,
$\lim_{t\to\infty} e_i(t,\lambda,\eta) = f_i(\lambda)$ for all $\lambda \ge 0$,
provided that $\eta \in (0, 2/(s_1+\lambda))$.  
Moreover,
\begin{align*}
	\lim_{t\to\infty} \mathbb{E}\!\left(\theta^\top A^{t}\Sigma A^{t}\theta \,\middle|\, X\right)
	= \rho \lim_{t\to\infty} \tr(\Sigma A^{2t})
	= 0,
\end{align*}
	since $\lim_{t\to\infty}A^t$ is equal to the zero matrix whenever $\eta \in (0,2/(s_{1}+\lambda))$. 
	Combining the last two arguments yields the second statement.
\end{proof}

\section{Proofs of Section~\ref{sec:tuning}}\label{App:B}

\begin{proof}[Proof of Lemma~\ref{denomnonzero}]
Denote by $s_1\geq s_2,\geq\dots\geq s_p\geq 0$ the oredered eigenvalues of $\hat{\Sigma}_n$ and recall that $n,p\geq 2$. For $n \ge p$, we have
\begin{equation*}
\gamma_n \hat m_1^{\,2}
= \gamma_n \left( \frac{1}{p} \sum_{i=1}^{p} s_i  \right)^{2}
\le \frac{1}{p} \sum_{i=1}^{p} s_i^2
= \hat m_2,
\end{equation*}
where the inequality follows from Jensen’s inequality and the fact that $\gamma_n \le 1$.
Equality can only hold when all eigenvalues of $\hat\Sigma_n$ coincide. As we show next, this event has probability zero.
Indeed, \citet[Theorem~1]{Okamoto1973} states that under assumption \ref{A2} and \ref{A3} with probability one, the non-zero eigenvalues of $\underline{\hat{\Sigma}}_n= n^{-1} Z\Sigma Z^\top=n^{-1}X X^\top$ are all distinct and the rank of $\underline{\hat{\Sigma}}_n$ is given by $n\wedge r\geq 1$, where $r$ is the rank of $\Sigma_n$. Because $\hat{\Sigma}_n$ and $\underline{\hat{\Sigma}}_n$ possess the same nonzero eigenvalues, the result carries over directly to $\hat{\Sigma}_n$. Hence, $\hat{\Sigma}_n$ has at least one nonzero eigenvalue and therefore the inequality in the preceding display is strict.

In the case $p > n$, we have
\begin{align*}
	\hat{m}_1^2  
	&= \left( \frac{1}{p} \sum_{i=1}^{p} s_i\right)^2 = \left( \frac{1}{\gamma_n} \frac{1}{n} \sum_{i=1}^{n} s_i \right)^2<  \left( \frac{1}{\gamma_n} \right)^2 \left( \frac{1}{n} \sum_{i=1}^{n} s_i^2 \right)= \frac{1}{\gamma_n} \left( \frac{1}{p} \sum_{i=1}^{p} s_i^2 \right)
	= \frac{1}{\gamma_n} \hat{m}_2,
\end{align*}
almost surely, where we have used the fact that $\hat{\Sigma}_n$ and $\underline{\hat{\Sigma}}_n$ share the same nonzero eigenvalues. The strict inequality follows from Jensen's inequality together with \citet[Theorem~1]{Okamoto1973}.

\end{proof}

\begin{proof}[Proof of Lemma~\ref{momqf1} and Lemma~\ref{momqf2}]
Recall that
\begin{equation*}
	\hat{\Sigma}_n = \frac{1}{n} \sum_{i=1}^{n} x_i x_i^\top
	= \frac{1}{n} \sum_{i=1}^{n} \Sigma_n^{1/2} z_i z_i^\top \Sigma_n^{1/2},
\end{equation*}
where $\{x_i\}_{i=1}^n$ are the rows of $X$ and $\{z_i\}_{i=1}^n$ are the rows of $Z$. 
For the first statements in Lemma~\ref{momqf1} and Lemma~\ref{momqf2}, we write
\begin{align}	\label{1stmomdectr}
\frac{1}{p^{1/2}} \tr(\hat{\Sigma}_n) - \frac{1}{p^{1/2}} \tr(\Sigma_n)
	&= \frac{1}{n} \sum_{i=1}^{n} \frac{1}{p^{1/2}} \Big( 
	\tr(\Sigma_n^{1/2} z_i z_i^\top \Sigma_n^{1/2}) - \tr(\Sigma_n) 
	\Big)\notag \\
	&= \frac{1}{n} \sum_{i=1}^{n} \frac{1}{p^{1/2}} \big( z_i^\top \Sigma_n z_i - \tr(\Sigma_n) \big),
\end{align}
and
\begin{align}\label{1stmomdecquad}
	\tilde{\beta}^\top \hat{\Sigma}_n \tilde{\beta} - \tilde{\beta}^\top \Sigma_n \tilde{\beta}\notag
	&= \frac{1}{n} \sum_{i=1}^{n} \Big( 
	\tilde{\beta}^\top \Sigma_n^{1/2} z_i z_i^\top \Sigma_n^{1/2} \tilde{\beta} 
	- \tilde{\beta}^\top \Sigma_n \tilde{\beta} 
	\Big)\notag \\
	&= \frac{1}{n} \sum_{i=1}^{n} \Big( 
	z_i^\top \Sigma_n^{1/2} \tilde{\beta} \tilde{\beta}^\top \Sigma_n^{1/2} z_i
	- \tr(\Sigma_n^{1/2} \tilde{\beta} \tilde{\beta}^\top \Sigma_n^{1/2}) 
	\Big).
\end{align}
Comparing the expressions in \eqref{1stmomdectr} and \eqref{1stmomdecquad}, we see that they differ only in the matrices $p^{-1/2}\Sigma_n$ and $\Sigma_n^{1/2} \tilde{\beta}\tilde{\beta}^\top \Sigma_n^{1/2}$. To handle both cases simultaneously, we introduce a symmetric positive semidefinite $p\times p$ matrix $A_n$ satisfying $\mathrm{tr}(A_n^2) \le C^2$. Observe that the collection of random variables
\[
\bigl\{ z_i^\top A_n z_i - \mathrm{tr}(A_n) \bigr\}_{i=1}^n
\]
is independent and has mean zero.

For any $M > 0$, the Markov inequality together with the preceding considerations gives
\begin{align*}
	\mathbb{P} \Bigg( \Bigl| \frac{1}{n} \sum_{i=1}^n \big( z_i^\top A_n z_i - \tr(A_n) \big) \Bigr| > M \Bigg)
	&\le \frac{1}{(n\, M)^2} \sum_{i=1}^n \mathbb{E} \Big[ \big( z_i^\top A_n z_i - \tr(A_n) \big)^2 \Big].
\end{align*}
Applying Lemma~\ref{quadform} to the expectation in the preceding display yields
\begin{align*}
	\mathbb{P} \Bigg( \Bigl| \frac{1}{n} \sum_{i=1}^n (z_i^\top A_n z_i - \tr(A_n)) \Bigr| > M \Bigg)
	&\le \frac{1}{n M^2} \, \nu_4 C_2 \tr(A_n^2) 
	\le \frac{1}{n M^2} \, \nu_4 C_2 C^2,
\end{align*}
and therefore
\begin{equation*}
\frac{1}{n} \sum_{i=1}^n (z_i^\top A_n z_i - \tr(A_n)) = O_{\mathbb{P}}\left(\frac{1}{n^{1/2}}\right).
\end{equation*}
Revisiting \eqref{1stmomdectr} and \eqref{1stmomdecquad}, this completes the first statements of Lemma~\ref{momqf1} and Lemma~\ref{momqf2}.

For the second statements of Lemma~\ref{momqf1} and Lemma~\ref{momqf2}, consider the decompositions
\begin{align}\label{sectrdec}
	\frac{1}{p} \mathrm{tr}(\hat{\Sigma}_n^2) 
	- \Biggl(\frac{1}{p}\mathrm{tr}(\Sigma_n^2) + \gamma_n \Bigl(\frac{1}{p}\mathrm{tr}(\Sigma_n)\Bigr)^2 \Biggr)
	&= \frac{1}{pn^2} \mathrm{tr} \Biggl( \Bigl( \sum_{i=1}^n \Sigma_n^{1/2} z_i z_i^\top \Sigma_n^{1/2} \Bigr)^2 \Biggr) 
	- \Biggl(\frac{1}{p}\mathrm{tr}(\Sigma_n^2) + \gamma_n \Bigl(\frac{1}{p}\mathrm{tr}(\Sigma_n)\Bigr)^2 \Biggr) \notag \\
	&= \frac{1}{n^2} \sum_{i=1}^n \sum_{i_1=1}^n \frac{1}{p} (z_i^\top \Sigma_n z_{i_1})^2 
	- \Biggl(\frac{1}{p}\mathrm{tr}(\Sigma_n^2) + \gamma_n \Bigl(\frac{1}{p}\mathrm{tr}(\Sigma_n)\Bigr)^2 \Biggr),
\end{align}
and
\begin{align}
	\tilde{\beta}^\top \hat{\Sigma}_n^2 \tilde{\beta} 
	- \Bigl(\tilde{\beta}^\top \Sigma_n^2 \tilde{\beta} + \frac{1}{n} \mathrm{tr}(\Sigma_n) \, \tilde{\beta}^\top \Sigma_n \tilde{\beta} \Bigr)
	&= \tilde{\beta}^\top \Biggl( \frac{1}{n} \sum_{i=1}^n \Sigma_n^{1/2} z_i z_i^\top \Sigma_n^{1/2} \Biggr)^2 \tilde{\beta} 
	- \Bigl(\tilde{\beta}^\top \Sigma_n^2 \tilde{\beta} + \frac{1}{n} \mathrm{tr}(\Sigma_n) \, \tilde{\beta}^\top \Sigma_n \tilde{\beta} \Bigr) \notag \\
	&= \frac{1}{n^2} \sum_{i=1}^n \sum_{i_1=1}^n 
	z_{i_1}^\top (\Sigma_n^{1/2} \tilde{\beta} \tilde{\beta}^\top \Sigma_n^{1/2}) z_i \; z_i^\top \Sigma_n z_{i_1} \label{secquaddec1}\\
	&\quad
	- \Bigl(\tr(\Sigma_n^{1/2}\tilde{\beta}\tilde{\beta}^\top\Sigma_n^{1/2}\Sigma_n) + \frac{1}{n} \mathrm{tr}(\Sigma_n) \, \mathrm{tr}(\Sigma_n^{1/2}\tilde{\beta}\tilde{\beta}^\top \Sigma_n^{1/2}  \Bigr).\label{secquaddec2}
\end{align}
Comparing the expressions in \eqref{sectrdec} and \eqref{secquaddec1}-\eqref{secquaddec2}, we observe that they differ only in the matrices $\Sigma_n^{1/2} \tilde{\beta} \tilde{\beta}^\top \Sigma_n^{1/2}$ and $p^{-1} \Sigma_n$. To handle both cases simultaneously, we introduce a symmetric positive semidefinite $p \times p$ matrix $B_n$ satisfying 
\(\mathrm{tr}(B_n) \le C\) and \(\mathrm{tr}(B_n^2) \le C^2\).
Then,
\begin{align*}
	\frac{1}{n^2} \sum_{i=1}^n \sum_{i_1=1}^n z_{i_1}^\top B_n z_i \, z_i^\top \Sigma_n z_{i_1}
	&- \Big( \tr(B_n \Sigma_n) + \frac{1}{n} \tr(B_n) \tr(\Sigma_n) \Big)\\
	&= \frac{1}{n} \sum_{i=1}^n \frac{1}{n} \Big( z_i^\top B_n z_i \, z_i^\top \Sigma_n z_i - \tr(B_n) \tr(\Sigma_n) \Big) \; \mathbf{(I)} \notag \\
	&\quad + \frac{1}{n} \sum_{i=1}^n \sum_{\substack{i_1=1 \\ i_1 \neq i}}^n \frac{1}{n} 
	\Big( z_{i_1}^\top B_n z_i \, z_i^\top \Sigma_n z_{i_1} - \tr(B_n \Sigma_n) \Big) \; \mathbf{(II)} \\
	&\quad - \frac{1}{n} \tr(B_n \Sigma_n) \; \mathbf{(III)}
\end{align*}
Since $|\mathbf{(III)}| = n^{-1} \tr(B_n \Sigma_n) = n^{-1} C \, \tr(B_n) \le n^{-1} C^2$ by Assumption~\ref{A2}, it remains to control the expected values of $(\mathbf{I})$ and $(\mathbf{II})$ in order to derive the second statements of the lemmas.

Consider the following decomposition of $\mathbf{(I)}$:
\begin{align}\label{(I)}
	\begin{split}
		\mathbf{(I)} 
		&= \frac{1}{n} \sum_{i=1}^n \frac{1}{n} \Big( z_i^\top B_n z_i \, z_i^\top \Sigma_n z_i - \tr(B_n) \tr(\Sigma_n) \Big) \\
		&= \frac{1}{n^2} \sum_{i=1}^n z_i^\top B_n z_i \, \Big( z_i^\top \Sigma_n z_i - \tr(\Sigma_n) \Big) + \frac{1}{n^2} \sum_{i=1}^n \tr(\Sigma_n) \, \Big( z_i^\top B_n z_i - \tr(B_n) \Big).
	\end{split}
\end{align}
To bound the expectation of the first term in the second line of~\eqref{(I)}, we invoke the triangle inequality and the Cauchy--Schwarz inequality, which gives
\begin{align}\label{I_i}
	\mathbb{E} \Biggl( \Bigl| \frac{1}{n^2} \sum_{i=1}^n z_i^\top B_n z_i \, \Big( z_i^\top \Sigma_n z_i - \tr(\Sigma_n) \Big) \Bigr|\Biggr)
	\le \frac{1}{n^2} \sum_{i=1}^n \sqrt{ \mathbb{E} \big[ (z_i^\top B_n z_i)^2 \big] \, 
		\mathbb{E} \big[ (z_i^\top \Sigma_n z_i - \tr(\Sigma_n))^2 \big] }.
\end{align}
Using the inequality $(a-b)^2 \le 2(a^2 + b^2)$ for $a,b \in \mathbb{R}$ and Lemma~\ref{quadform}, we obtain
\begin{align}\label{quad:exp}
	\begin{split}
	\mathbb{E}\!\left[(z_i^\top B_n z_i)^2\right]
	&= \mathbb{E}\!\left[(z_i^\top B_n z_i - \tr(B_n) + \tr(B_n))^2 \right] \\
	&\le 2\,\mathbb{E}\!\left[(z_i^\top B_n z_i - \tr(B_n))^2\right] + 2\,\tr(B_n)^2 \\
	&\le 4 C_2 \nu_4\, \tr(B_n^2) + 2\,\tr(B_n)^2 \\
	&\le 4 C_2 \nu_4 C^2 + 2 C^2 
	\end{split}
\end{align}
and
\begin{align*}
	\mathbb{E}\!\left[(z_i^\top \Sigma_n z_i - \tr(\Sigma_n))^2\right]
	\le 2 C_2 \nu_4\, \tr(\Sigma_n^2)
	\le p \,(2 C_2 \nu_4 C^2).
\end{align*}
Therefore, the bound in \eqref{I_i} reduces to
\begin{align}\label{I_i_2}
	\begin{split}
		\mathbb{E}\Biggl(\Bigl|
		\frac{1}{n^2}\sum_{i=1}^n 
		z_i^\top B_n z_i \, \big(z_i^\top \Sigma_n z_i - \tr(\Sigma_n)\big)
		\Bigr|\Biggr)
		&\le \frac{1}{n^2}\sum_{i=1}^n 
		\sqrt{\left(4 C_2\nu_4 C^2 + 2C^2\right)\big(p\,(2 C_2\nu_4 C^2)\big)} \\
		&\le K_1 \frac{\sqrt{p}}{n},
	\end{split}
\end{align}
for some constant $K_1>0$.

Similarly, the expectation of the second term in the second line of~\eqref{(I)} can be bounded using the Jensen inequality, recalling that 
$\{z_i^\top B_n z_i - \tr(B_n)\}_{i=1}^n$ are independent, mean-zero random variables and invoking Lemma~\ref{quadform}. We then obtain
\begin{align}\label{I_ii}
	\begin{split}
		\frac{\tr(\Sigma_{n})}{n}
		\mathbb{E}\Biggl(\Bigl|\frac{1}{n}
		\sum_{i=1}^{n}
		\bigl( z_{i}^{\top} B_{n} z_{i} - \tr(B_{n}) \bigr)
		\Bigr|\Biggr)
		&\le 
		\frac{\tr(\Sigma_{n})}{n}
		\Bigl[ \mathbb{E}\Bigl(\bigl|\frac{1}{n}\sum_{i=1}^{n}
		z_{i}^{\top} B_{n} z_{i} - \tr(B_{n}) \bigr|^{2}\Bigr) \Bigr]^{1/2} \\
		&\le 
		\frac{\tr(\Sigma_{n})}{n}
		\Bigl[ \frac{1}{n^{2}}
		\sum_{i=1}^{n}
		\mathbb{E}\Bigl(\bigl| z_{i}^{\top} B_{n} z_{i} - \tr(B_{n}) \bigr|^{2}\Bigr)
		\Bigr]^{1/2} \\
		&\le 
		\frac{p}{n}\, C
		\Bigg( \frac{2 C_{2}\nu_{4}\, \tr(B_{n}^{2})}{n} \Bigg)^{1/2} \\
		&\le 
		K_{2}\,\frac{p}{n^{3/2}},
	\end{split}
\end{align}
for some constant $K_{2}>0$. Using the decomposition of $\mathbf{(I)}$ in~\eqref{(I)} together with the bounds established in~\eqref{I_i_2} and~\eqref{I_ii}, it follows that
\begin{equation}\label{O_I}
	\mathbf{(I)}
	= 
	\frac{1}{n}\sum_{i=1}^{n}\frac{1}{n}
	\Big( z_{i}^{\top} B_{n} z_{i}\, z_{i}^{\top} \Sigma_{n} z_{i}
	- \tr(B_{n})\, \tr(\Sigma_{n}) \Big)
	= O_{\mathbb{P}}\!\left( \frac{\sqrt{p}}{n} \,\vee\, \frac{p}{n^{3/2}} \right).
\end{equation}

We are now going to show that
\begin{align}
		\mathbf{(II)}
		&=
		\frac{1}{n} \sum_{i=1}^{n}
		\sum_{\substack{i_{1}=1 \\ i_{1}\neq i}}^{n}
		\frac{1}{n}
		\Bigl(
		z_{i_{1}}^{\top} B_{n} z_{i}\,
		z_{i}^{\top} \Sigma_{n} z_{i_{1}}
		- \tr(B_{n}\Sigma_{n})
		\Bigr)=O_{\mathbb{P}}\left(\frac{1}{n^{1/2}}\lor \frac{\sqrt{p}}{n}\right).
\end{align}
We define
\begin{equation*}
E_{i,i_1}\coloneqq z_{i_1}^{\top} B_{n} z_{i}\; z_{i}^{\top}\Sigma_{n} z_{i_1}
- \tr\bigl(B_{n} \Sigma_{n}\bigr),\quad i\neq i_1.
\end{equation*}
Note that, for each fixed \(i\), the collection \(\{E_{i,i_1}\}_{i_1=1,i_1\neq i}^{n}\)
is conditionally independent given \(z_i\) and has mean zero. 
By the Jensen inequality, we obtain 
\begin{align}\label{qfmom2_1} \mathbb{E}\bigg[\bigg|\frac{1}{n}\sum_{i=1}^n \frac{1}{n}\sum_{\substack{i_1=1\\i_1\neq i}}^n E_{i,i_1}\bigg|\bigg] &\le \mathbb{E}\bigg[\Big(\frac{1}{n}\sum_{i=1}^n \frac{1}{n}\sum_{\substack{i_1=1\\i_1\neq i}}^n E_{i,i_1}\Big)^2\bigg]^{1/2} = \biggl[\frac{1}{n^{4}} \sum_{\substack{i,j=1\\}}^n \sum_{\substack{i_1\neq i\\ j_1\neq j}}^n \mathbb{E}\bigl(E_{i,i_1}E_{j,j_1}\bigr)\biggr]^{1/2}. 
\end{align}
Next, observe that $\mathbb{E}(E_{i,i_1} E_{j,j_1}) = 0$ whenever the index pairs $\{i,i_1\}$ and $\{j,j_1\}$ are disjoint. 
Consequently, only pairs with a nonempty intersection contribute to the quadruple sum in \eqref{qfmom2_1}.
 We distinguish the following cases:
\begin{align}\label{cases1}
\mathbb{E}[E_{i,i_1} E_{j,j_1}] =
\begin{cases}
	0, & \text{if } \{i,i_1\} \cap \{j,j_1\} = \emptyset,\\[1mm]
	\mathbb{E}[E_{i,i_1}^2], & \text{if } \{i,i_1\}=\{j,j_1\}\\[1mm]
	\mathbb{E}[(z_i^\top B_n\Sigma_n z_i-\tr(B_n\Sigma_n))^2], & \text{if}  \{i,i_1\}\cap\{j,j_1\}=\{i\},\\[1mm]
	\mathbb{E}[(z_{i_1}^\top B_n\Sigma_n z_{i_1}-\tr(B_n\Sigma_n))^2], & \text{if} \{i,i_1\}\cap\{j,j_1\}=\{i_1\}, 
\end{cases}
\end{align}
where in the case $i=j_1$ and $i_1=j$, we use that $\mathbb{E}(E_{i,i_1}E_{i_1,i})=\mathbb{E}(E_{i,i_1}E_{i,i_1})=\mathbb{E}(E_{i,i_1}^2)$.
By the tower property, Lemma~\ref{quadform} for the conditional expectation, 
and the Cauchy--Schwarz inequality, we obtain
\begin{align}\label{E_iequal-aos}
	\begin{split}
		\mathbb{E}[E_{i,i_1}^2]
		&= \mathbb{E}\bigl[ \mathbb{E}(E_{i,i_1}^2 \mid z_i) \bigr] \\[1mm]
		&\le \mathbb{E}\Bigl( 2 C_2 \nu_4 \, \tr( B_n z_i z_i^{\top} \Sigma_n^2 z_i z_i^{\top} B_n ) \Bigr) \\[1mm]
		&= 2C_2 \nu_4 \, \mathbb{E}\bigl[ (z_i^{\top} \Sigma_n^2 z_i) (z_i^{\top} B_n^2 z_i) \bigr] \\[1mm]
		&\le 2C_2 \nu_4 \,
		\Bigl[ \mathbb{E} \bigl(| z_i^{\top} \Sigma_n^2 z_i |^2\bigr) \Bigr]^{1/2} \;
		\Bigl[ \mathbb{E} \bigl(| z_i^{\top} B_n^2 z_i |^2\bigr) \Bigr]^{1/2},
	\end{split}
\end{align}
for $i_1 \in \{1,\dots,n\}$ with $i \neq i_1$.
Using the same arguments as in \eqref{quad:exp}, we get
\begin{align}\label{quad:exp2}
	\begin{split}
		\mathbb{E}\!\left[(z_i^\top \Sigma_n z_i)^2\right]
		&= \mathbb{E}\!\left[(z_i^\top \Sigma_n z_i - \tr(\Sigma_n) + \tr(\Sigma_n))^2 \right] \\
		&\le 2\,\mathbb{E}\!\left[(z_i^\top \Sigma_n z_i - \tr(\Sigma_n))^2\right] + 2\,\tr(\Sigma_n)^2 \\
		&\le 4  C_2 \nu_4\, \tr(\Sigma_n^2) + 2\,\tr(\Sigma_n)^2 \\
		&\le 4 p C_2 \nu_4 C^2 + 2 p^2 C^2.
	\end{split}
\end{align}
Combining \eqref{quad:exp}, \eqref{E_iequal-aos} and \eqref{quad:exp2}, and noting that 
$\Sigma_n$ and $B_n$ are symmetric positive semidefinite matrices satisfying
$\|\Sigma_n\|_2 \le C$ and $\|B_n\|_2 \le C$, it follows that
\begin{align}\label{K3}
	\begin{split}
		\mathbb{E}[E_{i,i_1}^2]
		&\le C_2 \nu_4 \,
		\bigl[ \mathbb{E}\bigl(|z_i^{\top} \Sigma_n^2 z_i|^2\bigr) \bigr]^{1/2} \,
		\bigl[ \mathbb{E}\bigl(|z_i^{\top} B_n^2 z_i|^2\bigr) \bigr]^{1/2} \\[1mm]
		&\le  C^2 C_2 \nu_4 \,
		\bigl[ \mathbb{E}\bigl(|z_i^{\top} \Sigma_n z_i|^2\bigr) \bigr]^{1/2} \,
		\bigl[ \mathbb{E}\bigl(|z_i^{\top} B_n z_i|^2\bigr) \bigr]^{1/2} \\[1mm]
		&\le C^2 C_2 \nu_4 \,
		\bigl( 4 p C_2 \nu_4 C^2 + 2 p^2 C^2 \bigr)^{1/2} 
		\bigl( 4 C_2 \nu_4 C^2 + 2 C^2 \bigr)^{1/2} \\[1mm]
		&\le p\, K_3,
	\end{split}
\end{align}
for some constant $K_3>0$.
For the cases $\{i,i_1\}\cap\{j,j_1\}=\{i_1\}$ and $\{i,i_1\}\cap\{j,j_1\}=\{i\}$, observe that for all $i\in \{1,\dots,n\}$, we have  by Lemma~\ref{quadform} 
\begin{align}\label{cases:45}
	\mathbb{E}[(z_{i}^\top B_n\Sigma_n z_{i}-\tr(B_n\Sigma_n))^2]\leq 2 C_2 \nu_4 \tr(B_n^2\Sigma_n^2)\leq 2 C_2 \nu_4 C^4\eqqcolon K_4.
\end{align}	
With these bounds in hand we return to \eqref{qfmom2_1}.
Combining \eqref{qfmom2_1} and \eqref{cases1}, together with the upper bounds in
\eqref{K3} and \eqref{cases:45}, we obtain
\begin{align}\label{K4}
	\begin{split}
	\mathbb{E}\bigg(\bigg|\frac{1}{n}\sum_{i=1}^n \frac{1}{n}\sum_{\substack{i_1=1\\i_1\neq i}}^n E_{i,i_1}\bigg|\bigg)&\leq \biggl[\frac{1}{n^{4}}
	\sum_{\substack{i,j=1\\}}
	\sum_{\substack{i_1\neq i\\ j_1\neq j}}
	\mathbb{E}\bigl(E_{i,i_1}E_{j,j_1}\bigr)\biggr]^{1/2}\\
	&\leq \biggl[2 K_3 \frac{p}{n^2}+ 4 K_4\frac{1}{n}\biggr]^{1/2}.
	\end{split}
\end{align} 
Therefore, we conclude that
\begin{align}\label{O_II-aos}
	\mathbf{(II)}
	=
	\frac{1}{n}\sum_{i=1}^{n}
	\sum_{\substack{i_{1}=1\\ i_{1}\neq i}}^{n}
	\frac{1}{n}
	\bigl(
	z_{i_{1}}^{\top}B_{n}z_{i}\,
	z_{i}^{\top}\Sigma_{n}z_{i_{1}}
	- \tr(B_{n}\Sigma_{n})
	\bigr)
	=
	O_{\mathbb{P}}\!\left(\frac{\sqrt{p}}{n}\,\vee\, \frac{1}{n^{1/2}}\right).
\end{align}
Combining \eqref{O_I} with \eqref{O_II-aos}, we find
\begin{align*}\label{O_sum-aos}
	\mathbf{(I)}+\mathbf{(II)}
	=
	O_{\mathbb{P}}\!\left(
	\frac{\sqrt{p}}{n}
	\;\vee\;
	\frac{1}{n^{1/2}}
	\;\vee\;
	\frac{p}{n^{3/2}}
	\right)
	=
	O_{\mathbb{P}}\!\left(
	\frac{1}{n^{1/2}}
	\;\vee\;
	\frac{p}{n^{3/2}}
	\right),
\end{align*}
where the final reduction follows from the facts that
\[
p\le n \implies \frac{\sqrt{p}}{n}\le \frac{1}{\sqrt{n}},
\qquad
p>n \implies \frac{\sqrt{p}}{n} < \frac{p}{n^{3/2}}.
\]
This completes the proof.
\end{proof}

   \begin{proof}[Proof of Theorem~\ref{est1}]
Using the definition of $\hat{\sigma}_n^2$ in \eqref{est:def}, we obtain
\begin{align}\label{decsigmadifI}
	|\hat{\sigma}_n^2 - \sigma^2|&=\left|\frac{1}{\tilde{m}_2} 
	\Bigl(\hat{m}_2 \frac{\|y\|_2^2}{n} - \frac{\hat{m}_1 \|X^\top y\|_2^2}{n^2} 
	\Bigr)-\sigma^2\right|\notag \\
	&=
	\left|\frac{1}{\tilde{m}_2} 
	\Bigl(\hat{m}_2 \frac{\|y\|_2^2}{n} - \frac{\hat{m}_1 \|X^\top y\|_2^2}{n^2} 
	- \tau^2 \hat{m}_2 \hat{m}_1 + \tau^2  \hat{m}_2 \hat{m}_1 \Bigr) - \sigma^2
	\right|,
\end{align}
almost surely.
Consider the following decompositions:
\begin{align}
	\frac{\|y\|_2^2}{n} 
	&= \beta^\top \hat{\Sigma}_n \beta 
	+ \frac{2}{n}\, u^\top X \beta 
	+ \frac{\|u\|_2^2}{n}, \label{decomp1} \\[1mm]
	\frac{\|X^\top y\|_2^2}{n^2} 
	&= \beta^\top \hat{\Sigma}_n^2 \beta 
	+ \frac{2}{n^2}\, u^\top X X^\top X \beta 
	+ \frac{u^\top \hat{\underline{\Sigma}}_n u}{n}, \label{decomp2} \\[1mm]
	\sigma^2 
	&= \frac{1}{\tilde{m}_2} \Bigl( \hat{m}_2 - \gamma_n \hat{m}_1^2 \Bigr) \sigma^2, \quad \text{almost surely}, \label{decomp3} \\[1mm]
	\gamma_n \hat{m}_1^2 
	&=\frac{\hat{m}_1}{n} \mathrm{tr}(\hat{\underline{\Sigma}}_n). \label{decomp4}
\end{align}
Using \eqref{decomp1}--\eqref{decomp4} and applying the triangle inequality 
in \eqref{decsigmadifI}, we obtain
\begin{align}\label{decsigmadifII-aos}
	\begin{split}
		|\hat{\sigma}_n^2 - \sigma^2| 
		&\le 
		\frac{\hat{m}_2}{\tilde{m}_2} 
		\Bigl|
		\beta^\top \hat{\Sigma}_n \beta - \tau^2 \hat{m}_1 
		+ \frac{2}{n} u^\top X \beta 
		+ \frac{\|u\|_2^2}{n} - \sigma^2
		\Bigr| \\[1mm]
		&\quad + 
		\frac{\hat{m}_1}{\tilde{m}_2} 
		\Bigl|
		\tau^2 \hat{m}_2 - \beta^\top \hat{\Sigma}_n^2 \beta 
		- \frac{2}{n^2} u^\top X X^\top X \beta 
		+ \sigma^2 \gamma_n \hat{m}_1 
		- \frac{1}{n} u^\top \hat{\underline{\Sigma}}_n u
		\Bigr|,
	\end{split}
\end{align}
almost surely.
Note that Lemma~\ref{momqf1}, together with the subsequent discussion, implies that under the assumptions of the theorem
\begin{equation}\label{momforest1}
	\hat{m}_1 \;\overset{\mathbb{P}}{\longrightarrow}\; \int_0^\infty t \, dH(t) > 0, \quad 
	\tilde{m}_2=\hat{m}_2 - \gamma_n \hat{m}_1^2 \;\overset{\mathbb{P}}{\longrightarrow}\; \int_0^\infty t^2 \, dH(t) > 0.
\end{equation}
Using the tower property and the conditional Markov inequality, for any $\varepsilon>0$, we obtain
\begin{align}\label{uXbeta}
	\begin{split}
		\mathbb{P}\!\left(\left|\frac{u^\top X\beta}{n}\right| > \varepsilon \right)
		= \mathbb{E}\!\left[\, \mathbb{P}\!\left(\left|\frac{u^\top X\beta}{n}\right| > \varepsilon \,\middle|\, X \right) \right] \le 
		\mathbb{E}\!\left[\, 1 \wedge 
		\frac{\mathbb{E}\!\left((u^\top X\beta)^2   \,\middle|\, X \right)}{(n\,\varepsilon)^{2}} \right].
	\end{split}
\end{align}
Using the independence of $u$, $X$, and $\beta$, and \eqref{momforest1}, we further obtain
\begin{align*}
	\mathbb{E}\bigg(\frac{(u^\top X\beta)^2}{n^2} \,\bigg|\, X\bigg) 
	&= \mathbb{E}\bigg(\frac{u^\top X \beta \beta^\top X^\top u}{n^2} \,\bigg|\, X\bigg) = \frac{\sigma^2}{n} \, \mathbb{E}\bigg(\frac{\tr(X \beta \beta^\top X^\top)}{n} \,\bigg|\, X\bigg) \\
	&= \frac{\tau^2 \sigma^2}{n} \hat{m}_1
	= o_{\mathbb{P}}(1), \quad \text{as } n \to \infty.
\end{align*}
By an application of the dominated convergence theorem to the upper bound in~\eqref{uXbeta}, it follows that 
$n^{-1} u^\top X\beta = o_{\mathbb{P}}(1)$. Similarly,
\begin{align}\label{uXXtopXbeta}
	\mathbb{P}\!\left(\left|\frac{u^\top XX^\top X\beta}{n^{2}}\right|>\varepsilon\right)
	=
	\mathbb{E}\!\left[
	\mathbb{P}\!\left(
	\left|\frac{u^\top XX^\top X\beta}{n^{2}}\right|>\varepsilon
	\,\bigg|\, X
	\right)
	\right] \le 
	\mathbb{E}\!\left(
	1 \,\wedge\,
	\frac{
		\mathbb{E}\!\left( (u^\top XX^\top X\beta)^{2}\,\big|\,X \right)
	}{n^4\,\varepsilon^{2}}
	\right).
\end{align}
By Assumption~\ref{A2}, we have
$s_{\max}(\hat{\Sigma}_n)
\le s_{\max}(\Sigma_n)\, s_{\max}(Z^\top Z/n)
\le C\, s_{\max}(Z^\top Z/n)$.
Moreover, using the independence of $u$, $X$, and $\beta$, and applying Lemma~\ref{lebip}, we obtain
\begin{align*}
	\mathbb{E}\!\left(
	\frac{(u^\top XX^\top X\beta)^{2}}{n^{4}}
	\,\Big|\, X
	\right)&=
	\frac{\tau^{2}\sigma^{2}}{pn}\,\tr(\hat{\Sigma}_n^{3}) \le
	\frac{\tau^{2}\sigma^{2}}{n}\, s_{\max}^{3}(\hat{\Sigma}_n) \\
	&\le
	\frac{\tau^{2}\sigma^{2}}{n}\, C^{3} s_{\max}^{3}(Z^\top Z/n)
	= o_{\mathbb{P}}(1).
\end{align*}
By an application of the dominated convergence theorem to the upper bound in~\eqref{uXXtopXbeta}, it follows that 
$n^{-2}u^\top XX^\top X\beta = o_{\mathbb{P}}(1)$.

So far, we have treated the bilinear forms that appear in the upper bound of~\eqref{decsigmadifII-aos}. 
We now turn to the difference between the quadratic forms and their corresponding conditional expectations in the same upper bound.
By the law of large numbers,
\[
\frac{1}{n}\|u\|_2^2 - \sigma^2 \;\overset{\mathbb{P}}{\longrightarrow}\; 0.
\]
Using the conditional Markov inequality, Lemma~\ref{quadform}, the identity $\gamma_n \hat{m}_1 = n^{-1}\tr(\underline{\hat{\Sigma}}_n)$, \eqref{momforest1}, and the independence of $u$ and $X$, we obtain, for any $\varepsilon>0$,
\begin{align}\label{quadu1}
	\begin{split}
	\mathbb{P}\!\left(
	\left| \frac{u^\top \underline{\hat{\Sigma}}_n u}{n} - \sigma^2 \gamma_n \hat{m}_1 \right| > \varepsilon \,\Big|\, X
	\right)
	&\le 
	\mathbb{P}\!\left(
	\frac{\sigma^2}{n} \left| \frac{1}{\sigma^2} u^\top \underline{\hat{\Sigma}}_n u - \tr(\underline{\hat{\Sigma}}_n) \right| > \varepsilon \,\Big|\, X
	\right) \\
	&\le 
	1 \wedge 
	\frac{1}{\varepsilon^2} \, 
	\mathbb{E}\!\left(
	\frac{\sigma^4}{n^2} 
	\left| \frac{1}{\sigma^2} u^\top \underline{\hat{\Sigma}}_n u - \tr(\underline{\hat{\Sigma}}_n) \right|^2 
	\,\Big|\, X
	\right) \\
	&\le 
	1 \wedge 
	\frac{1}{(n\varepsilon)^2} \, 2\, C_2 \nu_{4,u} \, \tr(\underline{\hat{\Sigma}}_n^2) \\
	&= 
	1 \wedge 
	\frac{\gamma_n}{n} \frac{2\, C_2 \nu_{4,u} }{\varepsilon^2} \, \hat{m}_2
	\;\overset{\mathbb{P}}{\longrightarrow}\; 0,\quad\text{as}\,\,n\to\infty.
	\end{split}
\end{align}
Similarly, for arbitrary $\varepsilon>0$, an application of the conditional Markov inequality, \eqref{momforest1}, the independence of $\beta$ and $X$, and Lemma~\ref{quadform} yields
\begin{align}\label{quadbeta1}
	\begin{split}
	\mathbb{P}\Big(\big|\beta^\top \hat{\Sigma}_n \beta - \tau^2 \hat{m}_1\big| > \varepsilon \,\Big|\, X \Big)
	&\le 
	1 \wedge \frac{1}{\varepsilon^2} \, 
	\mathbb{E}\Bigg(
	\frac{\tau^4}{p^2} \Big|\frac{p}{\tau^2} \beta^\top \hat{\Sigma}_n \beta - \tr(\hat{\Sigma}_n) \Big|^2 \,\Big|\, X
	\Bigg) \\
	&\le 
	1 \wedge \frac{2\, C_2 \, \nu_{4,\beta}}{p \, \varepsilon^2} \, \hat{m}_2
	\;\overset{\mathbb{P}}{\longrightarrow}\; 0,\quad\text{as}\,\,n\to\infty.
	\end{split}
\end{align}
Analogously, for the quadratic form involving $\hat{\Sigma}_n^2$, we have
\begin{align}\label{quadbeta2}
	\begin{split}
	\mathbb{P}\Big(\big|\beta^\top \hat{\Sigma}_n^2 \beta - \tau^2 \hat{m}_2\big| > \varepsilon \,\Big|\, X \Big)
	&\le 
	1 \wedge \frac{1}{\varepsilon^2} \, 
	\mathbb{E}\Bigg(
	\frac{\tau^4}{p^2} \Big|\frac{p}{\tau^2} \beta^\top \hat{\Sigma}_n^2 \beta - \tr(\hat{\Sigma}_n^2) \Big|^2 \,\Big|\, X
	\Bigg) \\
	&\le 
	1 \wedge \frac{2\,C_2 \, \nu_{4,\beta}}{(p \, \varepsilon)^2} \, \tr(\hat{\Sigma}_n^4) \\
	&\le 
	1 \wedge \frac{2\,C_2 \, \nu_{4,\beta}}{p \, \varepsilon^2} \, s_{\max}^4(\hat{\Sigma}_n) 
	\;\overset{\mathbb{P}}{\longrightarrow}\; 0,\quad \text{as}\,\,n\to\infty,
	\end{split}
\end{align}
where the last step uses Lemma~\ref{lebip}.

Once again, by the tower property, the dominated convergence theorem, and equations~\eqref{quadu1}, \eqref{quadbeta1}, and \eqref{quadbeta2}, it follows that
\[
\frac{1}{n} u^\top \hat{\underline{\Sigma}}_n u - \sigma^2 \gamma_n \hat{m}_1 = o_{\mathbb{P}}(1), \qquad
\beta^\top \hat{\Sigma}_n \beta - \tau^2 \hat{m}_1 = o_{\mathbb{P}}(1), \qquad
\beta^\top \hat{\Sigma}_n^2 \beta - \tau^2 \hat{m}_2 = o_{\mathbb{P}}(1).
\]
Therefore, we conclude that $|\hat{\sigma}_n^2 - \sigma^2|=o_{\mathbb{P}}(1)$.

\medskip

Building on the results for $\hat{\sigma}_n^2$, it is straightforward to derive  $|\hat{\tau}_n^2 - \tau^2| =o_{\mathbb{P}}(1)$.
Using the definition of $\hat{\tau}_n^2$ in \eqref{est:def}, we have
\begin{align}\label{dectaudifI}
	\begin{split}
		|\hat{\tau}_n^2 - \tau^2| 
		&= \frac{1}{\tilde{m}_2} \Biggl| 
		\frac{\|X^\top y\|_2^2}{n^2} - \gamma_n \hat{m}_1 \frac{\|y\|_2^2}{n} - \tau^2 (\hat{m}_2 - \gamma_n \hat{m}_1^2) 
		\Biggr| \\
		&\le \frac{1}{\tilde{m}_2} \Biggl| 
		\beta^\top \hat{\Sigma}_n^2 \beta - \tau^2 \hat{m}_2 + \frac{2}{n} u^\top X X^\top X \beta + \frac{1}{n} u^\top \hat{\underline{\Sigma}}_n u - \sigma^2 \gamma_n \hat{m}_1
		\Biggr| \\
		&\quad + \frac{\gamma_n \hat{m}_1}{\tilde{m}_2} \Biggl|
		\beta^\top \hat{\Sigma}_n \beta - \tau^2 \hat{m}_1 + \frac{2}{n} u^\top X \beta + \frac{\|u\|_2^2}{n} - \sigma^2
		\Biggr|,
	\end{split}
\end{align}
almost surely. In the upper bound of \eqref{dectaudifI}, we have the same quadratic forms with their corresponding conditional means, as well as bilinear forms, as in the upper bound of \eqref{decsigmadifII-aos}. For the factors $\tilde{m}_2^{-1}$ and $\gamma_n \hat{m}_1 \tilde{m}_2^{-1}$, we recall the convergence results in \eqref{momforest1}. Using the same arguments as for $\hat{\sigma}_n^2$, we therefore conclude that $
|\hat{\tau}_n^2 - \tau^2| =o_{\mathbb{P}}(1)$.
\end{proof}

\begin{proof}[Proof of Proposition~\ref{prop4m}]
	For some $\delta>0$ and $\varepsilon>0$, we may write
	\[
	\mathbb{E}\!\left[ z_{ij}^{2}\, \mathbf{1}\{|z_{ij}|\ge \varepsilon\sqrt{n}\}\right]
	= \mathbb{E}\!\left[ \frac{|z_{ij}|^{2+\delta}}{|z_{ij}|^{\delta}}\, \mathbf{1}\{|z_{ij}|\ge \varepsilon\sqrt{n}\}\right],
	\]
	since $|z_{ij}|\ge \varepsilon\sqrt{n}>0$ on the event under consideration.
	By Assumption~\ref{A3}, the $(4+\delta)$th moments of the $Z_{ij}$ are uniformly 
	bounded, and therefore their $(2+\delta)$th moments are uniformly bounded as well. 
	Consequently,
	\[
	\mathbb{E}\!\left[ z_{ij}^{2}\,\mathbf{1}\{|z_{ij}|\ge \varepsilon\sqrt{n}\}\right]
	\le \frac{\nu_{4+\delta}}{\varepsilon^{\delta} n^{\delta/2}}
	\;\longrightarrow\; 0, \qquad \text{as}\,\,n\to\infty .
	\]
	Therefore,
	\begin{align}\label{lindbergcondalpha}
		\begin{split}
			\frac{1}{\varepsilon^2 n^2}
			\sum_{i=1}^n \sum_{j=1}^p
			\mathbb{E}\!\left(z_{ij}^2 \mathbf{1}\{|z_{ij}|\ge \varepsilon\sqrt{n}\}\right)
			&= 
			\frac{p}{n\varepsilon^2}
			\frac{1}{n}\sum_{i=1}^n \frac{1}{p}\sum_{j=1}^p
			\mathbb{E}\!\left(z_{ij}^2 \mathbf{1}\{|z_{ij}|\ge \varepsilon\sqrt{n}\}\right) \\
			&\le 
			\gamma_n \,
			\frac{1}{n}\sum_{i=1}^n \frac{1}{p}\sum_{j=1}^p
			\frac{\nu_{4+\delta}}{\varepsilon^{\delta+2} n^{\delta/2}} \\
			&= 
			\gamma_n \,
			\frac{\nu_{4+\delta}}{\varepsilon^{\delta+2} n^{\delta/2}}
			\longrightarrow 0,
		\end{split}
	\end{align}
	as $n\to\infty$, where we use that $\gamma_n \to \gamma \in (0,\infty)$ by Assumption~\ref{A1}.
	
	Now replace $\varepsilon$ by
	\[
	\varepsilon_n = \bigl(n^{-\alpha}\bigr)^{1/(\delta+2)}, \qquad \alpha\in(0,\delta/2).
	\]
	Since $\varepsilon_n^{\delta+2}=n^{-\alpha}$ the expression appearing in the last line of \eqref{lindbergcondalpha} becomes
	\[
	\gamma_n \frac{\nu_{4+\delta}}{\varepsilon_n^{\delta+2} n^{\delta/2}}
	= \gamma_n\,\nu_{4+\delta}\,n^{\alpha-\delta/2}.
	\]
	Using Assumption~\ref{A1} and since $\alpha < \delta/2$, it follows that
	\[
	\gamma_n\,\nu_{4+\delta}\,n^{\alpha-\delta/2} \;\longrightarrow\; 0,
	\qquad \text{as}\,\,n\to\infty.
	\]
	This completes the proof.
\end{proof}

\begin{remark}\label{rm:remark}
	
	For the upcoming proofs, we recall some facts from random matrix theory. In Subsection~\ref{sec:TunedRGD}, we introduced the notion of a Stieltjes transform for any cumulative distribution function $G$ supported on $[0,\infty)$. For $z \in \mathbb{C} \setminus \mathbb{R}_{\geq 0}$, the Stieltjes transform of the empirical spectral distribution function $F_{\hat{\Sigma}_n}$ can be written as
	\begin{align*}
		m_n(z) \coloneqq m_{F_{\hat{\Sigma}_n}}(z)=\int_0^\infty\frac{1}{s-z}dF_{\hat{\Sigma}_n}(s) 
		= \frac{1}{p} \tr \bigl( (\hat{\Sigma}_n - z I_p)^{-1} \bigr) 
		= \frac{1}{p} \sum_{j=1}^{p} \frac{1}{s_j - z}.
	\end{align*}
	
Similarly, in Subsection~\ref{sec:TunedRGD} we introduced the empirical spectral distribution function of 
\(
\underline{\hat{\Sigma}}_{n}=n^{-1}XX^{\top}
\),
denoted by \(F_{\underline{\hat{\Sigma}}_{n}}\), and its associated Stieltjes transform \(m_{F_{\underline{\hat{\Sigma}}_n}}(z)=v_n(z)\), for $z\in \mathbb{C}\setminus\mathbb{R}_{\geq 0}$.  
The empirical spectral distribution functions  \(F_{\underline{\hat{\Sigma}}_{n}}\) and \(F_{\hat{\Sigma}_{n}}\) are related through  
\begin{equation}\label{disrel}
	F_{\hat{\Sigma}_{n}}(x)
	=\Bigl(1-\frac{1}{\gamma_{n}}\Bigr)\, \mathbf{1}_{[0,\infty)}(x)
	+\frac{1}{\gamma_{n}}\, F_{\underline{\hat{\Sigma}}_{n}}(x),\qquad \text{for all }x\in\mathbb{R},
\end{equation}
because the matrices \(\hat{\Sigma}_{n}\) and \(\underline{\hat{\Sigma}}_{n}\) differ only by 
\(|p-n|\) zero eigenvalues.

Combining Theorem~\ref{Pan2010} with the relation in \eqref{disrel}, one can deduce that, almost surely,
\[
F(x)\coloneqq\lim_{n\to \infty} F_{\hat{\Sigma}_n}(x)
=  \Bigl(1-\frac{1}{\gamma}\Bigr)\mathbf{1}_{[0,\infty)}(x)
\;+\; \frac{1}{\gamma}\,\bar F(x),\qquad \text{for all continuity points of $\bar{F}$}.
\]
Since the continuity points of $\bar{F}$ form a dense subset of $\mathbb{R}$, for any $x\in \mathbb{R}$ we can find a sequence of continuity points $\{x^c_k\}_{k\ge 1}$ satisfying $x^c_{k} > x^c_{k+1}$ for all $k\ge 1$ and $x^c_k \to x$ as $k \to \infty$. Using the right-continuity of $\mathbf{1}_{[0,\infty)}$ and $\bar{F}$, the limiting relation can then be extended to
\[
F(x)=  \Bigl(1-\frac{1}{\gamma}\Bigr)\mathbf{1}_{[0,\infty)}(x)
\;+\; \frac{1}{\gamma}\,\bar F(x),\qquad \text{for all  }x\in \mathbb{R}.
\]

For the corresponding Stieltjes transforms \(m_{F_{\hat{\Sigma}_{n}}}\) and
\(m_{F_{\underline{\hat{\Sigma}}_{n}}}\), the relation in \eqref{disrel} implies that, for each 
\(z \in \mathbb{C}\setminus \mathbb{R}_{\ge 0}\),
\begin{align}\label{Stielcon}
	\begin{split}
		m_{n}(z)
		= -\Bigl(1-\frac{1}{\gamma_{n}}\Bigr)\frac{1}{z}
		+ \frac{1}{\gamma_{n}}\, v_{n}(z).
	\end{split}
\end{align}
For $z \in \mathbb{C} \setminus \mathbb{R}_{\ge 0}$, with $z = u + i v$, we decompose
\begin{align}\label{samesign2}
	\begin{split}
		v_n(z) = \int_0^\infty \frac{1}{x - z} \, dF_{\underline{\hat{\Sigma}}_n}(x)
		&= \int_0^\infty \frac{x - u}{(x - u)^2 + v^2} \, dF_{\underline{\hat{\Sigma}}_n}(x)\\
		&\quad+ i \int_0^\infty \frac{v}{(x - u)^2 + v^2} \, dF_{\underline{\hat{\Sigma}}_n}(x).
	\end{split}
\end{align}
Note that \eqref{samesign2} implies that $\Im(v_n(z))$ has the same sign as $\Im(z)$. An analogous property holds for the limiting Stieltjes transform $m_{\bar{F}}(z)=v(z)$, that is, $\Im(v(z))$ also shares the sign of $\Im(z)$.
Using \eqref{samesign2}, it is straightforward to verify that for $z \in \mathbb{C} \setminus \mathbb{R}$,
\begin{align}\label{bound:realandim}
|\Im v_n(z)| \le \frac{1}{|\Im(z)|} \quad \text{and} \quad
|\Re v_n(z)| \le \frac{1}{2\,|\Im(z)|},
\end{align}
where the key step in establishing the bound for the real part is the observation that the function
\[
g(a) = \frac{a}{a^2 + b^2}, \quad a \ge 0, \, b \in \mathbb{R} \setminus \{0\},
\]
attains its maximum at $a = |b|$.

In the case where $z\in \mathbb{R}_{<0}$, we trivially obtain that $v_n(z)\leq |z|^{-1}$. The same bounds apply to $v(z)$ in the respective cases $z\in \mathbb{C}\setminus \mathbb{R}$ and $z\in \mathbb{R}_{<0}$.
Moreover, for every \(z \in \mathbb{C}\setminus \mathbb{R}_{\ge 0}\), the integrands corresponding to the real and imaginary parts of \(v_n(z)\), viewed as functions of \(x\), are continuous and bounded. 
Thus, applying Theorem~\ref{Pan2010} together with \eqref{Stielcon}, and using the Portmanteau theorem for the real and imaginary parts appearing in \eqref{samesign2}, we obtain, for each \(z\in \mathbb{C}\setminus \mathbb{R}_{\geq 0}\),
\begin{align}\label{limitStil}
	m(z) \;\coloneqq\; m_{F}(z)
	\;=\; -\Bigl(1 - \frac{1}{\gamma}\Bigr)\frac{1}{z}
	\;+\; \frac{1}{\gamma}\, v(z),
\end{align}
almost surely.

\end{remark}

\begin{proof}[Proof of Theorem~\ref{LPTh1}] 
Let \(\{\Theta_n\}_{n \ge 1}\), with \(\Theta_n \in \mathbb{R}^{p \times p}\), 
be symmetric positive semidefinite matrices satisfying \(\|\Theta_n\|_F = O(p^{-1/2})\) for all \(n\).
First, we establish that, for each $z \in \mathbb{C} \setminus \mathbb{R}_{\ge 0}$,
\begin{equation*}
	\tr\bigl(\Theta (\hat{\Sigma}_n -z  I_p)^{-1}\bigr) - 
	\tr\bigl(\Theta(-z v(z) \Sigma_n - z I_p)^{-1}\bigr)
	\;\overset{\text{a.s.}}{\longrightarrow}\; 0.
\end{equation*}
In the second step, we set $\Theta_n = p^{-1} g(\Sigma_n)$ for a continuous function $g : [0,\infty) \to \mathbb{R}$ and show that, for each $z \in \mathbb{C} \setminus \mathbb{R}_{\ge 0}$,
\begin{equation*}
	\frac{1}{p} \tr\bigl(g(\Sigma_n)(-z v(z) \Sigma_n - z I_p)^{-1}\bigr)
	= \int_{0}^\infty \frac{g(t)}{-z v(z) t - z} \, dF_{\Sigma_n}(t)
	\;\longrightarrow\; \int_{0}^\infty \frac{g(t)}{-z v(z) t - z} \, dH(t).
\end{equation*}
The extension to functions $g : [0,\infty) \to \mathbb{R}$ with finitely many discontinuities follows by the same arguments as in \citet{LedPeche2011}.

We begin with some notation and elementary algebra. Let
\[
\hat{\Sigma}_n = \frac{1}{n}\sum_{j=1}^{n} x_j x_j^{\top},
\]
where $x_1,\dots,x_n$ denote the rows of $X$. For each $j$, define the leave-one-out matrix
\[
\hat{\Sigma}_{n,-j} \coloneqq \frac{1}{n}\sum_{k=1,k\neq j}^{n} x_k x_k^{\top}.
\]
For any $z \in \mathbb{C}\setminus \mathbb{R}_{\ge 0}$, set
\[
\hat{R}_n(z) \coloneqq (\hat{\Sigma}_n - z I_p)^{-1},
\qquad 
\hat{R}_{n,-j}(z) \coloneqq (\hat{\Sigma}_{n,-j} - z I_p)^{-1}.
\]
By the Sherman--Morrison--Woodbury formula, we have, for $z \in \mathbb{C} \setminus \mathbb{R}_{\ge 0}$,
\begin{equation}\label{She-Woo}
	\hat{R}_n(z) 
	= \Bigl(\hat{\Sigma}_{n,-j} - z I_p + \frac{x_j x_j^\top}{n}\Bigr)^{-1} 
	= \hat{R}_{n,-j}(z) 
	- \frac{\hat{R}_{n,-j}(z)\, \bigl(\frac{x_j x_j^\top}{n}\bigr)\, \hat{R}_{n,-j}(z)}
	{1 + \frac{1}{n}\, x_j^\top \hat{R}_{n,-j}(z)\, x_j}.
\end{equation}
Using \eqref{Stielcon}, we may write, for $z \in \mathbb{C} \setminus \mathbb{R}_{\ge 0}$,
\begin{align}\label{PrepStilt}
	-\,z\,v_n(z) 
	\;=\; 1 - \gamma_n \;-\; \frac{z}{n}\,\tr\bigl(\hat{R}_n(z)\bigr).
\end{align}
Using \eqref{She-Woo}, we obtain
\begin{align*}
	-\frac{z}{n}\tr\big(\hat{R}_n(z)\big)
	&= \frac{1}{n} \sum_{j=1}^p \frac{-z}{s_i - z} = \gamma_n - \frac{1}{n} \tr\big(\hat{\Sigma}_n \hat{R}_n(z)\big) = \gamma_n - \frac{1}{n^2} \sum_{j=1}^n \tr\big(x_j x_j^\top \hat{R}_n(z)\big) \\
	&= \gamma_n - \frac{1}{n} \sum_{j=1}^n \frac{\frac{1}{n}x_j^\top \hat{R}_{n,-j}(z) x_j}{1 + \frac{1}{n}x_j^\top \hat{R}_{n,-j}(z) x_j}.
\end{align*}
By \eqref{PrepStilt} and the preceding display,
\begin{align}\label{lambdavn}
	\begin{split}
		-z\, v_n(z)
		&= 1 - \gamma_n - \frac{z}{n}\,\tr\!\big(\hat{R}_n(z)\big) = 1 - \frac{1}{n}\sum_{j=1}^{n}
		\frac{\frac{1}{n}\, x_j^{\top}\hat{R}_{n,-j}(z)\,x_j}
		{1 + \frac{1}{n}\, x_j^{\top}\hat{R}_{n,-j}(z)\,x_j} \\
		&= \frac{1}{n}\sum_{j=1}^{n}
		\frac{1}{1 + \frac{1}{n}\, x_j^{\top}\hat{R}_{n,-j}(z)\,x_j}.
	\end{split}
\end{align}
For $z \in \mathbb{C}\setminus\mathbb{R}_{\ge 0}$, define
\[
R_n(z) \coloneqq \big(-z\, v(z)\,\Sigma_n - z I_p\big)^{-1}.
\]
Note that both $\hat{R}_n(z)$ and $R_n(z)$ are invertible for all such $z$.
Applying the resolvent identity yields
\[
\hat{R}_n(z) - R_n(z)
= \hat{R}_n(z)\big(-z\,v(z)\,\Sigma_n - \hat{\Sigma}_n\big) R_n(z).
\]
Thus, for any $z \in \mathbb{C}\setminus\mathbb{R}_{\ge 0}$,
\begin{align}\label{stiltdec}
	\begin{split}
		\hat{R}_n(z) - R_n(z)
		&= \hat{R}_n(z)\big(-z\,v(z)\,\Sigma_n - \hat{\Sigma}_n\big) R_n(z) \\
		&= -z\,v(z)\,\hat{R}_n(z)\Sigma_n R_n(z)
		- \frac{1}{n}\sum_{j=1}^n \hat{R}_n(z)\,x_j x_j^{\top} R_n(z) \\
		&= \bigl(-z\,v(z) + z\,v_n(z)\bigr)\hat{R}_n(z)\Sigma_n R_n(z) \\
		&\quad - z\,v_n(z)\,\hat{R}_n(z)\Sigma_n R_n(z)
		- \frac{1}{n}\sum_{j=1}^n \hat{R}_n(z)\,x_j x_j^{\top} R_n(z).
	\end{split}
\end{align}
Next, multiplying \eqref{stiltdec} on the left by \(\Theta_n\), and using \eqref{She-Woo} and \eqref{lambdavn}, we obtain
\begin{align}\label{decStiC}
	\Theta_n (\hat{R}_n(z) - R_n(z))
	= (-z v(z) + z v_n(z))\, \Theta_n \hat{R}_n(z) \Sigma_n R_n(z) + \Theta_n A_n R_n(z),
\end{align}
where
\begin{align}
		A_n 
		&\coloneqq - z\, v_n(z)\, \hat{R}_n(z) \Sigma_n - \frac{1}{n} \sum_{j=1}^n \hat{R}_n(z)\, x_j x_j^\top \notag\\
		&= \frac{1}{n} \sum_{j=1}^n \Biggl( \frac{1}{1 + \frac{1}{n}\, x_j^\top \hat{R}_{n,-j}(z) x_j}\, \hat{R}_n(z) \Sigma_n - \hat{R}_n(z)\, x_j x_j^\top \Biggr) \notag\\
		&= \frac{1}{n} \sum_{j=1}^n \frac{1}{1 + \frac{1}{n}\, x_j^\top \hat{R}_{n,-j}(z) x_j} \Bigl( \hat{R}_n(z)\Sigma_n - \hat{R}_{n,-j}(z)x_j x_j^\top \Bigr)\notag\\
		&=\frac{1}{n} \sum_{j=1}^n \frac{1}{1 + \frac{1}{n}\, x_j^\top \hat{R}_{n,-j}(z) x_j}  \Bigl( \hat{R}_{n,-j}(z)\Sigma_n - \hat{R}_{n,-j}(z)x_j x_j^\top\label{AdecC1}\\
		&\quad- \frac{1}{1+\frac{1}{n}x_j^\top\hat{R}_{n,-j}(z)x_j}\hat{R}_{n,-j}(z)\bigl(\frac{x_jx_j^\top}{n}\bigr)\hat{R}_{n,-j}(z)\Sigma_n \Bigr) \label{AdecC2}.
\end{align}
Before proceeding with the proof of the theorem, we state some elementary facts that will be used repeatedly.

\medskip

First, for each \(z \in \mathbb{R}_{<0}\), it is immediate that
\begin{align}\label{Resolb:real}
\|R_n(z)\|_2 \le \frac{1}{|z|}.
\end{align}
Second, for \(z=u+iv\) with \(v\neq 0\), we compute
\begin{align}\label{samesign1}
	\begin{split}
		zv(z)
		&= \int_{0}^{\infty} \frac{z}{\,s-z\,}\, d\bar F(s)
		= \int_{0}^{\infty} \frac{(u+iv)(s-u+iv)}{(s-u)^2+v^2}\, d\bar F(s) \\
		&= \int_{0}^{\infty} \frac{u(s-u)-v^2}{(s-u)^2+v^2}\, d\bar F(s)
		\;+\; i \int_{0}^{\infty} \frac{sv}{(s-u)^2+v^2}\, d\bar F(s).
	\end{split}
\end{align}
Thus, from \eqref{samesign1}, we conclude that 
$\Im(z v(z))$ and $\Im(z)$ share the same sign.
Third, for $z \in \mathbb{C}\setminus\mathbb{R}$ and $t \ge 0$, 
\begin{align}\label{R_nupper}
	\begin{split}
		\biggl|\frac{1}{-z v(z) t - z}\biggr|^2 
		&= \frac{1}{|z v(z) t + z|^2} = \frac{1}{\bigl(\Re(z v(z) t + z)\bigr)^2 + \bigl(\Im(z v(z) t + z)\bigr)^2} \\[1mm]
		&\le \frac{1}{\bigl(\Im(z v(z))\, t + \Im(z)\bigr)^2} 
		\le \frac{1}{(\Im z)^2},
	\end{split}
\end{align}
where the last inequality follows from the fact that $\Im(z v(z))$ has the same sign as $\Im(z)$. Since the eigenvalues of $R_n(z)$ are 
$\{(-z v(z) t_j - z)^{-1}\}_{j=1}^{p}$, where $\{t_j\}_{j=1}^{p}$ are the 
eigenvalues of $\Sigma_n$, we obtain, for each $z \in \mathbb{C}\setminus\mathbb{R}$,
\begin{align}\label{Resolb:complex}
	\|R_n(z)\|_2
	= \max_{1\le j \le p} \bigl|(-z v(z) t_j - z)^{-1}\bigr|
	\le \frac{1}{|\Im(z)|}.
\end{align}
Since the same computations can be done for $v_n(z)$ instead of $v(z)$, we conclude that, for each $z \in \mathbb{C}\setminus\mathbb{R}$,
\begin{align*}
\|(-z v_n(z) \Sigma_n - z I_p)^{-1}\|_2\leq  \frac{1}{|\Im(z)|}.
\end{align*}
It is straightforward to verify that, the same bounds as in \eqref{Resolb:real} and \eqref{Resolb:complex} 
apply to
\(\|\hat{R}_n(z)\|_2\), and 
\(\|\hat{R}_{n,-j}(z)\|_2\), 
with $|\Im(z)|^{-1}$ for \(z \in \mathbb{C}\setminus \mathbb{R}\) and 
$|z|^{-1}$ for \(z \in \mathbb{R}_{< 0}\).
\medskip

Recall \eqref{decStiC}. In order to complete the first step of the proof, it 
remains to verify that, for each $z \in \mathbb{C}\setminus \mathbb{R}_{\ge 0}$,
\begin{equation}\label{conv1}
	\operatorname{tr}(\Theta_n A_n R_n(z)) \;\overset{\text{a.s.}}{\longrightarrow}\; 0
\end{equation}
and
\begin{equation}\label{conv2}
	(-z v(z) + z v_n(z))\, \operatorname{tr}(\Theta_n \hat{R}_n(z) \Sigma_n R_n(z)) 
	\;\overset{\text{a.s.}}{\longrightarrow}\; 0.
\end{equation}
We begin with the claim in \eqref{conv2}.  
By the Portmanteau theorem, applied to the real and imaginary parts of \(v_n(z)\) (see \eqref{samesign2}), it follows that for each \(z \in \mathbb{C}\setminus \mathbb{R}_{\ge 0}\),
\[
v_n(z)\;\xrightarrow{\text{a.s.}}\; v(z).
\]
Moreover, since for each $z \in \mathbb{C} \setminus \mathbb{R}_{\ge 0}$ the 
spectral norms of $\hat{R}_n(z)$, $R_n(z)$, and $\Sigma_n$ are uniformly 
bounded in $n$, and $\|\Theta_n\|_{F} = O(p^{-1/2})$ and $\Theta_n$ is a symmetric positive semidefinite matrix for every $n$, it follows that
\begin{align*}
\bigl|\operatorname{tr}(\Theta_n \hat{R}_n(z) \Sigma_n R_n(z))\bigr| &\le \|\hat{R}_n(z) \Sigma_n R_n(z)\|_2 \operatorname{tr}(\Theta_n )\\ 
&\leq  K_1(z)\, \operatorname{tr}(\Theta_n)
\le K_1(z)\, \|\Theta_n\|_F \|I_p\|_F\\ 
&= \sqrt{p}\, K_1(z)\, \|\Theta_n\|_F\;=O(1),
\end{align*}
for some constants $K_1(z)>0$ depending only on $z$. Therefore, we conclude that, for each $z \in \mathbb{C} \setminus \mathbb{R}_{\ge 0}$,
\[
\bigl(- z v(z) + z v_n(z)\bigr)\, 
\operatorname{tr}(\Theta_n \hat{R}_n(z) \Sigma_n R_n(z)) 
\;\;\overset{\text{a.s.}}{\longrightarrow}\;\; 0.
\]

For \eqref{conv1}, we use the representation of $A_n$ from \eqref{AdecC1}--\eqref{AdecC2} together with $x_j = \Sigma_n^{1/2} z_j$ to obtain
\begin{align}\label{Rest:LPmod}
	\begin{split}
		\bigl| \tr(\Theta_n A_n R_n(z)) \bigr|
		&\le
		\left|
		\frac{1}{n} \sum_{j=1}^n
		\frac{
			\tr\!\bigl(\Sigma_n^{1/2} R_n(z) \Theta_n \hat R_{n,-j}(z) \Sigma_n^{1/2}\bigr)
			-
			z_j^\top \Sigma_n^{1/2} R_n(z) \Theta_n \hat R_{n,-j}(z) \Sigma_n^{1/2} z_j
		}{
			1 + n^{-1} x_j^\top \hat R_{n,-j}(z) x_j
		}
		\right| \\[0.4em]
		&\quad+
		\left|
		\frac{1}{n^2} \sum_{j=1}^n
		\frac{
			z_j^\top \Sigma_n^{1/2} \hat R_{n,-j}(z) \Sigma_n R_n(z) \Theta_n \hat R_{n,-j}(z) \Sigma_n^{1/2} z_j
		}{
			\bigl(1 + n^{-1} x_j^\top \hat R_{n,-j}(z) x_j\bigr)^2
		}
		\right|.
	\end{split}
\end{align}
For convenience, we define
\begin{align*}
B_{-j}&=B_{-j}^{(n)}(z) \coloneqq \Sigma_n^{1/2} R_n(z)\, \Theta_n \, \hat R_{n,-j}(z) \, \Sigma_n^{1/2}\\
D_{-j}&=D_{-j}^{(n)}(z)\coloneqq \Sigma_n^{1/2}\hat{R}_{n,-j}(z)\Sigma_nR_n(z)\Theta_n\hat{R}_{n,-j}(z)\Sigma_n^{1/2}.
\end{align*}
When $z \in \mathbb{R}_{<0}$, $\hat R_{n,-j}(z)$ is positive definite and hence 
$x^\top \hat R_{n,-j}(z) x \ge 0$ for all $x \in \mathbb{R}^p$ and $j\in \{1,\dots,n\}$. Using this bound together with \eqref{Rest:LPmod}, we deduce that for each $z \in \mathbb{R}_{<0}$,
\begin{align}\label{Rest:LPmod:real}
	\begin{split}
		 \bigl|\tr(\Theta_n A_n R_n(z)) \bigr|
		&\le
		\frac{1}{n} \sum_{j=1}^n
		\bigl| z_j^\top B_{-j} z_j-\tr\!\bigl(B_{-j}\bigr) 
		\bigr|
		+
		\frac{1}{n^2} \sum_{j=1}^n
		\bigl|z_j^\top D_{-j} z_j
		\bigr|.
	\end{split}
\end{align}
In the case $z = u + i v$, $v \neq 0$ and for an arbitrary $j\in \{1,\dots,n\}$, consider the spectral decomposition 
$\hat{\Sigma}_{n,-j} = \sum_{k=1}^p \lambda_k u_k u_k^\top$. 
For any nonzero vector $x \in \mathbb{R}^p$ and $y_k = u_k^\top x$, $k \in \{1, \dots, p\}$, we have
\begin{align}\label{decstilt:quad}
	\begin{split}
		\frac{z}{n} x^\top \hat{R}_{n,-j}(z) x 
		&= \frac{1}{n} \sum_{k=1}^p y_k^2 \frac{z}{\lambda_k - z} 
		= \frac{1}{n} \sum_{k=1}^p y_k^2 \frac{(u + i v)(\lambda_k - u + i v)}{(\lambda_k - u)^2 + v^2} \\
		&= \frac{1}{n} \sum_{k=1}^p y_k^2 \frac{u (\lambda_k - u) - v^2}{(\lambda_k - u)^2 + v^2} 
		+ \frac{i}{n} \sum_{k=1}^p y_k^2 \frac{v \lambda_k}{(\lambda_k - u)^2 + v^2}.
	\end{split}
\end{align}
From \eqref{decstilt:quad}, we see that 
$\Im\bigl(n^{-1} z\, x^\top \hat R_{n,-j}(z) x\bigr)$ has the same sign as $\Im(z)$ and therefore
\begin{align}\label{upperb:LPmod2}
	\begin{split}
		|z| \, \Bigl| 1 + \frac{1}{n} x^\top \hat R_{n,-j}(z) x \Bigr|
		&= \Bigl| z + \frac{z}{n} x^\top \hat R_{n,-j}(z) x \Bigr|
		\ge \Bigl| \Im \Bigl( z + \frac{z}{n} x^\top \hat R_{n,-j}(z) x \Bigr) \Bigr| \\
		&= \Bigl| \Im(z) + \Im \Bigl( \frac{z}{n} x^\top \hat R_{n,-j}(z) x \Bigr) \Bigr|
		\ge |\Im(z)|.
	\end{split}
\end{align}
Hence, for all $j\in \{1,\dots,n\}$ and $z\in \mathbb{C}\setminus\mathbb{R}$, we have 
\begin{align}\label{upperb:LPmod3}
	\bigg|\frac{1}{1+\frac{1}{n}x_j^\top \hat{R}_{n,-j}(z)x_j}\bigg|\leq \frac{|z|}{|\Im(z)|}.
\end{align}
Therefore, in the case $z\in \mathbb{C}\setminus\mathbb{R}$, we use \eqref{Rest:LPmod} and \eqref{upperb:LPmod3} to obtain 
\begin{align}\label{Rest:LPmod:complex}
	\begin{split}
		\bigl| \tr(\Theta_n A_n R_n(z)) \bigr|
		&\le\frac{|z|}{|\Im(z)|}
		\frac{1}{n} \sum_{j=1}^n \left|z_j^\top B_{-j} z_j-
		\tr\!\bigl(B_{-j}\bigr) 
		\right|
		+
		\frac{|z|^2}{\Im(z)^2}
		\frac{1}{n^2} \sum_{j=1}^n
		\Bigl|z_j^\top D_{-j} z_j
		\Bigr|.
	\end{split}
\end{align}

To complete the argument that, for each $z \in \mathbb{C} \setminus \mathbb{R}_{\ge 0}$,
$\tr(\Theta_n A_n R_n(z)) \;\overset{\text{a.s.}}{\longrightarrow}\; 0,$
it remains to show that
\[
\frac{1}{n} \sum_{j=1}^n
\bigl|z_j^\top B_{-j} z_j - \tr(B_{-j})\bigr| 
\;\overset{\text{a.s.}}{\longrightarrow}\; 0,
\quad \text{and} \quad
\frac{1}{n^2} \sum_{j=1}^n
\bigl|z_j^\top D_{-j} z_j\bigr|
\;\overset{\text{a.s.}}{\longrightarrow}\; 0.
\]
Recall that the matrices involved in $D_{-j}$ and $B_{-j}$ are 
$R_n(z)$, $\hat{R}_n(z)$, $\hat{R}_{n,-j}(z)$, $\Sigma_n$, and $\Theta_n$ for each $j\in \{1,\dots,p\}$.  
Moreover, we have established that, for each $z \in \mathbb{C} \setminus \mathbb{R}_{\ge 0}$, 
the matrices $R_n(z)$, $\hat{R}_n(z)$, and $\hat{R}_{n,-j}(z)$ have uniformly bounded spectral norms in $n$.
Likewise, $\Sigma_n$ has uniformly bounded spectral norm in $n$ by Assumption~\ref{A2}, 
and $\|\Theta_n\|_{F} = O(p^{-1/2})$. We write $X'_n\coloneqq\Sigma_n^{1/2}R_n(z)$ and $Y'_n\coloneqq\hat{R}_{n,j}(z)\Sigma_n^{1/2}$ for $z\in \mathbb{C}\setminus\mathbb{R}_{\geq 0}$. Then, for each $z \in \mathbb{C} \setminus \mathbb{R}_{\ge 0}$ and 
$j \in \{1,\dots,p\}$, we have
\begin{align}\label{spectralb:B1}
	\begin{split}
		\operatorname{tr}(B_{-j} B_{-j}^\ast) &= \|X'_n\Theta_n Y'_n\|_F^2\leq \|X'_n\|_2^2\,\|Y'_n\|_2^2\, \|\Theta_n\|_F^2
		\\
		&\le K_2(z)\, \|\Theta_n\|_F^2 
		\le \frac{K_3(z)}{p},
	\end{split}
\end{align}
for some constants $K_2(z), K_3(z) > 0$. 
Consequently, for some $\delta>0$, we obtain
\begin{align}\label{spectralb:B2}
	\operatorname{tr}\bigl((B_{-j} B_{-j}^\ast)^{1+\delta}\bigr) 
	\le \operatorname{tr}(B_{-j} B_{-j}^\ast)^{1+\delta} 
	\le \biggl(\frac{K_3(z)}{p}\biggr)^{1+\delta}.
\end{align}
Using the same arguments with
\(X_n'' \coloneqq \Sigma_n^{1/2}\hat{R}_{n,-j}(z)\Sigma_n R_n(z)\)
and
\(Y_n'' \coloneqq \hat{R}_{n,-j}(z)\Sigma_n^{1/2}\)
in place of \(X_n\) and \(Y_n\), it follows that for each
\(z \in \mathbb{C} \setminus \mathbb{R}_{\ge 0}\) and any \(\delta>0\),
\begin{align}\label{spectralb:D}
	\operatorname{tr}(D_{-j} D_{-j}^\ast) 
	\le \frac{K_4(z)}{p}, \qquad
	\operatorname{tr}\bigl((D_{-j} D_{-j}^\ast)^{1+\delta}\bigr) 
	\le \biggl(\frac{K_4(z)}{p}\biggr)^{1+\delta},
\end{align}
for some constant $K_4(z) > 0$. By the independence of $z_j$ and $B_{-j}$, Lemma~\ref{quadform} with 
$q = 2+\varepsilon/2, \varepsilon>0$ and \eqref{spectralb:B1}--\eqref{spectralb:B2}, we obtain for each $z \in \mathbb{C} \setminus \mathbb{R}_{\ge 0}$ and 
$j \in \{1,\dots,p\}$,
\begin{align}\label{Bbound}
	\begin{split}
		\mathbb{E}\!\left(|z_j^\top B_{-j} z_j - \tr(B_{-j})|^{q}\mid B_{-j}\right)
		&\le 
		C_q\!\left[\bigl(\nu_4\,\tr(B_{-j}B_{-j}^*)\bigr)^{q/2}
		+ \nu_{2q}\,\tr\!\bigl((B_{-j}B_{-j}^*)^{q/2}\bigr)\right] \\[0.4em]
		&\le C_q \!\left[(\nu_4^{q/2}+\nu_{2q})\biggl(\frac{K_3(z)}{p}\biggr)^{q/2}\right]=
		\frac{K_5(z)}{p^{\,1+\varepsilon/4}},
	\end{split}
\end{align}
for some constant $K_5(z)>0$. Using Markov's inequality and Jensen's inequality, together with the upper bound in \eqref{Bbound}, we obtain for each $z \in \mathbb{C} \setminus \mathbb{R}_{\ge 0}$ and arbitrary $\delta>0$,
\begin{align*}
	\mathbb{P}\Bigg(
	\frac{1}{n} \sum_{j=1}^n \bigl| z_j^\top B_{-j} z_j - \operatorname{tr}(B_{-j}) \bigr| > \delta
	\Bigg)
	&\le  
	\frac{1}{\delta^q} \, \mathbb{E} \Bigg(
	\Big(\frac{1}{n} \sum_{j=1}^n \bigl| z_j^\top B_{-j} z_j - \operatorname{tr}(B_{-j}) \bigr| \Big)^q
	\Bigg) \\[1mm]
	&\le \frac{1}{\delta^q} \, \mathbb{E} \Bigg(
	\frac{1}{n} \sum_{j=1}^n \bigl| z_j^\top B_{-j} z_j - \operatorname{tr}(B_{-j}) \bigr|^q
	\Bigg) \le \frac{K_6(z)}{p^{\,1+\varepsilon/4}}
	= \frac{K_6(z)}{(\gamma_n n)^{\,1+\varepsilon/4}},
\end{align*}
for some constant $K_6(z) > 0$. Assumption~\ref{A1} implies that the sequence $\{\gamma_n\}_{n \ge 1}$ is bounded. 
Hence,
\[
\sum_{n=1}^\infty \frac{K_6(z)}{(\gamma_n n)^{\,1+\varepsilon/4}} < \infty,
\]
and it follows from the Borel-Cantelli lemma that for each $z\in \mathbb{C}\setminus \mathbb{R}_{\geq 0}$,
\[
\frac{1}{n} \sum_{j=1}^n |z_j^\top B_{-j} z_j - \tr(B_{-j})|
\;\overset{\text{a.s.}}{\longrightarrow}\; 0,
\quad \text{as } n \to \infty.
\]
For any fixed $j \in \{1,\ldots,n\}$, applying the inequality 
$(a+b)^q \le 2^{\,q-1}(a^q + b^q)$ for $a,b \ge 0$ with $q = 2 + \varepsilon/2$ and $\varepsilon>0$, 
and using Lemma~\ref{quadform} together with the bounds in~\eqref{spectralb:D}, 
we obtain, for each $z \in \mathbb{C} \setminus \mathbb{R}_{\ge 0}$, that
\begin{align}\label{D:bound1}
	\begin{split}
	\mathbb{E}(|z_j^\top D_{-j} z_j|^q)&\leq	2^{q-1}\mathbb{E}(|z_j^\top D_{-j} z_j-\tr(D_{-j})|^q) +2^{q-1}\mathbb{E}(|\tr(D_{-j})|^q)\\
	&\leq \frac{K_7(z)}{p^{q/2}} +2^{q-1}\mathbb{E}(|\tr(D_{-j})|^q)
	\end{split}
\end{align}
for some constant $K_7(z)>0$.
Using the uniform spectral bounds on 
$R_n(z)$, $\hat{R}_n(z)$, $\hat{R}_{n,-j}(z)$, and $\Sigma_n$, 
together with the assumption $\|\Theta_n\|_{F} = O(p^{-1/2})$, 
it follows that, for each $z \in \mathbb{C} \setminus \mathbb{R}_{\ge 0}$ 
and each $j \in \{1,\dots,p\}$,
\begin{align}\label{D:bound2}
	2^{q-1} \, \mathbb{E}\bigl(|\operatorname{tr}(D_{-j})|^q\bigr)
	\leq K_8(z)\, \operatorname{tr}(\Theta_n)^q
	\leq p^{q/2} K_8(z) \, \|\Theta_n\|_{F}^q
	\le K_9(z),
\end{align}
for some constants $K_8(z), K_9(z) > 0$.
Combining the bounds in \eqref{D:bound1}--\eqref{D:bound2}, we deduce that, for each $z\in \mathbb{C}\setminus \mathbb{R}_{\geq 0}$ and $j\in \{1,\dots,p\}$,
\[\mathbb{E}(|z_j^\top D_{-j} z_j|^q)\leq \frac{K_7(z)}{p^{q/2}}+K_9(z).
\]
Hence, using Markov's inequality, the Jensen inequality and the preceeding display, we obtain for arbitrary $z\in \mathbb{C}\setminus \mathbb{R}_{\geq 0}$ and $\delta > 0$,
\begin{align*}
	\mathbb{P}\Bigg(
	\Big|\frac{1}{n^2} \sum_{j=1}^n z_j^\top D_{-j} z_j \Big| > \delta
	\Bigg)
	\le \frac{1}{(n\delta)^q}\mathbb{E}\Bigg(
	\Big|\frac{1}{n} \sum_{j=1}^n z_j^\top D_{-j} z_j \Big|^q 
	\Bigg)\leq \frac{1}{(n\delta)^q}
	\frac{1}{n} \sum_{j=1}^n \mathbb{E}\bigl( |z_j^\top D_{-j} z_j|^q 
	\bigr) \leq \frac{K_{10}(z)}{n^q},
\end{align*}
for some constant $K_{10}(z)>0$.  
Since $q>1$ yields $K_{10}(z) \sum_{n=1}^\infty n^{-q} < \infty$,
we conclude by the Borel-Cantelli lemma that for each $z\in \mathbb{C}\setminus \mathbb{R}_{\geq 0}$, \[
\frac{1}{n^2} \sum_{i=1}^n | z_i^\top D_{-i} z_i | \;\overset{\text{a.s.}}{\longrightarrow}\; 0.
\]
Putting all arguments together, we conclude that, for each $z \in \mathbb{C} \setminus \mathbb{R}_{\ge 0}$,
\begin{align}\label{fstepLPTh1}
\tr(\Theta_n \hat{R}_n(z)) - \tr(\Theta_n R_n(z))
\;\overset{\text{a.s.}}{\longrightarrow}\; 0.
\end{align}
This completes the first step of the proof.
 
The second step is considerably simpler. To this end, let $g:[0,\infty)\to\mathbb{R}$ be a continuous function. 
Note that, for all $n$, the matrix $g(\Sigma_n)$ has uniformly bounded spectral norm.
Moreover, by Assumption~\ref{A2}, we have
\[
\|p^{-1} g(\Sigma_n)\|_F \le p^{-1/2}\, \|g(\Sigma_n)\|_2 \le p^{-1/2}\, C',
\]
for some constant $C' > 0$. Hence, we may apply \eqref{fstepLPTh1} with $\Theta_n = p^{-1} g(\Sigma_n)$.
 
Recall that we have established that, for all $t \ge 0$ and $z \in \mathbb{C} \setminus \mathbb{R}$,
$|\Im(-z v(z) t - z)|^{-1} \le |\Im(z)|^{-1}$ (cf.~\eqref{R_nupper}), and therefore,
\[
\bigl|\Im\bigl((-z v(z) t - z)^{-1}\bigr)\bigr|
= \frac{|\Im(-z v(z) t - z)|}{\Re(-z v(z) t - z)^2 + \Im(-z v(z) t - z)^2}
\le \frac{1}{|\Im(-z v(z) t - z)|} \le \frac{1}{|\Im(z)|}.
\]
Moreover, applying the same argument used to obtain the bound on the real part in \eqref{bound:realandim}, we have, for all $t \ge 0$ and $z \in \mathbb{C} \setminus \mathbb{R}$,
\[
\bigl|\Re\bigl((-z v(z) t - z)^{-1}\bigr)\bigr|
= \frac{|\Re(-z v(z) t - z)|}{\Re(-z v(z) t - z)^2 + \Im(-z v(z) t - z)^2}
\le \frac{1}{2|\Im(-z v(z) t - z)|} \le \frac{1}{2|\Im(z)|}.
\]
In the case where $z \in \mathbb{R}_{<0}$, we obtain
\[
|- z v(z) t - z|^{-1} \le |z|^{-1}.
\]
Using Assumption~\ref{A2}, the continuity of $g$ and applying the Portmanteau theorem separately to the real and imaginary parts of the function
$t \;\mapsto\; g(t)\,(-z\, v(z)\, t - z)^{-1}$,
we obtain that, for each $z \in \mathbb{C} \setminus \mathbb{R}_{\ge 0}$,

\[
\frac{1}{p} \, \tr \bigl( g(\Sigma_n) (-z v(z) \Sigma_n - z I_p)^{-1} \bigr)
= \int_0^\infty \frac{g(t)}{-z v(z)\, t - z} \, dF_{\Sigma_n}(t)
\;\longrightarrow\;
\int_0^\infty \frac{g(t)}{-z v(z)\, t - z} \, dH(t),
\]
as \(n \to \infty\).
This completes the second step of the proof.

To extend the result to a function $g:[0,\infty)\to\mathbb{R}$ with finitely many discontinuities, we can use the same arguments as \citet[Theorem 2]{LedPeche2011}.

%To extend the result for a continuous function with finitely many discontinuities, we follow the same arguments as \citet[Theorem 2]{LedPeche2011}. The case that it holds for all continuous functions $g:[0,\infty)\mapsto\mathbb{R}$, $k=0$ discontinuities, has been established so far. Let us assume that it holds for all continuous functions $g:[0,\infty)\mapsto \mathbb{R}$ with $k\in \mathbb{N}$ discontinuities. Consider a continuous function $f$ with $k+1$ discontinuities and let the point $d\in \mathbb{R}_{\geq 0}$ be one of them. Construct the function $\tilde{f}:[0,\infty)\mapsto\mathbb{R}$ with $x\mapsto f(x)(x-d)$ and notice that $\tilde{f}$ has the same discontinuities as the function $f$, but is continuous in $d$. By the induction hypothesis, for each $z\in \mathbb{C}\setminus\mathbb{R}_{\geq 0}$, 
%\[\frac{1}{p}\tr(\tilde{f}(\Sigma_n)(\hat{\Sigma}_n-zI_p)^{-1})\overset{a.s.}{\longrightarrow}\int_0^\infty \frac{\tilde{f}(t)}{-z\,v(z)t-z}dH(t).
%\]

\end{proof}	

\begin{remark}\label{LpTh1:remark}
Recall from Theorem~\ref{Pan2010} that, for \(z \in \mathbb{C}^+\),
$v(z) = m_{\bar{F}}(z)$
is implicitly given by
\[
v(z) = -\Bigl(z - \gamma \int_0^\infty \frac{t }{1 + t v(z)}\, dH(t)\Bigr)^{-1}.
\]
Using the implicit representation of \(v(z)\) for \(z \in \mathbb{C}^+\), and after some straightforward algebra, one obtains
\[
v(z) = -(1 - \gamma)\frac{1}{z} + \gamma \int_0^\infty \frac{1}{- z v(z) t - z} \, dH(t).
\]
Combining the previous display with \eqref{limitStil}, we then conclude that, for \(z \in \mathbb{C}^+\),
\[
m(z) = -\Bigl(1 - \frac{1}{\gamma}\Bigr)\frac{1}{z} + \frac{1}{\gamma} v(z)
= \int_0^\infty \frac{1}{- z v(z) t - z} \, dH(t).
\]
Note that this implicit representation of \(m(z)\) can be extended to all \(z \in \mathbb{C} \setminus \mathbb{R}_{\ge 0}\) using Theorem~\ref{LPTh1}. 
Recall that, for each \(z \in \mathbb{C} \setminus \mathbb{R}_{\ge 0}\),
$
m_n(z) \;\overset{\text{a.s.}}{\longrightarrow}\; m(z),
$
(see Remark~\ref{rm:remark}). Combining this almost sure convergence with Theorem~\ref{LPTh1} applied to \(g(\Sigma_n) = I_p\), we obtain, for each \(z\in \mathbb{C}\setminus\mathbb{R}_{\geq 0}\),
\begin{equation}\label{imprep:m}
m_n(z)
= \frac{1}{p}\,\tr\!\bigl( (\hat{\Sigma}_n -z I_p)^{-1} \bigr)
\;\overset{\text{a.s.}}{\longrightarrow}\;
\int_0^\infty \frac{1}{- z v(z) t -z} \, dH(t) = m(z).
\end{equation}

\end{remark}

\medskip

\begin{proof}[Proof of Lemma~\ref{LPmod}]
For the first assertion of the lemma, consider \(g(\Sigma_n) = \Sigma_n\).
By Theorem~\ref{LPTh1}, for each $\lambda>0$ we have
\begin{align}\label{lemma:3.8:p1}
	\frac{1}{p}\,\tr\!\bigl(\Sigma_n (\hat{\Sigma}_n + \lambda I_p)^{-1}\bigr)
	\;\overset{\text{a.s.}}{\longrightarrow}\;
	\int_0^\infty \frac{t}{\lambda v(-\lambda) t + \lambda}\, dH(t)
	\eqqcolon h(\lambda),
\end{align}
where the null set associated with the almost sure convergence depends on the choice of $\lambda$. We first show that
\begin{align}\label{Lem3.7I}
	h(\lambda)
	= \frac{1}{\gamma}\,
	\frac{1 - \lambda v(-\lambda)}{1 - \gamma\bigl(1 - \lambda v(-\lambda)\bigr)},
\end{align}
for each $\lambda>0$.
In the second step, we verify that the convergence in \eqref{lemma:3.8:p1} holds for all \(\lambda>0\) on a single null set that does not depend on \(\lambda\). We rely on the following facts.  

First, using \eqref{imprep:m} for \(z = -\lambda, \lambda>0\), we obtain
\[
m_n(-\lambda)
= \frac{1}{p}\,\tr\!\bigl( (\hat{\Sigma}_n + \lambda I_p)^{-1} \bigr)
\;\overset{\text{a.s.}}{\longrightarrow}\;
\int_0^\infty\frac{1}{t\lambda v(-\lambda)+\lambda} = m(-\lambda).
\] Second, since \(v(-\lambda) > 0\), it follows that \(\lambda\, v(-\lambda) > 0\). Therefore, we obtain
\begin{align*}
	h(\lambda)=\int_0^\infty \frac{t}{\lambda v(-\lambda)\, t + \lambda} \, dH(t)
	&= \frac{1}{\lambda v(-\lambda)} \int_0^\infty \frac{t}{t + \frac{\lambda}{\lambda v(-\lambda)}} \, dH(t)= \frac{1}{\lambda v(-\lambda)} \int_0^\infty \frac{t + \frac{\lambda}{\lambda v(-\lambda)} - \frac{\lambda}{\lambda v(-\lambda)}}{t + \frac{\lambda}{\lambda v(-\lambda)}} \, dH(t)\\
	&= \frac{ 1 - \lambda h(\lambda)}{\lambda v(-\lambda)} = \frac{1 - \lambda m(-\lambda)}{\lambda v(-\lambda)}.
\end{align*}
Using \eqref{limitStil}, we deduce that
\[
1 - \lambda m(-\lambda) 
= \biggl( 1 - \biggl(1 - \frac{1}{\gamma}\biggr) - \frac{\lambda}{\gamma} \, v(-\lambda) \biggr) 
= \frac{1}{\gamma} \, \bigl(1 - \lambda v(-\lambda)\bigr),
\]
and hence,
\begin{align*}
h(\lambda)=\frac{1 - \lambda m(-\lambda)}{\lambda v(-\lambda)}
	= \frac{1}{\gamma} \frac{1 - \lambda v(-\lambda)}{\lambda v(-\lambda)}.
\end{align*}
This establishes the equality in \eqref{Lem3.7I}.

For convenience, for $n\in \mathbb{N}$, define the function $h_n:(0,\infty)\times \Omega\to \mathbb{R}$ by 
\[
h_n(\lambda,\omega) \coloneqq \frac{1}{p}\,\mathrm{tr}\!\Bigl(\Sigma_n (\hat{\Sigma}_n + \lambda I_p)^{-1}\Bigr).
\]
So far, we have established that for each \(\lambda > 0\) there exists a null set \(N_\lambda\) such that
\[
h_n(\lambda, \omega) \;\longrightarrow\; h(\lambda), 
\quad \text{as } n \to \infty, \text{ for all } \omega \in N_\lambda^{c}.
\]
We now show that the convergence above holds for all \(\lambda>0\) on a single null set that does not depend on \(\lambda\).

To this end, consider an arbitrary $m\in\mathbb{N}$, the dense subset $D_m\coloneqq[1/m,m]\cap \mathbb{Q}$ in $[1/m,m]$ and $N=\cup_{\lambda \in \mathbb{Q}\cap \mathbb{R}_{>0}} N_\lambda$. Since \(N\) is a countable union of sets of measure zero, it also has measure zero. We thus conclude that, for each \(\lambda \in D_m\),
\begin{align}\label{countcon1}
	h_n(\lambda, \omega) \;\longrightarrow\; h(\lambda), 
	\quad \text{as } n \to \infty, \text{ for all } \omega \in N^c.
\end{align}
For the remaining part of the proof fix an arbitrary $\omega\in N^c$ and we are going to suppress the dependence on $\omega$ for notational simplicity. Next, we collect several facts that will be useful in the remainder of the proof.

First, by \citet[Theorem~6.28]{Klenke}, we have
\[
\frac{d}{d\lambda} v(-\lambda)
= \frac{d}{d\lambda}\int_0^\infty \frac{1}{s+\lambda}\, d\bar{F}(s)
= - \int_0^\infty \frac{1}{(s+\lambda)^{2}}\, d\bar{F}(s).
\]
Hence, \(\tfrac{d}{d\lambda} v(-\lambda) < 0\) for all $\lambda>0$, which implies that \(\lambda \mapsto v(-\lambda)\) is strictly decreasing for $\lambda>0$.
Second, since $\lambda \mapsto v(-\lambda)$ is differentiable for $\lambda>0$, it is also continuous on $(0,\infty)$. 
Third, by the continuity of $\lambda \mapsto v(-\lambda)$ on $(0,\infty)$ and the representation of $h$ in \eqref{Lem3.7I}, it follows that $h$ is continuous on $(0,\infty)$ and uniformly continuous on $[1/m, m]$. 
Fourth, using the representation in \eqref{Lem3.7I} of $h$, we observe that the numerator is monotonically decreasing in $\lambda$, while the denominator is monotonically increasing in $\lambda$. 
It follows that \(h\) is monotonically decreasing in \(\lambda\). Furthermore, for each $n \in \mathbb{N}$, the function $h_n$ is monotonically decreasing in $\lambda$.
This follows directly from the spectral decomposition \(\hat{\Sigma}_n = \sum_{i=1}^p s_i u_i u_i^\top\).  
For \(0 < \lambda_1 \leq \lambda_2<\infty\) and all \(n \in \mathbb{N}\), we have
\[
h_n(\lambda_1) - h_n(\lambda_2)
= \frac{1}{p} \sum_{i=1}^p \left( \frac{u_i^\top \Sigma_n u_i}{s_i + \lambda_1} - \frac{u_i^\top \Sigma_n u_i}{s_i + \lambda_2} \right) \geq 0.
\]

We now continue with the proof. Since \(h\) is uniformly continuous on \([1/m, m]\) and \(D_m\) is dense in \([1/m, m]\), we may, for any \(\varepsilon>0\), find points
$
\frac{1}{m} = \lambda_0 < \lambda_1 < \cdots < \lambda_M = m,
$
such that
$
|h(\lambda_i) - h(\lambda_{i+1})| < \varepsilon/2$, for all $i \in \{0,\dots, M-1\}.
$
Now, let $\lambda \in [1/m, m]$ and $\varepsilon>0$ be arbitrary. 
Then there exists an index $i \in \{0,\dots, M-1\}$ such that
$\lambda \in [\lambda_i, \lambda_{i+1}]$.
Furthermore, by the monotonicity of the functions $\{h_n\}_{n\ge 1}$, we have
\begin{align}\label{monohlam}
	h_n(\lambda_i) \ge h_n(\lambda) \ge h_n(\lambda_{i+1}),
\end{align}
for all $n \in \mathbb{N}$.
Using \eqref{countcon1} we can choose an $n_0\in \mathbb{N}$, such that $|h_n(\lambda_i)-h(\lambda_i)|<\varepsilon/2$ and $|h_n(\lambda_{i+1})-h(\lambda_{i+1})|<\varepsilon/2$ for all $n\geq n_0$. Therefore, 
\begin{align}\label{hlamboundI}
	h_n(\lambda_i)&= h(\lambda_i)+ h_n(\lambda_i)-h(\lambda_i)< h(\lambda_i) +\frac{\varepsilon}{2}
\end{align}	
and
\begin{align}\label{hlamboundII}
	h_n(\lambda_{i+1})&=  h(\lambda_{i+1})+ h_n(\lambda_{i+1})- h(\lambda_{i+1}) >  h(\lambda_{i+1})-\frac{\varepsilon}{2},
\end{align}
for all $n\geq n_0$.
Putting together \eqref{monohlam}, \eqref{hlamboundI} and \eqref{hlamboundII}, we get
\begin{align*}
	h(\lambda_i)+\frac{\varepsilon}{2}> h_n(\lambda_i)\ge h_n(\lambda)\ge h_n(\lambda_{i+1})>h(\lambda_{i+1})-\frac{\varepsilon}{2},
\end{align*}
for all $n\geq n_0$.
Since $\lambda \in[\lambda_i,\lambda_{i+1}]$ and by the monotonicity of $h$ we deduce that $|h(\lambda_i)-h(\lambda)|<\varepsilon/2$ and $|h(\lambda_{i+1})-h(\lambda)|<\varepsilon/2$. Together with the previous display, we obtain
\[
h(\lambda) + \varepsilon \ge h_n(\lambda) \ge h(\lambda) - \varepsilon
\]
for \(n\geq n_0\) and therefore, for each \(\lambda \in [1/m, m]\),
\begin{align}\label{countcon2}
	h_n(\lambda, \omega) \;\longrightarrow\; h(\lambda), 
	\quad \text{as } n \to \infty, \text{ for all } \omega \in N^c.
\end{align}
Since the convergence in \eqref{countcon2} holds for every \(m \in \mathbb{N}\), we conclude that, for each \(\lambda > 0\),
\begin{align}\label{countcon}
	h_n(\lambda, \omega) \;\longrightarrow\; h(\lambda),
	\quad \text{as } n \to \infty, \text{ for all } \omega \in N^c.
\end{align}

For the second statement of the lemma, note that
\[
h_n'(\lambda) \coloneqq \frac{d}{d\lambda} h_n(\lambda)
= - \frac{1}{p}\,\mathrm{tr}\!\Bigl(\Sigma_n (\hat{\Sigma}_n + \lambda I_p)^{-2} \Bigr),
\]
for all $\lambda>0$.
Recall that
$\frac{d}{d\lambda} v(-\lambda)
= - \int_0^\infty \frac{1}{(s+\lambda)^2}\, d\bar{F}(s),$
and together with the representation of $h$ in \eqref{Lem3.7I}, we obtain that the derivative $\frac{d}{d\lambda}h(\lambda)$ exists for all $\lambda>0$ and satisfies
\begin{align*}
	h'(\lambda)
	\coloneqq \frac{d}{d\lambda} 
	\bigg(
	\frac{1}{\gamma}\,
	\frac{1-\lambda v(-\lambda)}{\lambda v(-\lambda)}
	\bigg)= -\,\frac{1}{\gamma}\,
	\frac{v(-\lambda)+\lambda\,\tfrac{d}{d\lambda} v(-\lambda)}{(\lambda v(-\lambda))^2}.
\end{align*}
Throughout the remainder of the proof, let $0 < \lambda_1 < \lambda_2 < \infty$, and recall that we consider a fixed $\omega \in N^c$.  
From \eqref{countcon}, it follows that, for each $\lambda \in [\lambda_1, \lambda_2]$, 
\[
h_n(\lambda) \longrightarrow h(\lambda) \quad \text{as } n \to \infty.
\]
We shall prove that $h_n'$ converges uniformly to $h'$ on $[\lambda_1, \lambda_2]$, which is stronger than the statement given in the theorem.

To this end, define
$M_n \coloneqq \sup_{x\in[\lambda_1,\lambda_2]}
\bigl| h_n'(x) - h'(x) \bigr|.$
Suppose, for the sake of contradiction, that
\(\limsup_{n\to\infty} M_n > 0\).
Then there exists a subsequence \(\{n_k\}_{k\geq 1}\) such that
\(M_{n_k} \to c > 0\) as $k\to \infty$.
We next collect several properties of the family of functions
\(\{h_{n_k}'\}_{k\ge 1}\)
that will be used to derive the desired contradiction.

First, observe that the functions \(\{h_{n_k}'\}_{k \ge 1}\) are continuous in \(\lambda\) and satisfy
\[
|h_{n_k}'(\lambda)| \le \frac{C}{\lambda_1^2}, \quad \text{for all } \lambda \in [\lambda_1, \lambda_2] \text{ and all } k \ge 1,
\]
by Assumption~\ref{A2}. Second, for each \(k \ge 1\), the functions \(h_{n_k}'\) are \(C' \lambda_1^{-3}\)-Lipschitz on \([\lambda_1, \lambda_2]\), for some constant \(C' > 0\). It follows that the family of functions \(\{h_{n_k}'\}_{k\ge 1}\) is equicontinuous on \([\lambda_1,\lambda_2]\).
Third, by the previous two points and \citet[Theorems 7.12 and 7.25]{Rudin1976}, there exists a further subsequence \(\{n_{k_l}\}_{l \ge 1}\) such that
\[
h_{n_{k_l}}' \;\longrightarrow\; g' \quad \text{uniformly on } [\lambda_1, \lambda_2]
\]
as \(l \to \infty\), for some continuous function \(g'\).

We conclude, by the fundamental theorem of calculus and the dominated convergence theorem, that for all \(t \in [\lambda_1,\lambda_2]\),
\begin{align*}
	h(t) - h(\lambda_1) 
	&= \lim_{l\to \infty} \bigl(h_{n_{k_l}}(t) - h_{n_{k_l}}(\lambda_1)\bigr) = \lim_{l\to \infty} \int_{\lambda_1}^t h_{n_{k_l}}'(x)\, dx \\
	&= \int_{\lambda_1}^t \lim_{l\to \infty} h_{n_{k_l}}'(x)\, dx = \int_{\lambda_1}^t g'(x)\, dx.
\end{align*}

It follows that \(g' = h'\) almost everywhere on \([\lambda_1,\lambda_2]\), and by continuity of \(g'\) and \(h'\), also for all \(\lambda \in [\lambda_1,\lambda_2]\).  
In other words, \(M_{n_{k_l}} \to 0\) as $l\to \infty$, which contradicts the assumption that \(\limsup_{n\to\infty} M_n > 0\).

\end{proof}

\begin{proof}[Proof of Theorem~\ref{Ridgetuned}]
The first statement of Theorem~\ref{Ridgetuned} follows almost immediately from Lemma~\ref{LPmod} and the decomposition of \(\mathrm{R}_{\Sigma_n}(\hat{\beta}_R(\lambda))\) in Lemma~\ref{LemforTh1}\ref{dec2}.  
For \(\lambda > 0\), we have by Lemma~\ref{LemforTh1}\ref{dec2},
\begin{align}\label{Rdec}
	\begin{split}
		\mathrm{R}_{\Sigma_n}(\hat{\beta}_R(\lambda)) 
		&= \tr(\Sigma_n F)= \frac{\tau^2 \lambda^2}{p}\, \tr\Bigl(\Sigma_n (\hat{\Sigma}_n + \lambda I_p)^{-2}\Bigr)
		+ \frac{\sigma^2}{n}\, \tr\Bigl(\Sigma_n (\hat{\Sigma}_n + \lambda I_p)^{-2} \hat{\Sigma}_n \Bigr) \\
		&= \Bigl(\frac{\lambda^2 \tau^2}{p} - \frac{\lambda \sigma^2}{n}\Bigr)\, \tr\Bigl(\Sigma_n (\hat{\Sigma}_n + \lambda I_p)^{-2}\Bigr)
		+ \frac{\sigma^2}{n}\, \tr\Bigl(\Sigma_n (\hat{\Sigma}_n + \lambda I_p)^{-1}\Bigr) \\
		&=  (\lambda \tau^2 - \sigma^2 \gamma_n)\, \frac{\lambda}{p} \tr\Bigl(\Sigma_n (\hat{\Sigma}_n + \lambda I_p)^{-2}\Bigr)
		+ \frac{\sigma^2 \gamma_n}{p} \tr\Bigl(\Sigma_n (\hat{\Sigma}_n + \lambda I_p)^{-1}\Bigr).
	\end{split}
\end{align}
Using Lemma~\ref{LPmod} in the last line of \eqref{Rdec} and the convergence of $\gamma_n\to \gamma\in(0,\infty)$, we obtain almost surely that, for all \(\lambda > 0\),
\begin{align}\label{Ridgelimit}
	\mathrm{R}_{\Sigma_n}(\hat{\beta}_R(\lambda)) 
	\;\longrightarrow\;\mathrm{R}(\lambda), \quad \text{as } n\to\infty,
\end{align}
where 
\begin{align*}
	\mathrm{R}(\lambda)= 
	\lambda(\lambda \tau^2 - \sigma^2 \gamma) 
	\Biggl( \frac{1}{\gamma} \frac{v(-\lambda) + \lambda \frac{d}{d\lambda} v(-\lambda)}{(\lambda v(-\lambda))^2} \Biggr)
	+ \sigma^2  \Biggl( \frac{1 - \lambda v(-\lambda)}{\lambda v(-\lambda)} \Biggr).
\end{align*}
 This completes the first statement.

For the second statement, recall that 
\(\lambda_n^* = \frac{\sigma^2}{\tau^2}\,\gamma_n\).  
Substituting \(\lambda = \lambda_n^*\) into the last line of \eqref{Rdec} yields
\[
\mathrm{R}_{\Sigma_n}\bigl(\hat{\beta}_R(\lambda_n^*)\bigr)
= 
\frac{\sigma^2 \gamma_n}{p}\,
\mathrm{tr}\!\bigl(\Sigma_n (\hat{\Sigma}_n + \lambda_n^* I_p)^{-1}\bigr).
\]
Likewise, substituting  
\(\lambda^* = \frac{\sigma^2}{\tau^2}\,\gamma\) into the limit expression \(\mathrm{R}(\lambda)\) gives
\[
\mathrm{R}(\lambda^*)
=
\sigma^2  
\Biggl(
\frac{1 - \lambda^* v(-\lambda^*)}{\lambda^* v(-\lambda^*)}
\Biggr).
\]
The strategy is to decompose 
\[
\mathrm{R}_{\Sigma_n}\bigl(\hat{\beta}_R(\hat{\lambda}_n)\bigr)
=\Bigl[\mathrm{R}_{\Sigma_n}\bigl(\hat{\beta}_R(\hat{\lambda}_n)\bigr)
-\mathrm{R}_{\Sigma_n}\bigl(\hat{\beta}_R(\lambda_n^*)\bigr)\Bigr]
+\mathrm{R}_{\Sigma_n}\bigl(\hat{\beta}_R(\lambda_n^*)\bigr),
\]
and to show, in a first step, that
\[
\mathrm{R}_{\Sigma_n}\bigl(\hat{\beta}_R(\lambda_n^*)\bigr)
\;\overset{a.s.}{\longrightarrow}\;
\mathrm{R}(\lambda^*),
\]
and, in a second step, that
\[
\mathrm{R}_{\Sigma_n}\bigl(\hat{\beta}_R(\hat{\lambda}_n)\bigr)
-\mathrm{R}_{\Sigma_n}\bigl(\hat{\beta}_R(\lambda_n^*)\bigr)
\;\overset{\mathbb{P}}{\longrightarrow}\;0.
\]
Finally, we verify that $\mathrm{R}(\lambda^*)=\min_{\lambda>0} \mathrm{R}(\lambda)$.

To this end, we first establish some straightforward facts. Recall from \eqref{defestimatorlambda} that \(\hat{\lambda}_n \in [\lambda_1,\lambda_2]\) for all \(n\), and that \(\lambda_n^* \in [\lambda_1,\lambda_2]\) for all sufficiently large \(n\), where \(0<\lambda_1 \leq \lambda_2 < \infty\).
We write $h_n(\lambda)\coloneqq p^{-1} \tr\big(\Sigma_n(\hat{\Sigma}_n+ \lambda I_p\big)^{-1})$ for $\lambda>0$ and note that the derivative of $h_n(\lambda)$ with respect to $\lambda$ is given by $h_n^\prime(\lambda)\coloneqq-p^{-1} \tr\big(\Sigma_n(\hat{\Sigma}_n+ \lambda I_p)^{-2}\big)$. By Assumption~\ref{A2}, we additionally have
$|h_n'(\lambda)| \le C\lambda_1^{-2}$, for all $ \lambda \in [\lambda_1, \lambda_2]$.
Hence, using the mean value theorem and the fact that \(\lambda_n^* \in [\lambda_1, \lambda_2]\) for all sufficiently large \(n\), we obtain
\begin{align*}
	\bigl|\mathrm{R}_{\Sigma_n}(\hat{\beta}_R(\lambda_n^*)) - \mathrm{R}_{\Sigma_n}(\hat{\beta}_R(\lambda^*)) \bigr|
	&=
	\Biggl|
	\frac{\sigma^2 \gamma_n}{p}\,\mathrm{tr}\!\bigl(\Sigma_n (\hat{\Sigma}_n + \lambda_n^* I_p)^{-1}\bigr)
	-
	\frac{\sigma^2 \gamma_n}{p}\,\mathrm{tr}\!\bigl(\Sigma_n (\hat{\Sigma}_n + \lambda^* I_p)^{-1}\bigr)
	\Biggr|\\
	&\le \frac{C'}{\lambda_1^2} \, |\lambda_n^* - \lambda^*|,
\end{align*}
for all sufficiently large $n$, where $C'>0$ is some constant.
Using \(|\lambda_n^* - \lambda^*| \to 0\) as \(n \to \infty\), the previous display, and \eqref{Ridgelimit} evaluated at \(\lambda = \lambda^*\), we conclude that
\begin{align}\label{IIlimopt}
	\begin{split}
		\mathrm{R}_{\Sigma_n}\bigl(\hat{\beta}_R(\lambda_n^*)\bigr)
		&=
		\frac{\sigma^2 \gamma_n}{p} \,
		\mathrm{tr}\!\bigl(\Sigma_n (\hat{\Sigma}_n + \lambda_n^* I_p)^{-1}\bigr)
		-
		\frac{\sigma^2 \gamma_n}{p} \,
		\mathrm{tr}\!\bigl(\Sigma_n (\hat{\Sigma}_n + \lambda^* I_p)^{-1}\bigr)\\
		&\quad+
		\frac{\sigma^2 \gamma_n}{p} \,
		\mathrm{tr}\!\bigl(\Sigma_n (\hat{\Sigma}_n + \lambda^* I_p)^{-1}\bigr)
		\;\overset{\text{a.s.}}{\longrightarrow}\; 
		R(\lambda^*).
	\end{split}
\end{align}
This completes the first step of the second statement.

We proceed with the second step. By the Jensen inequality, the Cauchy--Schwarz inequality, Assumption~\ref{A2}, and the triangle inequality, we obtain
\begin{align}
	\bigl|\mathrm{R}_{\Sigma_n}(\hat{\beta}_R(\hat{\lambda}_n)) - \mathrm{R}_{\Sigma_n}(\hat{\beta}_R(\lambda_n^*))\bigr|
	&= \Bigl|\mathbb{E}\Bigl[
	(\hat{\beta}_R(\hat{\lambda}_n)-\beta)^\top \Sigma_n (\hat{\beta}_R(\hat{\lambda}_n)-\beta) 
	-
	(\hat{\beta}_R(\lambda_n^*)-\beta)^\top \Sigma_n (\hat{\beta}_R(\lambda_n^*)-\beta)
	\;\Big|\; X\Bigr]\Bigr| \notag \\
	&\leq \mathbb{E}\Bigl[\bigl|
	(\hat{\beta}_R(\hat{\lambda}_n) + \hat{\beta}_R(\lambda_n^*) - 2\beta)^\top \Sigma_n (\hat{\beta}_R(\hat{\lambda}_n) - \hat{\beta}_R(\lambda_n^*))
	\bigr|\;\Big|\; X\Bigr] \notag \\
	&\le \mathbb{E}\Bigl[
	\bigl\| \Sigma_n (\hat{\beta}_R(\hat{\lambda}_n) + \hat{\beta}_R(\lambda_n^*) - 2\beta) \bigr\|_2
	\,
	\bigl\| \hat{\beta}_R(\hat{\lambda}_n) - \hat{\beta}_R(\lambda_n^*) \bigr\|_2
	\;\Big|\; X\Bigr] \notag \\
	&\le C \, \mathbb{E}\Bigl[
	\bigl\| \hat{\beta}_R(\hat{\lambda}_n) + \hat{\beta}_R(\lambda_n^*) - 2\beta \bigr\|_2
	\,
	\bigl\| \hat{\beta}_R(\hat{\lambda}_n) - \hat{\beta}_R(\lambda_n^*) \bigr\|_2
	\;\Big|\; X\Bigr] \notag \\
	&\le C \, \mathbb{E}\Bigl[
	\bigl\| \hat{\beta}_R(\hat{\lambda}_n) + \hat{\beta}_R(\lambda_n^*) \bigr\|_2
	\,
	\bigl\| \hat{\beta}_R(\hat{\lambda}_n) - \hat{\beta}_R(\lambda_n^*) \bigr\|_2
	\;\Big|\; X\Bigr] \label{RidgeriskupperIi} \\
	&\quad + 2 C \, \mathbb{E}\Bigl[
	\|\beta\|_2
	\,
	\bigl\| \hat{\beta}_R(\hat{\lambda}_n) - \hat{\beta}_R(\lambda_n^*) \bigr\|_2
	\;\Big|\; X\Bigr]. \label{RidgeriskupperIii}
\end{align}
Note that 
	\begin{align*}
		\hat{\beta}_R(\hat{\lambda}_n)+\hat{\beta}_R(\lambda^*_n))&=((X^\top X+n\hat{\lambda}_n I_p)^{-1} +(X^\top X+n\lambda_n^* I_p)^{-1})X^\top y
		\\&=((\hat{\Sigma}_n+\hat{\lambda}_n I_p)^{-1} +(\hat{\Sigma}_n+\lambda_n^* I_p)^{-1})\frac{X^\top y}{n}.
	\end{align*}
We define  
\begin{align*}
	A\coloneqq (\hat{\Sigma}_n+\hat{\lambda}_n I_p)^{-1} +(\hat{\Sigma}_n+\lambda_n^* I_p)^{-1}\quad\text{and}\quad B\coloneqq (\hat{\Sigma}_n+\hat{\lambda}_n I_p)^{-1} -(\hat{\Sigma}_n+\lambda_n^* I_p)^{-1}.
\end{align*}
Let $\{s_i\}_{i=1}^p$ denote the eigenvalues of $\hat{\Sigma}_n$.  
By the definition of $\hat{\lambda}_n$, we have $\hat{\lambda}_n \in [\lambda_1, \lambda_2]$. Moreover, $\lambda_n^* \in [\lambda_1, \lambda_2]$ for all sufficiently large $n$.
Using these facts, the triangle inequality yields
\begin{align}\label{BoundA}
	\begin{split}
	\|A\|_2
	&= \max_{1\leq i\leq p} \biggl| \frac{1}{s_i + \hat{\lambda}_n} + \frac{1}{s_i + \lambda_n^*} \biggr| \\
	&\le \max_{1\leq i\leq p} \biggl( \biggl| \frac{1}{s_i + \hat{\lambda}_n} \biggr| + \biggl| \frac{1}{s_i + \lambda_n^*} \biggr| \biggr) \le \frac{2}{\lambda_1},
	\end{split}
\end{align}
for all sufficiently large $n$. Similarly, invoking a Lipschitz–continuity argument and writing 
$g_i(\lambda) \coloneqq (s_i+\lambda)^{-1}$, $i\in\{1,\ldots,p\}$ and $\lambda>0$, 
for the eigenvalues of $(\hat{\Sigma}_n+\lambda I_p)^{-1}$, with 
$g_i'(\lambda) \coloneqq \frac{d}{d\lambda} g_i(\lambda) = -(s_i+\lambda)^{-2}$,
we obtain, for all sufficiently large $n$, that
	\begin{align}\label{LipschitzB}
		\begin{split}
		\|B\|_2
		&= \max_{1\leq i\leq p}
		\left\{
		\left|
		\frac{1}{s_i+\hat{\lambda}_n}
		- 
		\frac{1}{s_i+\lambda_n^*}
		\right|
		\right\}\\
		&\le 
		\max_{1\leq i\leq p}
		\left\{
		\sup_{\lambda\in[\lambda_1,\lambda_2]}
		|g_i'(\lambda)|\,\big|\hat{\lambda}_n-\lambda_n^*\big|
		\right\}
		\le 
		\frac{1}{\lambda_1^{2}}\,\big|\hat{\lambda}_n-\lambda_n^*\big|.
		\end{split}
	\end{align}
Using the bounds in \eqref{BoundA} and \eqref{LipschitzB} for the quantity in
\eqref{RidgeriskupperIi}, and noting that
\[
s_{\max}(\hat{\Sigma}_n)
\le 
s_{\max}(\Sigma_n)\, s_{\max}(n^{-1}Z^\top Z)
\le 
C\, s_{\max}(n^{-1}Z^\top Z),
\]
we obtain
\begin{align}\label{RidgeriskupperII}
		C\,\mathbb{E}\!\left(
		\|\hat{\beta}_R(\hat{\lambda}_n)+\hat{\beta}_R(\lambda_n^*)\|_2\,
		\|\hat{\beta}_R(\hat{\lambda}_n)-\hat{\beta}_R(\lambda_n^*)\|_2
		\,\big|\, X
		\right)
		&\le 
		C\,\frac{2}{\lambda_1^3}\,
		\mathbb{E}\!\left(
		\frac{\|X^\top y\|_2^2}{n^2}\,
		|\hat{\lambda}_n-\lambda_n^*|
		\,\big|\, X
		\right)\notag
		\\[0.4em]
		&\le 
		C^2\,\frac{2\,s_{\max}(n^{-1}Z^\top Z)}{\lambda_1^3}\,
		\mathbb{E}\!\left(
		\frac{\|y\|_2^2}{n}\,
		|\hat{\lambda}_n-\lambda_n^*|
		\,\big|\, X
		\right).
\end{align}
Using the decomposition
\[
\frac{\|y\|_2^2}{n}
= \frac{\|X\beta + u\|_2^2}{n}
\le \frac{(\|X\beta\|_2 + \|u\|_2)^2}{n}
\le 2\,\beta^\top \hat{\Sigma}_n \beta
+ 2\,\frac{\|u\|_2^2}{n},
\]
and substituting this bound into the conditional expectation in
\eqref{RidgeriskupperII}, we obtain
\begin{align}\label{RidgeriskupperIII}
	\mathbb{E}\!\left(
	\frac{\|y\|_2^2}{n}\,
	|\hat{\lambda}_n - \lambda_n^*|
	\,\big|\, X
	\right)
	&\le
	2\,\mathbb{E}\!\left(
	\beta^\top \hat{\Sigma}_n \beta\,
	|\hat{\lambda}_n - \lambda_n^*|
	\,\big|\, X
	\right)
	+
	2\,\mathbb{E}\!\left(
	\frac{\|u\|_2^2}{n}\,
	|\hat{\lambda}_n - \lambda_n^*|
	\,\big|\, X
	\right).
\end{align}
Next, we show that
\[
\mathbb{E}\Bigl(
\beta^\top \hat{\Sigma}_n \beta \, |\hat{\lambda}_n - \lambda_n^*|
\,\big|\, X
\Bigr)
\;\overset{\mathbb{P}}{\longrightarrow}\; 0
\quad\text{and}\quad
\mathbb{E}\left(
\frac{\|u\|_2^2}{n} \, |\hat{\lambda}_n - \lambda_n^*|
\,\big|\, X
\right)
\;\overset{\mathbb{P}}{\longrightarrow}\; 0.
\]
By Assumption~\ref{A5} and the Cauchy--Schwarz inequality, we have
\[
\mathbb{E}\bigl(\|\beta\|_2^4\bigr)
\leq \sum_{i=1}^p \mathbb{E}(\beta_i^4)
+ \sum_{\substack{i,j=1 \\ i\neq j}}^p \mathbb{E}(\beta_i^4)^{1/2} \mathbb{E}(\beta_j^4)^{1/2}
\le \nu_{4,\beta}.
\]
Together with the conditional Cauchy--Schwarz inequality and the independence of $\beta$ and $X$,
this implies that
\begin{align}\label{IIIupperbound}
	\begin{split}
		\mathbb{E}\Bigl(
		\beta^\top \hat{\Sigma}_n \beta \, |\hat{\lambda}_n - \lambda_n^*|
		\,\big|\, X
		\Bigr)
		&\le
		s_{\max}(\hat{\Sigma}_n)\,
		\mathbb{E}\Bigl(
		\|\beta\|_2^2 \, |\hat{\lambda}_n - \lambda_n^*|
		\,\big|\, X
		\Bigr)
		\\[0.3em]
		&\le
		s_{\max}(\hat{\Sigma}_n)\,
		\mathbb{E}\bigl(\|\beta\|_2^4\bigr)^{1/2}
		\,
		\mathbb{E}\Bigl(
		|\hat{\lambda}_n - \lambda_n^*|^2
		\,\big|\, X
		\Bigr)^{1/2}
		\\[0.3em]
		&\leq
		C \, s_{\max}(n^{-1} Z^\top Z)\,\,
	\nu_{4,\beta}^{1/2}\,
		\,
		\mathbb{E}\Bigl(
		|\hat{\lambda}_n - \lambda_n^*|^2
		\,\big|\, X
		\Bigr)^{1/2}.
	\end{split}
\end{align}
By Theorem~\ref{est1}, the continuous mapping theorem, and Assumption~\ref{A1}, we have
\[
|\hat{\lambda}_n - \lambda_n^*| \le |\hat{\lambda}_n - \lambda^*| + |\lambda_n^* - \lambda^*| \;\overset{\mathbb{P}}{\longrightarrow}\; 0,
\qquad
|\hat{\lambda}_n - \lambda_n^*|^k \;\overset{\mathbb{P}}{\longrightarrow}\; 0,\,k\in \mathbb{N}.
\]
Moreover, there exists a constant \(K<\infty\) such that \(|\hat{\lambda}_n - \lambda_n^*| \le K\) for all \(n \in \mathbb{N}\), which follows from the boundedness of the sequence \(\{\lambda_n^*\}_{n\ge 1}\) and the truncation of the estimator \(\hat{\lambda}_n\).
Hence, by the Markov inequality, the tower property, and the dominated convergence theorem, 
we obtain for any $\varepsilon>0$ and $k\in \mathbb{N}$,
\begin{align}\label{exp:lambdaconv}
	\mathbb{P}\Bigl(\, \bigl| \mathbb{E}\bigl[|\hat{\lambda}_n - \lambda_n^*|^k \mid X\bigr] \bigr| > \varepsilon \,\Bigr)
	&\le 
	\frac{\mathbb{E}\bigl(\mathbb{E}[|\hat{\lambda}_n - \lambda_n^*|^k \mid X]\bigr)}{\varepsilon}= \frac{\mathbb{E}(|\hat{\lambda}_n - \lambda_n^*|^k)}{\varepsilon} 
	\;\longrightarrow\; 0,
\end{align}
as $n \to \infty$.
Revisiting \eqref{IIIupperbound}, we conclude that, by Lemma~\ref{lebip} and the previous considerations,
\begin{align}\label{RidgeriskupperIV}
	\mathbb{E}\Bigl(
	\beta^\top \hat{\Sigma}_n \beta \, |\hat{\lambda}_n - \lambda_n^*|
	\,\big|\, X
	\Bigr)
	\;\overset{\mathbb{P}}{\longrightarrow}\; 0.
\end{align}
Similarly, by Assumption~\ref{A5} and the Cauchy--Schwarz inequality, we have
\[
\mathbb{E}\!\left[\frac{\|u\|_2^4}{n^2}\right]
= \frac{1}{n^2} \sum_{i=1}^n \mathbb{E}(u_i^{4})
+ \frac{1}{n^2} \sum_{\substack{i,j=1\\ i\neq j}}^{n} \mathbb{E}(u_i^{2} u_j^{2})
\le \nu_{4,u}.
\]
Together with the conditional Cauchy--Schwarz inequality and the independence of $u$ and $X$, this implies
\begin{align*}
	\mathbb{E}\left(\frac{\|u\|_2^2}{n} \, |\hat{\lambda}_n - \lambda_n^*| \,\big|\, X \right)
	&\le 
	\mathbb{E}\left[\frac{\|u\|_2^4}{n^2}\right]^{1/2} \,
	\mathbb{E}\Bigl(|\hat{\lambda}_n - \lambda_n^*|^2 \,\big|\, X \Bigr)^{1/2} \\
	&\le 
	\nu_{4,u}^{1/2} \,\,
	\mathbb{E}\Bigl(|\hat{\lambda}_n - \lambda_n^*|^2 \,\big|\, X \Bigr)^{1/2}.
\end{align*}
Using \eqref{exp:lambdaconv},
it follows that
\begin{equation}\label{RidgeriskupperV}
	\mathbb{E}\left(\frac{\|u\|_2^2}{n} \, |\hat{\lambda}_n - \lambda_n^*| \,\big|\, X \right) \;\overset{\mathbb{P}}{\longrightarrow}\; 0,
\end{equation}
as \(n \to \infty\).
Putting together the arguments above, we conclude that the quantity in \eqref{RidgeriskupperIi} converges in probability to zero.

Next, we show that the quantity in \eqref{RidgeriskupperIii} also converges in probability to zero using similar arguments as above.  
Observe that, 
\begin{align*}
	\|\hat{\beta}_R(\hat{\lambda}_n) - \hat{\beta}_R(\lambda_n^*)\|_2^2
	&\le \|B\|_2^2 \, \frac{\|X^\top y\|_2^2}{n^2} 
	\le C \, s_{\max}(n^{-1} Z^\top Z) \|B\|_2^2 \, \frac{\|y\|_2^2}{n} \\
	&\le C \, s_{\max}(n^{-1} Z^\top Z) \, \frac{1}{\lambda_1^4} \, |\hat{\lambda}_n - \lambda_n^*|^2 \, \frac{\|y\|_2^2}{n},
\end{align*}
for all sufficiently large \(n\), and
\[
\left[C \, s_{\max}(n^{-1} Z^\top Z) \frac{1}{\lambda_1^4} 
\mathbb{E}\left(\frac{\|y\|_2^2}{n} \, |\hat{\lambda}_n - \lambda_n^*|^2 \,\bigg|\, X \right) \right]^{1/2} \;\overset{\mathbb{P}}{\longrightarrow}\; 0, \quad \text{as } n \to \infty.
\]
Therefore, for the expression in \eqref{RidgeriskupperIii}, it follows from the conditional Cauchy--Schwarz inequality, the equality
\(\mathbb{E}(\|\beta\|_2^2 \mid X) = \mathbb{E}(\|\beta\|_2^2) = \tau^2\), 
and the previous display that
\begin{align*}
		\mathbb{E}\Bigl(\|\beta\|_2 \, \|\hat{\beta}_R(\hat{\lambda}_n) - \hat{\beta}_R(\lambda_n^*)\|_2 \,\big|\, X \Bigr)
		&\le \mathbb{E}(\|\beta\|_2^2 \mid X)^{1/2} \,
		\mathbb{E}\bigl[\|\hat{\beta}_R(\hat{\lambda}_n) - \hat{\beta}_R(\lambda_n^*)\|_2^2 \mid X\bigr]^{1/2} \\
		&\le \tau \, \left[C \, s_{\max}(n^{-1} Z^\top Z) \frac{1}{\lambda_1^4} 
		\mathbb{E}\left(\frac{\|y\|_2^2}{n} \, |\hat{\lambda}_n - \lambda_n^*|^2 \,\bigg|\, X \right) \right]^{1/2} 
		\;\overset{\mathbb{P}}{\longrightarrow}\; 0.
\end{align*}
Putting everything together, we have shown that the quantities in \eqref{RidgeriskupperIi} and \eqref{RidgeriskupperIii} both converge to zero in probability. Hence,
\begin{equation*}
\big|\mathrm{R}_{\Sigma_n}(\hat{\beta}_R(\hat{\lambda}_n)) - \mathrm{R}_{\Sigma_n}(\hat{\beta}_R(\lambda_n^*))\big| \;\overset{\mathbb{P}}{\longrightarrow}\; 0,\quad \text{as }n\to\infty.
\end{equation*}

It remains to verify that \(\lambda^*\) satisfies
$R(\lambda^*) = \min_{\lambda > 0} R(\lambda)$.  
Consider the spectral decomposition \(\hat{\Sigma}_n = \sum_{i=1}^{p} s_i v_i v_i^\top\) and the decomposition in Lemma~\ref{LemforTh1}\ref{dec2}. Then, for all \(\lambda > 0\) and $n\in \mathbb{N}$, we have
\begin{align}\label{optlam*}
	\begin{split}
		\text{R}_{\Sigma_n}(\hat{\beta}_R(\lambda_n^*)) 
		=
		 \sum_{i=1}^{p} v_i^\top \Sigma_n v_i \, f_i(\lambda^*) \le \sum_{i=1}^{p} v_i^\top \Sigma_n v_i \, f_i(\lambda)  
		= \text{R}_{\Sigma_n}(\hat{\beta}_R(\lambda)).
	\end{split}
\end{align}
Using \eqref{Ridgelimit}, \eqref{IIlimopt} and \eqref{optlam*}, we conclude that
\begin{align*}
		R(\lambda^*)\leq R(\lambda),
\end{align*}
for all $\lambda>0$ and therefore proving the optimality of $\lambda^*$.
	
\end{proof}

\begin{proof}[Proof of Theorem~\ref{GDtuned}]
In the proof of the theorem, we will frequently simplify notation by writing 
\(\hat{\beta}_{t}(\hat{\lambda}_n,\hat{\eta}_n)\) in place of 
\(\hat{\beta}_{t}(\hat{\lambda}_n,\hat{\eta}_n(\hat{\lambda}_n))\).  
This is justified because the step size \(\hat{\eta}_n(\hat{\lambda}_n)\) is tuned using the same
regularization parameter that appears in the first argument of the gradient-descent estimator.  

For the first statement of the theorem, we use the same bounds as in the proof of Theorem~\ref{Ridgetuned} (cf.\ \eqref{RidgeriskupperIi} and \eqref{RidgeriskupperIii}) to obtain
\begin{align}
	\bigl|
	\mathrm{R}_{\Sigma_n}\!\bigl(\hat{\beta}_{t}(\hat{\lambda}_n,\hat{\eta}_n)\bigr)
	&-
	\mathrm{R}_{\Sigma_n}\!\bigl(\hat{\beta}_{t}(\lambda_n^*,\hat{\eta}_n)\bigr)
	\bigr|\notag\\
	&\le
	C \Biggl[
	\mathbb{E}\!\Bigl(
	\|\hat{\beta}_{t}(\hat{\lambda}_n,\hat{\eta}_n)
	+\hat{\beta}_{t}(\lambda_n^*,\hat{\eta}_n)\|_2\,
	\|\hat{\beta}_{t}(\hat{\lambda}_n,\hat{\eta}_n)
	-\hat{\beta}_{t}(\lambda_n^*,\hat{\eta}_n)\|_2
	\,\big|\, X \Bigr)  \label{GDupperI}
	\\ &\qquad\quad
	+\,2\,\mathbb{E}\!\Bigl(
	\|\beta\|_2\,
	\|\hat{\beta}_{t}(\hat{\lambda}_n,\hat{\eta}_n)
	-\hat{\beta}_{t}(\lambda_n^*,\hat{\eta}_n)\|_2
	\,\Big|\, X \Bigr)
	\Biggr]. \label{GDupperII}
\end{align}
We now develop all the upper bounds required to show that the quantities in \eqref{GDupperI} and \eqref{GDupperII} converge to zero in probability.

By Proposition~\ref{prop1}, the RGD estimator $\hat{\beta}_{t}(\lambda,\hat{\eta}_n)$ admits the representation
\[
\hat{\beta}_{t}(\lambda,\hat{\eta}_n)
=
\hat{\beta}_{R}(\lambda)
- A_\lambda^{\,t}\, \hat{\beta}_{R}(\lambda)
+ A_\lambda^{\,t}\,\theta,
\qquad 
\lambda > 0,\ t\in\mathbb{N},
\]
where $A_\lambda \coloneqq I_p - \hat{\eta}_n(\lambda)(\hat{\Sigma}_n + \lambda I_p)$.
Introduce
\[
B_\lambda \coloneqq (\hat{\Sigma}_n + \lambda I_p)^{-1}\bigl(I_p - A_\lambda^t\bigr),
\qquad 
\lambda > 0,\ t\in \mathbb{N},
\]
so that, by the triangle inequality,
\begin{align}\label{GDriskupperII}
	\begin{split}
	\bigl\|
	\hat{\beta}_{t}(\hat{\lambda}_n,\hat{\eta}_n)
	+
	\hat{\beta}_{t}(\lambda_n^*,\hat{\eta}_n)
	\bigr\|_2
	&=
	\bigl\|
	(B_{\hat{\lambda}_n} + B_{\lambda_n^*})\tfrac{X^\top y}{n}
	+
	\bigl(A_{\hat{\lambda}_n}^{\,t}
	+
	A_{\lambda_n^*}^{\,t}\bigr)\theta
	\bigr\|_2 \\[0.4em]
	&\le
	\bigl\|(B_{\hat{\lambda}_n} + B_{\lambda_n^*})\tfrac{X^\top y}{n}\bigr\|_2
	+
	\bigl\|\bigl(A_{\hat{\lambda}_n}^{\,t}
	+
	A_{\lambda_n^*}^{\,t}\bigr)\theta\bigr\|_2.
	\end{split}
\end{align}
Denote by $s_1\geq s_2\geq\dots\geq s_p$ the ordered eigenvalues of $\hat{\Sigma}_n$.  
It is straightforward to verify that $B_\lambda$ is simultaneously diagonalizable for all $\lambda>0$ and $t\in \mathbb{N}$. We now bound the spectral norm of $B_{\hat{\lambda}_n} + B_{\lambda_n^*}$ using the triangle inequality and Lemma~\ref{ext} for 
$x = \hat{\eta}_n(\hat{\lambda}_n)(s_i+\hat{\lambda}_n)$ and 
$x = \hat{\eta}_n(\lambda_n^*)(s_i+\lambda_n^*)$. We obtain
\begin{align}\label{GDriskupperIII}
	\begin{split}
		\bigl\|B_{\hat{\lambda}_n} + B_{\lambda_n^*}\bigr\|_2
		&= \max_{1\leq i\leq p} \Biggl\{\Biggl| 
		\frac{1-(1-\hat{\eta}_n(\hat{\lambda}_n)(s_i+\hat{\lambda}_n))^t}{s_i+\hat{\lambda}_n} 
		+ \frac{1-(1-\hat{\eta}_n(\lambda_n^*)(s_i+\lambda_n^*))^t}{s_i+\lambda_n^*} 
		\Biggr|\Biggr\} \\
		&\le \max_{1\leq i\leq p} \Biggl\{
		\Biggl|\frac{1-(1-\hat{\eta}_n(\hat{\lambda}_n)(s_i+\hat{\lambda}_n))^t}{s_i+\hat{\lambda}_n}\Biggr|
		+ \Biggl|\frac{1-(1-\hat{\eta}_n(\lambda_n^*)(s_i+\lambda_n^*))^t}{s_i+\lambda_n^*}\Biggr|
		\Biggr\} \\
		&\le \max_{1\leq i\leq p} \Biggl\{
		t \, \max\{1, |1-\hat{\eta}_n(\hat{\lambda}_n)(s_i+\hat{\lambda}_n)|^{t-1}\} \, \hat{\eta}_n(\hat{\lambda}_n) \\
		&\quad\quad + t \, \max\{1, |1-\hat{\eta}_n(\lambda_n^*)(s_i+\lambda_n^*)|^{t-1}\} \, \hat{\eta}_n(\lambda_n^*) 
		\Biggr\}.
	\end{split}
\end{align}
Recall that $\hat{\eta}_n(\lambda) = (\hat{s}_1 + \lambda)^{-1}$ for $\lambda>0$ and that for all $\lambda \in [\lambda_1, \lambda_2]$, we have $\hat{\eta}_n(\lambda) \le \lambda_1^{-1}$.  
By the definition of the estimator $\hat{\lambda}_n$, we have $\hat{\lambda}_n \in [\lambda_1, \lambda_2]$. Moreover, $\lambda_n^* \in [\lambda_1, \lambda_2]$ for all sufficiently large $n$ due to the convergence $\lambda_n^* \to \lambda^* \in (\lambda_1, \lambda_2)$.
Hence, the upper bound in \eqref{GDriskupperIII} reduces to
\begin{align}\label{GDriskupperIV}
	\bigl\|B_{\hat{\lambda}_n} + B_{\lambda_n^*}\bigr\|_2 \le \frac{2t}{\lambda_1} \, \max\Biggl\{ 1, \Bigl| 1 - \frac{s_1 + \lambda_2}{\lambda_1} \Bigr|^{t-1} \Biggr\} \eqqcolon K_1(s_1),
\end{align}
for all sufficiently large $n$.
For the second quantity in \eqref{GDriskupperII}, we use the uniform boundedness of $\|\theta_n\|_2$ in $n$, the triangle inequality, the bound $\hat{\eta}_n(\lambda)\le \lambda_1^{-1}$ for all $\lambda \in [\lambda_1,\lambda_2]$, and the upper bounds on the estimators $\hat{\lambda}_n$ and $\lambda_n^*$ for all sufficiently large $n$. Then, for all sufficiently large $n$, we have
\begin{align}\label{GDriskupperIV2}
	\begin{split}
	\bigl\|(A^t_{\hat{\lambda}} + A^t_{\lambda_n^*}) \theta\bigr\|_2
	&\le \bigl\|(A^t_{\hat{\lambda}} + A^t_{\lambda_n^*})\bigr\|_2 \, \bigl\|\theta\bigr\|_2 \\
	&\le C_\theta \, \max_{1\leq i\leq p} \Bigl| (1 - \hat{\eta}_n(\hat{\lambda}_n)(s_i+\hat{\lambda}_n))^t + (1 - \hat{\eta}_n(\lambda_n^*)(s_i+\lambda_n^*))^t \Bigr| \\
	&\le 2 C_\theta \, \max \Biggl\{ 1, \Bigl| 1 - \frac{s_1 + \lambda_2}{\lambda_1} \Bigr|^t \Biggr\} \eqqcolon K_2(s_1).
	\end{split}
\end{align}
Observe that
\begin{align}\label{GDriskupperV}
	\begin{split}
		\bigl\|\hat{\beta}_{t}(\hat{\lambda}_n,\hat{\eta}_n)-\hat{\beta}_{t}(\lambda_n^*,\hat{\eta}_n)\bigr\|_2
		&= \bigl\|(B_{\hat{\lambda}_n} - B_{\lambda_n^*})\tfrac{X^\top y}{n} + (A^t_{\hat{\lambda}} - A^t_{\lambda_n^*})\theta\bigr\|_2 \\
		&\le \bigl\|(B_{\hat{\lambda}_n} - B_{\lambda_n^*})\tfrac{X^\top y}{n}\bigr\|_2 + \bigl\|(A^t_{\hat{\lambda}} - A^t_{\lambda_n^*})\theta\bigr\|_2.
	\end{split}
\end{align}
We now use a Lipschitz argument to bound $\|B_{\hat{\lambda}_n} - B_{\lambda_n^*}\|_2$ and $\|A^t_{\hat{\lambda}_n}-A^t_{\lambda_n^*}\|_2$. 
Note that $A^t_{\hat{\lambda}}$ and $A^t_{\lambda_n^*}$ are simultaneously diagonalizable. 
For each $i \in \{1,\dots,p\}$, define
\[
g_i(\lambda) = \bigl(1 - \hat{\eta}_n(\lambda)(s_i + \lambda)\bigr)^t, \quad \lambda >0,
\]
and note that the derivative with respect to $\lambda$ is
\begin{align*}
	g_i'(\lambda) 
	\coloneqq \frac{d}{d\lambda} g_i(\lambda)
	= t \, \bigg(1 - \hat{\eta}_n(\lambda)(s_i + \lambda)\bigg)^{t-1} \bigg(\hat{\eta}_n(\lambda)^2 (s_i + \lambda) - \hat{\eta}_n(\lambda)\bigg).
\end{align*}
Furthermore, for all $i \in \{1,\dots,p\}$,
\begin{align*}
	\sup_{\lambda\in [\lambda_1,\lambda_2]}|g_i'(\lambda)|
	\le 
	t \, \max \Bigg\{ 1, \bigg| 1 - \frac{s_1 + \lambda_2}{\lambda_1} \bigg|^{t-1} \Bigg\}
	\Bigg| \frac{s_1 + \lambda_2}{\lambda_1^2} + \frac{1}{\lambda_1} \Bigg| \eqqcolon K_3(s_1).
\end{align*}
Therefore, we conclude that, for all sufficiently large $n$,
\begin{align}\label{GDriskupperVI}
	\begin{split}
	\|(A^t_{\hat{\lambda}} - A^t_{\lambda_n^*}) \theta\|_2
	&\le C_\theta \, \max_{i \in \{1,\dots,p\}} \Bigg| \bigg(1 - \hat{\eta}_n(\hat{\lambda}_n)(s_i + \hat{\lambda}_n)\bigg)^t - \bigg(1 - \hat{\eta}_n(\lambda_n^*)(s_i + \lambda_n^*)\bigg)^t \Bigg| \\
	&\le C_\theta \, K_3(s_1) \, \bigl|\hat{\lambda}_n - \lambda_n^*\bigr|.
	\end{split}
\end{align}
%By Lemma~\ref{lebip}, Theorem~\ref{est1} and the continuous mapping theorem, 
%\[
%C_\theta \, K_3(s_1) \, |\hat{\lambda}_n - \lambda_n^*| \;\overset{p}{\longrightarrow}\; 0 \quad \text{as } n \to \infty.
%\]
Next, we bound the spectral norm of $B_{\hat{\lambda}_n} - B_{\lambda_n^*}$. To this end, for $i \in \{1,\dots,p\}$, define
\begin{equation}\label{def:h}
h_i(\lambda) = \frac{1 - \bigl(1 - \hat{\eta}_n(\lambda)(s_i + \lambda)\bigr)^t}{s_i + \lambda}, \quad \lambda > 0.
\end{equation}
The derivative of $h_i$ with respect to $\lambda$ is
\begin{equation}\label{der:h}
h_i'(\lambda)
\coloneqq \frac{d}{d\lambda} h_i(\lambda)
= \frac{-g_i'(\lambda)}{s_i+\lambda} - \frac{h_i(\lambda)}{s_i+\lambda}.
\end{equation}
By Lemma~\ref{ext} and the bound for $\sup_{\lambda\in [\lambda_1,\lambda_2]}|g_i'(\lambda)|$, we obtain for each
$i \in \{1,\dots,p\}$,
	\begin{align*}
		\sup_{\lambda \in[\lambda_1,\lambda_2]}|h^\prime_i(\lambda)|&\leq \frac{K_3(s_1)}{\lambda_1}+\sup_{\lambda \in[\lambda_1,\lambda_2]} \frac{\hat{\eta}_n(\lambda)}{\lambda_1}\max\bigg\{1, \bigl|1-\hat{\eta}(\lambda)(s_i+\lambda)\bigr|^{t-1}\bigg\} \,t 
		\\&\leq \frac{K_3(s_1)}{\lambda_1}+ \frac{1}{\lambda_1^2}\max\bigg\{1, \Bigl|1-\frac{s_1+\lambda_2}{\lambda_1}\Bigr|^{t-1}\bigg\} \, t \eqqcolon K_4(s_1).
	\end{align*}
Therefore, for all sufficiently large $n$,
	\begin{align}\label{GDriskupperVII}
		\begin{split}
		\|B_{\hat{\lambda}_n} - B_{\lambda_n^*}\|_2
		&= \max_{1\leq i\leq p} \bigl| h_i(\hat{\lambda}_n) - h_i(\lambda_n^*) \bigr| \le \max_{1\leq i\leq p} 
		\bigg\{ \sup_{\lambda \in [\lambda_1,\lambda_2]} |h_i'(\lambda)| \, \bigl| \hat{\lambda}_n - \lambda_n^* \bigr| \bigg\} \\
		&\leq K_4(s_1) \bigl| \hat{\lambda}_n - \lambda_n^* \bigr|.
		\end{split}
	\end{align}
Next, we show that the expression in \eqref{GDupperI} converges in probability to zero.  
Using the Cauchy--Schwarz inequality, we have 
\begin{align}\label{GDriskupperVIII}
	\begin{split}
		\mathbb{E}\bigg(
		\bigl\|&\hat{\beta}_{t}(\hat{\lambda}_n,\hat{\eta}_n)
		+\hat{\beta}_{t}(\lambda_n^*,\hat{\eta}_n)\bigr\|_2
		\,
		\bigl\|\hat{\beta}_{t}(\hat{\lambda}_n,\hat{\eta}_n)
		-\hat{\beta}_{t}(\lambda_n^*,\hat{\eta}_n)\bigr\|_2
		\;\Big|\; X
		\bigg) \\
		&\leq 	\mathbb{E}\bigg(\bigl\|\hat{\beta}_{t}(\hat{\lambda}_n,\hat{\eta}_n)
		+\hat{\beta}_{t}(\lambda_n^*,\hat{\eta}_n)\bigr\|_2^2\Big|X\bigg)^{1/2}\mathbb{E}\bigg(\bigl\|\hat{\beta}_{t}(\hat{\lambda}_n,\hat{\eta}_n)
		-\hat{\beta}_{t}(\lambda_n^*,\hat{\eta}_n)\bigr\|_2^2\Big|X\bigg)^{1/2}.
	\end{split}
\end{align}
By the bounds in \eqref{GDriskupperII}--\eqref{GDriskupperIV2}, using the inequality
$(a+b)^2 \le 2a^2 + 2b^2$ for $a,b \in \mathbb{R}$, together with the uniform boundedness
of $\|\theta_n\|_2$ in $n$, we obtain
\begin{align}\label{B:add}
	\begin{split}
		\mathbb{E}\Bigl(\bigl\|
		\hat{\beta}_{t}(\hat{\lambda}_n,\hat{\eta}_n)
		+
		\hat{\beta}_{t}(\lambda_n^*,\hat{\eta}_n)
		\bigr\|_2^2 \,\big|\, X \Bigr)
		&\le
		\mathbb{E}\Bigg[
		\Bigl(
		\bigl\|
		(B_{\hat{\lambda}_n} + B_{\lambda_n^*})\tfrac{X^\top y}{n}
		\bigr\|_2
		+
		\bigl\|
		(A_{\hat{\lambda}_n}^{\,t}
		+
		A_{\lambda_n^*}^{\,t})\theta
		\bigr\|_2
		\Bigr)^2
		\,\bigg|\, X
		\Bigg] \\
		&\le
		\mathbb{E}\Bigg[
		\Bigl(
		K_1(s_1)\bigl\|\tfrac{X^\top y}{n}\bigr\|_2
		+
		K_2(s_1)\|\theta\|_2
		\Bigr)^2
		\,\bigg|\, X
		\Bigg] \\
		&\le
		2K_1^2(s_1)\,
		\mathbb{E}\bigg(
		\tfrac{\|X^\top y\|_2^2}{n^2}
		\,\Big|\, X
		\bigg)
		+
		2K_2^2(s_1)\,C_\theta^2 .
	\end{split}
\end{align}
Using $n^{-2}\|X^\top y\|_2^{2} \le s_{1}\, n^{-1}\|y\|_2^{2}$ and Lemma~\ref{lebip}, we have
\begin{equation}\label{exp:Op1}
\mathbb{E}\left(\frac{\|X^\top y\|_2^2}{n^2}\bigg|\,X \right)\leq 2 s_1\biggl(\, \frac{\tau^2}{p}\tr(\hat{\Sigma}_n)+\sigma^2\biggr)\leq 2s_1(\tau^2\,s_1+\sigma^2)=O_{\mathbb{P}}(1).
\end{equation}
Similarly, using \eqref{GDriskupperV}--\eqref{GDriskupperVII} and $(a+b)^2\leq 2a^2+2b^2$, for $a,b\in \mathbb{R}$, we have
\begin{align}\label{B:sub}
	\begin{split}
		\mathbb{E}\bigg(\bigl\|\hat{\beta}_{t}(\hat{\lambda}_n,\hat{\eta}_n)
		-\hat{\beta}_{t}(\lambda_n^*,\hat{\eta}_n)\bigr\|_2^2\Big|X\bigg)&\leq 2\,K_4(s_1)^2\,s_1\,\mathbb{E}\left(\bigl|\hat{\lambda}_n-\lambda_n^*\bigr|^2\,\frac{\|y\|_2^2}{n}\;\Big|X\right)\\
		&\quad+ 2\,C_\theta^2\,K_3(s_1)^2\,\mathbb{E}\bigg(\bigl|\hat{\lambda}_n-\lambda_n^*\bigr|^2\;\Big|X\bigg).
\end{split}
\end{align}	
Applying the same arguments as in \eqref{RidgeriskupperIII} and onward, now with $|\hat{\lambda}_{n}-\lambda_{n}^{*}|^2$ in place of $|\hat{\lambda}_{n}-\lambda_{n}^{*}|$, we obtain
\[
\mathbb{E}\!\left(
\bigl| \hat{\lambda}_{n}-\lambda_{n}^{*}\bigr|^2\,\frac{\|y\|_{2}^2}{n} \,\Big|\, X\right)
\overset{\mathbb{P}}{\longrightarrow}0,
\]
as $n\to\infty$. Moreover,
\[
\mathbb{E}\!\left(
\lvert \hat{\lambda}_n-\lambda_n^* \rvert^2
\,\Big|\, X
\right)
\overset{\mathbb{P}}{\longrightarrow} 0,
\]
as $n\to \infty$, by \eqref{exp:lambdaconv}.
Putting the arguments together from \eqref{GDriskupperVIII}--\eqref{B:sub}, we conclude that the expression in \eqref{GDupperI} converges in probability to zero. It remains to show that the quantity in \eqref{GDupperII} also converges in probability to zero. 
By the Cauchy--Schwarz inequality for the conditional expectation, Assumption~\ref{A5}, and \eqref{B:sub}, we have
\begin{align*}
	\mathbb{E}\Bigl[\, 
	\|\beta\|_2 \, 
	\|\hat{\beta}_{t}(\hat{\lambda}_n,\hat{\eta}_n)-\hat{\beta}_{t}(\lambda_n^*,\hat{\eta}_n)\|_2 
	\,\Big|\, X \,\Bigr]
	&\le 
	\Bigl\{ 
	\mathbb{E}\bigl(\|\beta\|_2^2 \,\big|\, X\bigr) 
	\Bigr\}^{1/2} 
	\Bigl\{ 
	\mathbb{E}\bigl(\|\hat{\beta}_{t}(\hat{\lambda}_n,\hat{\eta}_n)-\hat{\beta}_{t}(\lambda_n^*,\hat{\eta}_n)\|_2^2 \,\big|\, X\bigr) 
	\Bigr\}^{1/2} \\
	&\le 
	\tau \Biggl\{
	2 K_4(s_1)^2 s_1 \, 
	\mathbb{E}\left(\bigl|\hat{\lambda}_n-\lambda_n^*\bigr|^2 \,\frac{\|y\|_2^2}{n} \,\Big|\, X\right) \\
	&\quad + 2 C_\theta^2 K_3(s_1)^2 \, 
	\mathbb{E}\Bigl(\bigl|\hat{\lambda}_n-\lambda_n^*\bigr|^2 \,\big|\, X\Bigr)
	\Biggr\}^{1/2}.
\end{align*}
The upper bound converges in probability to zero, as argued previously. 
Since the expressions in \eqref{GDupperI} and \eqref{GDupperII} both converge in probability to zero, we conclude that
\begin{align*}
	\big|\mathrm{R}_{\Sigma_n}\big(\hat{\beta}_{t}(\hat{\lambda}_n,\hat{\eta}_n(\hat{\lambda}_n))\big)
	- \mathrm{R}_{\Sigma_n}\big(\hat{\beta}_{t}(\lambda_n^*,\hat{\eta}_n(\lambda_n^*))\big)\big|
	\;\overset{\mathbb{P}}{\longrightarrow}\; 0,
\end{align*}
as $n\to \infty$.
This completes the first statement.
	
For the second statement, we use the same proof idea as before, but bounding the derivatives and other expressions is now considerably simpler. We begin with a simple consideration. Under the assumptions of the theorem, 
$s_1 = s_{\max}(\hat{\Sigma}_n) > 0$ almost surely for all $n \in \mathbb{N}$.
Using the inequality
\[
\mathbb{P}(0 < s_1 < \hat{s}_1 < K) 
\ge \mathbb{P}(s_1 < \hat{s}_1 < K) + \mathbb{P}(s_1 > 0) - 1
= \mathbb{P}(s_1 < \hat{s}_1 < K),
\]
which holds for each $n \in \mathbb{N}$, we conclude that
\[
\mathbb{P}(0 < s_1 < \hat{s}_1 < K) \to 1 \quad \text{as } n \to \infty,
\]
since $\mathbb{P}(s_1 < \hat{s}_1 < K) \to 1$ as $n \to \infty$ by assumption. By the same arguments as in Remark~\ref{rem:GD}, we obtain, for $\lambda>0$,
\[
\|A_\lambda\|_2 
= \bigl\| I_p - \hat{\eta}_n(\lambda)(\hat{\Sigma}_n + \lambda I_p) \bigr\|_2
= \max\Bigl\{ \bigl| 1 - \hat{\eta}_n(\lambda)(s_1 + \lambda) \bigr|, \;\;
\bigl| 1 - \hat{\eta}_n(\lambda)(s_p + \lambda) \bigr| \Bigr\}.
\]
Furthermore, on the event 
$E_n \coloneqq \{ 0 < s_1 < \hat{s}_1 < K \}$ 
it follows that, for all $\lambda > 0$,
\[
\frac{1}{K + \lambda} < \frac{1}{\hat{s}_1 + \lambda} < \frac{1}{s_1 + \lambda}<\frac{1}{\lambda}.
\]
Next we show that on the event $E_n$ it cannot hold that
\[
\bigl| 1 - \hat{\eta}_n(\lambda)(s_1 + \lambda) \bigr| 
\ge \bigl| 1 - \hat{\eta}_n(\lambda)(s_p + \lambda) \bigr|.
\]
We show this by contradiction. Assume for the sake of argument that the statement holds true. Then,
\begin{align*}
	\bigg| 1 - \frac{1}{\hat{s}_1+\lambda}(s_1+\lambda) \bigg| 
	&\ge \bigg| 1 - \frac{1}{\hat{s}_1+\lambda}(s_p+\lambda) \bigg|\\ 
	&\Longleftrightarrow \bigl|\hat{s}_1-s_1\bigr|\geq \bigl|\hat{s}_1-s_p\bigr| \\
	&\Longleftrightarrow s_p\geq s_1,
\end{align*}
where the last inequality cannot hold true since $s_1 > s_p$, almost surely. 
Combining the previous arguments, we obtain that on the event $E_n$,
\begin{align}\label{A:upper}
	\|A_\lambda\|_2 
	= \Bigl| 1 - \hat{\eta}_n(\lambda)(s_p + \lambda) \Bigr| 
	\le 1 - \frac{\lambda}{K + \lambda} \in (0,1).
\end{align}
Define $F(\lambda)\coloneqq 1 - \frac{\lambda}{K+\lambda}$ for $\lambda\in [\lambda_1,\lambda_2]$ and note that for each $\lambda \in [\lambda_1, \lambda_2]$,
\[
\frac{(t+1) F(\lambda)^t}{t F(\lambda)^{t-1}}
= \frac{(t+1) F(\lambda)}{t}
= F(\lambda) \Bigl(1 + \frac{1}{t}\Bigr)
\longrightarrow F(\lambda) \in (0,1),
\qquad \text{as } t \to \infty.
\]
It then follows from the ratio test that
\[
\lim_{t \to \infty} t F(\lambda)^{t-1} = 0,
\]
and consequently,
\begin{equation}\label{F:boundint}
	\sup_{t \ge 1} t F(\lambda)^{t-1} < \infty.
\end{equation}
It immediately follows from the triangle inequality that
\begin{align}\label{upperA+}
	\begin{split}
		\bigl\|A^t_{\hat{\lambda}_n} + A^t_{\lambda_n^*}\bigr\|_2 
		&= \max_{1 \le i \le p} 
		\Bigl| \bigl(1 - \hat{\eta}_n(\hat{\lambda}_n)(s_i + \hat{\lambda}_n)\bigr)^t 
		+ \bigl(1 - \hat{\eta}_n(\lambda_n^*)(s_i + \lambda_n^*)\bigr)^t \Bigr| \le 2.
	\end{split}
\end{align}
Similarly, by the triangle inequality, and recalling that 
$\hat{\eta}_n(\hat{\lambda}_n) \le \lambda_1^{-1}$ and, for all sufficiently large $n$, 
$\hat{\eta}_n(\lambda_n^*) \le \lambda_1^{-1}$, 
we have, for all sufficiently large $n$,
\begin{align}\label{upperA+B}
	\begin{split}
	\bigl\|B_{\hat{\lambda}_n}+B_{\lambda_n^*}\bigr\|_2 
	&\le \max_{1\leq i\leq p} 
	\Bigg\{ 
	\Bigg| \frac{1 - (1 - \hat{\eta}_n(\hat{\lambda}_n)(s_i+\hat{\lambda}_n))^t}{s_i+\hat{\lambda}_n} 
	+ \frac{1 - (1 - \hat{\eta}_n(\lambda_n^*)(s_i+\lambda^*_n))^t}{s_i+\lambda^*_n} \Bigg| 
	\Bigg\} \\
	&\le \max_{1\leq i\leq p} 
	\Bigg\{ \Bigg| \frac{1 - (1 - \hat{\eta}_n(\hat{\lambda}_n)(s_i+\hat{\lambda}_n))^t}{s_i+\hat{\lambda}_n} \Bigg| \Bigg\} 
	+ \max_{1\leq i\leq p} 
	\Bigg\{ \Bigg| \frac{1 - (1 - \hat{\eta}_n(\lambda_n^*)(s_i+\lambda^*_n))^t}{s_i+\lambda^*_n} \Bigg| \Bigg\} \\
	&\le \frac{4}{\lambda_1}.
	\end{split}
\end{align}
For $\|A^t_{\hat{\lambda}_n} - A^t_{\lambda_n^*}\|_2$ and $\|B_{\hat{\lambda}_n} - B_{\lambda_n^*}\|_2$, we again apply a Lipschitz argument.
By the triangle inequality and the considerations above, for all 
$\lambda \in [\lambda_1, \lambda_2]$ and all $i \in \{1, \dots, p\}$, we have
\begin{align*}
	|g_i'(\lambda)| 
	&= \Bigl| t (1 - \hat{\eta}_n(\lambda)(s_i+\lambda))^{t-1} \bigl( \hat{\eta}_n(\lambda)^2 (s_i+\lambda) - \hat{\eta}_n(\lambda) \bigr) \Bigr| \\
	&\le \Bigl| t (1 - \hat{\eta}_n(\lambda)(s_i+\lambda))^{t-1} \Bigr|\, 
	\Bigl( \bigl|\hat{\eta}_n(\lambda)^2 (s_i+\lambda)\bigr| + \bigl|\hat{\eta}_n(\lambda)\bigr| \Bigr) \\
	&\le t F(\lambda_1)^{t-1} \bigg( \frac{K + \lambda_2}{\lambda_1^2} + \frac{1}{\lambda_1} \bigg).
\end{align*}
Using this bound, we obtain for all sufficiently large $n$,
	\begin{align}\label{upperA-}
		\begin{split}
		\bigl\|A^t_{\hat{\lambda}_n}-A^t_{\lambda_n^*}\bigr\|_2&= \max_{1\leq i\leq p}\{|(1-\hat{\eta}_n(\hat{\lambda}_n)(s_i+\hat{\lambda}_n))^t-(1-\hat{\eta}_n(\lambda_n^*)(s_i+\lambda^*_n))^t|\}
		\\&\leq \max_{1\leq i\leq p}\bigg\{\sup_{\lambda\in [\lambda_1,\lambda_2]}|g^\prime_i(\lambda)|\, \bigl|\hat{\lambda}_n-\lambda_n^*\bigr|\bigg\}
		\\&\leq tF(\lambda_1)^{t-1}\, \bigg(\frac{K+\lambda_2}{\lambda_1^2}+\frac{1}{\lambda_1}\bigg)\, \bigl|\hat{\lambda}_n-\lambda_n^*\bigr|.
		\end{split}
	\end{align}
Considering \eqref{def:h}, \eqref{der:h}, and the arguments above, we obtain, for all $\lambda \in [\lambda_1, \lambda_2]$ and all $i \in \{1, \dots, p\}$,
\begin{align*}
	\sup_{\lambda\in [\lambda_1,\lambda_2]}|h_i'(\lambda)| 
	&\le \sup_{\lambda\in [\lambda_1,\lambda_2]}\biggl| \frac{g_i'(\lambda)}{s_i + \lambda} \biggr| + \sup_{\lambda\in [\lambda_1,\lambda_2]}\biggl| \frac{h_i(\lambda)}{s_i + \lambda} \biggr|\\
	&\le \frac{1}{\lambda_1} \, t F(\lambda_1)^{t-1} \Biggl( \frac{K + \lambda_2}{\lambda_1^2} + \frac{1}{\lambda_1} \Biggr) + \frac{2}{\lambda_1^2}.
\end{align*}
Using this bound, we obtain, for all sufficiently large $n$,
\begin{align}\label{upper:B-D}
	\begin{split}
	\bigl\|B_{\hat{\lambda}_n}-B_{\lambda_n^*}\bigr\|_2 
	&= \max_{1\leq i\leq p} \bigl|g_i(\hat{\lambda}_n) - g_i(\lambda_n^*)\bigr| 
	= \max_{1\leq i\leq p} \Bigl\{ \sup_{\lambda \in [\lambda_1, \lambda_2]} |g_i'(\lambda)| \, \bigl|\hat{\lambda}_n - \lambda_n^*\bigr| \Bigr\} \\
	&\le \Biggl( \frac{1}{\lambda_1} \, t F(\lambda_1)^{t-1} \bigg( \frac{K + \lambda_2}{\lambda_1^2} + \frac{1}{\lambda_1} \bigg) + \frac{2}{\lambda_1^2} \Biggr)\, \bigl|\hat{\lambda}_n - \lambda_n^*\bigr|.
	\end{split}
\end{align}
Revisiting the arguments from the first part of the proof, and using the upper bounds
in \eqref{upperA+}--\eqref{upper:B-D} together with \eqref{F:boundint}, we conclude that
\[
\sup_{t \in \mathbb{N}}\,
\Bigl|
\mathrm{R}_{\Sigma_n}\!\bigl(\hat{\beta}_{t}(\hat{\lambda}_n,\hat{\eta}_n(\hat{\lambda}_n))\bigr)
-
\mathrm{R}_{\Sigma_n}\!\bigl(\hat{\beta}_{t}(\lambda_n^*,\hat{\eta}_n(\lambda_n^*))\bigr)
\Bigr|
\overset{\mathbb{P}}{\longrightarrow} 0,
\qquad \text{as } n \to \infty.
\]
\end{proof}

\begin{proof}[Proof of Theorem~\ref{consestnonran}]
	For convenience, define 
	\[m_1\coloneqq \frac{1}{p}\tr(\Sigma_n),\qquad m_2\coloneqq \frac{1}{p}\tr(\Sigma_n^2)
	\]
	For $|\hat{\sigma}_n^2 - \sigma_n^2|$, we obtain the following decomposition using the triangle inequality:
	\begin{align}\label{var:dec}
		\begin{split}
		|\hat{\sigma}_n^2-\sigma_n^2|&=\biggl|\hat{\sigma}_n^2-\frac{1}{\tilde{m}_2}\bigl(\hat{m}_2\beta^\top \hat{\Sigma}_n\beta -\hat{m}_1\beta^\top \hat{\Sigma}_n^2\beta\bigr)+\frac{1}{\tilde{m}_2}\bigl(\hat{m}_2\beta^\top \hat{\Sigma}_n\beta -\hat{m}_1\beta^\top \hat{\Sigma}_n^2\beta\bigr)-\sigma_n^2\biggr|\\
		&\leq\biggl|\hat{\sigma}_n^2-\frac{1}{\tilde{m}_2}\bigl(\hat{m}_2\beta^\top \hat{\Sigma}_n\beta -\hat{m}_1\beta^\top \hat{\Sigma}_n^2\beta\bigr)-\sigma_n^2\biggr|\\&\quad+\biggl|\frac{1}{\tilde{m}_2}\bigl(\hat{m}_2\beta^\top \hat{\Sigma}_n\beta -\hat{m}_1\beta^\top \hat{\Sigma}_n^2\beta\bigr)\biggr|,
		\end{split}
	\end{align}
	almost surely. 
We begin by showing that
\begin{equation}\label{sig:nonran:upper}
	\biggl|\frac{1}{\tilde{m}_2}\Bigl(\hat{m}_2 \, \beta^\top \hat{\Sigma}_n \beta - \hat{m}_1 \, \beta^\top \hat{\Sigma}_n^2 \beta\Bigr)\biggr| = o_{\mathbb{P}}(1).
\end{equation}
Recall that $\tau_n^2 = \|\beta\|_2^2$ and write $\tilde{\beta} = \beta / \tau_n$. Then,
\begin{align*}
	\hat{m}_1\, \beta^\top \hat{\Sigma}_n^{2}\beta
	= \tau_n^{2}
	\Big\{\hat{m}_1-m_1+m_1\Big\}
	\Big\{
	\tilde{\beta}^\top \hat{\Sigma}_n^{2}\tilde{\beta}
	- \tilde{\beta}^\top \Sigma_n^{2}\tilde{\beta}
	- \frac{1}{n}\tr(\Sigma_n)\,\tilde{\beta}^\top \Sigma_n \tilde{\beta}
	+ \tilde{\beta}^\top \Sigma_n^{2}\tilde{\beta}
	+ \frac{1}{n}\tr(\Sigma_n)\,\tilde{\beta}^\top \Sigma_n \tilde{\beta}
	\Big\}.
\end{align*}
For notational convenience, define
\begin{align*}
	A &\coloneqq \hat{m}_1-m_1, 
	\qquad B \coloneqq m_1, \qquad
	E \coloneqq \tau_n^{2}\!\left(
	\tilde{\beta}^\top \hat{\Sigma}_n^{2}\tilde{\beta}
	- \tilde{\beta}^\top \Sigma_n^{2}\tilde{\beta}
	- \frac{1}{n}\tr(\Sigma_n)\,\tilde{\beta}^\top \Sigma_n\tilde{\beta}
	\right),\\[0.3em]
	F &\coloneqq \tau_n^{2}\!\left(
	\tilde{\beta}^\top \Sigma_n^{2}\tilde{\beta}
	+ \frac{1}{n}\tr(\Sigma_n)\,\tilde{\beta}^\top \Sigma_n\tilde{\beta}
	\right).
\end{align*}
Under the assumptions of the theorem and Lemmas~\ref{momqf1} and~\ref{momqf2}, the above quantities satisfy
\[
A = O_{\mathbb{P}}\!\left( \frac{1}{(pn)^{1/2}} \right),\quad
B = O(1),\quad
E = O_{\mathbb{P}}\!\left( \frac{1}{n^{1/2}} \vee \frac{p}{n^{3/2}} \right),\quad
F = O\left( 1 + \gamma_n \right).
\]
With this notation we obtain the decomposition
\[
\hat{m}_1\, \beta^\top \hat{\Sigma}_n^{2}\beta
= (A + B)(E + F).
\]
It is straightforward to verify that
\[
AE = O_{\mathbb{P}}\biggl(\frac{1}{np^{1/2}} \vee \frac{\sqrt{p}}{n^{2}}\biggr), 
\qquad
BE = O_{\mathbb{P}}\biggl(\frac{1}{n^{1/2}} \vee \frac{p}{n^{3/2}}\biggr),
\]
and
\[
AF
= O_{\mathbb{P}}\Big((pn)^{-1/2}\Big)\, O_{\mathbb{P}}\big(1+\gamma_n\big)
= O_{\mathbb{P}}\biggl(\frac{1}{(pn)^{1/2}} \vee \frac{p^{1/2}}{n^{3/2}}\biggr).
\]
Furthermore, for \(p \ge n\), the following hierarchy of terms holds:
\begin{equation}\label{np:relation}
	\frac{\sqrt{p}}{n} \le \frac{p}{n^{3/2}} \le \frac{p^{3/2}}{n^2} \le \frac{p^2}{n^{5/2}}.
\end{equation}
In the case \(p \le n\), we have 
\begin{equation}\label{np:relation2}
	\frac{\sqrt{p}}{n^2}\leq \frac{1}{np^{1/2}},\quad\frac{p}{n^{3/2}}\leq \frac{1}{n^{1/2}}\quad\text{and}\quad \frac{p^{1/2}}{n^{3/2}}\leq \frac{1}{(pn)^{1/2}}.
\end{equation} 
Therefore, under the assumption \(p=o(n^{5/4})\), we conclude that
\begin{align}\label{hm1quad2}
	\begin{split}
		\hat m_1\, \beta^\top \hat{\Sigma}_n^{2}\beta
		&= BF + o_{\mathbb{P}}(1) \\
		&= \tau_n^2\, m_1
		\Biggl( \tilde{\beta}^\top \Sigma_n^{2} \tilde{\beta}
		+ \frac{1}{n} \tr(\Sigma_n)\, \tilde{\beta}^\top \Sigma_n \tilde{\beta} \Biggr)
		+ o_{\mathbb{P}}(1).
	\end{split}
\end{align}
Similarly, we define 
\begin{align*}
	A' &\coloneqq \tau_n^2 \!\left( \tilde\beta^\top \hat\Sigma_n \tilde\beta
	- \tilde\beta^\top \Sigma_n \tilde\beta \right), &
	B' &\coloneqq \tau_n^2 \, \tilde\beta^\top \Sigma_n \tilde\beta,\\[0.3em]
	E' &\coloneqq \hat{m}_2 - m_2 - \gamma_n m_1^{\,2}, &
	F' &\coloneqq m_2 + \gamma_n m_1^{\,2}.
\end{align*}
Under the assumption of the theorem and using Lemma~\ref{momqf1} and~\ref{momqf2}, the above quantities satisfy
\[
A' = O_{\mathbb{P}}\!\left( \frac{1}{n^{1/2}} \right),\quad
B' = O(1),\quad
E' = O_{\mathbb{P}}\!\left( \frac{1}{n^{1/2}} \vee \frac{p}{n^{3/2}} \right),\quad
F' = O\left( 1 + \gamma_n \right).
\]
By the assumptions of the theorem, using \eqref{np:relation} and \eqref{np:relation2}, we conclude that
\[
A' E' = o_{\mathbb{P}}(1),\quad B' E' = o_{\mathbb{P}}(1)
\qquad\text{and}\qquad A'F'=o_{\mathbb{P}}(1).
\]
Therefore, it follows that
\begin{align}\label{h2quad1}
	\begin{split}
	\hat m_2\,\beta^\top \hat\Sigma_n \beta
	&=\tau_n^2\,\tilde\beta^\top \hat\Sigma_n \tilde\beta \,\frac{1}{p}\tr(\hat\Sigma_n^2) =(A'+B')(E'+F')=B'F' + o_{\mathbb{P}}(1) \\
	&=\tau_n^2\,\tilde\beta^\top \Sigma_n \tilde\beta\,(m_2+\gamma_n m_1^2)
	+o_{\mathbb{P}}(1).
	\end{split}
\end{align}
Putting together the expansions from \eqref{hm1quad2} and \eqref{h2quad1}, we obtain
\begin{align*}
	\hat m_2\,\beta^\top \hat\Sigma_n\beta
	-\hat m_1\,\beta^\top \hat\Sigma_n^2\beta
	&= \tau_n^2\,\tilde\beta^\top \Sigma_n \tilde\beta\, (m_2+\gamma_n m_1^2) - \tau_n^2\,m_1\Bigl(\tilde\beta^\top \Sigma_n^2\tilde\beta
	+ \frac{1}{n}\tr(\Sigma_n)\,\tilde\beta^\top \Sigma_n\tilde\beta\Bigr)
	+ o_{\mathbb{P}}(1).
\end{align*}
As in Assumption~\ref{A6}, we write
$\Delta_1 = \tilde\beta^\top\Sigma_n\tilde\beta - m_1$ and  
$\Delta_2 = \tilde\beta^\top\Sigma_n^2\tilde\beta - m_2$. Hence, we can rewrite the previous display as
\begin{align*}
	\hat m_2\,\beta^\top \hat\Sigma_n\beta
	-\hat m_1\,\beta^\top \hat\Sigma_n\beta&= \tau_n^2 \Delta_1(m_2+\gamma_n m_1^2)+\tau_n^2m_1(m_2+\gamma_n m_1^2)\\
	&\qquad-\tau_n^2 m_1(\Delta_2+m_2 + \gamma_n m_1 \Delta_1+\gamma_nm_1^2)
	+ o_{\mathbb{P}}(1) \\
	&= \tau_n^2(\Delta_1 m_2 - m_1\Delta_2) + o_{\mathbb{P}}(1).
\end{align*}
Therefore,
\[
\hat{m}_{2}\,\beta^{\top}\hat{\Sigma}_{n}\beta
-
\hat{m}_{1}\,\beta^{\top}\hat{\Sigma}_{n}\beta
=  o_{\mathbb{P}}(1),
\]
whenever 
\(\tau_{n}^{2}\bigl(\Delta_{1} m_{2} - m_{1}\Delta_{2}\bigr)=o(1)\),
a condition that is satisfied under 
Assumption~\ref{A2} together with Assumption~\ref{A6}. For $\tilde m_2$, recall from Lemma~\ref{denomnonzero} that 
$\tilde m_2>0$ with probability one for all $n,p\geq 2$. 
Moreover, Lemma~\ref{Bodver} implies that $\tilde m_2 = m_2 +  o_{\mathbb{P}}(1)$. 
By \eqref{convtrgSig} and since $H(0)\neq 1$ by Assumption~\ref{A4}, we have
\begin{align}\label{mom2:nonran}
m_2 
= \frac{1}{p}\tr(\Sigma_n^2)
= \int x^2\, dF_{\Sigma_n}(x)
\;\longrightarrow\;
\int x^2\, dH(x)
\neq 0.
\end{align}
This completes the proof of the assertion in~\eqref{sig:nonran:upper}.

Although the following arguments are similar to those used in the proof of Theorem~\ref{est1}, they cannot be omitted, as the setting considered here is substantially different. In particular, the non-randomness of $\beta$ and the fact that we allow $p/n \to \infty$ require separate and careful treatment.
We now analyze the first term in the upper bound of~\eqref{var:dec} and establish that
\begin{equation} \label{Idec}
	I \coloneqq \biggl|
	\hat{\sigma}_n^{2}
	- \frac{1}{\tilde{m}_2} \Bigl( \hat{m}_2\, \beta^{\top}\hat{\Sigma}_n \beta 
	- \hat{m}_1\, \beta^{\top}\hat{\Sigma}_n^{2} \beta \Bigr)
	- \sigma_n^{2}
	\biggr|
	= o_{\mathbb{P}}(1).
\end{equation}
Note that
\begin{align}\label{decom:xy:y}
\frac{\|y\|_2^{2}}{n}
= \beta^{\top} \hat{\Sigma}_n \beta
+ \frac{2}{n}\, u^{\top} X\beta
+ \frac{1}{n}\, u^{\top} u,
\quad
\frac{\|X^{\top} y\|_2^{2}}{n^{2}}
= \beta^{\top} \hat{\Sigma}_n^{2} \beta
+ \frac{2}{n^{2}}\, u^{\top} X X^{\top} X \beta
+ \frac{1}{n}\, u^{\top} \underline{\hat{\Sigma}}_n u.
\end{align}
Moreover,
\begin{equation*}
	\sigma_n^{2}
	= \frac{1}{\tilde{m}_2}\Bigl( \hat{m}_2 - \gamma_n \hat{m}_1^{2} \Bigr)\sigma_n^{2}
	= \frac{1}{\tilde{m}_2}\Bigl( \hat{m}_2
	- \frac{1}{n}\,\hat{m}_1\, \mathrm{tr}(\underline{\hat{\Sigma}}_n) \Bigr)\sigma_n^{2},
	\qquad \text{almost surely.}
\end{equation*}
Recalling the definition of $\hat{\sigma}_n^2$ in \eqref{est:def}, applying the triangle inequality together with the above decompositions of
$n^{-1}\|y\|_{2}^{2}$ and $n^{-2}\|X^{\top} y\|_{2}^{2}$, the expression in~\eqref{Idec}
is bounded by
\begin{align}\label{sig:upper:I}
	\begin{split}
	I
	&\le 
	\frac{\hat{m}_2}{\tilde{m}_2}
	\Biggl\{
	\left|\frac{2}{n}\, u^{\top} X\beta\right|
	+ \left|\frac{u^{\top}u}{n} - \sigma_n^{2}\right|
	\Biggr\}\\
	&\quad+ \frac{\hat{m}_1}{\tilde{m}_2}
	\Biggl\{
	\left|\frac{2}{n^{2}}\, u^{\top} X X^{\top} X \beta\right|
	+ \left|
	\frac{1}{n}\, u^{\top} \underline{\hat{\Sigma}}_n u
	- \frac{\sigma_n^{2}}{n}\, \mathrm{tr}\!\left(\underline{\hat{\Sigma}}_n\right)
	\right|
	\Biggr\}.
	\end{split}
\end{align}
We begin by analyzing the first term in the upper bound of \eqref{sig:upper:I}. Note that by Lemma~\ref{momqf1} and Assumption~\ref{A2},
\begin{align*}
	\hat{m}_2
	= \hat{m}_2 - m_2 - \gamma_n m_1^2 + m_2 + \gamma_n m_1^2
	= O_{\mathbb{P}}(1+\gamma_n).
\end{align*}
In order to establish that 
\begin{equation}\label{bilin:nonran}
	\frac{\gamma_n}{n} u^\top X \beta = o_{\mathbb{P}}(1),
\end{equation}
it suffices to show that
\[
\mathbb{E}\!\left[
\frac{(\gamma_n u^{\top} X\beta)^{2}}{n^{2}}
\;\Big|\; X
\right] = o_{\mathbb{P}}(1).
\]
Observe that, by Assumption~\ref{A2}, Lemma~\ref{momqf2} and $p^2n^{-3}\leq p^{5/2}n^{-3}\to 0$ as $n,p\to\infty$, we have
\begin{align*}
	\mathbb{E}\!\left[
	\frac{(\gamma_n u^{\top} X\beta)^{2}}{n^{2}}
	\,\Big|\, X
	\right]
	&=
	\frac{p^{2}}{n^{4}}
	\,
	\mathbb{E}\!\left(
	\beta^{\top} X^{\top} uu^{\top} X\beta\, 
	\,\Big|\, X
	\right)=
	\frac{p^{2}\sigma_n^{2}}{n^{4}}
	\, \beta^{\top}X^{\top} X\beta \\
	&=
	\frac{p^{2}\sigma_n^{2}\tau_n^{2}}{n^{3}}
	\,\tilde{\beta}^{\top} \hat{\Sigma}_n \tilde{\beta} =
	\frac{p^{2}\sigma_n^{2}\tau_n^{2}}{n^{3}}
	\,\tilde{\beta}^{\top} \Sigma_n \tilde{\beta}
	+ o_{\mathbb{P}}(1)
	= o_{\mathbb{P}}(1).
\end{align*}
Next, since the entries of $u$ are independent with uniformly bounded fourth moments in $n$ and 
$p^{2}n^{-3}\to 0$, we conclude using the Markov inequality that
\begin{align}\label{u:LLN}
\gamma_{n}\,\Bigl|\frac{u^{\top}u}{n}-\sigma_{n}^{2}\Bigr|
= o_{\mathbb{P}}(1)
\end{align}
and that 
\[
	\frac{\hat{m}_2}{\tilde{m}_2}
\Biggl\{
\left|\frac{2}{n}\, u^{\top} X\beta\right|
+ \left|\frac{u^{\top}u}{n} - \sigma_n^{2}\right|
\Biggr\}=o_{\mathbb{P}}(1).
\]
Similarly, for the second term in the upper bound of \eqref{sig:upper:I}, we have, using Assumption~\ref{A2} and Lemma~\ref{momqf1},
\begin{align}\label{mom1:nonran}
	\hat{m}_1
	=\hat{m}_1-m_1+m_1
	= O_{\mathbb{P}}(1).
\end{align}
In order to establish that $n^{-2} u^{\top} X X^{\top} X\beta= o_{\mathbb{P}}(1)$, we show that 
\[	\mathbb{E}\!\left[
\frac{(u^{\top} X X^{\top} X\beta)^{2}}{n^{4}}
\,\Big|\, X
\right]=o_{\mathbb{P}}(1).
\]
Note that,
\begin{align}\label{bilinform3upperb}
	\begin{split}
	\mathbb{E}\!\left[
	\frac{(u^{\top} X X^{\top} X\beta)^{2}}{n^{4}}
	\,\Big|\, X
	\right]
	&=
	\frac{1}{n^{4}}
	\,\mathbb{E}\!\biggl[
	\beta^{\top} X^{\top} X X^{\top} uu^{\top} X X^{\top} X\beta\,
	\,\Big|\, X
	\biggr]  \\
	&\le
	\frac{\sigma_n^{2}}{n^{4}}
	\,\beta^{\top} (X^{\top} X)^{3}\beta
	=
	\frac{\sigma_n^{2}}{n}
	\,\beta^{\top} \hat{\Sigma}_n^{3}\beta.
	\end{split}
\end{align}
We write the entries of $\underline{\hat{\Sigma}}_{n}=n^{-1}Z\Sigma_n Z^\top$ as 
$a_{jk} = n^{-1} e_{j}^{\top}\underline{\hat{\Sigma}}_{n} e_{k}$ for 
$j,k \in \{1,\ldots,n\}$, where $\{e_{1},\ldots,e_{n}\}$ denotes the standard 
basis of $\mathbb{R}^{n}$.
Then,
\begin{align*}
	\beta^{\top} \hat{\Sigma}_n^{3} \beta
	= \tau_n^{2}\, \tilde{\beta}^{\top} \hat{\Sigma}_n^{3} \tilde{\beta} \le
	\|\hat{\Sigma}_n\|_{2}\,
	\tau_n^{2}\, \tilde{\beta}^{\top} \hat{\Sigma}_n^{2} \tilde{\beta}
	= 
	\|\underline{\hat{\Sigma}}_n\|_{2}\,
	\tau_n^{2}\, \tilde{\beta}^{\top} \hat{\Sigma}_n^{2} \tilde{\beta}.
\end{align*}
Taking the expectation of $\|\underline{\hat{\Sigma}}_n\|_{2}$ and using
$\|\underline{\hat{\Sigma}}_n\|_{2} \le \|\underline{\hat{\Sigma}}_n\|_{F}$ together with the Jensen
inequality, we obtain
\begin{align}\label{Frobinq}
	\mathbb{E}\!\left(\|\underline{\hat{\Sigma}}_n\|_{2}\right)
	\le 
	\mathbb{E}\!\left[
	\Bigl(\sum_{j=1}^{n}\sum_{k=1}^{n} a_{jk}^{2}\Bigr)^{1/2}
	\right]
	\le
	\left[
	\mathbb{E}\Bigl(\sum_{j=1}^{n}\sum_{k=1}^{n} a_{jk}^{2}\Bigr)
	\right]^{1/2}.
\end{align}
Recall that $z_{1},\ldots,z_{n}$ denote the row vectors of $Z$.
For $j\neq k$, we have
\begin{align*}
	\mathbb{E}(a_{jk}^{2})
	&= 
	\mathbb{E}\!\left[(e_j^{\top} \underline{\hat{\Sigma}}_n e_k)^{2}\right]
	= \frac{1}{n^{2}}\,
	\mathbb{E}\!\left[(z_j^{\top}\Sigma_n z_k)^{2}\right] \\
	&=
	\frac{1}{n^{2}}\,
	\mathbb{E}\!\left[
	\mathbb{E}\!\left(z_j^{\top}\Sigma_n z_k\, z_k^{\top}\Sigma_n z_j \,\Big|\, z_j\right)
	\right]
	=
	\frac{1}{n^{2}}\,\mathbb{E}\!\left(z_j^{\top}\Sigma_n^{2} z_j\right)
	= 
	\frac{1}{n^{2}}\,\mathrm{tr}(\Sigma_n^{2}).
\end{align*}
For $j=k$, Lemma~\ref{quadform} yields
\begin{align*}
	\mathbb{E}(a_{jj}^{2})
	=
	\frac{1}{n^{2}}\,\mathbb{E}\!\left[(z_j^{\top}\Sigma_n z_j)^{2}\right]
	\le
	\frac{C'}{n^{2}}\,\mathrm{tr}(\Sigma_n)^{2},
\end{align*}
for some constant $C'>0$.
Therefore, the upper bound in \eqref{Frobinq} simplifies to
\begin{align*}
	\mathbb{E}\!\left(\bigl\|\underline{\hat{\Sigma}}_n\bigr\|_{2}\right)
	&\le
	\Biggl[
	\sum_{j=1}^{n} \frac{C'}{n^{2}}\tr(\Sigma_n)^{2}
	+ \sum_{j=1}^{n}\sum_{\substack{k=1 \\ k\neq j}}^{n}
	\frac{1}{n^{2}}\tr(\Sigma_n^{2})
	\Biggr]^{1/2} \\
	&\le
	\biggl[
	\frac{p^{2}}{n} C'' 
	+ p C^{2}
	\biggr]^{1/2},
\end{align*}
for some constant $C''>0$.
Hence,
\[\bigl
\|\hat{\Sigma}_n\bigr\|_{2}
= \bigl\|\underline{\hat{\Sigma}}_n\bigr\|_{2}
= O_{\mathbb{P}}\!\left(\frac{p}{\sqrt{n}}+\sqrt{p}\right).
\]
Furthermore, by Assumption~\ref{A2} 
\[
\frac{1}{n}\tr(\Sigma_n)\,\tilde{\beta}^{\top}\Sigma_n\tilde{\beta}
= \gamma_n m_1\,\tilde{\beta}^{\top}\Sigma_n\tilde{\beta}
\le \gamma_n C^{2},
\]
and using Lemma~\ref{momqf2} together with the assumptions of the theorem, we have 
\begin{align*}
	\tilde{\beta}^{\top}\hat{\Sigma}_n^{2}\tilde{\beta} = O_{\mathbb{P}}\left(\frac{1}{\sqrt{n}}\vee \frac{p}{n^2}\right) +O_{\mathbb{P}}\bigl(1+\gamma_n\bigr)=O_{\mathbb{P}}\bigl(1+\gamma_n\bigr).
\end{align*}
Returning to \eqref{bilinform3upperb}, the bounds obtained above imply that
\begin{align*}
	\mathbb{E}\!\left(
	\frac{(u^{\top} X X^{\top} X\beta)^{2}}{n^{4}}
	\,\Big|\, X
	\right)
	&\le
	\frac{1}{n}\,
	\sigma_n^{2}\tau_n^{2}\,
	\tilde{\beta}^{\top}\hat{\Sigma}_n^{3}\tilde{\beta} \le
	\frac{1}{n}\,
	\sigma_n^{2}\tau_n^{2}\,
	\tilde{\beta}^{\top}\hat{\Sigma}_n^{2}\tilde{\beta}
	\,\bigl\|\hat{\Sigma}_n\bigr\|_{2} =O_{\mathbb{P}}\!\left(\frac{p}{n^{3/2}}+ \frac{\sqrt{p}}{n}\right)O_{\mathbb{P}}\!\left(1+\gamma_n\right)\\
	&=O_{\mathbb{P}}\!\left(
	\frac{\sqrt{p}}{n}
	+\frac{p}{n^{3/2}}
	+\frac{p^{3/2}}{n^{2}}
	+\frac{p^{2}}{n^{5/2}}
	\right)=O_{\mathbb{P}}\!\left(
	\frac{\sqrt{p}}{n}\vee
	\frac{p^{2}}{n^{5/2}}
	\right).
\end{align*}
Since $p=o(n^{5/4})$ and 
\[\frac{\sqrt{p}}{n}=\biggl(\frac{p}{n^{5/4}}\biggr)^{1/2}\frac{1}{n^{3/8}}\to0,
\] the bound in
\eqref{bilinform3upperb} implies
\[
n^{-4}\,\mathbb{E}\!\left[(u^{\top}XX^{\top}X\beta)^{2}\mid X\right]
= o_{\mathbb{P}}(1).
\]
It follows that, 
\begin{equation}\label{bilin:nonran2}
n^{-2}\,u^{\top}XX^{\top}X\beta=o_{\mathbb{P}}(1).
\end{equation}
It remains to show that
\[
\frac{1}{n}u^{\top}\underline{\hat{\Sigma}}_{n}u
-
\frac{\sigma_{n}^{2}}{n}\tr(\underline{\hat{\Sigma}}_{n})
=o_{\mathbb{P}}(1).
\]
By the independence of $u$ and $X$, we may apply Lemma~\ref{quadform} with 
$n^{-1}\underline{\hat{\Sigma}}_{n}$ and $\tilde{u}=\sigma_{n}^{-1}u$:
\begin{align*}
	\mathbb{E}\Biggl\{
	\Bigl|
	\frac{1}{n}u^{\top}\underline{\hat{\Sigma}}_{n}u
	-\frac{\sigma_{n}^{2}}{n}\tr\bigl(\underline{\hat{\Sigma}}_{n}\bigr)
	\Bigr|^{2}\,
	\bigg|X
	\Biggr\}
	&=
	\sigma_{n}^{4}\,
	\mathbb{E}\Biggl\{
	\Bigl|
	\frac{1}{n}\tilde{u}^{\top}\underline{\hat{\Sigma}}_{n}\tilde{u}
	-\frac{1}{n}\tr\bigl(\underline{\hat{\Sigma}}_{n}\bigr)
	\Bigr|^{2}
	\bigg|X
	\Biggl\}                                       \\
	&\le
	\nu_{4}^{2}\,\frac{C'\nu_{u,4}}{n^{2}}\,
	\tr\bigl(\underline{\hat{\Sigma}}_{n}^{2}\bigr).
\end{align*}
Hence, for any $\varepsilon>0$,
\begin{align*}
	\mathbb{P}\Biggl\{
	\biggl|
	\frac{1}{n}u^{\top}\underline{\hat{\Sigma}}_{n}u
	-\frac{\sigma_{n}^{2}}{n}\tr\bigl(\underline{\hat{\Sigma}}_{n}\bigr)
	\biggr|>\varepsilon
	\Biggr\}
	&=
	\mathbb{E}\Biggl[
	\mathbb{P}\Biggl\{
	\Bigl|
	\frac{1}{n}u^{\top}\underline{\hat{\Sigma}}_{n}u
	-\frac{\sigma_{n}^{2}}{n}\tr\bigl(\underline{\hat{\Sigma}}_{n}\bigr)
	\Bigr|>\varepsilon\,
	\bigg|X
	\Biggr\}
	\Biggr]                                     \\
	&\le
	\frac{1}{\varepsilon^{2}}
	\mathbb{E}\Biggl[
	1\wedge
	\mathbb{E}\Biggl\{
	\Bigl|
	\frac{1}{n}u^{\top}\underline{\hat{\Sigma}}_{n}u
	-\frac{\sigma_{n}^{2}}{n}\tr\bigl(\underline{\hat{\Sigma}}_{n}\bigr)
	\Bigr|^{2}\,
	\bigg|X
	\Biggr\}
	\Biggr]                                     \\
	&\le
	\frac{1}{\varepsilon^{2}}
	\wedge
	\left(
	\frac{\nu_{4}^{2}C_{2}\nu_{u,4}}{\varepsilon^{2}n^{2}}\,
	\tr\bigl(\underline{\hat{\Sigma}}_{n}^{2}\bigr)
	\right)                                    =
	\frac{1}{\varepsilon^{2}}
	\wedge
	\left(
	\frac{p}{n^{2}}\,C''\,
	\frac{1}{p}\tr\bigl(\hat{\Sigma}_{n}^{2}\bigr)
	\right),
\end{align*}
where we use $\tr(\underline{\hat{\Sigma}}_{n}^{2})
=\tr(\hat{\Sigma}_{n}^{2})$.
Since $\hat{m}_2 = O_{\mathbb{P}}(1+\gamma_n)$, and under the assumptions of the theorem
\[
\frac{p}{n^{2}} = \left(\frac{\sqrt{p}}{n}\right)^2 \to 0
\quad\text{and}\quad
\frac{p^2}{n^3} = \left(\frac{p}{n^{5/4}}\right)^2 \frac{1}{\sqrt{n}} \to 0,
\]
we conclude that
\begin{align}\label{u:quadform}
\frac{1}{n} \, u^{\top} \underline{\hat{\Sigma}}_n \, u
- \frac{\sigma_n^2}{n} \, \mathrm{tr}\bigl(\underline{\hat{\Sigma}}_n\bigr)
\;\overset{\mathbb{P}}{\longrightarrow}\; 0.
\end{align}

The second part is straightforward, since we can utilize the facts
established in the first part of the proof. We consider the decomposition
	\begin{align}\label{sig:dec}
	\begin{split}
		|\hat{\tau}_n^2-\tau_n^2|&=\biggl|\hat{\tau}_n^2-\frac{1}{\tilde{m}_2}\bigl(\beta^\top \hat{\Sigma}_n^2\beta -\gamma_n\hat{m}_1\beta^\top \hat{\Sigma}_n\beta\bigr)+\frac{1}{\tilde{m}_2}\bigl(\beta^\top \hat{\Sigma}_n^2\beta -\gamma_n\hat{m}_1\beta^\top \hat{\Sigma}_n\beta\bigr)-\tau_n^2\biggr|\\
		&\leq\biggl|\hat{\tau}_n^2-\frac{1}{\tilde{m}_2}\bigl(\beta^\top \hat{\Sigma}_n^2\beta -\gamma_n\hat{m}_1\beta^\top \hat{\Sigma}_n\beta\bigr)\biggr|\\
		&\quad+\biggl|\frac{1}{\tilde{m}_2}\bigl(\beta^\top \hat{\Sigma}_n^2\beta -\gamma_n\hat{m}_1\beta^\top \hat{\Sigma}_n\beta\bigr)-\tau_n^2\biggr|,
	\end{split}
\end{align}
almost surely.
First, we show that the second qunatity in the upper bound of \eqref{sig:dec} converges in probability to zero.
Recall that
\[
A = \hat{m}_1-m_1=O_{\mathbb{P}}\bigl((np)^{-1/2}\bigr), 
\qquad   
B = m_1=O(1),
\]
\[
A' = \tau_n^{2}\!\left(\tilde{\beta}^\top \hat{\Sigma}_n\tilde{\beta}
-\tilde{\beta}^\top \Sigma_n\tilde{\beta}\right)=O_{\mathbb{P}}\bigl(n^{-1/2}\bigr),
\qquad
B' = \tau_n^{2}\tilde{\beta}^\top \Sigma_n \tilde{\beta}=O(1).
\]
Therefore,
\begin{align*}
	\tau_n^2\gamma_n\hat m_1\,\tilde{\beta}^\top \hat{\Sigma}_n \tilde{\beta}
	&= \tau_n^2\gamma_n 
	\Big( \hat{m}_1-m_1+m_1\Big) 
	\Big(\tilde{\beta}^\top\hat{\Sigma}_n\tilde{\beta}
	- \tilde{\beta}^\top\Sigma_n\tilde{\beta}
	+ \tilde{\beta}^\top\Sigma_n\tilde{\beta}\Big) \\
	&= \gamma_n (A+B)(A' + B').
\end{align*}
It is straightforward to verify that
\[
\gamma_n A A' = O_{\mathbb{P}}\Biggl(\frac{\sqrt{p}}{n^2}\Biggr), \quad
\gamma_n A B' = O_{\mathbb{P}}\Biggl(\frac{\sqrt{p}}{n^{3/2}}\Biggr), \quad
\gamma_n B A' = O_{\mathbb{P}}\Biggl(\frac{p}{n^{3/2}}\Biggr).
\]
Since $p = o(n^{5/4})$, we conclude that
\[
\tau_n^2 \, \gamma_n \, \hat{m}_1 \, \tilde{\beta}^\top \hat{\Sigma}_n \tilde{\beta}
= \tau_n^2 \, \gamma_n \, m_1 \, \tilde{\beta}^\top \Sigma_n \tilde{\beta} + o_{\mathbb{P}}(1).
\]
For $\beta^\top \hat{\Sigma}_n^2\beta$, we obtain by the uniform boundedness of $\tau_n$ and Lemma~\ref{momqf2} that
	\begin{align*}
	\beta^\top \hat{\Sigma}_n^2\beta=\tau_n^2\tilde{\beta}^\top \hat{\Sigma}_n^2\tilde{\beta}
		= \tau_n^2\bigg(\tilde{\beta}^\top \Sigma_n^2\tilde{\beta} +\gamma_n \frac{1}{p}\tr(\Sigma_n)\tilde{\beta}^\top \Sigma_n\tilde{\beta}\bigg) +o_{\mathbb{P}}(1).
	\end{align*}
Combining the preceding two displays, we obtain
\begin{align*}
	\beta^\top \hat{\Sigma}_n^{2}\beta
	- \gamma_n \hat m_1\, \beta^\top \hat{\Sigma}_n \beta
	&=
	\tau_n^{2}\!\left(
	\tilde{\beta}^\top \Sigma_n^{2}\tilde{\beta}
	+ \gamma_n\,m_1\,
	\tilde{\beta}^\top \Sigma_n \tilde{\beta}
	\right)
	- \tau_n^{2}\gamma_n\,m_1\,
	\tilde{\beta}^\top \Sigma_n \tilde{\beta}
	+o_{\mathbb{P}}(1) \\
	&= \tau_n^{2}\tilde{\beta}^\top \Sigma_n^{2}\tilde{\beta}
	+ o_{\mathbb{P}}(1) \\
	&= \tau_n^{2}\Delta_2
	+ \tau_n^{2}m_2
	+ o_{\mathbb{P}}(1).
\end{align*}
We already established that $\tilde{m}_2 = m_2 + o_{\mathbb{P}}(1)$ and that
$m_2 \to \int_0^\infty x^2\, dH(x) \neq 0$. Since $\tau_n^2\Delta_2 = o(1)$ by
Assumption~\ref{A6}, we have
\begin{align*}
	\frac{1}{	\tilde{m}_2}
		\beta^\top \hat{\Sigma}_n^{2}\beta
		- \gamma_n \hat m_1 \beta^\top \hat{\Sigma}_n \beta
	- \tau_n^2
=o_{\mathbb{P}}(1).
\end{align*}

For the first quantity in the upper bound of \eqref{sig:dec}, we
decompose $n^{-1}\|y\|_2^2$ and $n^{-2}\|X^\top y\|_2^2$ as in \eqref{decom:xy:y}, and use the identity
\[
\sigma_n^2
= \frac{1}{\tilde{m}_2}\bigl( \hat{m}_2
- \gamma_n \hat{m}_1^2 \bigr)\sigma_n^2,
\qquad \text{almost surely.}
\] 
Recalling the definition of $\hat{\tau}_n^2$ in \eqref{est:def}, we obtain
	\begin{align*}
	\biggl|\hat{\tau}_n^2-\frac{1}{\tilde{m}_2}\bigl(\beta^\top \hat{\Sigma}_n^2\beta -\gamma_n\hat{m}_1\beta^\top \hat{\Sigma}_n\beta\bigr)\biggr|&=\Biggl|\bigg(\frac{1}{\tilde{m}_2} \frac{\|X^\top y\|_2^2}{n^2} - \frac{\gamma_n \hat{m}_1}{\tilde{m}_2} \frac{\|y\|_2^2}{n}\bigg)-\frac{1}{\tilde{m}_2} \bigg(\beta^\top \hat{\Sigma}_n^2\beta - \gamma_n \hat{m}_1 \beta^\top\hat{\Sigma}_n\beta \bigg)\Biggr|
		\\&\leq \frac{1}{\tilde{m}_2}\bigg|\frac{2}{n^2}u^\top X X^\top X \beta\bigg|+\frac{1}{\tilde{m}_2}\bigg|\frac{1}{n}u^\top \underline{\hat{\Sigma}}_nu-\gamma_n\hat{m}_1 \frac{1}{n}u^\top u\bigg|+\frac{2\gamma_n\hat{m}_1}{\tilde{m}_2}\bigg|\frac{1}{n}u^\top X\beta\bigg|
	\end{align*}
By \eqref{mom2:nonran}, \eqref{bilin:nonran}, \eqref{mom1:nonran}, \eqref{bilin:nonran2} and Lemma~\ref{Bodver}, it remains to verify that
\[
\bigg|\frac{1}{n}u^\top \underline{\hat{\Sigma}}_n u
- \gamma_n \hat{m}_1 \frac{1}{n}u^\top u
\bigg|
= o_{\mathbb{P}}(1).
\]
Using $\gamma_n \hat{m}_1
= n^{-1}\tr\bigl(\underline{\hat{\Sigma}}_n\bigr)$, we obtain
\begin{align*}
	\bigg|\frac{1}{n}u^\top \underline{\hat{\Sigma}}_n u
	- \gamma_n \hat{m}_1\frac{1}{n}u^\top u
	\bigg|
	&=
	\Biggl|
	\frac{1}{n}u^\top \underline{\hat{\Sigma}}_n u
	- \frac{\sigma_n^2}{n}\tr(\underline{\hat{\Sigma}}_n)
	- \gamma_n\hat{m}_1
	\left(\frac{1}{n}u^\top u - \sigma_n^2\right)
	\Biggr| \\
	&\le
	\Biggl|\frac{1}{n}u^\top \underline{\hat{\Sigma}}_n u
	- \frac{\sigma_n^2}{n}\tr\bigl(\underline{\hat{\Sigma}}_n\bigr)
	\Biggr|
	+
	\Biggl|\gamma_n\hat{m}_1
	\left(\frac{1}{n}u^\top u - \sigma_n^2\right)\Biggr|.
\end{align*}
Using \eqref{u:LLN}, \eqref{mom1:nonran} and \eqref{u:quadform}, we conclude that
\[
\Biggl|\frac{1}{n}u^\top \underline{\hat{\Sigma}}_n u
- \frac{\sigma_n^2}{n}\tr\bigl(\underline{\hat{\Sigma}}_n\bigr)
\Biggr| = o_{\mathbb{P}}(1),
\qquad
\Biggl|\gamma_n\hat{m}_1
\left(\frac{1}{n}u^\top u - \sigma_n^2\right)\Biggr|
= o_{\mathbb{P}}(1). \qedhere
\]

\end{proof}

\section{Further technical results}\label{App:C}

In this section, we present several auxiliary results that will be used in the proofs of Appendix~\ref{App:B}.

\begin{lemma}\label{lebip}
Under Assumptions~\ref{A1} and \ref{A3}, we have
\begin{align}\label{largeeig:p}
	s_{\max}\bigl(n^{-1} Z^\top Z\bigr) \;\overset{\mathbb{P}}{\longrightarrow}\; (1+\sqrt{\gamma})^2 
	\qquad \text{as }n\to\infty.
\end{align}

\end{lemma}

\begin{remark}	
A substantial literature establishes necessary and sufficient conditions for the almost sure convergence of the largest eigenvalue of the sample covariance matrix $n^{-1}Z^\top Z$, primarily under i.i.d.\ assumptions. In particular, \citet{BaiYin1988,BaiYin1993} show that if $\gamma_n = p/n \to \gamma \in (0,\infty)$, then the largest eigenvalue satisfies
\begin{align}\label{largeeig:as}
	s_{\max}(n^{-1}Z^\top Z) \xrightarrow{\text{a.s.}} (1+\sqrt{\gamma})^{2} 
	\qquad \text{as }n\to\infty.
\end{align}
Here, $Z = Z_n$ denotes the $n \times p$ upper-left block of a doubly infinite i.i.d.\ array with mean zero, variance one, and finite fourth moment. Moreover, \citet{notelargeeig} show that finiteness of the fourth moment is also necessary for the almost sure convergence in~\eqref{largeeig:as}.

\citet{BaiYin1993} further showed that the convergence in probability in~\eqref{largeeig:p} holds under the same framework, except that the finite fourth-moment assumption can be replaced by the condition
\begin{align}\label{BaiYincond}
	n^{2}\,\mathbb{P}\bigl(|z_{11}| \ge \sqrt{n}\bigr) = o(1),
\end{align}
which is readily seen to be weaker than the assumption of a bounded fourth moment.
Condition~\eqref{BaiYincond} also implies that $\mathbb{E}(|z_{11}|^{4-\varepsilon}) < \infty$ for each $\varepsilon>0$; see \citet[Remark~3]{BaiYin1993}. In fact, one can even show that \eqref{BaiYincond} lies strictly between the assumption of a bounded fourth moment and the requirement that all $(4-\varepsilon)$-th moments are finite, and thus constitutes only a minor relaxation in terms of a bounded moment condition.
Related discussions can be found in \citet{Silverstein1989}.

More recently, \citet[Theorem~1.8]{ChafTik} show that convergence in~\eqref{largeeig:p} holds for $\gamma_n \to \gamma \in (0,\infty)$ under a different collection of distributional assumptions.
Specifically, for each $n\in \mathbb{N}$, the rows of $Z$ are i.i.d.\ copies of an isotropic random vector $\xi^{(n)}$ (possibly varying with $n$) in $\mathbb{R}^p$, together with conditions controlling the tails of norms of projections of the rows. In particular, when $\xi$ does not vary with $n$ and its components are i.i.d.\ with bounded fourth moments, \citet[Corollary~1.9]{ChafTik} show that~\eqref{largeeig:p} holds.

In contrast, our setting assumes only that the entries of $Z=Z_n$ are independent for each $n\in \mathbb{N}$ and allows their distributions to vary with $i$, $j$, and $n$, but requires uniformly bounded $(4+\delta)$-th moments for some $\delta>0$.

\end{remark}

\begin{proposition}\label{prop:lebip}
Consider Assumptions~\ref{A1} and \ref{A3}, and let $\{\eta_n\}_{n \ge 1}$ be a sequence of positive numbers satisfying $\eta_n \log (n) \to \infty$. Then
\[
\sum_{i=1}^n \sum_{j=1}^{p} 
\mathbb{P}\bigl(|z_{ij}| \ge \eta_n \sqrt{n}\bigr) = o(1) \quad \text{as } n \to \infty.
\]
\end{proposition}
\begin{proof}
	By Markov's inequality, we have
	\begin{align*}
		\sum_{i=1}^n \sum_{j=1}^p 
		\mathbb{P}\bigl(|z_{ij}| \ge \eta_n \sqrt{n}\bigr)
		&\le \frac{1}{(\eta_n n^{1/2})^{4+\delta}}
		\sum_{i=1}^n \sum_{j=1}^p 
		\mathbb{E}\bigl(|z_{ij}|^{4+\delta}\bigr)\leq \gamma_n\frac{\nu_{4+\delta}}{\eta_n^{4+\delta} n^{\delta/2}} 
	\end{align*}
Since $\eta_n \log n \to \infty$, there exists an $n_0 \in \mathbb{N}$ such that
\[
\eta_n > \frac{1}{\log(n)}, \qquad \text{for all } n \ge n_0.
\]
Because $\eta_n > 0$ for all $n$, raising both sides to the power $4+\delta$ yields
\[
\eta_n^{4+\delta} > \frac{1}{\log(n)^{4+\delta}}, \qquad \text{for all } n \ge n_0.
\]
Multiplying by $n^{\delta/2}$ yields
	\[
	\eta_n^{4+\delta} n^{\delta/2} > \frac{n^{\delta/2}}{\log (n)^{4+\delta}} \quad \text{for all }n\geq n_0.
	\]
	Since, for all $\delta>0$, $n^{\delta/2}\log (n)^{-(4+\delta)} \to \infty$ as $n \to \infty$, 
	it follows that 
	\[
	\eta_n^{4+\delta} n^{\delta/2} \to \infty\qquad \text{as }n\to\infty,
	\]
	which completes the proof.
\end{proof}
%An admissible choice is
%\[
%k_n = n^{1/24},
%\qquad
%\eta_n = n^{-1/6}.
%\]

\begin{proof}[Proof of Lemma~\ref{lebip}]
Let $\{\eta_n\}_{n\ge 1}$ be a deterministic sequence such that $\eta_n>0$ for all $n\in\mathbb{N}$, 
$\eta_n \to 0$, and 
$\eta_n \log (n) \to \infty$ as $n \to \infty$. 
Let $\{k_n\}_{n\ge 1}$ be a sequence of nonnegative integers satisfying 
\begin{align*}
\frac{k_n}{\log (n)} \to \infty,
\qquad
\eta_n^{1/6}\frac{k_n}{\log (n)} \to 0
\qquad\text{and}\qquad\frac{\eta_n\sqrt{n}}{k_n^3}\to\infty\qquad \text{as } n \to \infty.
\end{align*}
In the next step, we truncate and center the entries of \(Z = Z_n\) for each
\(n \in \mathbb{N}\).
For \(i = 1, \ldots, n\) and \(j = 1, \ldots, p\), define the truncated variables
\[
\hat z_{ij} = z_{ij}\,\mathbf{1}\bigl\{|z_{ij}| < \eta_n \sqrt{n}\bigr\},
\]
and their centered versions
\[
\tilde z_{ij} = \hat z_{ij} - \mathbb{E}[\hat z_{ij}].
\]
Let
\[
\hat Z = (\hat z_{ij})_{1 \le i \le n,\; 1 \le j \le p}
\quad \text{and} \quad
\tilde Z = (\tilde z_{ij})_{1 \le i \le n,\; 1 \le j \le p}
\]
denote the corresponding \(n \times p\) matrices.
Moreover, define
\[
 s^{(n)}_{\max}= s_{\max}\!\left(n^{-1} Z^\top Z\right), \qquad
\hat s^{(n)}_{\max} = s_{\max}\!\left(n^{-1} \hat Z^\top \hat Z\right), \qquad
\tilde s^{(n)}_{\max} = s_{\max}\!\left(n^{-1} \tilde Z^\top \tilde Z\right).
\]
We begin by showing that  
\[s^{(n)}_{max}-\tilde{s}^{(n)}_{max}\overset{\mathbb{P}}{\longrightarrow}0.
\]
By the triangle inequality for the spectral norm,
we obtain
\begin{align}\label{largeigbound}
	\begin{split}
		\Bigl|\bigl(\hat s^{(n)}_{\max}\bigr)^{1/2}-\bigl(\tilde s^{(n)}_{\max}\bigr)^{1/2}\Bigr|
		= \frac{1}{\sqrt n}\bigl|\|\hat Z\|_2-\|\tilde Z\|_2\bigr|
		\le \frac{1}{\sqrt n}\|\hat Z-\tilde Z\|_2 \le \frac{1}{\sqrt n}\|\hat Z-\tilde Z\|_F.
	\end{split}
\end{align}
Let \(E=\hat Z-\tilde Z\). Then \(E\) is an \(n\times p\) matrix with entries
\[E_{ij}=\mathbb E(\hat z_{ij})
=\mathbb E\!\left[z_{ij}\mathbf 1\bigl\{|z_{ij}|<\eta_n\sqrt n\bigr\}\right],
\qquad 1\le i\le n,\;1\le j\le p.
\]
Since $\mathbb{E}[z_{ij}] = 0$, we have
\[
\mathbb{E}\bigl[z_{ij} \mathbf{1}\{|z_{ij}| < \eta_n \sqrt{n}\}\bigr] 
= - \mathbb{E}\bigl[z_{ij} \mathbf{1}\{|z_{ij}| \ge \eta_n \sqrt{n}\}\bigr].
\]
Applying the Cauchy--Schwarz inequality and using $\mathbb{E}(z_{ij}^2)=1$, it follows that
\begin{align*}
	\bigl| \mathbb{E}\bigl[z_{ij}\mathbf{1}\{|z_{ij}|<\eta_n\sqrt{n}\}\bigr] \bigr|
	&= \bigl| \mathbb{E}\bigl[z_{ij}\mathbf{1}\{|z_{ij}|\ge \eta_n\sqrt{n}\}\bigr] \bigr| \le \bigl\{\mathbb{E}(z_{ij}^2)\bigr\}^{1/2} \; \bigl\{\mathbb{P}\bigl(|z_{ij}|\ge \eta_n\sqrt{n}\bigr)\bigr\}^{1/2} \\
	&= \mathbb{P}\bigl(|z_{ij}|\ge \eta_n\sqrt{n}\bigr)^{1/2},
\end{align*}
for all $1\le i\le n$ and $1\le j\le p$.
Therefore, it follows from \eqref{largeigbound}, Proposition~\ref{prop:lebip}, and the previous display that
\begin{align}\label{half:conv}
	\begin{split}
		\Bigl|\bigl(\hat s^{(n)}_{\max}\bigr)^{1/2}-\bigl(\tilde s^{(n)}_{\max}\bigr)^{1/2}\Bigr|
		&\le \frac{1}{\sqrt{n}}\|\hat{Z}-\tilde{Z}\|_F=
		\Biggl(\frac{1}{n}\sum_{i=1}^n\sum_{j=1}^p
		\Bigl\{\mathbb E\bigl[z_{ij}\mathbf 1\{|z_{ij}|<\eta_n\sqrt n\}\bigr]\Bigr\}^2
		\Biggr)^{1/2} \\
		&\le
		\Biggl(\frac{1}{n}\sum_{i=1}^n\sum_{j=1}^p
		\mathbb P\bigl(|z_{ij}|\ge \eta_n\sqrt n\bigr)
		\Biggr)^{1/2}
		= o(n^{-1/2}).
	\end{split}
\end{align}
Moreover, by Markov’s inequality and Jensen’s inequality, for any \(M>0\),
\begin{align}\label{half:conv2}
	\begin{split}
		\mathbb P\Bigl(\bigl(\hat s^{(n)}_{\max}\bigr)^{1/2}>M\Bigr)
		&\le
		\mathbb P\biggl(
		\frac{1}{\sqrt{n}}\|\hat{Z}\|_F>M
		\biggr)
		\le
		\frac{1}{M\sqrt n}
		\mathbb E\Biggl[
		\Bigl(\sum_{i=1}^n\sum_{j=1}^p \hat z_{ij}^2\Bigr)^{1/2}
		\Biggr] \\
		&\le
		\frac{1}{M\sqrt n}
		\Biggl(
		\sum_{i=1}^n\sum_{j=1}^p
		\mathbb E\bigl[z_{ij}^2\mathbf 1\{|z_{ij}|<\eta_n\sqrt n\}\bigr]
		\Biggr)^{1/2}
		\le
		\frac{(\gamma_n n)^{1/2}}{M}.
	\end{split}
\end{align}
Similarly, for any \(M>0\), by Markov’s inequality, Jensen’s inequality, and the bound
\((a-b)^2 \le 2a^2 + 2b^2\) for \(a,b \in \mathbb{R}\), we obtain
\begin{align}\label{half:conv3}
	\begin{split}
		\mathbb{P}\Bigl(\bigl(\tilde s_{\max}^{(n)}\bigr)^{1/2} > M\Bigr)
		&\le
		\frac{1}{M\sqrt{n}}
		\Biggl[\sum_{i=1}^n \sum_{j=1}^p \mathbb{E}\Bigl\{\bigl[\hat z_{ij}- \mathbb{E}(\hat{z}_{ij})\bigr]^2\Bigr\}\Biggr]^{1/2} \\
		&\le
		\frac{2}{M\sqrt{n}}
		\Biggl(
		\sum_{i=1}^n \sum_{j=1}^p
		\mathbb{E}\bigl(\hat z_{ij}^2\bigr)
		\Biggr)^{1/2}
		\le
		\frac{2(\gamma_n n)^{1/2}}{M}.
	\end{split}
\end{align}
Combining \eqref{half:conv}–\eqref{half:conv3} with Assumption~\ref{A1}, we obtain
\begin{align}\label{hatsmintildmin}
	\begin{split}
	\hat s^{(n)}_{\max}-\tilde s^{(n)}_{\max}
	&=
	\Bigl(\bigl(\hat  s^{(n)}_{\max}\bigr)^{1/2}-\bigl(\tilde s^{(n)}_{\max}\bigr)^{1/2}\Bigr)
	\Bigl(\bigl(\hat s^{(n)}_{\max}\bigr)^{1/2}+\bigl(\tilde s^{(n)}_{\max}\bigr)^{1/2}\Bigr)\\
	&=
	o_{\mathbb{P}}\bigl(n^{-1/2}\bigr)\,O_{\mathbb P}\bigl(n^{1/2}\bigr)
	=
	o_{\mathbb P}(1).
	\end{split}
\end{align}
Next, by Proposition~\ref{prop:lebip}, we have that for any $\varepsilon>0$,
\begin{align}\label{sminhats}
	\begin{split}
		\mathbb{P}\Bigl(\bigl|s^{(n)}_{\max} - \hat s^{(n)}_{\max}\bigr| > \varepsilon\Bigr)
		&\le \mathbb{P}\Bigl(s^{(n)}_{\max} \neq \hat s^{(n)}_{\max}\Bigr) \le
		\mathbb{P}\Biggl(
		\bigcup_{1 \le i \le n,\; 1 \le j \le p}
		\bigl\{|z_{ij}| \ge \eta_n \sqrt{n}\bigr\}
		\Biggr) \\
		&\le
		\sum_{i=1}^n \sum_{j=1}^p
		\mathbb{P}\bigl(|z_{ij}| \ge \eta_n \sqrt{n}\bigr)
		= o(1).
	\end{split}
\end{align}
Using \eqref{hatsmintildmin} and \eqref{sminhats}, we conclude that
\begin{align}\label{lebip:step1}
s^{(n)}_{\max}-\tilde s^{(n)}_{\max}
=
\bigl(s^{(n)}_{\max}-\hat s^{(n)}_{\max}\bigr)
+
\bigl(\hat s^{(n)}_{\max}-\tilde s^{(n)}_{\max}\bigr)
\overset{\mathbb{P}}{\longrightarrow}0,
\qquad\text{as } n\to\infty.
\end{align}

In the next step of the proof we establish that
\begin{align}\label{liminf:smax}
	\liminf_{n \to \infty} s_{\max}^{(n)} \ge (1+\sqrt{\gamma})^2
	\quad \text{almost surely}.
\end{align}
We begin with the following observation. Let $\mu_n$ denote the measure corresponding to the empirical spectral distribution function $F_{n^{-1}Z^\top Z}$. By Theorem~\ref{Pan2010}, in the case where $\Sigma_n = I_p$ for all $n \in \mathbb{N}$, the measure $\mu_n$ converges weakly, almost surely, to the Marchenko--Pastur law $\mu_{\mathrm{MP}}$ with parameter $\gamma$, whose support is given by 
\[ \operatorname{supp}(\mu_{\mathrm{MP}}) = \begin{cases} \bigl[(1-\sqrt{\gamma})^2,\,(1+\sqrt{\gamma})^2\bigr], & \gamma \le 1,\\[1mm] \{0\} \cup \bigl[(1-\sqrt{\gamma})^2,\,(1+\sqrt{\gamma})^2\bigr], & \gamma > 1, \end{cases} \] (see, e.g., \citealp[Chapter~3]{Silversteinbook2010}). 
To establish \eqref{liminf:smax} we argue by contradiction. Fix an $\omega$ for which the conclusion of Theorem~\ref{Pan2010} holds with $\Sigma_n = I_p$ for all $n \in \mathbb{N}$, and suppose that \eqref{liminf:smax} is false. Then there exists an \(\varepsilon_0>0\) and a subsequence \(\{n_k\}_{k\ge 1}\) such that
\[
s_{\max}^{(n_k)} \le (1+\sqrt{\gamma})^2 - \varepsilon_0
\quad \text{for all } k.
\]
Therefore, for each $k$, all eigenvalues of the matrix \(n_k^{-1} Z^\top Z\) lie in \((-\infty, (1+\sqrt{\gamma})^2 - \varepsilon_0]\).  
Consequently, the corresponding empirical spectral measure satisfies
\[
\mu_{n_k}\Bigl( ((1+\sqrt{\gamma})^2 - \varepsilon_0, \infty) \Bigr) = 0
\quad \text{for all } k.
\]
Since \(((1+\sqrt{\gamma})^2 - \varepsilon_0, \infty)\) is an open set, the Portmanteau theorem implies that
\[
\liminf_{k\to\infty} \mu_{n_k}\Bigl( ((1+\sqrt{\gamma})^2 - \varepsilon_0, \infty) \Bigr)
\;\ge\; 
\mu_{\mathrm{MP}}\Bigl( ((1+\sqrt{\gamma})^2 - \varepsilon_0, \infty) \Bigr).
\]
The left-hand side is zero, while the right-hand side is strictly positive, because $(1+\sqrt{\gamma})^2$ is the right edge of $\operatorname{supp}(\mu_{\mathrm{MP}})$. This is a contradiction. Hence, we conclude that \eqref{liminf:smax} holds. 

Next, we show that, combining \eqref{lebip:step1} and \eqref{liminf:smax}, 
\[
s_{\max}^{(n)} \;\overset{\mathbb{P}}{\longrightarrow}\; (1+\sqrt{\gamma})^2,
\] 
provided that 
\begin{align}\label{lebip:step2}
\limsup_{n\to\infty} \tilde s_{\max}^{(n)} \;\le\; (1+\sqrt{\gamma})^2 \quad \text{almost surely}.
\end{align}
Fix $\epsilon>0$ and note that
\begin{align*}
	\mathbb{P}\Big(s_{\max}^{(n)} > (1+\sqrt{\gamma})^2 + \epsilon \Big) 
	&= \mathbb{P}\Big(\tilde s_{\max}^{(n)} + (s_{\max}^{(n)} - \tilde s_{\max}^{(n)}) > (1+\sqrt{\gamma})^2 + \epsilon \Big) \\
	&\le \mathbb{P}\Big(\tilde s_{\max}^{(n)} > (1+\sqrt{\gamma})^2 + \epsilon/2 \Big) 
	+ \mathbb{P}\Big(|s_{\max}^{(n)} - \tilde s_{\max}^{(n)}| > \epsilon/2\Big).
\end{align*}
where both terms on the right-hand side vanish as $n\to\infty$ by \eqref{lebip:step1} and \eqref{lebip:step2}. Hence,
\[
\lim_{n\to\infty} \mathbb{P}\Big(s_{\max}^{(n)} > (1+\sqrt{\gamma})^2 + \epsilon \Big) = 0.
\]
Similarly, \eqref{liminf:smax} implies that
\[
\lim_{n \to \infty}
\mathbb{P}\Bigl( s_{\max}^{(n)} < (1+\sqrt{\gamma})^2 - \varepsilon \Bigr) = 0.
\]
The result follows by combining the two previous displays.

Therefore, it remains to show that \eqref{lebip:step2} holds.
The crucial step is to establish that, for the sequence $\{k_n\}_{n\ge 1}$ defined at the beginning of the proof 
and for any $x > (1+\sqrt{\gamma})^2$,
\begin{equation}\label{BaiYinres}
	\sum_{n=1}^{\infty} 
	\mathbb{E}\Biggl[ 
	\Biggl( \frac{\tilde s^{(n)}_{\max}}{x} \Biggr)^{k_n} 
	\Biggr] < \infty.
\end{equation}
For an arbitrary $x > (1+\sqrt{\gamma})^2$, define
\[
A_n = \bigl\{\tilde s^{(n)}_{\max} > x \bigr\}, 
\qquad 
A = \limsup_{n \to \infty} A_n.
\]
Note that if \eqref{BaiYinres} holds, then by Markov's inequality,
\begin{align*}
	\sum_{n=1}^{\infty} \mathbb{P}(A_n)
 \le \sum_{n=1}^{\infty} \mathbb{E} \Biggl[ \Biggl( \frac{\tilde{s}^{(n)}_{\max}}{x} \Biggr)^{k_n} \Biggr] 
	< \infty.
\end{align*}
It then follows from the Borel--Cantelli lemma that
$\mathbb{P}(A) = 0$.
Hence, for every \(\omega \notin A\), there exists an \(n_0(\omega)\) such that for all
\(n \ge n_0(\omega)\), $\tilde s_{\max}^{(n)}(\omega) \le x.
$
Therefore,
\[
\limsup_{n \to \infty} \tilde s_{\max}^{(n)} \le x
\quad \text{almost surely}.
\]
Since \(x > (1+\sqrt{\gamma})^2\) is arbitrary, we conclude that
\[
\limsup_{n \to \infty} \tilde s_{\max}^{(n)} \le (1+\sqrt{\gamma})^2
\quad \text{almost surely}.
\]
It remains to verify \eqref{BaiYinres}. To this end, we follow the approach of 
\citet[Theorem~3.1]{BaiYin1988} and begin with a few preliminary observations.

Using matrix multiplication and a proof by induction, it can be shown that
\begin{align*}
	\tr\bigl((\tilde{Z}^\top \tilde{Z})^k\bigr) 
	= \sum_{j_1=1}^p (\tilde{Z}^\top \tilde{Z})^{(k)}_{j_1 j_1} 
	= \sum_{j_1, \dots, j_k = 1}^{p} 
	\tilde z_{\bullet j_1}^\top \tilde z_{\bullet j_2} \,
	\tilde z_{\bullet j_2}^\top \tilde z_{\bullet j_3} \,
	\cdots \,
	\tilde z_{\bullet j_{k-1}}^\top \tilde z_{\bullet j_k} \,
	\tilde z_{\bullet j_k}^\top \tilde z_{\bullet j_1}.
\end{align*}
Note that in the previous display, each summand contains $k$ inner products. 
For each inner product, we write
\[
\tilde z_{\bullet j_r}^\top \tilde z_{\bullet j_{r+1}} 
= \sum_{i_r=1}^n \tilde z_{i_r j_r} \tilde z_{i_r j_{r+1}}, 
\quad r = 1, \dots, k,
\]
with the convention that $j_{k+1} = j_1$. 
Hence, the trace can be expressed as
\[
\tr\bigl((\tilde{Z}^\top \tilde{Z})^k\bigr) 
= \sum_{j_1, \dots, j_k = 1 }^{p}\sum_{i_1, \dots, i_k = 1}^n 
\tilde z_{i_1 j_1} \tilde z_{i_1 j_2} \,
\tilde z_{i_2 j_2} \tilde z_{i_2 j_3} \,
\cdots \,
\tilde z_{i_k j_k} \tilde z_{i_k j_1}.
\]
Going back to verify \eqref{BaiYinres}, for arbitrary $k \ge 0$ we use the previous display to bound
\begin{align}\label{sumdec}
	\begin{split}
		\mathbb{E}\Bigl[\bigl(\tilde{s}^{(n)}_{\max}\bigr)^{k}\Bigr] 
		&\le \mathbb{E}\Bigl[\tr\bigl((n^{-1}\tilde{Z}^\top \tilde{Z})^k\bigr)\Bigr] \\
		&= n^{-k} \sum_{j_1, \dots, j_k = 1 }^{p}\sum_{i_1, \dots, i_k = 1}^n 
		\mathbb{E}\Bigl[
		\tilde z_{i_1 j_1} \tilde z_{i_1 j_2} \,
		\tilde z_{i_2 j_2} \tilde z_{i_2 j_3} \,
		\cdots \,
		\tilde z_{i_k j_k} \tilde z_{i_k j_1}
		\Bigr].
	\end{split}
\end{align}
For the remaining part of the proof, we introduce some notation from graph theory.  

Let $k \ge 1$ be arbitrary. Given two $k$-tuples of integers
\[
(i_1, \dots, i_k) \in \{1, \dots, n\}^k, \qquad
(j_1, \dots, j_k) \in \{1, \dots, p\}^k,
\]
we define a directed multigraph as follows. Place the vertices 
$\{i_1, \dots, i_k\}$ on a line called the $I$-line, 
and the vertices $\{j_1, \dots, j_k\}$ on a line called the $J$-line. 
The edge set is
\[
E \coloneqq \{e_1, \dots, e_{2k}\},
\]
where for each $a\in  \{1, \dots, k\}$,
\[
e_{2a-1} \coloneqq \overrightarrow{j_a i_a}, \qquad
e_{2a} \coloneqq \overrightarrow{i_a j_{a+1}},
\]
with the convention that $j_{k+1} = j_1$. Here, for each edge, the first index denotes the initial vertex and the second index denotes the terminal vertex.
Two vertices are considered equal if and only if they have the same integer value and lie on the same line. 
In particular, even if $i_l$ and $j_m$ take the same integer value, they are treated as distinct vertices since they lie on different lines.  
Thus, a vertex is determined not only by its numerical value but also by its membership in the $I$- or $J$-line.  
We say that two edges $e_i$ and $e_j$ are coincident if they connect the same sets of vertices; that is, the direction of the edge does not matter, only the pair of vertices it connects. The graph constructed above is called a $W$-graph if each edge coincides with at least one other edge, distinct from itself.

Consider the expectation on the right-hand side of \eqref{sumdec}.  
For each summand, we have a given set of $I$-indices $(i_1, \dots, i_k) \in \{1, \dots, n\}^k$ and $J$-indices $(j_1, \dots, j_k) \in \{1, \dots, p\}^k$, which allows us to construct a directed multigraph as described above.  
Note that each random variable appearing in the expectation corresponds to exactly one edge in the constructed graph.  
Moreover, the expectation is zero whenever the corresponding graph is not a $W$-graph.  
This conclusion follows from the definition of coincident edges, 
together with the independence of the $\tilde{z}_{i,j}$ and the fact that each $\tilde{z}_{i,j}$ has mean zero. 
Henceforth we restrict our attention to $W$-graphs.

%We are now going to convince ourselves that if two edges are coincident, the corresponding random variables on the right-hand side of \eqref{sumdec} are equal. Recall, that the even edges, i.e. $e_1,e_2,\dots$, can be written as $e_{2a}=\overrightarrow{i_{a}j_{a+1}}$ for some $a\in\{1,\dots,k\}$. Hence, they point from the $I$-line to the $J$-line. Analogously, the edges $e_{2a-1}=\overrightarrow{j_ai_a}$ for some  $a\in\{1,\dots,k\}$ point from the $J$-line to the $I$-line. We want to recall that two indices are considered equal if there are equal in the same line. Now, there are three cases of coincident edges to distinguish: (1) If an edge $e_{2m}$ coincides with an edge $e_{2n}$ for $m,n\in\{1, \dots, k\}$, then $i_{m}=i_{n}$ and $j_{m+1}=j_{n+1}$. (2) If $e_{2m-1}$ coincides with $e_{2n-1}$ for some $m,n\in\{1, \dots, k\}$, then $i_{m}=i_{n}$ and $j_{m}=j_{n}$. (3) In the case where $e_{2m-1}$ coincides with $e_{2n}$,  for some $m,n\in\{1, \dots, k\}$, we have $i_{m}=i_{n}$ and $j_{m}=j_{n+1}$. Hence, in all three cases the $I$-indices of the corresponding edges are equal, as well as the $J$-indices, and since each random variable first lists the $I$-index and then the $J$-index, the random variables are also equal. 
An edge $e_{2a}$ is called a \emph{column innovation} if $j_{a+1} \notin \{j_1, \dots, j_a\}$, and an edge $e_{2a-1}$ is called a \emph{row innovation} if $i_a \notin \{i_1, \dots, i_{a-1}\}$, for $a \ge 1$. 
By convention, when $a=1$, the set $\{i_1, \dots, i_{a-1}\}$ is empty.  Hence, a row innovation introduces a new integer for the row index of $\tilde{Z}$, while a column innovation introduces a new integer for the column index of $\tilde{Z}$.  
In particular, by this definition, $e_1$ is always a row innovation. We denote the set of row and column innovations by $T_1$. We also note that in \citet{BaiYin1988}, the definitions of row and column innovations are reversed, because in their setup $Z$ is a $p \times n$ matrix and they consider $s_{\max} = s_{\max}(n^{-1} Z Z^\top)$.  
This difference is purely notational and does not affect the combinatorial arguments or the resulting bounds on the spectral norm, as we will see later.

An edge $e_i$ is said to be \emph{single up to $e_j$}, with $i \le j$, if it does not coincide with any other edge in the set $\{e_1,\dots,e_j\}$ except itself. Note that by definition every $T_1$-edge is single up to itself.

An edge $e_i$ is called a \emph{$T_3$-edge} if there exists an innovation $e_j$ with $j<i$ that is single up to $e_{i-1}$ and such that $e_i$ coincides with $e_j$.
Loosely speaking, a $T_3$-edge is the first repetition of a $T_1$-edge in terms of coincidence. The $T_3$-edges are further divided into regular and irregular $T_3$-edges.

Let $e_i$ be a $T_3$-edge. By definition, there exists at least one \(T_1\)-edge \(e_j\),
\(j<i\), which is single up to \(e_{i-1}\), and such that $e_j$ shares at least one vertex with the initial vertex of
\(e_i\). Up to step \(i-1\), only the initial vertex of \(e_i\) is specified, and there may
exist several \(T_1\)-edges sharing this vertex.
Once the terminal vertex of \(e_i\) is determined, that is, once the full edge
\(e_i\) is specified, this uniquely determines which of these candidate \(T_1\)-edges
coincides with \(e_i\).  
If there is exactly one such $T_1$-edge, then $e_i$ is called an \emph{irregular $T_3$-edge}.  
If there exist at least two such $T_1$-edges, then $e_i$ is called a \emph{regular $T_3$-edge}.  
In this way, the set of $T_3$-edges is partitioned into regular and irregular $T_3$-edges. The combinatorial importance of this definition is that regular $T_3$-edges allow multiple possible pairings with preceding $T_1$-edges, thereby increasing the number of admissible edge sequences in the graph. 
In contrast, irregular $T_3$-edges correspond to a unique coincidence, leaving no flexibility for alternative pairings, but more on this later.

An edge is called a $T_4$-edge if it is neither a $T_1$-edge nor a $T_3$-edge.  
A $T_4$-edge $e_i$ is said to be a $T_2$-edge if it does not coincide with any edge in the set $\{e_1, \dots, e_{i}\} \cap T_4$ except itself. In other words, a $T_2$-edge is single-up to itself, among the $T_4$-edges.  
By definition, we have $T_2 \subseteq T_4$.  

The $T_2$-edges can be further divided into two categories: 
\begin{enumerate}
	\item $T_{2,1}$-edges, which coincide with a $T_3$-edge, and
	\item $T_{2,2}$-edges, which do not coincide with any $T_3$-edge.
\end{enumerate}

Looseley speaking, $T_{2,1}$-edges are repetitions of a repetition in terms of coincidence. We refer to $T_{2,2}$-edges as \emph{noninnovations}, meaning that they are edges which are single up to themselves but are not $T_1$-edges.  
Intuitively, these edges do not introduce new vertices, but allow for additional ways to connect vertices that have already appeared.
Let $t = |T_2|$ denote the total number of $T_2$-edges, and let $\mu = |T_{2,1}|$. Then the number of $T_{2,2}$-edges is $t - \mu$. 

 By definition, each $T_2$-edge does not coincide with any other $T_2$-edge, besides itself. Furthermore, observe that the definition of coincidence is an equivalence relation on $E$. It follows, that the $T_2$-edges partition the $T_4$-edges into equivalence classes under coincidence.
We denote by $n_i$, $i \in \{1, \dots, \mu\}$, the number of $T_4$-edges that coincide with the $i$-th $T_{2,1}$-edge, and note that $n_i \ge 1$.  
Similarly, we denote by $m_j$, $j \in \{1, \dots, t-\mu\}$, the number of $T_4$-edges that coincide with the $j$-th $T_{2,2}$-edge and note that $m_j \ge 2$.  
Indeed, for any $T_{2,2}$-edge, we count at least the edge itself and, by the definition of a $W$-graph, the edge it coincides with.

Corresponding to the above definitions of edges, we record some useful facts:
\begin{enumerate}
	\item\label{fact1} Note that by definition of the sets $T_1$, $T_3$, and $T_4$, we have
	\[
	E=\{e_1, \dots, e_{2k}\} = T_1 \cup T_3 \cup T_4
	\]
	and $T_2\subseteq T_4$.
	\item\label{fact2} Let $l$ denote the total number of innovations, $r$ the number of row innovations, 
	and $c$ the number of column innovations. Then $l = r + c$. Moreover, $e_1$ is always a row innovation, and since in a $W$-graph each innovation eventually coincides with another edge, it follows that $1 \leq \ell \leq k$. By definition, there is a one-to-one correspondence between $T_1$-edges and $T_3$-edges, implying
	$|T_3| = l.$
 
	\item \label{fact3} Recall that the graph contains a total of $2k$ edges.  
	Since the number of innovations equals the number of $T_3$-edges, it follows that the number of $T_4$-edges is
	$|T_4| = 2k - 2l$.

	\item \label{fact5} Since $T_2 \subseteq T_4$, we have
	\[
	t = |T_2| \le |T_4| = 2k - 2l,
	\]
	and since $T_{2,1} \subseteq T_2$, it follows that $\mu \le t$. Furthermore, by the definition of $T_2$-edges $t=0$ implies that $|T_4|=0$.
	
	\item \label{fact6} Next we show that $l \ge \mu$. By definition, each $T_{2,1}$-edge coincides with exactly one $T_3$-edge. Indeed, let $e_i$ be a $T_{2,1}$-edge. By the definition of $T_2$-edges, $e_i$ does not coincide with any other $T_2$-edge, besides itself, and hence it coincides with exactly one $T_3$-edge. By Fact~\ref{fact2}, each $T_3$-edge corresponds to exactly one $T_1$-edge. Therefore, each $T_{2,1}$-edge corresponds to exactly one $T_1$-edge, which immediately implies $l \ge \mu$. As a consequence, there are $l - \mu$ $T_1$-edges that do not coincide with any $T_{2,1}$-edge.  
	By definition, each of these $l-\mu$ $T_1$-edges also does not coincide with any $T_{2,2}$-edge, and hence cannot coincide with any edge in $T_4 \setminus T_2$.  
	It follows that each of these $l-\mu$ $T_1$-edges coincides with exactly one $T_3$-edge.
	\item \label{fact4}  
Using the definitions of $T_1$, $T_2$, $T_3$, and $T_4$-edges, along with the facts above, we can decompose the $2k$ edges of $E$ into equivalence classes under coincidence. By Fact~\ref{fact6}, there are $l-\mu$ $T_1$-edges that are coincident only with themselves and exactly one $T_3$-edge, and no other edges in $E$. Hence, each of these $l-\mu$ $T_1$-edges forms an equivalence class containing exactly $2$ elements. 
Similarly, for a $T_{2,1}$-edge $e_i$, there exists exactly one coincident $T_3$-edge and exactly one coincident $T_1$-edge. By the definition of $n_i$, the equivalence class of $e_i$ therefore contains $n_i + 2$ elements.  
Analogously, $m_j$ denotes the number of elements in the equivalence class of the $j$-th $T_{2,2}$-edge under coincidence.  

Consequently, we obtain
\begin{align}\label{graphdec}
	\begin{split}
		2k &= |T_1| + |T_3| + |T_4|= 2l + |T_4|= 2l + \sum_{i=1}^\mu n_i + \sum_{j=1}^{t-\mu} m_j \\
		&= 2(l-\mu) + \sum_{i=1}^\mu (n_i + 2) + \sum_{j=1}^{t-\mu} m_j.
	\end{split}
\end{align}

	\end{enumerate}

Two $W$-graphs are said to be \emph{isomorphic} if one can be obtained from the other by a permutation of the $I$-vertices in $\{1, \dots, n\}$ and a permutation of the $J$-vertices in $\{1, \dots, p\}$.  	Under this definition, the collection of $W$-graphs is partitioned into isomorphism classes.
	
A $W$-graph is called \emph{canonical} if $i_1 = j_1 = 1$, and for each $a \ge 2$,
\begin{align}\label{canonicalrule}
i_a \le \max\{i_1, \dots, i_{a-1}\} + 1, \qquad
j_a \le \max\{j_1, \dots, j_{a-1}\} + 1.
\end{align}
By construction, each isomorphism class contains exactly one canonical $W$-graph.  
Indeed, existence follows by relabeling the $I$- and $J$-vertices so that the first
vertex appearing in each class is labeled $1$, the next distinct vertex is labeled
$2$, and so on, separately for the $I$- and $J$-vertices.
To prove uniqueness, suppose there exist two canonical $W$-graphs in the same isomorphism class. We show that they must coincide.
Consider an edge $e_i$, $i>1$. The initial vertex of $e_i$ is already determined
by the preceding edges $\{e_1, \dots, e_{i-1}\}$. If $e_i$ is an innovation edge, then by \eqref{canonicalrule} its terminal vertex is also uniquely determined. Otherwise, $e_i$ only connects to vertices already introduced by $\{e_1, \dots, e_{i-1}\}$. Since vertex labels are determined by their order of first appearance, it follows that the choice of the terminal vertex of $e_i$ is uniquely determined.
It follows 
that the two graphs coincide edge by edge, and hence each isomorphism class 
contains exactly one canonical $W$-graph.

	\begin{figure}
		\centering
		\begin{tikzpicture}[
			x=5cm,
			y=3cm,->,>=stealth', shorten >=4pt, auto,node distance=3cm,
			thick,main node/.style={circle,draw,font=\sffamily\Large\bfseries}
			]
			\node at (2.25,0) {\(I\)};
			\node at (2.25,2) {\(J\)};
			\draw (0,0) node[draw, circle, label={below:\(i_1=i_4=1\)}] {} 
			-- (1,0) node[draw, circle, label={below:\(i_2=i_5=i_7=i_8=2\)}] {}
			-- (2,0) node[draw, circle, label={below:\(i_3=i_6=3\)}] {}
			-- (2.2,0) ;
			\draw (0,2) node[draw, circle, label={above:\(j_1=j_5=1\)}] {}
			-- (1,2) node[draw, circle, label={above:\(j_2=j_4=j_6=2\)}] {}
			--(2,2) node[draw, circle, label={above:\(j_3=j_7=j_8=3\)}] {} 
			-- (2.2,2) ;
			\path
			(0,2) edge[red] node [left] {e1}  (0,0)
			(0,0) edge[bend left, purple] node [left] {e2}  (1,2)
			(1,2) edge[blue] node [left] {e3}  (1,0)
			(1,0) edge[bend right, magenta] node [left] {e4} (2,2)
			(2,2) edge[cyan] node [right] {e5} (2,0)
			(2,0) edge[bend left=35, olive] node [left] {e6} (1,2)
			(1,2) edge[bend left, purple] node [right] {e7} (0,0)
			(0,0) edge[bend left, red] node [left] {e8} (0,2)
			(0,2) edge[bend right] node [left] {e9} (1,0)
			(1,0) edge[bend right,blue] node [left] {e10} (1,2)
			(1,2) edge[bend left=35, olive] node [right] {e11} (2,0)
			(2,0) edge[bend right, cyan] node [right] {e12} (2,2)
			(2,2) edge[bend right, magenta] node [right] {e13} (1,0)
			(1,0) edge[bend right=50, magenta] node [right] {e14} (2,2)
			(2,2) edge[magenta] node [right] {e15} (1,0)
			(1,0) edge[bend right] node [right] {e16} (0,2); 
		\end{tikzpicture}
		\caption{The figure illustrates an example of a canonical $W$-graph that highlights the definitions introduced above. 
			Coincident edges are indicated by the same color. 
			In this example, $k = 8$, 
			$T_1 = \{e_1, e_2, e_3, e_4, e_5\}$,
			$T_3 = \{e_7, e_8, e_{10}, e_{12}, e_{13}\}$,
			and $T_4 = \{e_6, e_9, e_{11}, e_{14}, e_{15}, e_{16}\}$,
			where
			$T_2 = \{e_6, e_9, e_{14}\}$.
			Here, $e_{14}$ is a $T_{2,1}$-edge, while $e_6$ and $e_9$ are $T_{2,2}$-edges. 
			Moreover, $e_8$ is an example of an irregular $T_3$-edge, whereas $e_7$ is a regular $T_3$-edge.
		}
		\label{Figure2}
	\end{figure} 
\medskip 
With these facts and definitions in hand, we decompose the sum on the right-hand side of \eqref{sumdec} according to the number of innovations. 
Recall that $l$ denotes the number of $T_1$-edges in a $W$-graph. 
For each fixed $l$, we sum over all index sequences $(i_1,\dots,i_k)$ and $(j_1,\dots,j_k)$ that are compatible with a $W$-graph having exactly $l$ innovations. 
Within each such configuration, we further sum over all possible arrangements of $T_1$-, $T_3$-, and $T_4$-edges, all canonical $W$-graphs corresponding to that arrangement, and all $W$-graphs isomorphic to a given canonical graph. 
Formally, we write
\[
n^{-k} \sum_{l=1}^{k} \sum\nolimits^\prime \sum\nolimits^{\prime\prime} \sum\nolimits^{\prime\prime\prime} 
\mathbb{E} \Bigl( \tilde z_{i_1 j_1} \tilde z_{i_1 j_2} \, \tilde z_{i_2 j_2} \tilde z_{i_2 j_3} \cdots \tilde z_{i_k j_k} \tilde z_{i_k j_1} \Bigr),
\]
where:
\begin{itemize}
	\item $\sum\nolimits^\prime$ sums over all arrangements of $T_1$-, $T_3$-, and $T_4$-edges;
	\item $\sum\nolimits^{\prime\prime}$ sums over all canonical $W$-graphs for a given arrangement; 
	\item $\sum\nolimits^{\prime\prime\prime}$ sums over all $W$-graphs isomorphic to a given canonical graph.
\end{itemize}

We now bound the number of terms arising in each of the above sums, starting with 
$\sum_{l=1}^k \sum\nolimits^\prime$.
Recall that the first edge in a $W$-graph is always a row innovation and hence belongs to $T_1$.
Therefore, for a fixed $l \in \{1,\dots,k\}$, there are
\[
\binom{2k-1}{\,l-1\,}\leq \binom{2k}{\,l\,} 
\]
possible choices for the locations of the remaining $l-1$ $T_1$-edges among the edges
$\{e_2,\dots,e_{2k}\}$. Once the positions of the $T_1$-edges are fixed, there remain $2k-l$ edges.
Since $|T_3|=l$, there are
\[
\binom{2k-l}{\,l\,}
\]
ways to choose the positions of the $T_3$-edges among the remaining edges, and the remaining
$2k-2l$ positions are occupied by $T_4$-edges.
By the Vandermonde identity,
\[
\binom{2k}{l}
=
\sum_{c=0}^l \binom{k}{c}\binom{k}{l-c}.
\]
Therefore,
\[
\Bigl|\sum_{l=1}^k \sum\nolimits^\prime\Bigr|
\le
\sum_{l=1}^k \binom{2k}{l}\binom{2k-l}{l}
=
\sum_{l=1}^k \sum_{c=0}^l
\binom{k}{c}\binom{k}{l-c}\binom{2k-l}{l}.
\]
Bounding the number of elements of $\sum\nolimits^{\prime\prime\prime}$ is straightforward.
Observe that every column innovation introduces a new $J$-index, while every row innovation
introduces a new $I$-index. Since the first edge is always a row innovation, a $W$-graph with
$c$ column innovations and $r$ row innovations contains exactly $c+1$ distinct $J$-indices
and $r$ distinct $I$-indices.
As the $I$-indices take values in $\{1,\dots,n\}$ and the $J$-indices take values in
$\{1,\dots,p\}$, each isomorphism class contains
\[
p(p-1)\cdots(p-c)\; n(n-1)\cdots(n-r+1)
\]
distinct graphs. Using the bounds
$p(p-1)\cdots(p-c)\le p^{c+1}$, 
$n(n-1)\cdots(n-r+1)\le n^{r}$, and recalling that $l=c+r$,
we obtain
\[
\Bigl|\sum\nolimits^{\prime\prime\prime}\Bigr| \le p^{c+1} n^{r}
= p^{c+1} n^{\,l-c}.
\]

Next, we bound the number of canonical graphs corresponding to a fixed arrangement of $T_1$, $T_3$, and $T_4$ edges. 
Equivalently, we ask: for a given edge-type arrangement, what is the maximal number of distinct canonical graphs that can be formed?
Assume for the moment that the number of $T_2$-edges $t$ satisfies $t \ge 1$.    
Recall that in a canonical $W$-graph, the vertices of each $T_1$-edge are uniquely determined by the preceding edges.  
Specifically, the initial vertex of a $T_1$-edge is either an existing vertex determined by previous edges or the first vertex in the graph, and the terminal vertex is the new integer assigned according to the innovation rule. Similarly, if $e_i$ with $i \ge 2$ is an irregular $T_3$-edge, its vertices are uniquely determined by the single $T_1$-edge that is single up to $e_{i-1}$. Hence, once the arrangement of the different types of edges is fixed in a canonical $W$-graph, the vertices of $T_1$-edges and irregular $T_3$-edges are uniquely determined, and therefore these edges do not add any further combinatorial complexity.

The situation is different for regular $T_3$-edges. For example, in
Figure~\ref{Figure2}, the edge $e_7$ is a regular $T_3$-edge. Given the arrangement
of edge types up to $e_6$, there are two admissible choices for completing
$e_7$. Indeed, the edges $e_2$ and $e_3$ are $T_1$-edges that are single up to $e_6$ and
each shares one vertex with the initial vertex of $e_7$. Consequently, when
constructing $e_7$ as a regular $T_3$-edge, its terminal vertex must be identified
with the corresponding vertex of one of these candidate $T_1$-edges. In this
example, one may therefore choose to connect $j_4 = 2$ either with $i_4 = 1$ or
with $i_4 = 2$.

\citet[Lemma~5.5]{Silversteinbook2010} show that, for a regular $T_3$-edge $e_i$ in a canonical $W$-graph, the number of $T_1$-edges that are single up to $e_{i-1}$ and share at least one vertex with the initial vertex of $e_i$ is bounded by $t+1$.  
Consequently, there are at most $t+1$ possible ways to draw a single regular $T_3$-edge.
Moreover, \citet[Lemma 5.6]{Silversteinbook2010} show that the total number of regular $T_3$-edges is at most twice the number of $T_2$-edges.
Since $t=|T_2|\leq |T_4|=2k-2l$, the total number of ways to assign vertices to all regular $T_3$-edges is bounded by
\begin{align}\label{T4bound1}
(t+1)^{2t} \le (t+1)^{2(2k-2l)}.
\end{align}
We now consider the $T_4$-edges. Recall that we assume for the moment that there are $t \ge 1$ $T_2$-edges in the graph. As established the $T_2$-edges partition the $T_4$-edges in equivalence classes under coincidence. 
Hence, for a given set of $t$ $T_2$-edges, there are at most $t^{2k-2l}$ possible assignments of the vertices of these $T_2$-edges to the remaining $2k-2l$ $T_4$-edges.

Next, we estimate the number of ways to choose the $t$ $T_2$-edges given the indices
$\{i_1,\dots,i_k\}$ and $\{j_1,\dots,j_k\}$ using a crude bound.
Note that in a canonical $W$-graph each $I$-vertex of an edge takes values in
$\{1,\dots,r\}$, while each $J$-vertex takes values in $\{1,\dots,c+1\}$.
Hence, there are at most $r(c+1)$ possible choices for a single edge.
It follows that the number of ways to choose $t$ distinct $T_2$-edges is bounded by
\[
\binom{r(c+1)}{t}
\;\le\;
\binom{k^2}{t}
\;\le\;
k^{2t}.
\]
Combining these bounds, the total number of ways to draw the $2k-2l$ $T_4$-edges is at most
\begin{align}\label{T4bound2}
k^{2t} \, t^{2k-2l} \le k^{2t} (t+1)^{2k-2l}.
\end{align}
It remains to consider the case $t=0$. By \citet[Lemma 5.6]{Silversteinbook2010}, there are no regular $T_3$-edges in this case. Moreover, by Fact~\ref{fact5}, there are no $T_4$-edges when $t=0$. Hence, when $t=0$, the graph contains only $T_1$-edges and irregular $T_3$-edges and $l=k$. By the previous discussion, there is a unique way to assign vertices to both $T_1$-edges and irregular $T_3$-edges, so that $\sum\nolimits^{\prime\prime}$ is only a single summand in this case.  
Therefore, combining \eqref{T4bound1} and \eqref{T4bound2}, together with the discussion of the case $t=0$, and recalling that $t \in \{0,\dots,2k-2l\}$, we conclude that the total number of summands in $\sum\nolimits^{\prime\prime}$ is bounded by
\[
\sum_{t=0}^{2k-2l} (t+1)^{4k-4l} k^{2t} (t+1)^{2k-2l}
=
\sum_{t=0}^{2k-2l} k^{2t} (t+1)^{6k-6l}.
\]
Using the bounds from above, the total number of summands of 
$\sum_{l=1}^k \sum^\prime \sum^{\prime\prime} \sum^{\prime\prime\prime}$ can be bounded by
\begin{equation}\label{fastfertig1}
\sum_{l=1}^k \sum_{c=0}^l 
\binom{k}{c} \binom{k}{l-c} \binom{2k-l}{l} 
\sum_{t=0}^{2k-2l} k^{2t} (t+1)^{6k-6l} p^{c+1} n^{\,l-c}.
\end{equation}

We now turn to bounding the expectation
\[
\mathbb{E}\bigl(\tilde z_{i_1 j_1} \tilde z_{i_1 j_2} \tilde z_{i_2 j_2} \tilde z_{i_2 j_3} \cdots \tilde z_{i_k j_k} \tilde z_{i_k j_1}\bigr).
\]
First, note that for all \(i \in \{1,\dots,n\}\) and \(j \in \{1,\dots,p\}\),
the random variables \(\tilde z_{ij}\) are independent and satisfy
\[
\mathbb{E}(\tilde z_{ij})
= \mathbb{E}(\hat z_{ij}) - \mathbb{E}(\hat z_{ij})
= 0,
\]
\[
|\tilde z_{ij}| \le 2\, \eta_n \sqrt{n},
\]
and
\[
\mathbb{E}(\tilde z_{ij}^2)
= \mathbb{E}\bigl[\bigl(\hat z_{ij}-\mathbb{E}(\hat z_{ij})\bigr)^2\bigr]
\le \mathbb{E}(\hat{z}_{ij}^2)
\le \mathbb{E}(z_{ij}^2)
= 1.
\]
Furthermore, since \(|\tilde z_{ij}| \le 2\,\eta_n\sqrt n\) almost surely and
\(\mathbb E(\tilde z_{ij}^2)\le 1\), we have for all \(l \ge 2\),
\[
\mathbb E\bigl(|\tilde z_{ij}|^l\bigr)
\le (2\,\eta_n\sqrt n)^{\,l-2}.
\]
For $l \ge 3$, using Hölder’s inequality and since Assumption~\ref{A3} implies uniformly bounded fourth moments,
\[
\mathbb{E}\bigl( |\tilde z_{ij}|^l \bigr)
\le \mathbb{E}\bigl( |\tilde z_{ij}|^3 \bigr)\,\bigl(2\eta_n \sqrt n\bigr)^{\,l-3}
\le \mathbb{E}\bigl( \tilde z_{ij}^4 \bigr)^{1/2}
\mathbb{E}\bigl( \tilde z_{ij}^2 \bigr)^{1/2}
\bigl(2\eta_n\sqrt n\bigr)^{\,l-3}
\le \bigl(2^4\nu_4\bigr)^{1/2}\bigl(2\eta_n\sqrt n\bigr)^{\,l-3}.
\]
We define
\[
C \coloneqq \bigl(2^4\nu_4\bigr)^{1/2}.
\]
As already established, in the case $t=0$ there are only $T_1$-edges and
irregular $T_3$-edges in the graph. Hence, each random variable appears
exactly twice, and by independence,
\[
\mathbb{E}\Bigl(
\tilde{z}_{i_1j_1}\tilde{z}_{i_1j_2}\tilde{z}_{i_2j_2}\tilde{z}_{i_2j_3}
\cdots
\tilde{z}_{i_kj_k}\tilde{z}_{i_kj_1}
\Bigr)
=
\prod_{a=1}^k
\mathbb{E}\bigl(\tilde{z}_{i_aj_a}^2\bigr)
\le 1.
\]
Now assume $t \ge 1$. There are $\mu$ $T_{2,1}$-edges and $t-\mu$
$T_{2,2}$-edges. By the decomposition \eqref{graphdec} into equivalence
classes under coincidence and the independence of the entries
$\tilde z_{ij}$, the expectation factorizes as
\begin{align}\label{expectationgraph}
	\mathbb{E}\Bigl(
	\tilde{z}_{i_1j_1}\tilde{z}_{i_1j_2}\cdots
	\tilde{z}_{i_kj_k}\tilde{z}_{i_kj_1}
	\Bigr)
	&=
	\prod_{v=1}^{l-\mu}
	\mathbb{E}\bigl(\tilde{z}_{i_{I_v}j_{I_v}}^{2}\bigr)
	\prod_{v=1}^{t-\mu}
	\mathbb{E}\bigl(\tilde{z}_{i_{N_v}j_{N_v}}^{m_v}\bigr)
	\prod_{v=1}^{\mu}
	\mathbb{E}\bigl(\tilde{z}_{i_{M_v}j_{M_v}}^{n_v+2}\bigr).
\end{align}
Here:
\begin{itemize}
	\item $(i_{I_v},j_{I_v})$, $v=1,\dots,l-\mu$, denote the vertices corresponding to
	$T_1$-edges that only coincide with
	exactly one irregular $T_3$-edge;
	\item $(i_{N_v},j_{N_v})$, $v=1,\dots,t-\mu$, denote the vertices corresponding to the
	equivalence classes under coincidence generated by $T_{2,2}$-edges, each class
	having size $m_v\ge2$;
	\item $(i_{M_v},j_{M_v})$, $v=1,\dots,\mu$, denote the vertices corresponding to the
	equivalence classes under coincidence generated by $T_{2,1}$-edges, each class
	having size $n_v+2\ge3$.
\end{itemize}
Note that
\begin{align*}
	\prod_{v=1}^{\mu}
	\mathbb{E}\bigl(\tilde{z}_{i_{M_v}j_{M_v}}^{\,n_v+2}\bigr)
	&\le
	\prod_{v=1}^{\mu}
	C(2\eta_n\sqrt{n})^{\,n_v-1}
	\le
	C^{\mu}(2\eta_n\sqrt{n})^{\sum_{v=1}^{\mu}(n_v-1)}.
\end{align*}
Similarly,
\begin{align*}
	\prod_{v=1}^{t-\mu}
	\mathbb{E}\bigl(\tilde{z}_{i_{N_v}j_{N_v}}^{\,m_v}\bigr)
	&\le
	\prod_{v=1}^{t-\mu}
	(2\eta_n\sqrt{n})^{\,m_v-2}
	=
	(2\eta_n\sqrt{n})^{\sum_{v=1}^{t-\mu}(m_v-2)}.
\end{align*}
Therefore, using \eqref{expectationgraph} and the bound
\(\mathbb{E}(\tilde z_{ij}^2)\le 1\), we obtain
\begin{align}\label{fastfertig2}
	\begin{split}
	\mathbb{E}\Bigl(
	\tilde{z}_{i_1j_1}\cdots \tilde{z}_{i_kj_1}
	\Bigr)
	&\le
	C^{\mu}
	(2\eta_n\sqrt{n})^{\sum_{v=1}^{t-\mu}(m_v-2)+\sum_{v=1}^{\mu}(n_v-1)} \\
	&=
	C^{\mu}(2\eta_n\sqrt{n})^{\,2k-2l-2(t-\mu)-\mu},
	\end{split}
\end{align}
where in the last line we used \[
2k-2l=\sum_{v=1}^{t-\mu}m_v+\sum_{v=1}^\mu n_v.
\]
Recall that $k = k_n$ and $\eta_n$ satisfy
\[
\frac{k_n}{\log (n)} \to \infty,
\qquad
\eta_n^{1/6}\frac{k_n}{\log (n)} \to 0
\qquad\text{and}\qquad\frac{\eta_n\sqrt{n}}{k_n^3}\to\infty\qquad \text{as } n \to \infty.
\]
Choose $N \in \mathbb{N}$ such that $k_n \ge C\geq 1$ for all $n \ge N$ and recall that $\mu\leq t$.
Then the upper bound in \eqref{fastfertig2}, for all $n \ge N$, simplifies to
\begin{equation}\label{fastfertig3}
	C^\mu (2\eta_n\sqrt{n})^{2k-2l-2(t-\mu)-\mu}
	\;\le\;
	k^t (2\eta_n\sqrt{n})^{2k-2l-t}.
\end{equation}
Combining \eqref{fastfertig1}, \eqref{fastfertig2} and \eqref{fastfertig3}, we obtain, for all $n\geq N$,
\begin{align}\label{momentupperboundtr}
	\begin{split}
		\mathbb{E}\Bigl(\bigl(\tilde{s}^{(n)}_{\max}\bigr)^k\Bigr)
		&\le n^{-k}
		\sum_{l=1}^k \sum_{c=0}^l
		\binom{k}{c}\binom{k}{l-c}\binom{2k-l}{l}
		\, p^{c+1} n^{\,l-c} \\
		&\quad\times\biggl[
		\sum_{t=0}^{2k-2l}
		k^{2t}(t+1)^{6(k-l)} k^t
		(2\eta_n\sqrt{n})^{2k-2l-t}\biggr].
	\end{split}
\end{align}
Rewriting the inner sum gives
\[
\sum_{t=0}^{2k-2l}
k^{2t}(t+1)^{6(k-l)} k^t
(2\eta_n\sqrt{n})^{2k-2l-t}
=
(2\eta_n\sqrt{n})^{2k-2l}
\sum_{t=0}^{2k-2l}
k^{3t}(t+1)^{6(k-l)}(2\eta_n\sqrt{n})^{-t}.
\]
Since
\[
n^{-k} (\sqrt{n})^{2k-2l} n^l = 1,
\]
the bound in \eqref{momentupperboundtr} simplifies to
\begin{align}\label{momentupperboundtr1}
	\begin{split}
	\mathbb{E}\Bigl(\bigl(\tilde{s}^{(n)}_{\max}\bigr)^k\Bigr)
	&\le
	p \sum_{l=1}^k \sum_{c=0}^l
	\binom{k}{c}\binom{k}{l-c}\binom{2k-l}{l}
	\Bigl(\frac{p}{n}\Bigr)^c \\
	&\quad\times\Biggl[
	\sum_{t=0}^{2k-2l}
	k^{3t}(t+1)^{6(k-l)}(2\eta_n\sqrt{n})^{-t}\Biggr]\,(2\eta_n)^{2k-2l}.
	\end{split}
\end{align}
Since we assume that $\eta_n \sqrt{n}/k_n^3 \to \infty$, it follows that
$2\eta_n \sqrt{n}/k_n^3 > 1$ for all sufficiently large $n$ and we
define $\tilde{\eta}_n \coloneqq 2\eta_n$. Using the inequality $a^{-t}(t+1)^b \le a (b/\log a)^b$, which holds
for $a > 1$, $b > 0$, and $t > 0$ (see \citet[p.~520]{BaiYin1988}), with
$a = \tilde{\eta}_n \sqrt{n}/k^3$ and $b = 6(k-l)$. Consequently, for all
$l \in \{1,\dots,k-1\}$ and all sufficiently large $n$, we obtain
	\begin{align}\label{fastfertig4}
		\begin{split}
		\sum_{t=1}^{2k-2l}k^{3t}(t+1)^{6(k-l)}(\tilde{\eta}_n\sqrt{n})^{-t}&= \sum_{t=1}^{2k-2l}\bigg(\frac{\tilde{\eta}_n\sqrt{n}}{k^3}\bigg)^{-t}(t+1)^{6(k-l)}\\
		&\leq 2k\, \biggl(\frac{\tilde{\eta}_n\sqrt{n}}{k^3}\biggr)\left(\frac{6(k-l)}{\log\bigl(\frac{\tilde{\eta}_n\sqrt{n}}{k^3}\bigr)}\right)^{6(k-l)}\\
		&\leq 2\,\bigl(\tilde{\eta}_n\sqrt{n}\bigr)\,\Biggl(\frac{6(k-l)}{\log(\tilde{\eta}_n)+\frac{1}{2}\log(n)-3\log(k)}\Biggr)^{6(k-l)}.
		\end{split}
	\end{align}
Now, observe that the assumption $\eta_n \log(n) \to \infty$ implies that
\[
\eta_n \ge \frac{1}{\log(n)}, \quad \text{for all sufficiently large } n.
\]
Taking logarithms and using that $\eta_n \to 0$, we obtain
\[
0 \ge \log(\eta_n) \ge - \log(\log(n)), \quad \text{for all sufficiently large } n.
\]
Since $\tilde{\eta}_n$ is merely a rescaled (by a constant) version of $\eta_n$, the same bound holds for $\tilde{\eta}_n$ for all sufficiently large $n$.

Similarly, the condition $\delta_n^{1/6} k_n / \log(n) \to 0$ implies that
\[
k_n \le \frac{\log(n)}{\eta_n^{1/6}}, \quad \text{for all sufficiently large } n.
\]
Taking logarithms and using the previous considerations, we obtain
\[
\log k_n \le \log \log(n) - \frac{1}{6} \log(\eta_n) \le \frac{7}{6} \log \log(n), \quad \text{for all sufficiently large } n.
\]
Hence, for all sufficiently large $n$,
\[
\log(\tilde{\eta}_n) + \frac{1}{2}\log n - 3 \log k_n 
\ge \frac{1}{2}\log (n) - \frac{27}{6} \log( \log (n)) \ge \frac{1}{3} \log (n).
\]
Since $\tilde{\eta}_n \sqrt{n} \le n$ for all sufficiently large $n$, and using \eqref{fastfertig4}, we obtain that, for all sufficiently large $n$ and all $l \in \{1, \dots, k-1\}$,
\begin{align*}
	\sum_{t=1}^{2k-2l} k^{3t} (t+1)^{6(k-l)} (\tilde{\eta}_n \sqrt{n})^{-t} 
	\le 2\, n \, \Biggl( \frac{18 k}{\log (n)} \Biggr)^{6(k-l)}.
\end{align*}
Furthermore, in the case $t = 0$ or $k = l$, we have
\[
k^{3t}(t+1)^{6(k-l)}(\tilde{\eta}_n \sqrt{n})^{-t} = 1 \le 2n,
\]
and therefore, for all $l \in \{1, \dots, k\}$, there exists an $N_1 \in \mathbb{N}$ such that for all $n \ge N_1$,
\begin{align}\label{fastfertig5}
	\sum_{t=0}^{2k-2l} k^{3t} (t+1)^{6(k-l)} (2 \eta_n \sqrt{n})^{-t} 
	\le 2\, n \, \Biggl( \frac{18 k}{\log (n)} \Biggr)^{6(k-l)}.
\end{align}
	
Using \citet[Lemma~2.1]{BaiYin1988} with $1 \le c \le l \le k$ and for any $\delta>0$, we have
\begin{equation}\label{BaiYinLem2.1}
	\binom{k}{c}\binom{k}{l-c}\binom{2k-l}{l}\delta^{k-l}
	\le \bigl(1+\sqrt{\delta}\bigr)^{2k}\binom{k}{l}\binom{2l}{2c}.
\end{equation}
We now show that \eqref{BaiYinLem2.1} remains valid for all $0 \le c \le l \le k$.
First consider the case $c=0<l\le k$. Then
\[
\binom{k}{c}\binom{k}{l-c}\binom{2k-l}{l}\delta^{k-l}
= \binom{k}{l}\binom{2k-l}{l}\delta^{k-l}.
\]
Since all terms are nonnegative, we may bound
\begin{align*}
	\binom{k}{l}\binom{2k-l}{l}\delta^{k-l}
	\le \binom{k}{l}\sum_{m=0}^k \binom{2k-m}{m}\delta^{k-m} = \binom{k}{l}\sum_{l^*=0}^k \binom{k+l^*}{2l^*}\delta^{l^*},
\end{align*}
where $l^*=k-m$. Observing that $\binom{k+l^*}{2l^*}\le \binom{2k}{l^*}$ for all $0\le l^*\le k$ and that
\[
\sum_{l^*=0}^k \binom{2k}{2l^*}\delta^{l^*}
\le \sum_{l^*=0}^{2k} \binom{2k}{l^*}\bigl(\sqrt{\delta}\bigr)^{\,l^*},
\]
we obtain, by the binomial theorem,
\[
\binom{k}{l}\sum_{l^*=0}^k \binom{k+l^*}{2l^*}\delta^{l^*}
\le \binom{k}{l}\sum_{l^*=0}^{2k} \binom{2k}{l^*}\bigl(\sqrt{\delta}\bigr)^{\,l^*}
= \bigl(1+\sqrt{\delta}\bigr)^{2k}\binom{k}{l}.
\]
Next, if $c=l<k$, inequality \eqref{BaiYinLem2.1} reduces to
$\delta^k \le (1+\delta)^{2k}$, which is immediate. Finally, when $c=l=k$, both sides of \eqref{BaiYinLem2.1} equal one. Therefore, inequality \eqref{BaiYinLem2.1} holds for all $0 \le c \le l \le k$.

Using \eqref{fastfertig5} and \eqref{BaiYinLem2.1} for $\delta=\tilde{\eta}_n$, the right hand side of \eqref{momentupperboundtr1} can be upper bounded, for all $n\geq N_1$, by
	\begin{align*}
		\mathbb{E}\Bigl(\bigl(s^{(n)}_{max}\bigr)^{k}\Bigr)&\leq 2np\, \sum_{l=1}^k\sum_{c=0}^l \binom{k}{l}\binom{2l}{2c}\gamma_n^c\bigl(1+\sqrt{\tilde{\eta}_n}\bigr)^{2k}\Biggl(\frac{18\,\tilde{\eta}_n^{1/6}k}{\log(n)}\Biggr)^{6(k-l)}.
		\end{align*}
The upper bound further simplifies to
\begin{align*}	
	\mathbb{E}\Bigl((s^{(n)}_{max})^{k}\Bigr)&\leq 2np\, \bigl(1+\tilde{\eta}_n\bigr)^{2k}\sum_{l=1}^k\binom{k}{l}\Biggl(\frac{18\,\tilde{\eta}_n^{1/6}k}{\log(n)}\Biggr)^{6(k-l)}\sum_{c=0}^{2l} \binom{2l}{c}\bigl(\sqrt{\gamma_n}\bigr)^c\\
	&\leq 2np\, \bigl(1+\tilde{\eta}_n\bigr)^{2k}\sum_{l=0}^k\binom{k}{l}\bigl\{(1+\sqrt{\gamma_n})^2\bigr\}^l\,\Biggl[\biggl(\frac{18\,\tilde{\eta}_n^{1/6}k}{\log(n)}\biggr)^{6}\Biggr]^{k-l}\\
	&=2np\, \bigl(1+\tilde{\eta}_n\bigr)^{2k}\left(\bigl(1+\sqrt{\gamma_n}\bigr)^{2}+ \biggl(\frac{18\,\tilde{\eta}_n^{1/6}k}{\log(n)}\biggr)^{6}\right)^k\\
	&=\left((2np)^{1/k}\bigl(1+\tilde{\eta}_n\bigr)^{2}\Biggl[\bigl(1+\sqrt{\gamma_n}\bigr)^{2}+ \biggl(\frac{18\,\tilde{\eta}_n^{1/6}k}{\log(n)}\biggr)^{6}\Biggr]\right)^k\eqqcolon D^k_n
	,
\end{align*}
for all $n\geq N_1$.
Observe that
	\begin{itemize}
		\item $(2np)^{1/k}\to 1,$ because $k_n/\log(n)\to \infty$ as $n\to \infty$,
		\item $(1+\sqrt{\tilde{\eta}_n})^2\to 1$, since $\eta_n\to 0$ and therefore, $\tilde{\eta}_n\to 0$.
		\item $(1+\sqrt{\gamma_n})^2\to (1+\sqrt{\gamma})^2$, because of Assumption~\ref{A1}, and 
		\item $\tilde{\eta}_n^{1/6}k_n/\log(n)\to0$, since we assume that  $\eta_n^{1/6}k_n/\log(n)\to0$.
	\end{itemize}
Hence, $D_n \to (1+\sqrt{\gamma})^2$. Therefore, there exists $\delta>0$ satisfying
\[
(1+\sqrt{\gamma})^2 < \delta < x < \infty
\]
and an $N_2 \in \mathbb{N}$ such that, for all $n \ge N_2$,
\[
D_n^k \le \delta^k.
\]
Consequently, for all $n \ge N_2$,
\[
\mathbb{E}\!\left[\frac{\bigl(\tilde{s}^{(n)}_{\max}\bigr)^{k_n}}{x^{k_n}}\right]
\le \left(\frac{\delta}{x}\right)^{k_n}.
\]
Since $k_n / \log (n) \to \infty$, there exists for any $\varepsilon>0$ an $n_0 \in \mathbb{N}$ such that
\[
k_n \ge \varepsilon \log (n), \quad \text{for all } n \ge n_0.
\]
To guarantee summability, we choose $\varepsilon$ sufficiently large so that
\[
K := -\varepsilon \log(\delta/x) > 1.
\]
This is possible because $\log(\delta/x)<0$.  
Hence, there exists an $N_3\in \mathbb{N}$ such that for all $n\geq N_3$,
\[
\left(\frac{\delta}{x}\right)^{k_n} 
\le \exp\,\bigl(k_n \log(\delta/x)\bigr)
\le \exp\,\bigl(\varepsilon \log (n) \cdot \log(\delta/x)\bigr)
= n^{-K}.
\]
Consider the decomposition
\[
\sum_{n=1}^\infty 
\mathbb{E}\!\left[\frac{(\tilde{s}^{(n)}_{\max})^{k_n}}{x^{k_n}}\right]
= \sum_{n=1}^{N_3-1} \mathbb{E}\!\left[\frac{(\tilde{s}^{(n)}_{\max})^{k_n}}{x^{k_n}}\right]
+ \sum_{n=N_3}^\infty \mathbb{E}\!\left[\frac{(\tilde{s}^{(n)}_{\max})^{k_n}}{x^{k_n}}\right].
\]
Since the first sum on the right-hand side of the preceding display is finite by the boundedness of the entries $\tilde{z}_{ij}$, and by our previous considerations
\(x^{-k_n}\mathbb{E}[(\tilde{s}^{(n)}_{\max})^{k_n}] \le n^{-K} \) with $K>1$ for all \(n \ge N_3\), the second sum is finite as well. Therefore,
\[
\sum_{n=1}^\infty 
\mathbb{E}\!\left[\frac{(\tilde{s}^{(n)}_{\max})^{k_n}}{x^{k_n}}\right] < \infty.
\]
This completes the proof.
\end{proof}

The following elementary lemma is used only in the proof of Theorem \ref{GDtuned} and allows us to abbreviate several arguments.
\begin{lemma}\label{ext}
	Let $m \in \mathbb{N}$ and $x\geq 0$. Then, $|1-(1-x)^m| \leq \max\{1,|(1-x)^{m-1}|\}|x|m$.	
\end{lemma}

\begin{proof}[Proof of Lemma~\ref{ext}]
	\begin{equation*}
		\begin{gathered}
			|1-(1-x)^m| = |m(1-\zeta)^{m-1}x|\leq \max\{1,|(1-x)|^{m-1}\}\,|x|\,m.
		\end{gathered} 
	\end{equation*}
	The first equality follows from the mean value theorem for some $\zeta$ in the open interval between $0$ and $x$.
	For the inequality, note that $|1-\zeta|^{m-1}$ is bounded by 1, for $\zeta\leq 2$ and increases monotonically for $\zeta>2$. Since $\zeta<x$, we have $|1-\zeta|^{m-1}\leq \max\{1,|1-\zeta|^{m-1}\}\leq \max\{1,|1-x|^{m-1}\}$.
\end{proof}

The next lemma of this section is from the book of \citet{Silversteinbook2010} and is stated without a proof.
\begin{lemma}[\citet{Silversteinbook2010}, Lemma B.26]\label{quadform}
	Let $A$ be an $n \times n $ non-random matrix and $x=(x_1,...,x_n)^\top$ be a random vector of independent entries. Assume that $\mathbb{E}(x_i)=0$, $\mathbb{E}(x_i^2)=1$ and $\mathbb{E}(x_j^l)\leq \nu_l$ for $l\leq 2q$. Then for any $q\geq 1$,
	\begin{align*}
		\mathbb{E}(|x^\top A x-\tr(A)|^q)\leq C_q \bigg((\nu_4 \tr(AA^*))^{q/2}+ \nu_{2q} \tr((AA^*)^{q/2})\bigg),
	\end{align*}
	where $C_q$ is a constant depending on $q$ only.
\end{lemma}

The following lemma is an almost immediate consequence of Lemma~\ref{momqf1} and Lemma~\ref{momqf2}. Since 
$p^{-1}\tr(\Sigma_n^2)\to\int t^2\,dH(t)$ by Assumptions~\ref{A2} and \ref{A4}, the first assertion below, 
together with Assumption~\ref{A4}, extends \citet[Theorem 3.1]{BODNAR2014} to the setting where 
$\gamma_n=o(n^{1/2})$.

\begin{lemma}\label{Bodver}
	Under Assumptions~\ref{A2} and~\ref{A3}, and if $\gamma_n=o(n^{1/2})$, we have
	\begin{align*}
		\tilde{m}_2
		= \frac{1}{p}\tr(\hat{\Sigma}_n^2)
		- \gamma_n\Bigl(\frac{1}{p}\tr(\hat{\Sigma}_n)\Bigr)^2
		= \frac{1}{p}\tr(\Sigma_n^2) + o_{\mathbb{P}}(1),
	\end{align*}
	and for every sequence $\tilde{\beta}=\tilde{\beta}_n$ with $\|\tilde{\beta}\|_2=1$, we have
	\begin{align*}
		\tilde{\beta}^\top \hat{\Sigma}_n^2 \tilde{\beta}
		- \frac{1}{n}\tr(\hat{\Sigma}_n)\,\tilde{\beta}^\top \hat{\Sigma}_n \tilde{\beta}
		= \tilde{\beta}^\top \Sigma_n^2 \tilde{\beta} + o_{\mathbb{P}}(1).
	\end{align*}
\end{lemma}

\begin{proof}
We use the notation
\[\hat{m}_1=\frac{1}{p}\tr\bigl(\hat{\Sigma}_n\bigr),\quad\hat{m}_2=\frac{1}{p}\tr\bigl(\hat{\Sigma}_n^2\bigr), \quad m_1=\frac{1}{p}\tr\bigl(\Sigma_n\bigr) \quad \text{and}\quad m_2=\frac{1}{p}\tr\bigl(\Sigma_n^2\bigr).
\]
Since $\gamma_n=o(n^{1/2})$ and using the second statement of Lemma~\ref{momqf1}, we obtain
\begin{align*}
	\tilde{m}_2
	&= \hat{m}_2-\gamma_n \hat{m}_1^2 \\[0.3em]
	&= \Bigl(\hat{m}_2- m_2-\gamma_n m_1^2\Bigl) + m_2+\gamma_n m_1^2 
	- \gamma_n\Bigl(
	\hat{m}_1-m_1+m_1\Bigr)^2 \\[0.3em]
	&= m_2 -2\gamma_n m_1\Bigl(\hat{m}_1-m_1\Bigr)-\gamma_n\Bigl(\hat{m}_1-m_1\Bigr)^2
	+ o_{\mathbb{P}}(1).
\end{align*}
By the first statement of Lemma~\ref{momqf1} and since 
$m_1 \le C$ by Assumption~\ref{A2}, we have
\begin{align*}
	2 \gamma_n \, m_1
	\Bigl( \hat{m}_1-m_1 \Bigr)
	= O_{\mathbb{P}}\!\left(\frac{\sqrt{p}}{n^{3/2}}\right)\quad\text{and}\quad
	\gamma_n \Bigl(\hat{m}_1-m_1 \Bigr)^2
	= O_{\mathbb{P}}\!\left(\frac{1}{n^2}\right).
\end{align*}
Since $\gamma_n = o(n^{1/2})$, both terms are $o_{\mathbb{P}}(1)$.  
Combining these arguments yields
\[
\tilde{m}_2
= m_2 + o_{\mathbb{P}}(1).
\]
Similarly, we establish the second assertion using Lemma~\ref{momqf1} and
Lemma~\ref{momqf2}. Observe that
\begin{align*}
	\tilde{\beta}^\top \hat{\Sigma}_n^2 \tilde{\beta}
	- \gamma_n\hat{m}_1\,\tilde{\beta}^\top \hat{\Sigma}_n \tilde{\beta}\,
	&=
	\Bigl(
	\tilde{\beta}^\top \hat{\Sigma}_n^2 \tilde{\beta}
	- \tilde{\beta}^\top \Sigma_n^2 \tilde{\beta}
	- \gamma_n\,m_1\tilde{\beta}^\top \Sigma_n \tilde{\beta}
	\Bigr)
	\\
	&\quad
	+ \tilde{\beta}^\top \Sigma_n^2 \tilde{\beta}
	+ \gamma_n\, m_1 \frac{1}{n}\tilde{\beta}^\top \Sigma_n \tilde{\beta}
	\\
	&\quad
	- \gamma_n\!
	\Bigl(\hat{m}_1-m_1+m_1
	\Bigl)
	\left(
	\tilde{\beta}^\top \hat{\Sigma}_n \tilde{\beta}
	- \tilde{\beta}^\top \Sigma_n \tilde{\beta}
	+ \tilde{\beta}^\top \Sigma_n \tilde{\beta}
	\right).
\end{align*}
By Lemma~\ref{momqf1} and Lemma~\ref{momqf2}, the following bounds hold:
\begin{align*}
	\gamma_n\,\Bigl(\hat{m}_1-m_1\Bigr)
	\,\tilde{\beta}^\top \Sigma_n \tilde{\beta}
	&= O_{\mathbb{P}}\!\left(\frac{\gamma_n}{(pn)^{1/2}}\right), \\
	\gamma_n\, m_1\left(
	\tilde{\beta}^\top \hat{\Sigma}_n \tilde{\beta}
	- \tilde{\beta}^\top \Sigma_n \tilde{\beta}
	\right)
	&= O_{\mathbb{P}}\!\left(\frac{\gamma_n}{n^{1/2}}\right), \\
	\gamma_n\,\Bigl(\hat{m}_1-m_1\Bigr)
	\left(
	\tilde{\beta}^\top \hat{\Sigma}_n \tilde{\beta}
	- \tilde{\beta}^\top \Sigma_n \tilde{\beta}
	\right)
	&= O_{\mathbb{P}}\!\left(\frac{\gamma_n}{p^{1/2}n}\right).
\end{align*}
Since $\gamma_n = o(n^{1/2})$, each of the three terms is $o_{\mathbb{P}}(1)$.
Combining these arguments completes the proof.

\end{proof}

%%  MR numbers will be added by VTeX.                      %%
%%                                                         %%
%%  Use \cite{...} to cite references in text.             %%
%%                                                         %%
%%%%%%%%%%%%%%%%%%%%%%%%%%%%%%%%%%%%%%%%%%%%%%%%%%%%%%%%%%%%%

%% if your bibliography is in bibtex format, uncomment commands:
%%\bibliographystyle{imsart-number} % Style BST file (imsart-number.bst or imsart-nameyear.bst)
% Bibliography file (usually '*.bib')

\end{document}